\documentclass[letterpaper, 11pt, oneside]{book}

    \usepackage[left=108pt, right=74pt, top=108pt, bottom=81pt, dvips, pdftex]{geometry}
    \usepackage[T1]{fontenc}

    \usepackage[hang,flushmargin]{footmisc}

    \usepackage{fancyhdr} 
    \renewcommand{\bibname}{References}

    \usepackage{amsmath,amsthm, amsfonts,amssymb} 
    \usepackage{setspace} 
    \usepackage{enumerate}
    \usepackage{color}
    \usepackage[linktocpage,bookmarksopen,bookmarksnumbered,
                pdftitle={Multivariate spectral multipliers},
                pdfauthor={Blazej Wrobel},%
                pdfsubject={Ph.D. Thesis},%
                pdfkeywords={Multivariate spectral multipliers}]{hyperref}
    \usepackage[titletoc]{appendix}

\newtheorem{thm}{Theorem}[section]
\newtheorem{thmc}{Theorem}[chapter]
\newtheorem{lem}[thm]{Lemma}
\newtheorem{lemc}[thmc]{Lemma}

\newtheorem{cor}[thm]{Corollary}
\newtheorem{corc}[thmc]{Corollary}
\newtheorem{pro}[thm]{Proposition}

\theoremstyle{remark}
\newtheorem*{remark}{Remark}
\newtheorem*{remark1}{Remark 1}
\newtheorem*{remark2}{Remark 2}

\theoremstyle{definition}
\newtheorem{defi}[thm]{Definition}

    \numberwithin{equation}{section}
    \numberwithin{figure}{section}

\newcommand{\mF}{\mathcal{F}}
\newcommand{\mS}{\mathcal{S}}
\newcommand{\Hi}{\mathcal{H}}
\newcommand{\mC}{\mathcal{C}}
\newcommand{\mB}{\mathcal{B}}
\newcommand{\mP}{\mathcal{P}}
\newcommand{\Rde}{\mathbb{R}_{\varepsilon}^d}
\newcommand{\Rdp}{\mathbb{R}_+^d}
\newcommand{\Rdpm}{\mathbb{R}_+^{d-1}}
\newcommand{\Rlp}{\mathbb{R}_+^l}
\newcommand{\la}{\lambda}
\newcommand{\M}{\mathcal{M}}
\newcommand{\F}{\mathcal{F}}
\newcommand{\mL}{\mathcal{L}}

\newcommand{\pp}{p\rightarrow p}

\newcommand{\TT}{t}
\newcommand{\ta}{\tau}
\newcommand{\Ija}{I_v^{\alpha,\beta}}
\newcommand{\Cze}{\log_+(|v|)(1+|v|)^{2\alpha+2\beta+\frac{11}{2}}}
\newcommand{\Pj}{\mathcal{P}}

\newcommand{\Ha}{\mathcal{H}}
\newcommand{\Hla}{\mathcal{H}_{\alpha}}
\newcommand{\Hlap}{\mathcal{H}_{\alpha+1}}

\newcommand{\Ree}{\textrm{Re}}
\newcommand{\Imm}{\textrm{Im}}

\newcommand{\T}{\mathcal{T}}

\newcommand{\ki}{\|\kappa\|_{\infty}}

\newcommand{\Leb}{\Lambda}
\newcommand{\mW}{\mathcal{W}}
\newcommand{\mM}{\mathcal{M}}
\newcommand{\mD}{\mathcal{D}}
\newcommand{\ph}{\varphi}

\newcommand{\tm}{\widetilde{m}}
\newcommand{\maxlogv}{\log_+(|v|)}

\newcommand{\nH}{\tilde{H}}
\newcommand{\bnH}{\tilde{\bf{H}}}
\newcommand{\pst}{\phi_p^{\ast}}
\newcommand{\nL}{\tilde{L}}
\newcommand{\bnL}{\tilde{\bf{L}}}
    \newcommand{\n}{\vspace{12pt}} 

    \newcommand{\newchapter}[3] 
	{                           
        \chapter[#2]{#3}
        \chaptermark{#1}
        \thispagestyle{myheadings}
	}

\DeclareMathOperator{\Real}{Re}
\DeclareMathOperator{\Ima}{Im}
\DeclareMathOperator{\Arg}{Arg}
\DeclareMathOperator{\Dom}{Dom}
\DeclareMathOperator{\Ran}{Ran}
\DeclareMathOperator{\supp}{supp}

\begin{document}


    \pagenumbering{roman}
    \pagestyle{plain}

    %
    %


     ~\vspace{-0.75in} 
    \begin{center}

\begin{small}
    \MakeUppercase{Scuola Normale Superiore, Pisa}\\ \vspace{0.1in}
    \MakeUppercase{Uniwersytet Wroc\l awski, Instytut Matematyczny}\\
\end{small}
    \vspace{1in}
      \begin{Huge}
           Multivariate Spectral Multipliers
        \end{Huge}\\\n
        PhD thesis by\\\n
        \begin{Large}  { B\l a\.{z}ej Wr\'obel}\end{Large}\\

        \rule{4in}{1pt}\\
        Advisors:\\
        \begin{Large}~prof.\ Fulvio Ricci (Scuola Normale Superiore, Pisa)\\ \vspace{0.1in}
        ~prof.\ dr hab.\ Krzysztof Stempak (Politechnika Wroc\l awska)\end{Large}
        \\\n\n

    \end{center}

    \newpage

    %
    %


    ~\\[7.75in] 
    \centerline{
                \copyright\ B\l a\.{z}ej Wr\'obel,
                            2014. All rights reserved.
               }
    \thispagestyle{empty}
    \addtocounter{page}{-1}

    \newpage

    %
    %



    \doublespacing

    %
    %

    \tableofcontents

\section*{List of symbols}
\begin{tabular}{rl}
        $\mathbb{Z},  \mathbb{Q},\mathbb{R}, \mathbb{C}$ & the sets of integer, rational, real and complex numbers\\
        $\mathbb{N}$ & the set of positive integers\\
        $\mathbb{N}_0$ & the set of non-negative integers\\
        $\mathbb{N}^d$ & the $d$-fold Cartesian product of $\mathbb{N}$\\
        $\mathbb{N}^d_0$ & the $d$-fold Cartesian product of $\mathbb{N}_0$\\
        $\mathbb{R}_+$ & the set of positive real numbers\\
        $\mathbb{R}_+^d$ & the $d$-fold Cartesian product of $\mathbb{R}_+$\\
        $\mathbb{T}$ & the $1$-dimensional torus\\
        $\mathbb{T}^d$ & the $d$-dimensional torus\\
         $\Arg(z)$ & the principal argument of $z\in\mathbb{C}$\\
          $\partial^{n}_jf$ & the $n$-th partial derivative of $f(x_1,\ldots,x_d)$ with respect to $x_j$\\
          $\partial^{\alpha}f$ & $\partial^{\alpha_1}_{1}\cdots \partial^{\alpha_d}_d f$\\
          $\nabla f$ & the gradient of $f,$ i.e.\ the vector $(\partial_1 f,\ldots, \partial_d f)$\\
          $f*g$ & the convolution of $f$ and $g$\\
    $G_1\cong G_2$ & a group isomorphism between $G_1$ and $G_2$ \\
$V^c$ & the complement of a set $V\subseteq X$ i.e.\ the set $X\setminus V$ \\
         $\mB(X)$ & the $\sigma$-algebra of all Borel sets in $X$\\
         $l^p(X)$ & the space of $p$-th power summable sequences on $X$\\
        $L^p(X,\nu)$ & the Lebesgue space over the measure space $(X,\nu)$\\
        $L^{p,\infty}(X,\nu)$ & the Lorentz space $L^{p,\infty}$ over $(X,\nu),$ by convention $L^{\infty,\infty}=L^{\infty}$\\
        $C_c(X)$ & the space of continuous compactly supported functions on $X$\\
        $C^1(X)$ & the space of functions with continuous first derivative\\
        $\mS(\mathbb{R}^d)$ & the Schwartz space on $\mathbb{R}^d$\\
        $C_c^{\infty}(X)$ & the space of smooth compactly supported functions on $X$\\
        $H^{\infty}(U)$ & the space of bounded holomorphic functions on an open $U\subset \mathbb{C}^d$
        \end{tabular}
    \newpage

    %
    %

    %

    ~\vspace{-1in} 
    \begin{flushright}
        \singlespacing
        B\l a\.{z}ej Wr\'obel\\
        ~2014
    \end{flushright}

    \centerline{\large Multivariate spectral multipliers}

    \centerline{\textbf{\underline{Abstract}}}

    This thesis is devoted to the study of multivariate (joint) spectral multipliers for systems of strongly commuting non-negative self-adjoint operators, $L_1,\ldots,L_d,$ on $L^2(X,\nu),$ where $(X,\nu)$ is a measure space. By strong commutativity we mean that the operators $L_r,$ $r=1,\ldots,d,$ admit a joint spectral resolution $E(\la).$ In that case, for a bounded function $m\colon [0,\infty)^d\to \mathbb{C},$ the multiplier operator $m(L)$ is defined on $L^2(X,\nu)$ by
    $$m(L)=\int_{[0,\infty)^d}m(\la)dE(\la).$$

    By spectral theory, $m(L)$ is then bounded on $L^2(X,\nu).$ The purpose of the dissertation is to investigate under which assumptions on the multiplier function $m$ it is possible to extend $m(L)$ to a bounded operator on $L^p(X,\nu),$ $1<p<\infty.$

    The crucial assumption we make is the $L^p(X,\nu),$ $1\leq p\leq \infty,$ contractivity of the heat semigroups corresponding to the operators $L_r,$ $r=1,\ldots,d.$ Under this assumption we generalize the results of \cite{cit:Me} to systems of strongly commuting operators. As an application we derive various multivariate multiplier theorems for particular systems of operators acting on separate variables. These include e.g.\ Ornstein-Uhlenbeck, Hermite, Laguerre, Bessel, Jacobi, and Dunkl operators. In some particular cases, we obtain presumably sharp results. Additionally, we demonstrate how a (bounded) holomorphic functional calculus for a pair of commuting operators, is useful in the study of dimension free boundedness of various Riesz transforms.



    \newpage

    %
    %

    \chapter*{\vspace{-1.5in}Acknowledgments}

First of all I would like to express my sincerest gratitude to my advisors: prof.\ Fulvio Ricci and prof.\ Krzysztof Stempak, for all their help and encouragement during my PhD studies and the preparation of the thesis.

I am extremely grateful to prof.\ Stempak, for his thoughtful guidance, invaluable advice, and countless helpful suggestions throughout my master and PhD studies. I also thank him for the freedom I have been given with regard to my research.

My deepest thanks go to prof.\ Ricci, for posing and discussing with me many interesting problems, for numerous conversations about the dissertation, and for his wise counsel. I am also grateful for all the kind hospitality he has shown me during the time I spent in Italy.

I am very much indebted to prof.\ Jacek Dziuba\'nski, for being my research supervisor at the first year of the PhD, and for valuable conversations, advice, and cooperation.

 I would also like to thank dr Marcin Preisner, for fruitful collaboration, as well as Gian Maria Dall'Ara, prof.\ Pawe{\l } G\l owacki, prof.\ Stefano Meda, dr hab.\ Adam Nowak, Tomasz Z.\ Szarek, Dario Trevisan, dr hab.\ Micha{\l } Wojciechowski, and dr hab.\ Jacek Zienkiewicz, for several useful discussions. I am grateful to prof.\ Leszek Skrzypczak, for his hospitality during the semester I spent in Pozna\'n, to my Italian peers, for their friendliness during my time in Pisa, to prof.\ Piotr Biler, prof.\ Ewa Damek, and Mrs.\ Elisabetta Terzuoli, for their help in preparing the co-tutelle agreement, and to dr Mariusz Mirek for a certain idea.

 Last, yet by no means least, special thanks go to my family. I am most grateful to my parents Irena and Tadeusz for giving me great opportunities to study and for their continuous support. Finally, I thank my beloved Ma\l gosia, for her patience during the (often to long) times of my absence.\vspace{0.01in}
 \begin{spacing}{0.2}
\begin{center}   \rule{400pt}{1pt}\end{center}
\end{spacing}
I gratefully acknowledge the support from the PhD programme '\'Srodowiskowe Studia Doktoranckie z Nauk Matematycznych' funded by the European Social Fund, and the National Science Center (NCN) 'Preludium' grant 2011/01/N/ST1/01785.

    \newpage

    %
    %

    \pagestyle{fancy}
    \pagenumbering{arabic}

    %
    %

    \newchapter{Introduction}{Introduction}{Introduction}
    \label{chap:Intro}


        %

        \section{Motivation and informal overview}
        \label{chap:Intro,sec:Motivation}
        The classical examples of multiplier operators are the ones connected with the Fourier series and the Fourier transform. In the context of the Fourier transform, an operator $T_m$ is called a multiplier operator (or multiplier for short), if it is of the form $$\mathcal{F}(T_m f)=m\,\mathcal{F}f;\qquad \mathcal{F}f(x)=\int_{\mathbb{R}^d} e^{-i \langle x, y\rangle}f(y)\,dy.$$ Initially, we only assume that the function $m$ in the definition (which is called a multiplier function or simply a multiplier) is measurable and bounded. Due to Plancherel's identity this implies of course the boundedness of $T_m$ on $L^2(\mathbb{R}^d,dx).$ Much more complicated is one of the classical problems of harmonic analysis to study the boundedness of $T_m$ on $L^p(\mathbb{R}^d,dx),$ $1\leq p\leq \infty,$ $p\neq 2.$ In fact, except for the cases $p=1$ and $p=2,$ even in this classical setting, there is no complete description of the multiplier spaces $$\mM_p=\{m\colon m \textrm{ is bounded and Borel measurable, } T_m\textrm{ is bounded on $L^p(\mathbb{R}^d,dx)$}\}.$$

        On the other hand there are many results which give sufficient conditions on $m,$ under which $m\in\mM_p.$ These are called multiplier theorems. The two most famous multiplier theorems are associated with the names of H\"ormander and Marcinkiewicz.

        We first recall the H\"{o}rmander multiplier theorem \cite[Theorem 2.5]{Horm1} (see also \cite[Theorem 7.9.5, p.243]{Horm2}).
        \begin{thm}
        \label{thm:Hormmult}
        Assume $m\colon\mathbb{R}^d\to \mathbb{C}$ is a bounded Borel measurable function such that for every $\gamma$ satisfying $|\gamma|\leq \lfloor d/2\rfloor+1$ we have
        \begin{equation}\label{chap:Intro,sec:Motivation,eq:Horm}
\sup_{0<R<\infty}R^{-d}\int_{R\leq|\xi|\leq 2R}|R^{|\gamma|}\partial^{\gamma}m(\xi)|^2\,d\xi<\infty.
\end{equation}
Then the operator $T_m$ is of weak type $(1,1)$ and bounded on $L^p(\mathbb{R}^d,dx),$ $1<p<\infty.$
\end{thm}
\begin{remark}
Condition \eqref{chap:Intro,sec:Motivation,eq:Horm} is implied by the Mikhlin (Mikhlin-H\"ormander) condition
\begin{equation*}
\sup_{\xi\in\mathbb{R}^d}|\la|^{|\gamma|}|\partial^{\gamma}m(\xi)|<\infty,\qquad |\gamma|\leq \lceil d/2\rceil.
\end{equation*}
\end{remark}
Observe that \eqref{chap:Intro,sec:Motivation,eq:Horm} is rotation invariant. Thus, one can guess that the multiplier functions admitted by the H\"ormander multiplier theorem should be somewhat close to being radial. Indeed, it is clear that we can take $m(\xi)$ to be e.g.\
\begin{equation}
\label{chap:Intro,sec:Motivation,eq:exaHorm}
\begin{split}
&  e^{-t|\xi|^2},\,t>0\textrm{ (heat semigroup)},\qquad |\xi|^{iv},\,v\in\mathbb{R}\textrm{ (imaginary powers)},\\
 &\frac{\xi_r}{|\xi|},\,r=1,\ldots,d \textrm{ (Riesz transforms)},\qquad \chi_{|\xi|\leq 1}(1-|\xi|^2)^{\delta} \textrm{ (Bochner-Riesz means)};\end{split}\end{equation} in the case of Bochner-Riesz means we take $\delta>0$ large enough. Moreover, Theorem \ref{thm:Hormmult} also gives a weak type $(1,1)$ result. On the other hand, the H\"ormander multiplier theorem does not apply if the function $m$ is a product function or a function which depends only on some of the variables $\xi_r,$ $r=1,\ldots,d.$

The other famous multiplier theorem is associated with the name of Marcinkiewicz. In modern terminology, the Marcinkiewicz multiplier theorem for the Fourier transform may be stated in the following way.
 \begin{thm}
        \label{thm:Marmult}
Assume  $m\colon\mathbb{R}^d\to \mathbb{C}$ is a bounded Borel measurable function such that for $\gamma_r=0,1,$ $r=1,\ldots,d,$ we have
\begin{equation}\label{chap:Intro,sec:Motivation,eq:Marcon}
(R_1\cdots R_d)^{-1}\int_{R_1<|\xi_1|<2R_1}\ldots\int_{R_d<|\xi_d|<2R_d}|R_1\xi_1|^{2\gamma_1}\cdots |R_d\xi_d|^{2\gamma_d}|\partial^{\gamma}m(\xi)|^2\,d\xi\leq C_{\gamma},
\end{equation}
uniformly in $R=(R_1,\ldots,R_d)\in \Rdp.$ Then $T_m$ is bounded on $L^p(\mathbb{R}^d,dx),$ $1<p<\infty.$
\end{thm}
\begin{remark}As with the H\"ormander multiplier theorem, condition \eqref{chap:Intro,sec:Motivation,eq:Marcon} is implied by the Mikhlin (Mikhlin-Marcinkiewicz) condition
\begin{equation*}
\sup_{\xi\in\mathbb{R}^d}|\xi_1|^{\gamma_1}\cdots|\xi_d|^{\gamma_d}|\partial^{\gamma}m(\xi)|<\infty,\qquad \gamma_r=0,1,\quad r=1,\ldots,d.
\end{equation*}
\end{remark}
Note that, for $d>1,$ neither Theorem \ref{thm:Hormmult} is stronger than Theorem \ref{thm:Marmult} nor vice versa. Indeed, on the one hand, the H\"ormander multiplier is better suited for radial multipliers, in which case it may require less smoothness than the Marcinkiewicz multiplier theorem. For example, for some values of $\delta>0,$ Theorem \ref{thm:Hormmult} is applicable to the Bochner-Riesz means given in \eqref{chap:Intro,sec:Motivation,eq:exaHorm}, whereas Theorem \ref{thm:Marmult} is not. On the other hand, the Marcinkiewicz multiplier theorem is valid for product multipliers, which is not true in general for the H\"ormander multiplier theorem. For instance, the product imaginary powers $m_{u}(\xi)=|\xi_1|^{iu_1}\cdots|\xi_d|^{iu_d},$ $u\in \mathbb{R}^d,$ satisfy \eqref{chap:Intro,sec:Motivation,eq:Marcon} but do not satisfy \eqref{chap:Intro,sec:Motivation,eq:Horm}. Moreover, Theorem \ref{thm:Marmult} is applicable in all the other cases listed in \eqref{chap:Intro,sec:Motivation,eq:exaHorm}, however, it does not give weak type $(1,1)$ results.

Let us now look at the Fourier multipliers from a slightly different perspective. If $m\colon \mathbb{R}^d\to\mathbb{C}$ is a Borel measurable function, the formal identity
$$i\partial_r f=\mF^{-1}(\xi_r\, \mF) f,\qquad r=1,\ldots,d,$$
leads to
$$m(i\partial)f=m(i\partial_1,\ldots,i\partial_d)f=\mF^{-1}(m\mF) f.$$
The above says that Fourier multipliers are precisely the 'joint' multipliers for the system of partial derivatives $i\partial=(i\partial_1,\ldots,i\partial_d).$ Put in other words, Fourier multipliers  provide a way to define a joint functional calculus for the system $i\partial.$ Thus, both Theorems \ref{thm:Hormmult} and \ref{thm:Marmult} give in fact sufficient conditions for the $L^p(\mathbb{R}^d,dx)$ boudedness of joint multipliers for the system $i\partial.$ Note that $i\partial$ consists of symmetric (in fact essentially self-adjoint) and pairwise commuting operators in $L^2(\mathbb{R}^d,dx).$

At this point a question arises: can we replace the system of partial derivatives by some other system of commuting operators? An attempt to answer this question is the main theme of this dissertation. We study joint (in other words: multivariate) multipliers for other systems of pairwise (strongly) commuting, self-adjoint operators on some space $L^2(X,\nu).$ Due to the use of the technique of imaginary powers, we often assume that the considered operators are non-negative.

If $L=(L_1,\ldots,L_d)$ is a system of self-adjoint non-negative operators on $L^2(X,\nu),$ such that their spectral projections commute pairwise, then the multivariate spectral theorem postulates the existence of the joint spectral measure $E$ for the system $L,$ defined by the identity
$$L_rf=\int_{[0,\infty)^d}\la_r dE(\la),\qquad \la=(\la_1,\ldots,\la_d),$$
for each $r=1,\ldots,d.$ Then, for a bounded Borel measurable function $m$ on $[0,\infty)^d,$ we define the multiplier operator $m(L)$ on $L^2(X,\nu)$ by
$$m(L)=m(L_1,\ldots,L_d)=\int_{[0,\infty)^d}m(\la) dE(\la).$$
From the multivariate spectral theorem it follows that this operator is bounded on $L^2(X,\nu).$ The problems we study in this dissertation concern the boundedness of $m(L)$ on some other $L^p(X,\nu)$ spaces, $1<p<\infty.$ We are trying to understand, under which assumptions on $L,$ can one obtain Marcinkiewicz type or H\"ormander type multiplier theorems.

Throughout a most part of the thesis we work in a rather broad generality. We are pursuing multiplier theorems that can be proved by only assuming something on the operators
$L_r,$ $r=1,\ldots,d,$ without reference to the particular geometric structure of the space $(X,\nu).$ Whenever we work in a general setting (e.g.\ in Chapter \ref{Chap:General}, Section \ref{chap:Ck,sec:genMar}, Section \ref{chap:Hinf,sec:genMarHinf} and Section \ref{chap:CkHinf,sec:genMarHinfCk}), the conditions we assume on the multiplier function $m$ and the multiplier theorems we obtain, shall be mostly of Marcinkiewicz type. The H\"ormander type multiplier theorems we obtain,  are placed in the context of some specific operators, see Section \ref{chap:Ck,sec:HorHan} and Section \ref{chap:CkHinf,sec:OA}.

The results we present in this dissertation can be also seen from the following perspective. We study the relation between functional calculi for each of the operators $L_r,$ $r=1,\ldots,d,$ and a joint functional calculus for the whole system $L.$ As we shall see, from this point of view, Marcinkiewicz type conditions seem more natural than H\"ormander type.

        \section{Summary of known results}
        \label{chap:Intro,sec:Summary}
        This section is an overview of existing multiplier results for systems of commuting operators.

        Before proceeding further a comment on the spectral multiplier problem for a single self-adjoint operator is in order. This topic has been already explored in a great variety of contexts, see e.g.\ \cite{s1}, \cite{s2}, \cite{s3}, \cite{CRW}, \cite{Hanonsemi}, \cite{s4}, \cite{s5}, \cite{s6}, \cite{s7}, \cite{Horm1}, \cite{vectvalMM}, \cite{cit:Me}, \cite{Mikhlin}, \cite{topics}, and the references within; we shall not give a detailed historical exposition. Let us only briefly describe these results for a single operator which pertain to the material presented in our thesis.

        Results contained in \cite{topics} are among the first general multiplier theorems. Under the assumption that the heat semigroup of a general non-negative self adjoint operator $L$ is Markovian, Stein proved that Laplace transform type multipliers defined by $m(\la)=\la\int_0^{\infty}e^{-\la t}\kappa(t)\,dt,$ were $\|\kappa\|_{L^{\infty}}<\infty,$ give rise to operators $m(L),$ which are bounded on all $L^p,$ $1<p<\infty.$ This theorem was then improved by Coifman, Rochberg and Weiss, see \cite{CRW}. Their approach relies on the so called transference methods and allows one to reduce the assumption in Stein's theorem to the requirement that $L$ generates a (not necessarily Markovian) contraction semigroup on all $L^p,$ $1\leq p\leq \infty.$ Note that Laplace transform type multipliers are holomorphic functions in the right complex half-plane $S_{\pi/2},$ that are uniformly bounded in every proper subsector $S_{\varphi}$ of $S_{\pi/2}.$

        In \cite{Hanonsemi}, by enhancing the transference methods, Cowling proved that, in order to obtain the boundedness of $m(L)$ on $L^p,$ for some $1<p<\infty,$ it is enough to assume that $m$ is holomorphic and uniformly bounded in a certain proper subsector $S_{\varphi_p}\subset S_{\pi/2},$ enclosing the positive real half-line. Moreover, the sector $S_{\varphi_p}$ approaches $(0,\infty)$, as $p\to 2.$ Using Mellin transform techniques, in \cite{cit:Me} Meda was able to further generalize Cowling's theorem. In particular, assuming a polynomial growth in $v$ for the $L^p,$ $1<p<\infty,$ norms of the imaginary power operators $L^{iv},$ he proved a general Marcinkiewicz type multiplier theorem, see \cite[Theorem 4]{cit:Me}.

        Many of the results and techniques presented in this thesis, see e.g.\ Chapter \ref{Chap:General} and Section \ref{chap:Ck,sec:genMar}, are multi-dimensional generalizations of Meda's \cite{cit:Me}.

        As for the joint spectral multipliers for a system of commuting self-adjoint operators there are relatively fewer results. Below we describe some of them in more details.

        The case of the system of partial derivatives on $L^2(\mathbb{R}^d,\,dx)$ has been already presented as a model example in the previous section. We should however remember that one of the first multivariate multiplier results, the Marcinkiewicz multiplier theorem, was about double Fourier series. From our point of view this result concerns the boundedness of joint spectral multipliers for the system of partial derivatives on $L^2([0,2\pi)^{2},dx).$ Indeed, for a bounded double sequence $\{m(j)\}_{j\in\mathbb{Z}^2},$ the multiplier operator for the Fourier series given by
        \begin{equation}
        \label{chap:Intro,sec:Summary,eq:Foursermult}
        \T_m(f)(x)=\frac{1}{4\pi^2}\sum_{j\in\mathbb{Z}^2}m(j)\langle f, e^{-i \langle j , \cdot \rangle}\rangle_{L^2([0,2\pi)^{2},\,dy)} e^{i \langle j , x \rangle},\qquad f\in L^2([0,2\pi)^{2},dx),\end{equation}
        coincides with $m(i\partial_1,i\partial_2)$ as an operator on $L^2([0,2\pi)^2,\,dx).$

        The Marcinkiewicz multiplier theorem, see \cite[Th\'eor\`{e}me 2]{Marorg}, in its original form (yet written to fit into modern terminology) states the following.

        \begin{thm}
        \label{thm:Marmultser}
Let  $\{m(j)\}_{j\in\mathbb{Z}^2}$ be a bounded complex valued double sequence. For $k\in\mathbb{N}_0^d$ set $I_k=I^1_{k_1}\times I^2_{k_2},$ where
$I^r_{k_r}=\{j_r\in\mathbb{Z}\colon 2^{k_r-1}\leq |j_r|<2^{k_r}\},$ $r=1,2.$ If
\begin{align}
&\sup_{k\in\mathbb{N}_0^2}\sum_{j\in I_{k}}|m(j)-m(j_1+1,j_2)-m(j_1,j_2+1)+m(j+{\bf 1})|<\infty, \label{chap:Intro,sec:Summary,eq:conMar1}\\
&\sup_{k_1\in\mathbb{N}_0,j_2\in\mathbb{Z}}\sum_{j_1\in I^1_{k_1}}|m(j_1+1,j_2)-m(j)|<\infty,\label{chap:Intro,sec:Summary,eq:conMar2}\\ &\sup_{k_2\in\mathbb{N}_0,j_1\in\mathbb{Z}}\sum_{j_2\in I^2_{k_2}}|m(j_1,j_2+1)-m(j)|<\infty,\label{chap:Intro,sec:Summary,eq:conMar3}
\end{align}
then $\T_m$ given by \eqref{chap:Intro,sec:Summary,eq:Foursermult} is bounded on all $L^p([0,2\pi)^2,dx)$ spaces, $1<p<\infty.$
\end{thm}
\begin{remark}Observe that the conditions \eqref{chap:Intro,sec:Summary,eq:conMar1}, \eqref{chap:Intro,sec:Summary,eq:conMar2} and \eqref{chap:Intro,sec:Summary,eq:conMar3} are discrete $l^1$ versions of \eqref{chap:Intro,sec:Motivation,eq:Marcon} for $\partial_1\partial_2 m,$ $\partial_1 m$ and $\partial_2 m,$ respectively. In fact these conditions are equivalent to the following: there exists an extension $\tilde{m}$ of $m$ to $\mathbb{R}^2_+,$ such that $\tilde{m}$ satisfies the $L^1$ variant of \eqref{chap:Intro,sec:Motivation,eq:Marcon} (i.e.\ with $2\gamma_r$ and $|\partial^{\gamma}\tilde{m}(\xi)|^2$ replaced by $\gamma_r$ and $|\partial^{\gamma}\tilde{m}(\xi)|,$ respectively).
\end{remark}

One of the first general cases of commuting operators, investigated in the context of a joint functional calculus, was that of sectorial operators (see \cite[Definition 1.1]{LanLanMer}). In \cite{Al} and \cite{AlFrMc} Albrecht et al.\ studied the existence of an $H^{\infty}$  joint functional calculus for a pair $L=(L_1,L_2)$ of commuting sectorial operators defined on a Banach space $B$. From their results it follows that, if each $L_r,$ $r=1,2,$ has an $H^{\infty}$ functional calculus in a certain sector (containing the positive real half-line) in the complex half-plane\footnote{That is, $\|m(L_r)\|\leq C \|m\|_{H^{\infty}(S_{\varphi_r})},$ where $S_{\varphi_r},$ $r=1,2,$ are these sectors, $\|\cdot\|$ is the operator norm on $B$, while $\|m\|_{H^{\infty}(U)}$ denotes the supremum norm on the space of holomorphic functions on $U$.}, then the system $L$ has an $H^{\infty}$ joint functional calculus in a product of slightly bigger sectors. This fact, for $B=L^p,$ $1<p<\infty,$ is employed in several places in the thesis, see e.g.\ Chapter \ref{chap:Riesz}. For some other results concerning holomorphic functional calculus for a pair of sectorial operators see \cite{LanLanMer}.

Marcinkiewicz type multivariate multiplier theorems for specific commuting operators (i.e sublaplacians and central derivatives) on the Heisenberg (and related) groups were investigated in \cite{Fr1}, \cite{Fr2}, \cite{Fr3}, \cite{Mu:RiSt1} and \cite{Mu:RiSt2} among others.

In \cite{Sik} Sikora proved a H\"ormander type multiplier theorem for a pair of non-negative self-adjoint operators $L_r$ acting on $L^2(X_r,\mu_r),$ $r=1,2,$ i.e.\ on separate variables\footnote{Then, the tensor product operators $L_1\otimes I$ and $I\otimes L_2$ commute strongly on $L^2(X_1\times X_2,\mu_1\otimes \mu_2)$}. In that article he assumes that the kernels of the heat semigroup operators $e^{-t_rL_r},$ $t_r>0,$ $r=1,2,$ satisfy certain Gaussian bounds and that the underlying measures $\mu_r$ are doubling. Theorem \ref{thm:genMarCk} from Section \ref{chap:Ck,sec:genMar} is, to some extent, an answer to a question posed in \cite[Remark 4]{Sik}.

The PhD thesis of Martini, \cite{Martini_Phd} (see also \cite{Martini_JFA} and \cite{Martini_Annales}), is a treatise of the subject of joint spectral multipliers for general Lie groups of polynomial growth. He proves various Marcinkiewicz type and H\"ormander type multiplier theorems, mostly with sharp smoothness thresholds.

         \section{Setting}
        \label{chap:Intro,sec:Setting} Let us now define rigorously the main objects of our interest. We study joint (or multivariate) spectral multipliers for systems of strongly commuting non-negative self-adjoint operators $L=(L_1,\ldots,L_d),$ $d\geq 1,$ on $L^2(X,\nu)$ for some $\sigma$-finite measure space $(X,\nu).$ Our methods are also applicable for $d=1,$ in which case we merely consider a single operator. By strong commutativity we mean that the spectral projections $E_{L_r}$ of the operators $L_r,$ $r=1,\ldots,d,$ commute pairwise. In this case one can define the joint spectral measure $E(\la)$ on $[0,\infty)^d,$ which, for every $r=1,\ldots,d,$ satisfies
        $$L_r=\int_{[0,\infty)^d} \la_r\, dE(\la)=\int_{[0,\infty)}\la_r\, dE_{L_r}(\la_r),\qquad \la=(\la_1,\ldots,\la_d).$$ Then, for a Borel measurable function $m,$ we define the spectral multiplier for the system $L$ by
        \begin{equation}
        \label{chap:Intro,sec:Setting,eq:mdef}
        m(L)=m(L_1,\ldots,L_d)=\int_{[0,\infty)^d} m(\la)dE(\la)=\int_{\sigma(L)}m(\la)dE(\la),
        \end{equation}
        on the natural domain
        \begin{equation}
        \label{chap:Intro,sec:Setting,eq:mdom}
        \Dom(m(L))=\bigg\{f\in L^2(X,\nu)\colon \int_{[0,\infty)^d}|m(\la)|^2dE_{f,f}(\la)<\infty\bigg\}.\end{equation}
        The symbol $\sigma(L)$ in \eqref{chap:Intro,sec:Setting,eq:mdef} denotes the joint spectrum of the system $L,$ i.e.\ the support of the joint spectral measure $E(\la).$ The reader is kindly referred to Appendix \ref{chap:App,sec:joint} for more details on the joint spectral measure and integration, and the multivariate spectral theorem.

        The particular choices of $m,$ $m=m_{r,t_r},$ $t_r>0,$ $r=1,\ldots,d,$  with $m_{r,t_r}(\la)=e^{-t_r\la_r},$ lead to the corresponding heat semigroups $\{\exp(-t_rL_r)\}_{t_r>0},$ $r=1,\ldots,d,$ being contractions on $L^2(X,\nu).$

        In this thesis we always assume that the heat semigroups for the operators $L_r,$ $r=1,\ldots,d,$ satisfy the contraction property
        \begin{equation*}
        \tag{CTR}
        \label{chap:Intro,sec:Setting,eq:contra}
        \|\exp(-t_rL_r)f\|_{L^p(X,\nu)}\leq \|f\|_{L^p(X,\nu)},\qquad f\in L^2(X,\nu)\cap L^p(X,\nu),
        \end{equation*}
        for all $t_r>0$ and $p\in[1,\infty].$ In many places, what we really need is the fact that each of the operators $L_r,$ $r=1,\ldots,d,$ has an $H^{\infty}$ functional calculus in every proper subsector of the right half-plane $S_{\pi/2}.$ This property is shared by the operators satisfying \eqref{chap:Intro,sec:Setting,eq:contra}, see \cite[Theorem 3]{Hanonsemi}. Restricting to operators generating $L^p$ contraction semigroups makes the presentation clearer. Moreover, all the exemplary operators we consider do satisfy \eqref{chap:Intro,sec:Setting,eq:contra}.

        The other assumption we often impose is the atomlessness condition \begin{equation*}\tag{ATL}\label{chap:Intro,sec:Setting,eq:noatomatzero}E_{L_r}(\{0\})=0,\qquad r=1,\ldots,d.\end{equation*}
        In particular, we always impose \eqref{chap:Intro,sec:Setting,eq:noatomatzero} when considering general operators, for instance in Chapter \ref{Chap:General}, Section \ref{chap:Ck,sec:genMar}, Section \ref{chap:Hinf,sec:genMarHinf} and Section \ref{chap:CkHinf,sec:genMarHinfCk}. Note that, assuming \eqref{chap:Intro,sec:Setting,eq:noatomatzero}, we can rewrite \eqref{chap:Intro,sec:Setting,eq:mdef} as $m(L)=\int_{\Rdp}m(\la)dE(\la),$ where $\Rdp=(0,\infty)^d.$ The assumption \eqref{chap:Intro,sec:Setting,eq:noatomatzero} is a purely technical one, making it easier to state our main results.

        \section{Notation and terminology}
        \label{chap:Intro,sec:Notation}
        For a vector of angles $\varphi=(\varphi_1,\ldots,\varphi_d)\in (0,\pi/2]^d,$ $r=1,\ldots,d,$ denote by ${\bf S}_{\varphi}$ the polysector (contained in the product of the right complex half-planes) $${\bf S}_{\varphi}=\{z\in \mathbb{C}^d\colon z_r\neq0,\quad |\Arg(z_r)|<\varphi_r,\quad r=1,\ldots,d\}.$$ If, instead of a vector of angles, we just consider a single angle $\varphi\in(0,\pi/2),$ then, slightly abusing the notation, we continue writing ${\bf S}_{\varphi}$ instead of ${\bf S}_{(\varphi,\ldots,\varphi)}.$ In particular, ${\bf S}_{\pi/2}$ means the $d$-fold product of the right half-planes. In the case when only one complex variable is considered we write $S_{\varphi},$ $S_{\pi/2}$ instead of ${\bf S}_{\varphi},$ ${\bf S}_{\pi/2},$ respectively. For an arbitrary $V\subseteq \mathbb{C}^d$ the symbol $\overline{V}$ denotes the closure of $V.$

         If $U$ is an open subset of $\mathbb{C}^d,$ the symbol $H^{\infty}(U)$ stands for the vector space of bounded functions on $U$ which are holomorphic. We equip this space with the supremum norm.

        If $\gamma$ and $\rho$ are real vectors (e.g.\ multi-indices), by $\gamma<\rho$ ($\gamma\leq \rho$) we mean that $\gamma_r<\rho_r$ ($\gamma_r\leq\rho_r$), for $r=1,\ldots,d.$ For any real number $x$ the symbol ${\bf x}$ denotes the vector $(x,\ldots,x)\in \mathbb{R}^d.$

        For two vectors $z,w\in \mathbb{C}^d$ we set $z^{w}=z_1^{w_1}\cdots z_d^{w_d},$ whenever it makes sense. In particular, for $\la=(\la_1,\ldots,\la_d)\in\Rdp$ and $u=(u_1,\ldots,u_d)\in\mathbb{R}^d,$ by $\la^{iu}$ we mean $\la_1^{iu_1}\cdots\la_d^{iu_d};$ similarly, for $N=(N_1,\ldots,N_d)\in\mathbb{N}^d,$ by $\la^N$ we mean $\la_1^{N_1}\cdots \la_d^{N_d}.$ This notation is also used for operators, i.e.\, for $u\in\mathbb{R}^d$ and $N\in\mathbb{N}^d$ we set
        $$L^{iu}=L_1^{iu_1}\cdots L_d^{iu_d},\qquad L^N=L_1^{N_1}\cdots L_d^{N_d}.$$
        Note that, due to the assumption on the strong commutativity, the order of the operators in the right hand sides of the above equalities is irrelevant.

        By $\langle z, w \rangle,$ $z,w\in\mathbb{C}^d$ we mean the usual inner product on $\mathbb{C}^d.$ Additionally, if instead of $w\in\mathbb{C}^d$ we take a vector of self-adjoint operators $L=(L_1,\ldots,L_d),$ then, by $\langle z, L \rangle$ we mean $\sum_{r=1}^d z_r L_r.$

        The symbol $\frac{d\la}{\la}$ (in some places we write $\frac{dt}{t}$ instead) stands for the product Haar measure on $(\Rdp,\cdot),$ i.e.\ $$\frac{d\la}{\la}=\frac{d\la_1}{\la_1}\cdots\frac{d\la_d}{\la_d}.$$ For a function $m\in L^1(\Rdp,\, \frac{d\la}{\la}),$ we define its $d$-dimensional Mellin transform by $$\M(m)(u)=\int_{\Rdp}\la^{-iu}\,m(\la)\,\frac{d\la}{\la},\qquad u\in\mathbb{R}^d.$$ In Appendix \ref{chap:App,sec:Mel} we list properties of the Mellin transform needed in our thesis.

        Throughout the thesis we use the variable constant convention, i.e.\ the constants (such as $C,$ $C_p$ or $C(p),$ etc.) may vary from one occurrence to another. In most cases we shall however keep track of the parameters on which the constant depends, (e.g.\ $C$ denotes a universal constant, while $C_p$ and $C(p)$ denote constants which may also depend on $p$).

        Additionally, the symbol $a\lesssim b$ means that $a\leq C b,$ with a constant $C$ independent of significant quantities. We also write $a\approx b,$ whenever $a\leq C b$ and $b \leq C a.$ Note that using the latter symbols results in hiding the track of the parameters on which the constant $C$ depends.

        Let $B_1,B_2$ be Banach spaces and let $F$ be a dense subspace of $B_1.$ We say that a linear operator $T\colon F\to B_2$ is bounded, if it has a (unique) bounded extension to $B_1.$

        We shall need the following definition.
        \begin{defi}
\label{defi:Marcon}
We say that $m\colon\Rdp\to \mathbb{C}$ satisfies a Marcinkiewicz condition of order $\rho=(\rho_1,\ldots,\rho_d)\in \mathbb{N}_0^d,$ if $m$ is a bounded function having partial derivatives up to order $\rho$ \footnote{i.e.\ $\partial^{\gamma}(m)$ exist for $\gamma=(\gamma_1,\ldots,\gamma_d)\leq \rho$} and for all multi-indices $\gamma=(\gamma_1,\ldots,\gamma_d)\leq\rho$
\begin{equation*}\tag{M} \label{chap:Intro,sec:Notation,eq:Marcon}
\|m\|_{(\gamma)}:=\bigg(\sup_{R_1,\ldots,R_d>0}\int_{R_1<\la_1<2R_1}\ldots\int_{R_d<\la_d<2R_d}|\la^{\gamma}\partial^{\gamma}m(\la)|^2\,\frac{d\la}{\la}\bigg)^{1/2}<\infty.
\end{equation*}
\end{defi}
If $m$ satisfies a Marcinkiewicz condition of order $\rho,$ then we set
\begin{equation*}
\|m\|_{Mar,\rho}:=\sup_{\gamma\leq \rho}\|m\|_{(\gamma)}.
\end{equation*}

        Throughout the thesis we employ the following useful terminology. We say that a single operator $L$ has a Marcinkiewicz functional calculus\footnote{In the single operator case it might seem better to use the term 'H\"ormander functional calculus', cf.\ \cite[Section 2]{cit:Me}. We use the name of Marcinkiewicz to accord with the naming of the multi-dimensional condition.} of order $\rho>0$, whenever the following holds: if the multiplier function $m$ satisfies the one-dimensional (i.e.\ with $d=1$) Marcinkiewicz condition \eqref{chap:Intro,sec:Notation,eq:Marcon} of order $\rho,$ then the multiplier operator $m(L)$ is bounded on $L^p(X,\nu),$ $1<p<\infty,$ and
         $\|m(L)\|_{L^p(X,\nu)\to L^p(X,\nu)}\leq C_{p}\|m\|_{Mar,\rho}.$ Additionally, we say that $L$ has an $H^{\infty}$ functional calculus, whenever we have the following: for each $1<p<\infty,$ there is a sector $S_{\varphi_p},$ $\varphi_p<\pi/2,$ such that, if $m$ is a bounded holomorphic function on $S_{\varphi_p},$ then $\|m(L)\|_{L^p(X,\nu)\to L^p(X,\nu)}\leq C_{p}\|m\|_{H^{\infty}(S_{\varphi_p})}.$

        An analogous terminology is used when considering a system of operators $L=(L_1,\ldots,L_d)$ instead of a single operator. To say that the system $L$ has a Marcinkiewicz joint functional calculus of order $\rho=(\rho_1,\ldots,\rho_d)\in \Rdp$ we require the following condition to be true: if the multiplier function $m$ satisfies the $d$-dimensional Marcinkiewicz condition \eqref{chap:Intro,sec:Notation,eq:Marcon} of order $\rho=(\rho_1,\ldots,\rho_d),$ then the multiplier operator $m(L)$ is bounded on $L^p(X,\nu),$ $1<p<\infty,$ and
         $\|m(L)\|_{L^p(X,\nu)\to L^p(X,\nu)}\leq C_{p}\|m\|_{Mar,\rho}.$ Similarly, we say that $L$ has an $H^{\infty}$ joint functional calculus, whenever the following holds: for each $1<p<\infty$ there is a polysector ${\bf S}_{\varphi_p},$ $\varphi_p<\pi/2,$ such that if $m$ is a bounded holomorphic function in several variables on ${\bf S}_{\varphi_p},$ then $\|m(L)\|_{L^p(X,\nu)\to L^p(X,\nu)}\leq C_{p}\|m\|_{H^{\infty}({\bf S}_{\varphi_p})}.$

Note that every system which has a Marcinkiewicz joint functional calculus of finite order, also has an $H^{\infty}$ joint functional calculus. Indeed, by the (multivariate) Cauchy's integral formula, every function which is holomorphic in a polysector ${\bf S}_{\varphi},$ $0<\varphi<\pi/2,$ satisfies the Marcinkiewicz condition \eqref{chap:Intro,sec:Notation,eq:Marcon} of arbitrary order $\rho\in \mathbb{N}_0^d.$ Moreover, we have $\|m\|_{Mar,\rho}\leq C_{\rho}\|m\|_{H^{\infty}({\bf S}_{\varphi})}.$

        From \cite[Theorem 3]{Hanonsemi} it follows that each $L_r,$ $r=1,\ldots,d,$ satisfying \eqref{chap:Intro,sec:Setting,eq:contra}, has an $H^{\infty}$ functional calculus. However, there are examples of operators which obey \eqref{chap:Intro,sec:Setting,eq:contra} yet do not have a Marcinkiewicz functional calculus. Indeed, the $H^{\infty}$ functional calculus is the strongest possible for the Ornstein-Uhlenbeck and Laguerre operators (connected with Laguerre polynomial expansions), see Section \ref{chap:Hinf,sec:ExOp}. Thus, general operators $L_r,$ $r=1,\ldots,d,$ for which \eqref{chap:Intro,sec:Setting,eq:contra} and \eqref{chap:Intro,sec:Setting,eq:noatomatzero} hold may:
         \begin{enumerate}[I)]
        \item \label{caseI} all have a Marcinkiewicz functional calculus (this case is studied in Chapter \ref{Chap:Ck}),
        \item \label{caseII} all have only an $H^{\infty}$ functional calculus (this case is studied in Chapter \ref{Chap:Hinf}),
        \item \label{caseIII} some of them have Marcinkiewicz while others $H^{\infty}$ functional calculi (this case is studied in Chapter \ref{Chap:CkHinf}).
        \end{enumerate}
         When considering general operators which may only satisfy II) we use the letter $L.$ When considering general operators which additionally have a Marcinkiewicz functional calculus we use the letter $A.$

        In many instances in the thesis the measure space $(X,\nu)$ is in fact a product space $(X_1\times\cdots \times X_d,\nu_1\otimes \cdots\otimes \nu_d).$ Assume that $T$ is a self-adjoint or bounded operator on $L^2(X_{r},\nu_{r}),$ for some $r=1,\ldots,d.$ Then, one can define the tensor product operators
        \begin{equation}
        \label{chap:Intro,sec:Notation,eq:tensnot}
        T\otimes I_{(r)}=I_{L^2(X_1,\nu_1)}\otimes \cdots \otimes I_{L^2(X_{r-1},\nu_{r-1})}\otimes T\otimes I_{L^2(X_{r+1},\nu_{r+1})} \otimes \cdots \otimes I_{L^2(X_d,\nu_d)},
        \end{equation}
        as, respectively, self-adjoint or bounded operators on $L^2(X,\nu).$ If $T$ is bounded on $L^2(X_r,\nu_r),$ then, for $f\in L^2(X,\nu),$ the operator $(T\otimes I_{(r)})f$ may be regarded as acting only on the $x_r$-th variable of $f$ with the other variables being fixed. Moreover, when $T$ is also bounded on $L^p(X_r,\nu_r),$ for some $1\leq p<\infty,$ then, by Fubini's theorem $T\otimes I_{(r)}$ is a bounded operator on $L^p(X,\nu)$ and
        \begin{equation}
        \label{chap:Intro,sec:Notation,eq:tensnormeq}
        \|T\|_{L^p(X_r,\nu_r)\to L^p(X_r,\nu_r)}=\|T\otimes I_{(r)}\|_{L^p(X,\nu)\to L^p(X,\nu)}.
        \end{equation}
        When we are in the product setting we often do not distinguish between $T$ and $T\otimes I_{(r)}.$

        Assume that the system $L=(L_1,\ldots,L_d)$ consists of self-adjoint non-negative operators acting on separate variables, that is, each $L_r$ acts on $\Dom(L_r)\subset L^2(X_r,\nu_r),$ $r=1,\ldots,d.$ Then the operators $L_r\otimes I_{(r)},$ $r=1,\ldots,d,$ are non-negative self-adjoint and strongly commuting on $L^2(X,\nu).$ Moreover, $E_{L_r}\otimes I_{(r)}$ is the spectral measure of the operator $L_r\otimes I_{(r)},$ $r=1,\ldots,d.$ From this observation it follows that, if, for some $r=1,\ldots,d,$ the operator $L_r$ satisfies the atomlessness condition \eqref{chap:Intro,sec:Setting,eq:noatomatzero}, then the same is true for $L_r\otimes I_{(r)}.$ Additionally, using \eqref{chap:Intro,sec:Notation,eq:tensnormeq} it is not hard to deduce that, if the operator $L_r$ satisfies the contractivity condition \eqref{chap:Intro,sec:Setting,eq:contra} (with respect to $L^p(X_r,\nu_r)$), then $L_r\otimes I_{(r)}$ satisfies \eqref{chap:Intro,sec:Setting,eq:contra} as well (with respect to $L^p(X,\nu)$).

        We kindly refer the reader to consult Appendix \ref{chap:App,sec:tens}, for all unexplained statements from the previous two paragraphs, and more details on tensor products of operators.

\section{Outline of the thesis}
\label{chap:Intro,sec:outline}
The non-introductory part of the dissertation is divided into five chapters.

In Chapter \ref{Chap:General} we prove a general multivariate multiplier theorem for strongly commuting operators. This theorem is then used to obtain the main results of Chapters \ref{Chap:Ck}, \ref{Chap:Hinf}, and \ref{Chap:CkHinf}.

The subdivision of the material in these chapters corresponds to the three cases I), II), and III) above, respectively. Each of Chapters \ref{Chap:Ck}, \ref{Chap:Hinf}, and \ref{Chap:CkHinf}, begins with a formulation of the following principle: a joint multiplier theorem can be deduced from separate multiplier theorems. The instances of this principle are the general Marcinkiewicz type multiplier theorems in each of the contexts I), II), and III). These theorems are proved by using the results of Chapter \ref{Chap:General} and the complexity of the proofs increases with the number of the chapter.

In Chapter \ref{Chap:Ck} we investigate systems of operators which satisfy I). After giving general results, we discuss: polynomial growth of the $L^p$ norms of imaginary powers (in particular for the Jacobi operator), H\"ormander type multipliers for the multi-dimensional Hankel transform, and Marcinkiewicz type multipliers for the Dunkl transform (in the product setting). Note that in the latter two topics we obtain sharper results then those given by the general theory. Moreover, in the case of the Hankel transform we also prove weak type $(1,1)$ results.

Chapter \ref{Chap:Hinf} is devoted to studying systems of operators satisfying II). We prove a general holomorphic Marcinkiewicz type multiplier theorem and apply it to systems of Hermite and Laguerre operators (connected with Laguerre polynomial expansions). Additionally, we prove a general holomorphic extension theorem for multipliers connected with operators that do not have a Marcinkiewicz functional calculus.

Chapter \ref{chap:Riesz} contains some applications of the $H^{\infty}$ joint functional calculus from Chapter \ref{Chap:Hinf} to the study of dimension free boundedness of certain Riesz transforms.

In Chapter \ref{Chap:CkHinf} we consider systems of operators satisfying III). The dominant part of this chapter is devoted to obtaining weak type $(1,1)$ results for a pair of operators consisting of the $d$-dimensional Ornstein-Uhlenbeck operator and an operator whose heat kernel satisfies certain Gaussian bounds.

Part of the material included in the thesis has already been published in \cite{cit:ja}, \cite{jaOU}, \cite{jaDun}, and in the joint paper \cite{dpw} with Dziuba\'nski and Preisner.
    %
    %
\newpage
    \newchapter{A general multivariate multiplier theorem}{A general multivariate multiplier theorem}{A general multivariate multiplier theorem}
    \label{Chap:General}
\numberwithin{equation}{chapter}
We present a complete proof of a certain multiplier theorem, which is a multivariate extension of Meda's \cite[Theorem 1]{cit:Me}.  Theorem \ref{thm:gen} is among our basic tools in proving the results of Chapters \ref{Chap:Ck}, \ref{Chap:Hinf} and \ref{Chap:CkHinf}. The theorem we present here is an enhanced version of our previous result \cite[Theorem 2.2]{cit:ja}. The novelty lies in the fact that we only assume strong commutativity, and no product structure or discrete spectrum assumptions are needed, contrary to \cite{cit:ja}.

Throughout this chapter we consider a system $L=(L_1,\ldots,L_d)$ of strongly commuting, non-negative, self-adjoint operators on some space $L^2(X,\nu)$ as in Section \ref{chap:Intro,sec:Setting}. We impose that the system $L$ satisfies the contractivity assumption \eqref{chap:Intro,sec:Setting,eq:contra} and the atomlessness condition \eqref{chap:Intro,sec:Setting,eq:noatomatzero}. In this chapter, for the sake of brevity we write $L^p$ and $\|\cdot\|_p$ instead of $L^p(X,\nu)$ and $\|\cdot\|_{L^p(X,\nu)},$ $1\leq p\leq\infty.$ For a linear operator $T$ the symbol $\|T\|_{p\to p}$ shall denote the norm of $T$ acting on $L^p.$

Consider a bounded Borel measurable function $m$ on $\Rdp.$ We associate with $m$ the multivariate multiplier operator $m(L)$ via \eqref{chap:Intro,sec:Setting,eq:mdef}. For $N=(N_1,\ldots,N_d)\in \mathbb{N}^d,$ $\TT=(t_1,\ldots,t_d)\in \Rdp$ and $\la=(\la_1,\ldots,\la_d)\in \Rdp,$ let
\begin{equation*}
m_{N,\TT}(\la)=\la^N\TT^N \exp\big(-2^{-1}\langle t,\la \rangle\big)\,m(\la).
\end{equation*}
The main theorem of this chapter is the following.
\begin{thmc}
\label{thm:gen}
Let $1<p<\infty$ be given. Assume $m\colon (0,\infty)^{d}\to \mathbb{C}$ is a bounded Borel measurable function such that for some $N \in \mathbb{N}^d$
\begin{equation}
\label{chap:General,eq:thm:gen}
m(L,N,p):=\int_{\mathbb{R}^{d}}\|L^{iu}\|_{\pp}\,\sup_{\TT\in \Rdp}|\M(m_{N,\TT})(u)|\,du<\infty.
\end{equation}
Then the multiplier operator $m(L)$ is bounded on $L^p$ and $$\|m(L)\|_{p\to p}\leq C_{p,d,N}\,m(L,N,p).$$
\end{thmc}
\begin{remark1} A comment on the measurability of the integrand in \eqref{chap:General,eq:thm:gen} is in order. Since $m$ is bounded, for each fixed $N\in\mathbb{N}^d$ and $u\in \mathbb{R}^d,$ $|\M(m_{N,\TT})(u)|$ is continuous in $t\in \Rdp.$ Hence, the supremum in \eqref{chap:General,eq:thm:gen} is a Borel measurable function of $u\in\mathbb{R}^d,$ for a full proof see Appendix \ref{chap:App,sec:Mel,sub:rem}. Moreover, it can be shown that the mapping $\mathbb{R}^d\ni u\to L^{iu}$ is continuous in the strong operator topology on $L^p,$ \label{chap:General,pag:strong} see Appendix \ref{chap:App,sec:joint,subsec:strongmeasLiu}. Thus, in order to be completely rigorous with regard to the measurability of the integrand in \eqref{chap:General,eq:thm:gen}, we should assume
 $$\int_{\mathbb{R}^{d}}\|L^{iu}f\|_{p}\,\sup_{\TT\in \Rdp}|\M(m_{N,\TT})(u)|\,du\leq C_{p,N,d}\|f\|_p,\qquad f\in L^2\cap L^p.$$
\end{remark1}
\begin{remark2}
The proof of Theorem \ref{thm:gen} we present here is modeled over the original proof of \cite[Theorem 1]{cit:Me} for the one-operator case. In \cite[Theorem 2.1]{Hanonultracon} the authors gave a simpler proof of \cite[Theorem 1]{cit:Me}. However, a closer look at their method reveals that it does not carry over to our multivariate case. The reason is that we initially do not know whether multivariate multipliers of Laplace transform type $\Rdp\ni\la\mapsto \la^{\bf 1}\int_{\Rdp}\exp(-\langle t,\la\rangle)\, \kappa(t)\,dt,$
with $\kappa$ being a bounded function on $\Rdp$ that may not have a product form,  produce bounded multiplier operators on $L^p.$
\end{remark2}
The key ingredient in the proof of Theorem \ref{thm:gen} is the $L^p$ boundedness of some auxiliary $g$-function $g_N.$ For any fixed $N\in\mathbb{N}^d$ let
\begin{equation}
g_N(f)=\left(\int_{\Rdp}\bigg|t^N L^N \exp({-\langle t, L\rangle}) f \bigg|^2\,\frac{dt}{t}\right)^{1/2},\qquad f\in L^2. \label{chap:General,eq:gnsecformula}
\end{equation}
Note that for each fixed $f\in L^2$ we have
\begin{equation}\label{chap:General,eq:Frechet}\partial_{t_1}^{N_1}\cdots \partial_{t_d}^{N_d}\exp({-\langle t, L\rangle}) f=(-1)^{|N|}L^N\exp({-\langle t, L\rangle}) f,\end{equation}
where the left hand side of \eqref{chap:General,eq:Frechet} is the Fr\'echet (partial) derivative in $L^2$ of the function $\Rdp\ni t\mapsto \exp({-\langle t, L\rangle})f;$ to justify \eqref{chap:General,eq:Frechet} we just use the multivariate spectral theorem. Thus we may also express $g_N$ as
\begin{equation}
 \label{chap:General,eq:gnfirformula} g_N(f)=\left(\int_{\Rdp}\bigg|t^N \bigg(\partial_{t_1}^{N_1}\cdots \partial_{t_d}^{N_d} \exp({-\langle t, L\rangle}) f\bigg) \bigg|^2\,\frac{dt}{t}\right)^{1/2}
\end{equation}
In many particular cases \eqref{chap:General,eq:gnfirformula} can be also understood pointwise.

It is not hard to see that $g_N$ is, up to a multiplicative constant, a well defined isometry on $L^2.$ Indeed, using Fubini's theorem (the first, third and fourth equality below) and the multivariate spectral theorem (the second and fourth equality below) we obtain
\begin{align}
\label{chap:General,eq:gnL2isometry}
\begin{split}
\|g_N(f)\|_{2}^2&=\int_{\Rdp}\int_{X}\bigg|t^N L^N \exp({-\langle t, L\rangle})f(x)\bigg|^2\,d\nu(x)\,\frac{dt}{t}\\
&=\int_{\Rdp}\int_{\sigma(L)}\bigg|t^N \la^{N}\exp({-\langle t, \la\rangle})\bigg|^2\, dE^{L}_{f,f}(\la)\,\frac{dt}{t}\\
&=\int_{\sigma(L)}\int_{\Rdp}\bigg|t^N \la^{N}\exp({-\langle t, \la\rangle})\bigg|^2\,\frac{dt}{t}\, dE^{L}_{f,f}(\la)=C_{N}\|f\|_{2}.
\end{split}
\end{align}

Moreover, for fixed $\la \in \Rdp,$ we have
\begin{equation*}
\M\big((\cdot)^N \la^{N}\exp({-\langle \cdot, \la\rangle})\big)(u)=\Gamma(N_1-iu_1)\cdots\Gamma(N_d-iu_d)\la^{iu},\qquad u\in\mathbb{R}^d.
\end{equation*}
Thus, using the multivariate spectral theorem and Plancherel's formula for the Mellin transform we rewrite $g_N(f)$ as
\begin{equation}
\label{chap:General,eq:gnformula}
g_N(f)=\frac{1}{(2\pi)^{d/2}}\bigg(\int_{\mathbb{R}^d}\big|\Gamma(N_1-iu_1)\cdots\Gamma(N_d-iu_d)L^{iu}f\big|^2\,du\bigg)^{1/2}.
\end{equation}

The following theorem expresses the crucial $L^p$ norm preserving property of $g_N.$
\begin{thmc}
\label{thm:gfun}
Let $1<p<\infty.$ Then, there exists a constant $C_{p,d,N}>0$ such that
\begin{equation}
\label{chap:General,eq:gfun}
(C_{p,d,N})^{-1}\|f\|_p\leq \|g_N(f)\|_p\leq C_{p,d,N} \|f\|_p,\qquad f\in L^p \cap L^2.
\end{equation}
\end{thmc}
\begin{remark} The formulae  for $g_N(f)$, \eqref{chap:General,eq:gnsecformula} and \eqref{chap:General,eq:gnformula}, are easily extendable beyond $N\in\mathbb{N}^d,$ to complex parameters $\alpha=(\alpha_1,\ldots,\alpha_d)$ with $\Real(\alpha_r)>0,$ $r=1,\ldots,d.$ Then, by using methods similar to those employed in this thesis, it can be proved that the generalized $g$-function $g_{\alpha}$ also satisfies \eqref{chap:General,eq:gfun}.\end{remark}

Theorem \ref{thm:gfun} can be derived from results of \cite{AlFrMc}. The proof we present here is more direct. It follows the scheme originated in \cite{Me_gfun} and generalized in \cite{cit:ja}. Here we also need the following multi-dimensional randomization lemma (cf. \cite[Lemma 5.2]{AlFrMc} and \cite[Lemma 1.3]{Me_gfun}). The proof of Lemma \ref{lem:randkhin} is based on a standard application of Khintchine's inequality and Minkowski's integral inequality, hence we omit it.
\begin{lemc}
\label{lem:randkhin}
Fix $1<p<\infty.$ Then, for each sequence $\{f_{j}\}_{j\in\mathbb{Z}^d},$ of functions in $L^p,$ we have
$$\bigg\|\bigg(\sum_{j\in\mathbb{Z}^d}|f_j|^2\bigg)^{1/2}\bigg\|_p\leq C_p \sup_{|a_{j_1}|\leq 1,\ldots,|a_{j_d}|\leq 1} \bigg\|\sum_{j\in\mathbb{Z}^d}a_{j_1}\cdots a_{j_d} f_j\bigg\|_p.$$
\end{lemc}
\begin{proof}[Proof of Theorem \ref{thm:gfun}]
Using \eqref{chap:General,eq:gnL2isometry} and a polarization argument we see that it suffices to prove the right hand side of \eqref{chap:General,eq:gfun}. Recall that $f\in L^2,$ so that (by Fubini's theorem and the multivariate spectral theorem) all the formulae appearing in the proof make sense $\nu$-a.e. Let $N\in \mathbb{N}^d$ be fixed. Take a smooth function $h$ on $\mathbb{R},$ supported in $[-\pi,\pi],$ and such that
\begin{equation} \label{chap:General,eq:resid}\sum_{l\in\mathbb{Z}} h(x-\pi l)=1,\qquad x\in\mathbb{R}.\end{equation}
Then, for $k=(k_1,\ldots,k_d)\in\mathbb{Z}^d,$ we set $$h_{k}(x)=h_{k_1}(x_1)\cdots h_{k_d}(x_d)=h(x_1-\pi k_1)\cdots h(x_d-\pi k_d),\qquad x=(x_1,\ldots,x_d)\in\mathbb{R}^d.$$

Next, for each $j\in\mathbb{Z}^d$ we define the functions $b_{j,k}$ on ${\bf S}_{\pi/2}$
by
\begin{align*}
b_{j,k}(z)&=\int_{\mathbb{R}^d}h_{k}(u)\Gamma(N_1-iu_1)\cdots\Gamma(N_d-iu_d)e^{-i\langle j ,u\rangle} z^{iu}\,du\\&=\prod_{r=1}^d\int_{\mathbb{R}}h_{k_r}(u_r)\Gamma(N_r-iu_r)e^{-ij_r u_r}z_r^{iu_r}\,du_r\equiv\prod_{r=1}^{d}b_{j_r,k_r}(z_r).
\end{align*}
It is not hard to see that the functions $b_{j,k}$ are bounded on $[0,\infty)^d,$ thus it makes sense to consider $b_{j,k}(L).$ Moreover, due to the product structure of $b_{j,k},$ we have
\begin{equation}
\label{chap:General,eq:prodbjk}
b_{j,k}(L)=b_{j_1,k_1}(L_1)\cdots b_{j_d,k_d}(L_d),\qquad j,k\in\mathbb{Z}^d.
\end{equation}

From Parseval's formula for the multiple Fourier series it follows that, for each $k\in\mathbb{Z}^d,$ the set $\{(2\pi)^{-d/2}e^{-i\langle j, u-\pi k\rangle }\}_{j\in\mathbb{Z}^d}$ forms an orthonormal basis in the space $L^2(Q_k,du),$ with $Q_k=[k_1\pi-\pi,k_1\pi+\pi]\times\cdots\times [k_d\pi-\pi,k_d\pi+\pi].$ Therefore, since $\supp(h_k)\subset Q_k$ and $|e^{i \pi \langle j, k\rangle}|=1,$ for each fixed $z \in {\bf S}_{\pi/2}$ we have
\begin{equation}
\label{chap:General,eq:Pars}
\begin{split}
&\int_{\mathbb{R}^d}\left|h_k(u)\Gamma(N_1-iu_1)\cdots\Gamma(N_d-iu_d)z^{iu}\right|^2\,du\\
&=\frac{1}{(2\pi)^{d/2}}\sum_{j\in\mathbb{Z}^d} \int_{\mathbb{R}^d}\left|h_k(u)\Gamma(N_1-iu_1)\cdots\Gamma(N_d-iu_d)e^{-i\langle j, u\rangle}z^{iu}\right|^2\,du.
\end{split}
\end{equation}

Now, using \eqref{chap:General,eq:gnformula}, \eqref{chap:General,eq:resid} and Minkowski's integral inequality, followed by the multivariate spectral theorem and \eqref{chap:General,eq:Pars}, we obtain
\begin{align*}
g_N(f)(x)&\leq (2\pi)^{-d/2} \sum_{k\in\mathbb{Z}^d}\left[\int_{\mathbb{R}^d}\left|h_k(u)\Gamma(N_1-iu_1)\cdots\Gamma(N_d-iu_d)L^{iu}f(x)\right|^2\,du\right]^{1/2}\\
&=(2\pi)^{-d}\sum_{k\in\mathbb{Z}^d}\left[\sum_{j\in\mathbb{Z}^d}\left|\int_{\mathbb{R}^d}h_k(u)\Gamma(N_1-iu_1)\cdots\Gamma(N_d-iu_d)e^{-i\langle j ,u\rangle}L^{iu}f(x)\,du\right|^2\right]^{1/2}\\
&=(2\pi)^{-d}\sum_{k\in\mathbb{Z}^d}\bigg(\sum_{j\in\mathbb{Z}^d}|b_{j,k}(L)(f)(x)|^2\bigg)^{1/2},
\end{align*}
and consequently,
\begin{equation*}
\|g_N(f)\|_p\leq (2\pi)^{-d}\sum_{k\in\mathbb{Z}^d}\bigg\|\bigg(\sum_{j\in\mathbb{Z}^d}|b_{j,k}(L)f|^2\bigg)^{1/2}\bigg\|_p.
\end{equation*}
Then from Lemma \ref{lem:randkhin} and \eqref{chap:General,eq:prodbjk} it follows that
\begin{align}
\nonumber
(2\pi)^{d}\|g_N(f)\|_p&\leq \sum_{k\in\mathbb{Z}^d}\sup_{|a_{j_1}|\leq 1,\ldots,|a_{j_d}|\leq 1}\bigg\|\left(\sum_{j_1\in\mathbb{Z}}a_{j_1} b_{j_1,k_1}(L_1)\right)\cdots\left(\sum_{j_d\in\mathbb{Z}}a_{j_d} b_{j_d,k_d}(L_d)\right)f\bigg\|_p\\ \label{chap:General,eq:ineqaim}
&:= \sum_{k\in\mathbb{Z}^d}\sup_{|a_{j_1}|\leq 1,\ldots,|a_{j_d}|\leq 1}\bigg\|m^{a_1}_{k_1}(L_1)\cdots m^{a_d}_{k_d}(L_d)f\bigg\|_p.
\end{align}

We shall now focus on estimating the $L^p$ operator norm of each of the operators $m^{a_r}_{k_r}(L_r),$ $r=1,\ldots,d,$ defined in \eqref{chap:General,eq:ineqaim}.
We claim that each of the functions $$m_{r}(z_r):=m^{a_r}_{k_r}(z_r)=\sum_{j_r\in\mathbb{Z}}a_{j_r} b_{j_r,k_r}(z_r)$$ is holomorphic on the right half-plane $S_{\pi/2}$ and has the following property: for every angle $0<\varphi<\pi/2$ there exist $C_{\varphi}>0,$ $c_{\varphi}>0,$ independent of $a_r$ and $k_r,$ such that
\begin{equation}
\label{chap:General,eq:claim}
\|m_r\|_{H^{\infty}(S_{\varphi})}\leq C_{\varphi}e^{-c_{\varphi}|k_r|}.
\end{equation}

If the claim is true, then coming back to \eqref{chap:General,eq:ineqaim} and using \cite[Theorem 3]{Hanonsemi} we finish the proof of Theorem \ref{thm:gfun}. Note that the bound in \cite[Theorem 3]{Hanonsemi} does not depend on the considered operator.

Thus we focus on proving \eqref{chap:General,eq:claim}. Since this inequality is one-dimensional, till the end of the proof of the claim we consider $j,k\in\mathbb{Z},$ $z\in S_{\pi/2},$ $u\in\mathbb{R}$ and $N\in\mathbb{N}.$ We use the following estimate, valid for each $\varepsilon>0,$
\begin{equation}
\label{chap:General,eq:Gammaest}
|\Gamma(N-iu)|+|\Gamma'(N-iu)|+|\Gamma''(N-iu)|\leq C_{\varepsilon,N}e^{-|u|(\pi/2-\varepsilon)}, \qquad u\in\mathbb{R}.
\end{equation}
Observe that $z=e^{w}\in S_{\varphi}$ if and only if $|\Ima(w)|<\varphi.$ In view of \eqref{chap:General,eq:Gammaest} and the fact that $\supp h_{k}\subseteq [k\pi-\pi,k\pi+\pi],$ for $\varepsilon$ small enough (say, $\varepsilon<\pi/2-|\varphi|$), we have
\begin{equation}
\label{chap:General,eq:claimp1}
|b_{j,k}(e^{w})|\leq \int_{\mathbb{R}}|h_{k}(u)|e^{-|u|(\pi/2-\varepsilon-|\varphi|)}\,du\leq C_{\varphi}e^{-c_{\varphi}|k|}.
\end{equation}
Moreover, using integration by parts in
$$b_{j,k}(e^{w})=\int_{\mathbb{R}}\left[h_{k}(u)\Gamma(N-iu)\right]e^{iu(w-j)}\,du,\qquad |\Ima(w)|<\varphi,$$
together with \eqref{chap:General,eq:Gammaest} and the fact that $h_{k},$ $h'_{k}$ and $h''_k$ are all supported in $[k\pi-\pi,k\pi+\pi],$ we see that
\begin{equation}
\label{chap:General,eq:claimp2}
|b_{j,k}(e^{w})|\leq C_{\varphi} \frac{e^{-c_{\varphi}|k|}}{|w-j|^2}.
\end{equation}
Hence, combining \eqref{chap:General,eq:claimp1} and \eqref{chap:General,eq:claimp2}, for $|\Ima(w)|<\varphi$ we arrive at
\begin{align*}
|m_k(e^{w})|=\bigg|\sum_{|j-w|<1/2}a_j b_{j,k}(e^{w})+\sum_{|j-w|\geq 1/2}a_j b_{j,k}(e^{w})\bigg|\leq C_{\varphi} e^{-c_{\varphi}|k|}.
\end{align*}
The proof of \eqref{chap:General,eq:claim} is finished, thus also the proof of Theorem \ref{thm:gfun} is completed.
\end{proof}

We are now ready to prove Theorem \ref{thm:gen}.
\begin{proof}[Proof of Theorem \ref{thm:gen}]
Take $f\in L^2\cap L^p.$
From the inversion formula for the Mellin transform and the multivariate spectral theorem we see that
\begin{equation}
\label{chap:General,eq:auxeq0}
t^{N}L^N\exp(-2^{-1}\langle t, L\rangle)m(L)f=\frac{1}{(2\pi)^d}\int_{\mathbb{R}^d}\M(m_{N,t})(u)L^{iu}f\,du.
\end{equation}

Consequently, since $t^{\bf 1} L^{\bf 1}\exp(-2^{-1}\langle t, L\rangle)$ is bounded on $L^2,$ we have
\begin{equation}
\label{chap:General,eq:auxeq}
t^{N+{\bf 1}}L^{N+{\bf 1}}\exp(-\langle t, L\rangle)m(L)f=\frac{1}{(2\pi)^d}\int_{\mathbb{R}^d}\M(m_{N,t})(u)t^{\bf 1}L^{\bf 1}\exp\big(-\frac12\langle t, L\rangle\big)(L^{iu}f)\,du.
\end{equation}
Note that, for each fixed $t\in\Rdp,$ both the integrals in \eqref{chap:General,eq:auxeq0} and \eqref{chap:General,eq:auxeq} can be considered as Bochner integrals of (continuous) functions taking values in $L^2.$

Then, at least formally, from Theorem \ref{thm:gfun} followed by \eqref{chap:General,eq:auxeq}, we obtain
\begin{align*}
&(C_{p,d,N+{\bf 1}})^{-1}\|m(L)f\|_p\leq \|g_{N+{\bf 1}}(m(L)(f))\|_p\\
&=\bigg\|\bigg(\int_{\Rdp}\bigg|\frac{1}{(2\pi)^d}\int_{\mathbb{R}^d}\M(m_{N,t})(u)tL\exp(-2^{-1}\langle t, L\rangle)(L^{iu}f)\,du\bigg|^2\,\frac{dt}{t}\bigg)^{1/2}\bigg\|_p.
\end{align*}
Hence, using Minkowski's integral inequality, it follows that $\|m(L)f\|_p$ is bounded by
\begin{equation*} (2\pi)^{-d}\,C_{p,d,N+{\bf 1}}\int_{\mathbb{R}^d}\sup_{\TT\in\Rdp}|\M(m_{N,\TT})(u)| \bigg\|\bigg(\int_{\Rdp}\bigg|tL\exp(-2^{-1}\langle t, L\rangle)(L^{iu}f)\bigg|^2\,\frac{dt}{t}\bigg)^{1/2}\bigg\|_p\,du.
\end{equation*}
Now, observing that
$$\left(\int_{\Rdp}\bigg|tL\exp(-2^{-1}\langle t, L\rangle)(L^{iu}f)\bigg|^2\,\frac{dt}{t}\right)^{1/2}=2^{d}g_{{\bf 1}}(L^{iu}f)$$
and using once again Theorem \ref{thm:gfun} (this time with $N={\bf 1}$), we arrive at
\begin{align}
\nonumber
\|m(L)f\|_p&\leq \pi^{-d}\,C_{p,d,N+{\bf 1}}\int_{\mathbb{R}^d}\|g_{{\bf 1}}(L^{iu}f)\|_p\,\sup_{\TT\in \Rdp}|\M(m_{N,\TT})(u)|\,du \\ \label{chap:General,eq:lastexp}
&\leq \pi^{-d}\,C_{p,d,N+{\bf 1}}C_{p,d,{\bf 1}}\int_{\mathbb{R}^d}\|L^{iu}\|_{p\to p}\,\sup_{\TT\in\Rdp}|\M(m_{N,\TT})(u)|\,du\,\|f\|_p.
\end{align}
Thus, the proof of Theorem \ref{thm:gen} is finished, provided we justify the formal steps above.

We proceed almost exactly as in \cite[p. 642]{cit:Me}. By \eqref{chap:General,eq:thm:gen} the last expression in \eqref{chap:General,eq:lastexp} converges. Therefore
$$\int_{\mathbb{R}^d}\M(m_{N,t})(u)t^{\bf 1}L^{\bf 1}\exp(-2^{-1}\langle t, L\rangle)(L^{iu}f)\,du$$
converges as an $L^p((X,\nu);L^2(\Rdp,\frac{dt}{t}))$-valued Bochner integral, provided the map
$$\mathbb{R}^d \ni u\mapsto \M(m_{N,t})(u)t^{\bf 1}L^{\bf 1}\exp(-2^{-1}\langle t, L\rangle)(L^{iu}f)$$
is measurable. In fact, using an argument similar to the one from \cite[p. 642]{cit:Me} (recall that also in our case $u\to L^{iu}$ is strongly continuous on $L^p$), we show that this map is continuous, hence concluding the proof of Theorem \ref{thm:gen}.
\end{proof}




    %
    %

    \newchapter{Systems of operators having Marcinkiewicz functional calculus}{Systems of operators having Marcinkiewicz functional calculus}{Systems of operators having Marcinkiewicz functional calculus}
    \numberwithin{equation}{section}
    \label{Chap:Ck}

In this chapter we consider systems of operators which have a finite order Marcinkiewicz functional calculus.

We start with a short overview. In Section \ref{chap:Ck,sec:genMar} we prove a general Marcinkiewicz type multiplier theorem, see Theorem \ref{thm:genMarCk}. Here we use Theorem \ref{thm:gen}. Section \ref{chap:Ck,sec:ExOp} provides a list of operators which have polynomially growing norms of imaginary powers. Then in Section \ref{chap:Ck,sec:BoundSome,subsec:Jac} we show how to obtain polynomial bounds for the imaginary powers of the Jacobi operator. The last two sections of Chapter \ref{Chap:Ck} do not rely on Theorem \ref{thm:genMarCk} and can be read independently. In Section \ref{chap:Ck,sec:HorHan} we present a multivariate H\"ormander type multiplier theorem for the Hankel transform. Finally, in Section \ref{chap:Ck,sec:MarDun} we provide a multivariate Marcinkiewicz type multiplier theorem for the Dunkl transform corresponding to the reflection group $G$ isomorphic to $\mathbb{Z}_2^d.$ In particular, as a corollary, we obtain analogous results for the multivariate Hankel transform.



        \section[A Marcinkiewicz type multiplier theorem]{A Marcinkiewicz type multiplier theorem}
        \label{chap:Ck,sec:genMar}

Consider a general system of self-adjoint, non-negative, strongly commuting operators $A=(A_1,\ldots,A_d),$ on some space $L^2(X,\nu).$ Recall that we impose that the operators $A_r,$ $r=1,\ldots,d,$ satisfy the assumptions of Section \ref{chap:Intro,sec:Setting}, in particular \eqref{chap:Intro,sec:Setting,eq:contra} and \eqref{chap:Intro,sec:Setting,eq:noatomatzero}. We keep the brief notation of Chapter \ref{Chap:General}, i.e.\ we write $L^p,$ $\|\cdot\|_p,$ and $\|\cdot\|_{p\to p}$  instead of $L^p(X,\nu),$ $\|\cdot\|_{L^p(X,\nu)},$ and $\|\cdot\|_{L^p(X,\nu)\to L^p(X,\nu)},$ $1\leq p\leq \infty,$ respectively.

Throughout the present section, we additionally assume that all the operators $A_r,$ $r=1,\ldots,d,$ have a Marcinkiewicz functional calculus. More precisely, we impose that there is a vector of positive real numbers $\sigma=(\sigma_1,\ldots,\sigma_d),$ such that for every $1<p<\infty$ and $r=1,\ldots,d,$
\begin{equation}
\label{chap:Ck,sec:genMar,eq:polynomial}
\|A_r^{iv}\|_{p\to p}\leq \mC(p,A_r) (1+|v|)^{\sigma_r|1/p-1/2|},\qquad v\in\mathbb{R}.
\end{equation}
In view of \cite[Theorem 4]{cit:Me}, for each $r=1,\ldots,d,$ the condition \eqref{chap:Ck,sec:genMar,eq:polynomial} implies that $A_r$ has a Marcinkiewicz functional calculus of every order $\rho_r>\sigma_r/2+1.$ Conversely, if $A_r$ has a Marcinkiewicz functional calculus of order $\rho_r$ then \eqref{chap:Ck,sec:genMar,eq:polynomial} holds with any $\sigma_r>2\rho_r.$

Observe also that \eqref{chap:Ck,sec:genMar,eq:polynomial} implies that for every $1<p<\infty$
\begin{equation}
\label{chap:Ck,sec:genMar,eq:polynomialfull}
\|A^{iu}\|_{p\to p}\leq \mC(p,A) \prod_{r=1}^d(1+|u_r|)^{\sigma_r|1/p-1/2|},\qquad u=(u_1,\ldots,u_d)\in\mathbb{R}^d.
\end{equation}
Note that the constants $\mC(p,A_r)$ in \eqref{chap:Ck,sec:genMar,eq:polynomial} and $\mC(p,A)$ in \eqref{chap:Ck,sec:genMar,eq:polynomialfull} are written in the calligraphic font, so that we can keep track of them in Theorem \ref{thm:genMarCk}.

The main theorem of this section, Theorem \ref{thm:genMarCk}, is a multivariate generalization of Meda's \cite[Theorem 4]{cit:Me}. It is also an enhancement of our previous result, \cite[Theorem 4.1]{cit:ja}. As in Theorems \ref{thm:gen} and \ref{thm:gfun} from Chapter \ref{Chap:General}, the novelty of Theorem \ref{thm:genMarCk} in comparison with our previous result lies in the fact that we only assume strong commutativity. Recall that the operators $A_r,$ $r=1,\ldots,d,$ satisfy all the assumptions of Section \ref{chap:Intro,sec:Setting}, $m(A)$ is defined by \eqref{chap:Intro,sec:Setting,eq:mdef}, while $\sigma=(\sigma_1,\ldots,\sigma_d)$ appears in \eqref{chap:Ck,sec:genMar,eq:polynomialfull}.
\begin{thm}
\label{thm:genMarCk}
Fix $1<p<\infty$ and assume that $m$ satisfies the Marcinkiewicz condition \eqref{chap:Intro,sec:Notation,eq:Marcon} of order $\rho> |1/p -1/2|\,\sigma+{\bf 1}.$ Then $m(A)$ is bounded on $L^p$ and $$\|m(A)\|_{p\to p}\leq C_{p,d}\,\mC(p,A)\, \|m\|_{Mar,\rho}.$$ In particular, if $\rho> \sigma/2+{\bf 1},$ then $m(A)$ is bounded on all $L^p$ spaces, $1<p<\infty.$
\end{thm}
Before proving Theorem \ref{thm:genMarCk} let us first state and prove a seemingly interesting corollary. Recall that the phrase 'has a Marcinkiewicz (joint) functional calculus' was explained in Section \ref{chap:Intro,sec:Notation}. The corollary below implies that a system $A=(A_1,\ldots,A_d)$ has a Marcinkiewicz joint functional calculus of a finite order if and only if each $A_r,$ $r=1,\ldots,d,$ has a Marcinkiewicz functional calculus of a finite order.
\begin{cor}
\label{cor:genMarCkequiv}
We have the following:
\begin{itemize}
\item[(i)] If, for each $r=1,\ldots,d,$ the operator $A_r$ has a Marcinkiewicz functional calculus of order $\rho_r,$ then the system $A=(A_1,\ldots,A_d)$ has a Marcinkiewicz joint functional calculus of every order greater than $\rho+{\bf 1}.$
\item[(ii)] If the system $A=(A_1,\ldots,A_d)$ has a Marcinkiewicz joint functional calculus of order $\rho,$ then, for each $r=1,\ldots,d,$ the operator $A_r$ has a Marcinkiewicz functional calculus of order $\rho_r.$
\end{itemize}
\end{cor}
\begin{proof}[Proof (sketch)]
To prove item (i), note that having a Marcinkiewicz functional calculus of order $\rho_r$ implies satisfying \eqref{chap:Ck,sec:genMar,eq:polynomial} with $\sigma_r=2\rho_r$. This observation, together with Theorem \ref{thm:genMarCk} implies the desired conclusion. The proof of item (ii) is even more straightforward, we just need to consider a function $m$ which depends only on some of the variables $\la_r,$ $r=1,\ldots,d.$
\end{proof}
\begin{proof}[Proof of Theorem \ref{thm:genMarCk}]
The proof mimics the proofs of \cite[Theorem 4]{cit:Me} and \cite[Theorem 4.1]{cit:ja}. However, for the sake of completeness, we present it in details. We shall prove that $m$ satisfies \eqref{chap:General,eq:thm:gen}. Then, the desired conclusion follows from Theorem \ref{thm:gen}.

Let $N\in \mathbb{N}^d,$ $N>\rho,$ and $\psi$ be a nonnegative, $C^{\infty}$ function supported in $[1/2,2]$ and such that $$\sum_{k=-\infty}^{\infty}\psi(2^k v )=1,\qquad v>0.$$ Then, for $\Psi_{j}(\la)=\psi(2^{j_1}\la_1)\cdots \psi(2^{j_d}\la_d),$ $$\sum_{j\in\mathbb{Z}^d}\Psi_{j}(\la)=1, \qquad \la\in\mathbb{R}^d_+.$$

Set $$c_{N_r,\rho_r,u_r}=\frac{(-1)^{\rho_r}}{(N_r-iu_r)\cdots (N_r-iu_r+\rho_r-1)}\qquad \textrm{and}\qquad c_{N,\rho,u}=\prod_{r=1}^d c_{N_r,\rho_r,u_r}.$$ Integrating by parts $\rho_r$ times in the $r$-th variable, $r=1,\ldots,d,$ we see that
\begin{align*}
\M(m_{N,t})(u)=c_{N,\rho,u}t^{iu}\sum_{j}\int_{\Rdp}\la^{N+\rho-iu}\partial^{\rho}\bigg(e^{-2^{-1}\langle {\bf 1},\la\rangle }m(\la_1/t_1,\ldots,\la_d/t_d)\Psi_j(\la)\bigg)\,\frac{d\la}{\la}.
\end{align*}
Leibniz's rule allows us to express the derivative $\partial^{\rho}$ as a weighted sum of derivatives of the form
\begin{align*}E^{j}_{\gamma,\delta,t}(\la)&= e^{-2^{-1}\langle {\bf 1},\la\rangle}t^{-\gamma}(\partial^{\gamma}m)(\la_1/t_1,\ldots,\la_d/t_d)2^{\langle j,\delta\rangle}
\prod_{r=1}^d\bigg(\frac{d^{\delta_r}}{d\la_r^{\delta_r}}\psi\bigg)(2^{j_r}\la_r),\end{align*}
where $\gamma=(\gamma_1,\ldots,\gamma_d)$ and $\delta=(\delta_1,\ldots,\delta_d)$ are multi-indices such that $\gamma+\delta\leq\rho.$ Proceeding further as in the proof of \cite[Theorem 4]{cit:Me}, we denote
$$I_{j,N,\gamma,\delta}(t,u)\equiv\int_{\Rdp}\la^{N+\rho-iu}E^{j}_{\gamma,\delta,t}(\la)\,\frac{d\la}{\la}.$$
Define $p_j=p_{j_1}\cdots p_{j_d}$ with $p_{j_r},$ $r=1,\ldots,d,$ given by
$$p_{j_r}=\begin{cases}
 2^{-j_r\rho_r}, &\mbox{if $j_r>0$}, \\
  2^{-j_r(N_r+\rho_r)}\exp(-2^{-j_r-1}), &\mbox{if $j_r\leq0.$}
       \end{cases} $$
We claim that
\begin{equation}
\label{chap:Ck,sec:genMar,eq:claim}
|I_{j,N,\gamma,\delta}(t,u)|\leq C_{N,\gamma,\delta}\,\|m\|_{Mar,\rho}\, p_{j}, \qquad j\in \mathbb{Z}^d,
\end{equation}
uniformly in $t\in\Rdp$ and $u\in\mathbb{R}^d.$

Assuming \eqref{chap:Ck,sec:genMar,eq:claim}, we obtain
\begin{align}
\nonumber
\sup_{t\in \Rdp}|\M(m_{N,t})(u)|&\leq C_N\prod_{r=1}^{d}(1+|u_r|)^{-\rho_r}\sum_{\gamma+\delta\leq \rho}C_{\gamma,\delta,\rho}\,\sum_{j\in \mathbb{Z}^d}\sup_{t\in\Rdp}|I_{j,N,\gamma,\delta}(t,u)|\\
&\leq C_{N,\rho}\|m\|_{Mar,\rho} \prod_{r=1}^{d}(1+|u_r|)^{-\rho_r}, \label{chap:Ck,sec:genMar,eq:estiHinf}
\end{align}
and consequently, with the aid of \eqref{chap:Ck,sec:genMar,eq:polynomialfull},
\begin{align*}
&\int_{\mathbb{R}^{d}}\|A^{iu}\|_{\pp}\,\sup_{\TT\in \Rdp}|\M(m_{N,\TT})(u)|\,du\\
&\leq C_{N,\rho}\,\mC(p,A) \|m\|_{Mar,\rho}\int_{\mathbb{R}^d}\prod_{r=1}^d (1+|u_r|)^{-\rho_r+\sigma_r\,|1/p -1/2|}\,du \leq  C_{\rho}\,\mC(p,A)\|m\|_{Mar,\rho}.
\end{align*}
Thus, to finish the proof of Theorem \ref{thm:genMarCk} it remains to show \eqref{chap:Ck,sec:genMar,eq:claim}.

From the change of variable $2^{j}\la\to \la$ it follows that
\begin{align*}
&|I_{j,N,\gamma,\delta}(t)|=2^{-\langle j,N+\rho-\gamma-\delta\rangle }\left|\int_{[1/2,2]^{d}}\la^{N+\rho-\gamma-iu}\exp(-2^{-1}\langle {\bf 2}^{-j},\la\rangle)\right. \\
&\times\left.\left(\frac{\la_1}{2^{j_1}t_1},\ldots,\frac{\la_d}{2^{j_d}t_d}\right)^{\gamma}\partial^{\gamma}(m)\left(\frac{\la_1}{2^{j_1}t_1},\ldots,\frac{\la_d}{2^{j_d}t_d}\right)\partial^{\delta}(\Psi)(\la)\,\frac{d\la}{\la}\right|.
\end{align*}
Thus, applying Schwarz's inequality we obtain
\begin{equation}
\label{chap:Ck,sec:genMar,eq:estI}
\begin{split}
&|I_{j,N,\gamma,\delta}(t)|\leq C_{\Psi} 2^{-\langle j,N+\rho-\gamma-\delta\rangle} \left(\int_{[1/2,2]^d}\left|\la^{N+\rho-\gamma}\exp(-2^{-1}\langle{\bf 2}^{-j},\la\rangle)\right|^2\,\frac{d\la}{\la}\right)^{1/2}\\
&\times \left(\int_{[1/2,2]^d}\left|\left(\frac{\la_1}{2^{j_1}t_1},\ldots,\frac{\la_d}{2^{j_d}t_d}\right)^{\gamma}\partial^{\gamma}(m)\left(\frac{\la_1}{2^{j_1}t_1},\ldots,\frac{\la_d}{2^{j_d}t_d}\right)\right|^2\,\frac{d\la}{\la} \right)^{1/2}.
\end{split}
\end{equation}
Moreover, it is not hard to see that
\begin{align}
\nonumber
&\left(\int_{[1/2,2]^d}\left|\la^{N+\rho-\gamma}\exp(-2^{-1}\langle{\bf 2}^{-j},\la\rangle)\right|^2\,\frac{d\la}{\la}\right)^{1/2}\\ \nonumber
&=\left(\prod_{r=1}^{d}\int_{[1/2,2]}\left|\la_r^{N_r+\rho_r-\gamma_r}\exp(-2^{-j_r-1}\la_r)\right|^2\,\frac{d\la_r}{\la_r}\right)^{1/2}\\
\label{chap:Ck,sec:genMar,eq:estint}
&\leq C_{N,\rho,\gamma}\prod_{r=1}^{d} \begin{cases}
 1, &\mbox{if $j_r>0$}, \\
  \exp(-2^{-j_r-1}), &\mbox{if $j_r\leq0.$}
       \end{cases}
\end{align}

Now, coming back to \eqref{chap:Ck,sec:genMar,eq:estI}, we use the assumption that $m$ satisfies the Marcinkiewicz condition of order $\rho$ together with \eqref{chap:Ck,sec:genMar,eq:estint} (recall that $\gamma+\delta\leq \rho< N$) to obtain \eqref{chap:Ck,sec:genMar,eq:claim}. The proof of Theorem \ref{thm:genMarCk} is thus finished.
\end{proof}

        \section[Examples of systems with Marcinkiewicz joint functional calculus]{Examples of systems with Marcinkiewicz joint functional calculus}
        \label{chap:Ck,sec:ExOp}
        We shall now focus on particular examples of systems of operators, which have a Marcinkiewicz functional calculus. The present section is not completely formal and can rather be regarded as a collection of these systems.

        In view of Theorem \ref{thm:genMarCk} and Corollary \ref{cor:genMarCkequiv} a system $A=(A_1,\ldots,A_d)$ has a Marcinkiewicz functional calculus if and only if the imaginary powers $A_r^{iu_r},$ $u_r\in\mathbb{R},$ $r=1,\ldots,d,$ are growing at most polynomially in $L^p$ norm, i.e.\ satisfy \eqref{chap:Ck,sec:genMar,eq:polynomial} for every $1<p<\infty.$ Thus we start by seeking operators which have polynomially growing imaginary powers.

        The classical example of this type is the Laplacian, $-\Delta=-\sum_{r=1}^d\partial_r^2.$ It is well known that the imaginary powers $(-\Delta)^{iv},$ $v\in\mathbb{R},$ are Calder\'on-Zygmund operators. Moreover, from the Calder\'on-Zygmund theory, it is not hard to deduce that
        $$\|(-\Delta)^{iv}\|_{L^p(\mathbb{R}^d,dx)\to L^p(\mathbb{R}^d,dx)}\leq C_p (1+|v|)^{d/2+1},\qquad v\in\mathbb{R}.$$ In fact, by varying the level of the Calder\'on-Zygmund decomposition and optimizing (see \cite{SikWr}), it is possible to improve the above inequality, obtaining
        \begin{equation}
        \label{chap:Ck,sec:ExOp,eq:pollLap}
        \|(-\Delta)^{iv}\|_{L^p(\mathbb{R}^d,dx)\to L^p(\mathbb{R}^d,dx)}\leq C_p (1+|v|)^{d/2},\qquad v\in\mathbb{R}.
        \end{equation}
        Using \eqref{chap:Ck,sec:ExOp,eq:pollLap} (for $d=1$), together with the boundedness of the Hilbert transform and Theorem \ref{thm:genMarCk}, we can reobtain Theorem \ref{thm:Marmult} (though without the sharp smoothness threshold).

        In \cite{SikWr} the authors proved that a bound similar to \eqref{chap:Ck,sec:ExOp,eq:pollLap} remains true in a much broader generality. Namely, assume that $(X,\mu,\zeta)$ is a space of homogenous type, with $\zeta$ being a (quasi) metric, and let $A$ be a self-adjoint non-negative operator on $L^2(X,\mu).$ Then, from \cite[Theorems 1 and 2]{SikWr} it follows that under certain assumption on the Gaussian bounds for the heat semigroup $\{e^{-tA}\}_{t>0},$ for every $1<p<\infty$ we have
        \begin{equation}\label{chap:Ck,sec:ExOp,eq:pollSik}\|A^{iv}\|_{L^p(X,\mu)\to L^p(X,\mu)}\leq C (1+|v|)^{d|1/p-1/2|}, \qquad v\in\mathbb{R}.\end{equation} Here we also need to assume that the measure of a ball $B_{\zeta}(x,R),$ $R>0,$ is proportional to $R^d,$ uniformly with respect to $x\in X.$

        Using the Feynman-Kac formula it can be shown that the Gaussian bounds needed in \cite[Theorem 2]{SikWr} are valid for the heat kernels of Schr\"odinger operators with non-negative potentials. Thus, from \eqref{chap:Ck,sec:ExOp,eq:pollSik} it follows that the imaginary powers $(-\Delta+V)^{iv},$ $v\in\mathbb{R},$ with $V\geq 0,$ satisfy
        \begin{equation}\label{chap:Ck,sec:ExOp,eq:pollSchr}\|(-\Delta+V)^{iv}\|_{L^p(\mathbb{R}^d,dx)\to L^p(\mathbb{R}^d,dx)}\leq (1+|v|)^{d|1/p-1/2|},\qquad v\in\mathbb{R}.\end{equation}
        From \eqref{chap:Ck,sec:ExOp,eq:pollSchr} and Theorem \ref{thm:genMarCk} we obtain multiplier theorems for systems of one-dimensional Schr\"odinger operators on $L^2(\mathbb{R}^d,dx),$ given by $A_r=-\partial_r^2+V_r(x_r),$ $V_r\geq 0,$ $r=1,\ldots,d.$ In particular, this applies to the system of harmonic oscillators $-\partial_r^2+x_r^2,$ $r=1,\ldots,d,$ thus giving a multivariate generalization of Thangavelu's multiplier theorem in the framework of Hermite expansions, \cite[Theorem 1]{Thanherm}.

        Another consequence of Theorem \ref{thm:genMarCk} and \cite[Theorem 2]{SikWr} is a multivariate multiplier theorem for Laguerre expansions of Hermite type. The relevant operators in this setting are self-adjoint extensions of the commuting operators \begin{equation}\label{chap:Ck,sec:ExOp,eq:LagHermop}-\partial_r^2+x_r^2+(\alpha_r^2-1/4)x_r^{-2},\qquad r=1,\ldots,d,\end{equation}
        defined on $C_{c}^{\infty}(\Rdp)\subseteq L^2(\Rdp,dx);$ here we assume $\alpha=(\alpha_1,\ldots,\alpha_d)\in[-1/2,\infty)^d.$ Note that if $\alpha_r\in (-1/2,1/2),$ $r=1,\ldots,d,$ the potential term in \eqref{chap:Ck,sec:ExOp,eq:LagHermop} is not positive. However, the relevant gaussian bounds for the heat semigroup are still valid, see \cite[Proposition 2.1]{NSRieszLagHerm}, thus it is possible to apply \cite[Theorem 2]{SikWr}.

        In some cases, even though the results of \cite{SikWr} are not applicable, it is possible to obtain polynomial bounds for the $L^p$ norms of the imaginary powers of certain operators simply by enhancing existing estimates. This method is used in Section \ref{chap:Ck,sec:BoundSome,subsec:Jac} to study the imaginary powers of the Jacobi operator.

        A different way to obtain multivariate multiplier theorems is to use directly Corollary \ref{cor:genMarCkequiv}. For example, from Corollary \ref{cor:genMarCkequiv} together with
        \cite[Theorem 1.1]{GosSt1}
        we obtain
        a multivariate Marcinkiewicz type multiplier theorem for the Hankel transform (this setting is connected with the system of Bessel operators).
        Similarly, using Corollary \ref{cor:genMarCkequiv} and
        \cite[Theorem 3.1]{Sifi1},
        we prove
        a multivariate Marcinkiewicz type multiplier theorem for the Dunkl transform in the simplest setting of a reflection group $G$ isomorphic to $\mathbb{Z}_2^d.$
        The two just announced
        multiplier results are not sharp. We shall see how to improve them in Section~\ref{chap:Ck,sec:MarDun}.

        \section[$L^p$ bounds of the imaginary powers of the Jacobi operator]{$L^p$ bounds of the imaginary powers of the Jacobi operator}
        \label{chap:Ck,sec:BoundSome,subsec:Jac}
        We prove polynomial bounds for the norms on $L^p$ of the imaginary powers of the Jacobi operator. This is achieved by improving some estimates from \cite{NSj1}. Since, for the most part, we use the notation from \cite{NSj1}, we refer the reader to the latter article for all symbols that are not properly explained here. Combining the bounds from this section with Theorem \ref{thm:genMarCk} immediately leads to a multivariate multiplier theorem for the expansions in Jacobi trigonometric polynomials. The results presented in this section are contained in \cite[Section 4]{cit:ja}.

        Analogous methods were also used to prove a multivariate multiplier theorem for the system of Dunkl harmonic oscillators, see \cite[Section 4]{jaDun}. Since the ideas of the proofs of Proposition \ref{prop:Jaa} and
        \cite[Lemma 4.2]{jaDun} are similar, and the proof of the latter is more  technical, we decided not to include the results for the system of Dunkl harmonic oscillators in this thesis.

          For parameters $\alpha,\beta>-1,$ the Jacobi operator is defined on $C_c^{\infty}(0,\pi)$ by $$\mathcal{J}^{\alpha,\beta}=-\frac{d}{d\theta^2}-\frac{\alpha-\beta+(\alpha+\beta+1)\cos \theta}{\sin \theta}\frac{d}{d\theta}+\left(\frac{\alpha+\beta+1}{2}\right)^2.$$ Then $\mathcal{J}^{\alpha,\beta}$ is symmetric on $L^2((0,\pi),\mu_{\alpha,\beta}),$ with the measure $$d\mu_{\alpha,\beta}(\theta)=\bigg(\sin \frac{\theta}{2}\bigg)^{2\alpha+1}\bigg(\cos \frac{\theta}{2}\bigg)^{2\beta+1}\,d\theta.$$
          Note that the results of \cite{SikWr} cannot be applied for the Jacobi operator, as the measure $\mu_{\alpha,\beta}$ does not satisfy the assumptions of \cite[Theorem 2]{SikWr}.

         Till the end of this section we assume that $\alpha+\beta>-1.$ Then the operator $\mathcal{J}^{\alpha,\beta}$ has a natural self-adjoint positive extension given in terms of Jacobi trigonometric polynomials. Indeed, let $\{\Pj_k^{\alpha,\beta}\}_{k\in\mathbb{N}_0}$ be the system of $L^2((0,\pi),\mu_{\alpha,\beta})$-normalized Jacobi trigonometric polynomials, as defined in \cite[pp.\ 4-5]{NSj0}. Then, for each $k\in\mathbb{N}_0,$
        \begin{equation}
        \label{chap:Ck,sec:BoundSome,subsec:Jac,eq:eigen}
        \mathcal{J}^{\alpha,\beta}\Pj_k^{\alpha,\beta}=\la_{k}^{\alpha,\beta}\Pj_k^{\alpha,\beta},\qquad \la_{k}^{\alpha,\beta}=\big(k+\frac{\alpha+\beta+1}{2}\big)^2.
        \end{equation}

        The above equation allows us to define, for a bounded function $m$ on $\{\la_{k}^{\alpha,\beta}\colon k\in \mathbb{N}_0\},$ the multiplier operator $m(\mathcal{J}^{\alpha,\beta})$ in terms of the discrete spectral series
        \begin{equation}\label{chap:Ck,sec:BoundSome,subsec:Jac,eq:discspect} m(\mathcal{J}^{\alpha,\beta})f=\sum_{k\in\mathbb{N}_0}m(\la_{k}^{\alpha,\beta})\langle f,\Pj_k^{\alpha,\beta}\rangle_{L^2((0,\pi),\mu_{\alpha,\beta})} \Pj_k^{\alpha,\beta},\qquad f\in L^2((0,\pi),\mu_{\alpha,\beta}).\end{equation}
        In particular, putting $m_v(\la)=\la^{iv},$ $v\in\mathbb{R},$ in \eqref{chap:Ck,sec:BoundSome,subsec:Jac,eq:discspect}, we obtain the imaginary powers \begin{equation}\label{chap:Ck,sec:BoundSome,subsec:Jac,eq:imapow} I^{\alpha,\beta}_{v}f:=\sum_{k\in\mathbb{N}_0}\bigg(k+\frac{\alpha+\beta+1}{2}\bigg)^{2iv}\langle f,\Pj_k^{\alpha,\beta}\rangle_{L^2((0,\pi),\mu_{\alpha,\beta})} \Pj_k^{\alpha,\beta},\qquad f\in L^2((0,\pi),\mu_{\alpha,\beta}).\end{equation}

        In this case, for quite a long time, one of the few information on the Jacobi-heat semigroup $\{\exp(-t\mathcal{J}^{\alpha,\beta})\}_{t>0},$ was the fact that it is positivity preserving and contractive on all $L^p((0,\pi),\mu_{\alpha,\beta}),$ $1\leq p\leq \infty,$ i.e.\ it satisfies \eqref{chap:Intro,sec:Setting,eq:contra}. There were no useful estimates for the kernel of the Jacobi-heat semigroup; in fact Gaussian bounds for this kernel were obtained only recently, see \cite{CouKerPetr1} and \cite{NSj1}. However, as far as the imaginary powers are concerned, it is perfectly enough to use the Jacobi-Poisson semigroup $\{\exp(-t(\mathcal{J}^{\alpha,\beta})^{1/2}\}_{t>0}$ instead of the Jacobi-heat semigroup. Indeed, for $v\in\mathbb{R},$ we formally have
        $$(\mathcal{J}^{\alpha,\beta})^{-iv}=\frac{1}{\Gamma(iv)}\int_0^{\infty}\exp\big(-t\mathcal{J}^{\alpha,\beta}\big)t^{iv}\frac{dt}{t}
        =\frac{1}{\Gamma(2iv)}\int_0^{\infty}\exp\big(-t(\mathcal{J}^{\alpha,\beta})^{1/2}\big)t^{2iv}\frac{dt}{t}.$$
        Due to the non linearity of the eigenvalues in \eqref{chap:Ck,sec:BoundSome,subsec:Jac,eq:eigen}, in the case of the Jacobi operator it is easier to work with the Poisson semigroup than with the heat semigroup. Note however, that in the case of the Dunkl harmonic oscillator it is better to use its heat semigroup instead.

        For $\alpha=\beta=(d-3)/2,$ $d\geq 2,$ using the connection between the Jacobi (Gegenbauer) polynomials and the spherical Laplacian on the $d$-dimensional unit sphere, we can obtain polynomial bounds for the norms of the Jacobi imaginary power operators, see the remarks at the end of \cite{SikWr}. Denote $\log_+ (x)=1+\max(\log(x),0),$ $x>0.$ In the general case, we have the following.

\begin{pro}
\label{prop:Jaa}
Assume that $\alpha\geq -1/2,$ $\beta\geq -1/2,$ and $\alpha+\beta>-1.$ Let $\Ija,$ $v\in\mathbb{R},$ $v\neq0,$  be the imaginary powers of $\mathcal{J}^{\alpha,\beta},$ given by \eqref{chap:Ck,sec:BoundSome,subsec:Jac,eq:imapow}. Then, for every $1<p<\infty,$
\begin{align*}&\|\Ija\|_{L^p((0,\pi),\mu_{\alpha,\beta})\to L^p((0,\pi),\mu_{\alpha,\beta})}\leq C_{\alpha,\beta,p} (\maxlogv)^{|2/p-1|}(1+|v|)^{(4\alpha+4\beta+11)|1/p-1/2|},\end{align*}
uniformly in $v\in\mathbb{R}.$
\end{pro}
\begin{proof} To prove the proposition we will use the subtle estimates obtained by Nowak and Sj\"{o}gren in \cite{NSj0} and some complex variable techniques as in \cite{funccalOu}.

In \cite{NSj0} the authors proved that the operator $\Ija$ is a Calder\'{o}n-Zygmund operator associated with the kernel
\begin{equation}\label{chap:Ck,sec:BoundSome,subsec:Jac,eq:calzygker} K_{v}^{\alpha,\beta}(\theta,\varphi)=c_{\alpha,\beta}\frac{1}{\Gamma(2iv)}\int_0^{\infty}t^{2iv-1}\sinh \frac{t}{2}\iint\,\frac{d\Pi_{\alpha}(s_1)d\Pi_{\beta}(s_2)}{(\cosh \frac{t}{2}-1+q)^{\alpha+\beta+2}}\,dt,\end{equation}
see \cite[(12) p.\ 9 and Proposition 4.1]{NSj0}. Here $q:=q(\theta,\varphi,s_1,s_2)$ with $$q(\theta,\varphi,s_1,s_2)=1-s_1\sin\frac{\theta}{2}\sin\frac{\varphi}{2}-s_2\cos\frac{\theta}{2}\cos\frac{\varphi}{2},$$ for $\theta,\varphi\in (0,\pi)$ and $s_1,s_2\in [-1,1]$; while $\Pi_{\alpha}$ denotes the probability measure
 $$d\Pi_{\alpha}(u)=\frac{\Gamma(\alpha+1)}{\sqrt{\pi}\Gamma(\alpha+1/2)}(1-u^2)^{\alpha-1/2},\qquad u\in[-1,1],$$
 if $\alpha>-1/2,$ or the sum of Dirac deltas $(\delta_{-1}+\delta_{1})/2,$ if $\alpha=-1/2.$ For further reference, we note that $0\leq q\leq2$ (see \cite[(21) p.\ 21]{NSj0}) and $|\partial_{\theta}q|\leq 2.$

Thus, by referring to the Calder\'{o}n-Zygmund theory and the Marcinkiewicz interpolation theorem we see that it suffices to show that the Calder\'{o}n-Zygmund constant $D_{v}$ in the smoothness condition for the kernel $K_{v}^{\alpha,\beta}(\theta,\varphi),$ is less than a constant times $\Cze.$ Since the function $v\rightarrow \frac{1}{\Gamma(2iv)}$ is continuous, looking closely at the proof of the smoothness bound from \cite[p.\ 25]{NSj0} we see that it suffices to focus on $|v|>2.$ For symmetry reasons we may consider only the derivatives with respect to the variable $\theta.$ From \eqref{chap:Ck,sec:BoundSome,subsec:Jac,eq:calzygker} it follows that
\begin{align*} \Gamma(2iv)\frac{\partial}{\partial\theta}\Ija=\int_0^{\infty}t^{2iv-1}\sinh \frac{t}{2}\iint\,d\Pi_{\alpha}(s_1)d\Pi_{\beta}(s_2)\frac{\partial_{\theta}q}{(\cosh \frac{t}{2}-1+q)^{\alpha+\beta+3}}\,dt\equiv J_{v}.\end{align*}

Assume $v>1$ and set $\frac{\pi}{3}<\phi<\frac{\pi}{2}.$ We consider the function $$h(z):=z^{2iv-1}\sinh \frac{z}{2}\iint\,d\Pi_{\alpha}(s_1)d\Pi_{\beta}(s_2)\frac{\partial_{\theta}q}{(\cosh \frac{z}{2}-1+q)^{\alpha+\beta+3}},$$ which is holomorphic in the right half-plane. Then we change the contour of integration in the integral formula for $J_{v}$ to $(e^{i\phi}0,e^{i\phi}\infty)$ (the other integrals can be easily seen to vanish), where for an angle $\phi,$ by $(0e^{i\phi},\infty e^{i\phi})$ we mean the ray $\{te^{i\phi}\,:\,t\geq0\}.$ Of course $h$ is not holomorphic at zero, but this difficulty can be easily overcome by a limiting process. Using the parametrization $\gamma(t)=te^{i\phi},$ $\gamma'(t)=e^{i\phi},$ we get
\begin{align*} |J_{v}|&\leq e^{-2v\phi}\int_0^{\infty}t^{-1}\big|\sinh \frac{\gamma(t)}{2}\big|\iint\,d\Pi_{\alpha}(s_1)d\Pi_{\beta}(s_2)\frac{|\partial_{\theta}q|}{|\cosh \frac{\gamma(t)}{2}-1+q|^{\alpha+\beta+3}}\,dt\\&= e^{-2v\phi} \bigg(\int_0^1\ldots+\int_1^{\frac{2}{\cos \phi}}\ldots+\int_{\frac{2}{\cos \phi}}^{\infty}\bigg)\equiv e^{-2v\phi} (J_1+J_2+J_3).\end{align*}
In the reasoning below the symbol $\lesssim$ indicates that the constant in the inequality is independent of both $v$ and $\phi;$ it may however depend on $\alpha$ or $\beta.$

We start by estimating $J_1.$ Let
\begin{align}
\label{chap:Ck,sec:BoundSome,subsec:Jac,eq:cosh}
A(t)=\cosh\frac{\gamma(t)}{2}=\cosh \frac{t\cos\phi}{2}\cos\frac{t\sin\phi}{2}+i\sinh\frac{t\cos\phi}{2}\sin\frac{t\sin\phi}{2},
\end{align}
and let $B(t)=A(t)-1+q.$ Clearly, since $\sin\phi>1/2$ and $\sinh x>\frac{1}{2}x,$ $\sin x>\frac{1}{2}x,$ for $0<x\leq 1,$ we see that $\Imm(B(t))\geq (\cos \phi)\, t^{2}/4,$ for $0<t\leq1.$ Consequently, if $t^2/16> \frac{1}{2}q,$ then $\cos \phi\,(t^2+q)\lesssim|B(t)|.$ On the other hand, if $t^2/16\leq \frac{1}{2}q,$ then since $1-x^2\leq\cos x$ and $(\sin \phi)^2\geq 1/2,$ we see that $\Ree(B(t))\geq -t^2/16+q \geq \frac{1}{2}q.$ Hence, in either case $\cos \phi\,(t^2+q)\lesssim|B(t)|.$ From the latter inequality, together with the bound $|\gamma(t)\sinh(\gamma(t)/2)|\lesssim1,$ $0<t<1,$ we get \begin{equation}\label{chap:Ck,sec:BoundSome,subsec:Jac,eq:io}J_1\lesssim (\cos\phi)^{-\alpha-\beta-3} \int_0^{1}\iint\,d\Pi_{\alpha}(s_1)d\Pi_{\beta}(s_2)\frac{\partial_{\theta}q}{(t^{2}+q)^{\alpha+\beta+3}}\,dt.\end{equation}

Now we pass to the estimation of $J_2.$ Let $\delta^2=\cosh(\cos\phi/2).$ From \eqref{chap:Ck,sec:BoundSome,subsec:Jac,eq:cosh}, the inequality $|-1+q|\leq1$ and the fact that $1+(\cos\phi/2)^2/2\lesssim \delta^2\approx 1,$ we see that if $\cos(t\sin\phi/2)\geq\frac{1}{\delta},$ then $|\textrm{Re}(B(t))|\geq C (\cos(\phi))^2;$ whereas if $\cos(t\sin\phi/2)<\frac{1}{\delta},$ then $|\textrm{Im}(B(t)|\geq C (\cos(\phi))^2.$ In either case, $|B(t)|\geq C(\cos\phi)^2,$ so that since $|\sinh(\gamma(t)/2)|\lesssim 1,$ for $t\in(1,\frac{2}{\cos\phi}),$ and $q,$ $\partial_{\theta}q$ are bounded we obtain \begin{equation}\label{chap:Ck,sec:BoundSome,subsec:Jac,eq:it}J_2\lesssim -\log(\cos\phi)(\cos\phi)^{-2\alpha-2\beta-6}\iint\,\frac{d\Pi_{\alpha}(s_1)d\Pi_{\beta}(s_2)}{q^{\alpha+\beta+2}}.\end{equation}
We now focus on $J_3.$ Since, $|A(t)|^2=(\cosh(t\cos\phi/2))^2-(\sin(t\sin\phi/2))^2$ we see that for $t\geq\frac{2}{\cos\phi},$ $|B(t)|\geq C e^{t\cos \phi/2}.$ Similarly, for such $t$ we have $$|\sinh(\gamma(t)/2)|^2=(\sinh(t\cos\phi/2))^2+(\sin(t\sin\phi/2))^2\lesssim e^{t\cos\phi}.$$ Combining the latter two inequalities together with the boundedness of $q,$ $\partial_{\theta}q,$ and the change of variable $t\cos\phi\rightarrow t,$ we obtain
\begin{equation}\label{chap:Ck,sec:BoundSome,subsec:Jac,eq:ir}J_3\lesssim \int_1^{\infty}e^{-t/2(\alpha+\beta+2)}\iint\,\frac{d\Pi_{\alpha}(s_1)d\Pi_{\beta}(s_2)}{q^{\alpha+\beta+2}}\,dt.\end{equation}
\indent From the estimates \eqref{chap:Ck,sec:BoundSome,subsec:Jac,eq:io}, \eqref{chap:Ck,sec:BoundSome,subsec:Jac,eq:it} and \eqref{chap:Ck,sec:BoundSome,subsec:Jac,eq:ir}, proceeding as in \cite[Section 4]{NSj0} (in the case of the bound \eqref{chap:Ck,sec:BoundSome,subsec:Jac,eq:io} we need to refer to \cite[Lemma 4.5]{NSj0}), for $\frac{\pi}{4}<\phi<\frac{\pi}{2},$ we derive the bound independent of $v>1,$
\begin{equation*} |J_{v}|\lesssim -\log(\cos\phi)(\cos\phi)^{-2\alpha-2\beta-6}\iint\,\frac{d\Pi_{\alpha}(s_1)d\Pi_{\beta}(s_2)}{q^{\alpha+\beta+2}}.\end{equation*}
Now, from \cite[Lemma 4.3]{NSj0} it follows that the Calder\'{o}n-Zygmund constant $D_{v}$ in the smoothness condition for the kernel $K_{v}^{\alpha,\beta}$ satisfies \begin{equation}\label{chap:Ck,sec:BoundSome,subsec:Jac,eq:CZbound} D_{v}\lesssim -\frac{e^{-2v\phi}}{\Gamma(2iv)}\log(\cos\phi)(\cos\phi)^{-2\alpha-2\beta-6}.\end{equation}
Since $(1+v)^{1/2}e^{-\pi v}\lesssim|\Gamma(2iv)|$ (see for instance \cite[Chapter 1]{2}), taking $\phi=\arctan(v)$ and using \eqref{chap:Ck,sec:BoundSome,subsec:Jac,eq:CZbound} we finally arrive at $$D_{v}\lesssim \log (v)(1+v)^{2\alpha+2\beta+\frac{11}{2}}\exp\left(v(\pi-2\arctan v)\right),$$
which easily produces the desire bound in the case $v>0.$ The reasoning in the case $v<-2$ is analogous, we need only to replace the ray $(0e^{i\phi},\infty e^{i\phi})$ by $(0e^{-i\phi},\infty e^{-i\phi}).$\end{proof}

In the end of this section we briefly introduce the multivariate multiplier theorem for the Jacobi expansions which is a consequence of Proposition \ref{prop:Jaa}. Let $$\Pj_k^{\alpha,\beta}(\theta)=\Pj_{k_1}^{\alpha_1,\beta_1}(\theta_1)\ldots\Pj_{k_d}^{\alpha_d,\beta_d}(\theta_d),\qquad \theta=(\theta_1,\ldots,\theta_d)\in (0,\pi)^d,\qquad k\in\mathbb{N}^d_0,$$ with $\alpha=(\alpha_1,\ldots,\alpha_d),$ $\beta=(\beta_1,\ldots,\beta_d),$ $\alpha+\beta>-\bf{1},$ be the system of $d$-dimensional Jacobi trigonometric polynomials. Then $\{\Pj^{\alpha,\beta}_k\}_{k\in\mathbb{N}^d_0}$ forms an orthonormal basis in $L^2:=L^2(((0,\pi)^d,d\mu_{\alpha,\beta})),$ with $d\mu_{\alpha,\beta}(\theta):=d\mu_{\alpha_1,\beta_1}(\theta_1)\times\cdots\times d\mu_{\alpha_d,\beta_d}(\theta_d).$ For a bounded function $m$ on $\Rdp$ we define the multiplier operator $m(\mathcal{J}^{\alpha,\beta})$ of the system of the one-dimensional Jacobi operators $\mathcal{J}^{\alpha,\beta}=(\mathcal{J}^{\alpha_1,\beta_1},\ldots,\mathcal{J}^{\alpha_d,\beta_d})$ by
\begin{equation}
\label{chap:Ck,sec:BoundSome,subsec:Jac,eq:multJac}
m(\mathcal{J}^{\alpha,\beta})f=\sum_{k\in\mathbb{N}^d_0}m(\la_{k_1}^{\alpha_1,\beta_1},\ldots,\la_{k_d}^{\alpha_d,\beta_d})\langle f, \Pj_{k}^{\alpha,\beta}\rangle_{L^2}\Pj_k^{\alpha,\beta},\qquad f\in L^2.
\end{equation}

Using Proposition \ref{prop:Jaa} together with Theorem \ref{thm:genMarCk} we get the following.

\begin{cor}
\label{cor:Corjaa}
Let $\alpha+\beta>-\bf{1}.$ Assume that $m$ satisfies the Marcinkiewicz condition \eqref{chap:Intro,sec:Notation,eq:Marcon} of order $\rho>|1/p-1/2|\,(4\alpha+4\beta+\bf{11})+\bf{1},$ with $1<p<\infty.$ Then the multiplier operator $m(\mathcal{J}^{\alpha,\beta})$ given by \eqref{chap:Ck,sec:BoundSome,subsec:Jac,eq:multJac} is bounded on $L^p((0,\pi)^d,d\mu_{\alpha,\beta}).$ In particular, if $\rho>2\alpha+2\beta+\bf{\frac{13}{2}},$ then $m(\mathcal{J}^{\alpha,\beta})$ is bounded on all $L^p((0,\pi)^d,d\mu_{\alpha,\beta})$ spaces, $1<p<\infty.$
\end{cor}
\begin{remark}
Note that the joint spectrum $\sigma(\mathcal{J}^{\alpha,\beta})$ of the system $\mathcal{J}^{\alpha,\beta}$ equals the discrete set $\{(\la_{k_1}^{\alpha_1,\beta_1},\ldots,\la_{k_d}^{\alpha_d,\beta_d})\colon k\in \mathbb{N}_0^{d}\}.$ Thus, in order to be rigorous, in Corollary \ref{cor:Corjaa} we should assume that $m$ is initially defined on $\sigma(\mathcal{J}^{\alpha,\beta}),$ and extends to a function defined on all of $\Rdp,$ that satisfies an appropriate Marcinkiewicz condition.
\end{remark}

        \section[H\"ormander type multipliers for the Hankel transform]{H\"ormander type multipliers for the Hankel transform}
        \label{chap:Ck,sec:HorHan}

        We prove a multivariate H\"ormander type multiplier theorem for the multi-dimensional Hankel transform. The present section together with Section \ref{chap:CkHinf,sec:OA} are the only ones in the thesis in which weak type (1,1) results are given. Note that we heavily exploit here the translation and dilation structure connected with the Hankel transform setting. The results presented in this section constitute a part of our joint article with Dziuba\'nski and Preisner, \cite{dpw}. The paper contains also $H^1$ results, however for the sake of homogeneity of the exposition we decided not to include them in the thesis.

        Throughout this section for a multi-index $\alpha =(\alpha_1,...,\alpha_d)$, $\alpha_r>-1\slash 2$, we consider the measure space $X=(\Rdp, d\nu_{\alpha})$, where  $$d\nu_{\alpha}(x)=d\nu_{\alpha_1}(x_1)\cdots d\nu_{\alpha_d}(x_d),\qquad d\nu_{\alpha_r}(x_r)=x_r^{2\alpha_r}\,dx_r, \quad r=1,\ldots,d.$$ The space $X$ equipped with the Euclidean distance is a space of homogeneous type. For the sake of brevity we write $L^p$ and $\|\cdot\|_p$ instead of $L^p(X)$ and $\|\cdot\|_{L^p(X)},$ $1\leq p\leq\infty.$ If $T$ is a linear operator, the symbol $\|T\|_{p\to p}$ denotes the norm of $T$ acting on $L^p.$

For a  function $f\in L^1,$ the multi-dimensional (modified) Hankel transform  is defined by
\begin{equation}
\label{chap:Ck,sec:HorHan,eq:HanTrandef}
\Ha_{\alpha}(f)(x)=\int_{\Rdp} f(\la) E_{x}(\la)\,d\nu_{\alpha}(\la),
\end{equation}
where $$E_{x}(\la)=\prod_{r=1}^d (x_r \la_r)^{-\alpha_r+1/2}J_{\alpha_r-1/2}(x_r \la_r)=\prod_{r=1}^dE_{x_r}(\la_r).$$ Here $J_{s}$ is the Bessel function of the first kind of order $s>-1,$ see \cite[Chapter 5]{2}. For the sake of brevity we usually write $\Ha$ instead of $\Hla$ and $\nu$ instead of $\nu_{\alpha}.$  The system $\{E_{x}\}_{x\in\Rdp}$ consists of the eigenvectors of the Bessel operator
\begin{equation*}
L_B=-\Delta-\sum_{r=1}^d \frac{2\alpha_r}{\la_r}\frac{\partial}{\partial \la_r};
\end{equation*}
that is, $L_B(E_{x})=|x|^2E_{x}.$ Also, the functions $E_{x_r},$ $r=1,\ldots,d$, are eigenfunctions of the one-dimensional Bessel operators
$$L_{r}=-\frac{\partial^2}{\partial {\la_r}^2}-\frac{2\alpha_r}{\la_r}\frac{\partial}{\partial \la_r},$$
namely, $L_{r}(E_{x_r})=x_r^2 E_{x_r}.$

It is known that $\Ha$ is an isometry on $L^2$ that satisfies $\Ha^{-1}=\Ha$ (see, e.g., \cite[Chapter 8]{T}).  Moreover, for  $f\in L^2,$ we have \begin{equation}\label{chap:Ck,sec:HorHan,eq:diago}L_r(f)=\Ha(\la_r^2\Ha f).\end{equation}

For $y\in X$ let $\tau^{y}$ be the $d$-dimensional generalized Hankel translation given by
$$ \Ha (\tau^yf)(x)=E_y(x)\Ha f(x).$$
Clearly,
$\tau^y f (x)=\tau^{y_1}\cdots\tau^{y_d}f (x)$, where for each  $r=1,\ldots,d$, the operator  $\tau^{y_r}$  is the one-dimensional Hankel translation acting on a function $f$ as a function of the  $x_r$  variable with the other variables fixed.  It is also known that, if $\alpha_r>0,$ for $r=1,\ldots,d,$ then $\tau^y$ is a contraction on all $L^p$ spaces, $1\leq p\leq\infty,$ and that
$$ \tau^yf(x)=\tau^xf(y).$$
For two reasonable functions $f$ and $g$ define their Hankel convolution as
$$f \natural g(x)= \int_{X}\, \tau^{x}f(y) g(y)\,d\nu(y).$$
It is not hard to check that $f\natural g=g\natural f$ and
\begin{equation}\label{chap:Ck,sec:HorHan,eq:convo}\Ha(f\natural g)(x)=\Ha f(x)\,\Ha g(x).\end{equation}
As a consequence of the contractivity of $\tau^{y},$ for $\alpha\in \Rdp,$ we also have
\begin{equation}
\label{chap:Ck,sec:HorHan,eq:young}
\|f\natural g\|_{1}\leq \|f\|_{1} \|g\|_1, \qquad f\in L^1,\quad g\in L^1.
\end{equation}

For details concerning translation, convolution, and transform in the Hankel setting we refer the reader to, e.g., \cite{Ha}, \cite{T}, and \cite{W}.

For a function $f\in L^1$ and $t>0,$ let $f_t$ denote the $L^1$-dilation of $f$ given by $$(f_t)(x)={t}^{Q}f(tx),$$ where $Q=\sum_{r=1}^d (2\alpha_r+1)$. Then we have:
\begin{equation}
\Ha (f_t)(x)=\Ha f(t^{-1} x),\qquad \label{chap:Ck,sec:HorHan,eq:dil+tra}\tau^{y}(f_t)(x)=(\tau^{ty}f)_t(x).
\end{equation}
Notice that $Q$ represents the dimension of $X$ at infinity, that is, $\nu(B(x,R))\sim R^Q$ for large $R.$

 Let $m\colon \Rdp \to\mathbb{C}$ be a bounded measurable function. Define the multiplier operator $\mathcal T_m$ by
\begin{equation}
\label{chap:Ck,sec:HorHan,eq:spHan}
\mathcal T_m^{\Hi}(f)=\mathcal T_m(f)=\Ha(m\Ha f).
\end{equation}
Clearly, $\mathcal T_m$ is bounded on $L^2.$ Also note that if $m(\la_1,\ldots,\la_d)=n(\la_1^2,\ldots,\la_d^2),$ for some bounded, measurable function $n$ on $\mathbb{R}^d,$ then from \eqref{chap:Ck,sec:HorHan,eq:diago} it can be deduced that the Hankel multiplier operator defined by \eqref{chap:Ck,sec:HorHan,eq:spHan} coincides with the joint spectral multiplier operator of the system $(L_1,\ldots,L_d)$ given by $n(L_1,\ldots,L_d).$ The smoothness requirements on $m$ that guarantee the boundedness of $\mathcal T_m$ on $L^p$ will be stated in terms of appropriate Sobolev space norms.

For $z\in \mathbb C$, $\text{Re}\, z>0$, let
$$ G_z(x)=\Gamma \big(z\slash 2\big)^{-1} \int_0^\infty (4\pi t)^{-d\slash 2 }e^{-|x|^2\slash 4t} e^{-t}t^{z\slash 2}\frac{dt}{t}$$
be the kernels of the Bessel potentials. Then
\begin{equation}\label{chap:Ck,sec:HorHan,eq:potential}\| G_z\|_{L^1(\mathbb R^d)}\leq \Gamma (\text{Re} \, z\slash 2)|\Gamma (z\slash 2)|^{-1} \ \ \text{ and}\ \  \mathcal F G_z(\xi)=(1+|\xi|^2)^{-z\slash 2},
\end{equation}
where $\mathcal F G_z(\xi)=\int_{\mathbb{R}^d}G_z(x)e^{-i<x,\xi>}\,dx$ is the Fourier transform.

By definition, a function $f\in W_2^s(\mathbb R^d)$, $s>0$,  if and only if there exists a function $h\in L^2(\mathbb R^d)$ such that $f=h* G_s$, and $\| f\|_{W^s_2(\mathbb R^d) } = \| h\|_{L^2(\mathbb R^d)}.$ For $s>0$ denote  $w^s(x)=(1+|x|)^{s},$ $x\in\mathbb{R}^d.$ Then the space $W_2^s(\mathbb{R}^d)$ can be equivalently characterized as the space of those functions $f$ in $L^2(\mathbb{R}^d,dx)$ such that $w^s\mF(f)\in L^2(\mathbb{R}^d,dx).$ Moreover, we have $\|f\|_{W^s_2}\approx\|w^s\mF(f)\|_{L^2(\mathbb{R}^d,dx)}.$

 Similarly, a function $f$ belongs to the potential space $\mathcal L_s^\infty(\mathbb R^d)$, $s>0$,  if there is a function $h\in L^\infty(\mathbb R^d)$ such that $f=h* G_s$ (see \cite[Chapter V]{singular}). Then $\| f\|_{\mathcal L^\infty_s (\mathbb R^d)}=\| h\|_{L^\infty (\mathbb R^d)}$.

Denote $A_{R_1,R_2} =\{x\in\mathbb{R}^d\,:\, R_1\leq|x|\leq R_2\}.$ The main result of this section are Theorems \ref{thm:HorHan} and \ref{thm:HorHanw}. In the statement of these theorems by $\psi$ we denote a $C_c^{\infty}(A_{1/2,2})$ function such that
\begin{equation}\label{chap:Ck,sec:HorHan,eq:suma}
\sum_{j\in\mathbb{Z}} \psi(2^{-j}\la)=1, \qquad \lambda\in \mathbb R^d\backslash \{0\}.
\end{equation}
\begin{thm}
\label{thm:HorHan} Assume that $\alpha_r\geq 1\slash 2$ for $r=1,...,d$.
 Let
 \begin{equation}\label{chap:Ck,sec:HorHan,eq:ppp}
 m(\la)=n(\la_1^2,\ldots,\la_d^2),
  \end{equation}
  where $n$ is a bounded function on $\mathbb{R}^d$ such that, for certain real number $\beta > Q/2$ and for some (equivalently, for every) non-zero radial function $\eta\in C_c^{\infty}(A_{1/2,2}),$ we have
\begin{equation}
\label{chap:Ck,sec:HorHan,eq:conmulto}
\sup_{j\in\mathbb{Z}}\|\eta(\cdot)n(2^j\cdot)\|_{W^{\beta}_2(\mathbb R^d)}<\infty.
\end{equation}
Then the multiplier operator $\mathcal T_m^{\Hi}$ defined by \eqref{chap:Ck,sec:HorHan,eq:spHan} is a Calder\'{o}n-Zygmund operator associated with the kernel
$$K(x,y)=\sum_{j\in\mathbb{Z}}\tau^y \Ha(\psi(2^{-j}(\la_1^2,\cdots,\la_d^2)) m(\la)) (x).$$
As a consequence $\mathcal T_m^{\Hi}$ extends to the bounded operator from $L^1$ to $L^{1,\infty}(X)$ and from $L^p$ to itself for  $1<p<\infty$.
\end{thm}
\begin{thm}\label{thm:HorHanw}
 If we relax the conditions on $\alpha_r$ assuming only that $\alpha_r> 0$, then the conclusion of
 Theorem \ref{thm:HorHan} holds provided there is $\beta >Q\slash 2 $ such that
 \begin{equation}\label{chap:Ck,sec:HorHan,eq:conmulto22}
\sup_{j\in\mathbb{Z}}\|\eta(\cdot)n(2^j\cdot)\|_{\mathcal L_{\beta}^\infty(\mathbb R^d)}<\infty.
\end{equation}
\end{thm}

The weak type $(1,1)$ estimate under assumption \eqref{chap:Ck,sec:HorHan,eq:conmulto22} could be proved by applying a general multiplier theorem of Sikora \cite[Theorem 2.1]{Sik}. However, in the case of the  Hankel transform Theorem \ref{thm:HorHanw} has a simpler proof based on Lemmata \ref{lem:Lporr} and \ref{lem:Linterrw}.

Hankel multipliers, mostly in one variable, attracted attention of many authors, see, e.g.\ \cite{BX}, \cite{dphanhar}, \cite{GS}, \cite{GT}, \cite{GosSt1}, and references therein.

In \cite{B} the authors considered multidimensional Hankel multipliers $m$ of Laplace transform type, that is,
 $$ m(\la)=|\la|^2\int_0^{\infty} e^{-t|\la|^2}\kappa(t)\, dt,\qquad \la \in \Rdp,$$
 where $\kappa \in L^{\infty} (0,\infty)$ (see \cite{topics}). To see that $m(\lambda)$ satisfies the assumptions of Theorems \ref{thm:HorHan} and \ref{thm:HorHanw} we set
  \begin{equation*} n(\lambda)=\Xi(\lambda) (\lambda_1+\ldots+\lambda_d) \int_0^\infty e^{-t(\lambda_1+...+\lambda_d)}\phi (t)\, dt,\end{equation*} where $\Xi \in C^\infty (\mathbb R^d\setminus \{0\})$, $\Xi (t\lambda)=\Xi (\lambda)$ for $t>0$, $\Xi (\lambda )=1$ for $\lambda \in \Rdp$, $\Xi (\lambda)=0$ for $\lambda_1+...+\lambda_d<|\lambda |\slash d.$ Then, clearly $m$ and $n$ are related by \eqref{chap:Ck,sec:HorHan,eq:ppp}. Moreover, we easily check that $n(\lambda)$ satisfies the Mikhlin condition, that is, $|\lambda|^{|\gamma|}|\partial^{\gamma}n(\lambda)|\leq C_{\gamma},$ for $\la \in \mathbb{R}^d$ and all multi-indices $\gamma$. Hence, (\ref{chap:Ck,sec:HorHan,eq:conmulto}) and (\ref{chap:Ck,sec:HorHan,eq:conmulto22}) hold with every $\beta>0$.

 A typical example of  multipliers of Laplace transform type are the imaginary powers $m_u(y)=|y|^{2iu},$ $u\in \mathbb{R},$ which correspond to $\kappa_u(t)=(\Gamma(1-iu))^{-1}t^{-iu}.$ In this case the resulting operators $\mathcal T_{m_u}$ coincide  with $(L_B)^{iu}.$
It is worth to remark that thanks to Theorem \ref{thm:HorHanw} we can prove substantially better bounds on $L^p(X),$ $1<p<\infty,$ than
 \begin{equation*}
 \|L^{iu}\|_{L^p(X)\to L^p(X)}\leq C_p e^{\pi|u||\frac{1}{2}-\frac{1}{p}|}\|f\|_{L^p(X)},\qquad f\in L^p(X),\quad u\in\mathbb{R},
 \end{equation*}
 which were obtained in \cite[Corollary 1.2]{B}. Namely, for arbitrary small $\varepsilon>0$ we have
 \begin{equation}
 \label{chap:Ck,sec:HorHan,eq:ourima}
 \|L^{iu}\|_{L^p(X)\to L^p(X)}\leq C_{p,\varepsilon} (1+|u|)^{(Q+2\varepsilon)|\frac{1}{2}-\frac{1}{p}|}\|f\|_{L^p(X)},\qquad f\in L^p(X),\quad u\in\mathbb{R}.
 \end{equation}
The proof of \eqref{chap:Ck,sec:HorHan,eq:ourima} can be found in \cite[Appendix]{dpw}.

It is perhaps worth to point out that in $d=1$ the assumptions \eqref{chap:Ck,sec:HorHan,eq:conmulto} and \eqref{chap:Ck,sec:HorHan,eq:conmulto22} could be given in terms of  function $m$ instead of $n$. However, in the multivariate case we assume
in \eqref{chap:Ck,sec:HorHan,eq:conmulto} a $W_2^\beta$-Sobolev regularity of a function $n(\lambda)$  which is related with $m(\lambda)$  by \eqref{chap:Ck,sec:HorHan,eq:ppp}.
It can be shown, see \cite[Appendix]{dpw}, that even for the classical Fourier multipliers supported in $A_{1\slash 2,2}$ the Sobolev norms of $n(\lambda)$ and $m(\lambda)$ are not comparable when we apply the change of variables \eqref{chap:Ck,sec:HorHan,eq:ppp}.

We start with proving some basic estimates. Throughout this section the symbol $\lesssim$ means that the constant in the inequality is independent of significant quantities; however it may depend on $s,$ $\beta,$ $\varepsilon,$ $d$ or $\alpha.$ Recall that $w^s(x)=(1+|x|)^{s}.$
\begin{lem}
\label{lem:Lporr}
 For every $s,\varepsilon>0$ there exists a constant $C(s,\varepsilon,d,\alpha)$ such that   if $m(\la)=n(\la_1^2,...,\la_d^2)$, $\supp \,n\subseteq A_{1/4,4}$, then
\begin{equation}\label{chap:Ck,sec:HorHan,eq:eq21}
\|\Ha (m) w^s\|_{2} \leq C(s,\varepsilon,d,\alpha) \|n\|_{W^{s+d\slash 2+\varepsilon}_2(\mathbb R^d)}.
\end{equation}
\end{lem}

\begin{proof} Since $m(\la)= g(\la_1^2,...,\la_d^2) e^{-|\la|^2},$ with $g(\la)= n(\la)e^{\la_1+...+\la_d},$ using the Fourier inversion formula for $g$, we get
\begin{equation*}\begin{split} (2\pi)^d\, m(\la)&=e^{-|\la|^2} \int_{\mathbb R^d}\mathcal{F}(g)(y)e^{i y_1\la_1^2+...+iy_d\la_d^2}\,dy=\int_{\mathbb R^d} \mathcal{F}(g)(y)e^{(-1+iy_1)\la_1^2+...+(-1+iy_d)\la_d^2}\,dy.
     \end{split}\end{equation*}
 Applying the Hankel transform and changing the order of integration, we obtain
\begin{equation}\label{chap:Ck,sec:HorHan,eq:inte}\Ha (m)(x)=(2\pi)^{-d}\int_{\mathbb R^d} \mathcal{F}(g)(y)\Ha (e_{\textbf{1}-iy})(x)\,dy,\end{equation}
where for $z=(z_1,\ldots,z_d)\in\mathbb{C}^d,$ $e_z(\la)=\prod_{r=1}^d e_{z_r}(\la_r)$ with $e_{z_r}(\la_r)=e^{-z_r\la_r^2}.$

Clearly, $$\Ha(e_{\textbf{1}-iy})(x)=\prod_{r=1}^{d}\Ha_r(e_{1-iy_r})(x_r),$$ with $\Ha_r$ denoting the one-dimensional Hankel transform acting on the $r$-th variable.
 It is well known that for $t>0,$ $\Ha_r(e_{t})(x_r)=C_{\alpha_r}t^{-(2\alpha_r+1)/2}\exp{(-x_r^2\slash 4t)},$ see \cite[p.\ 132]{2}. Moreover, for  fixed $x_r$, the functions $$z_r\mapsto \Ha_r(e_{z_r})(x_r)\qquad \textrm{and} \qquad z_r\mapsto C_{\alpha_r}z_r^{-(2\alpha_r+1)/2}\exp{\left(-\frac{x_r^2}{4{z_r}}\right)}$$ are holomorphic on $\{z_r\in\mathbb{C}\,:\, \textrm{Re} \, z_r>0\}$ (provided we choose   an appropriate holomorphic branch of the  power function $z_r^{-(2\alpha_r+1)/2}$). Hence, by the uniqueness of the holomorphic extension, we obtain $$\Ha_r(e_{1-iy_r})(x_r)=C_{\alpha_r} (1-iy_r)^{-(2\alpha_r+1)/2}\exp{\bigg(-\frac{x_r^2}{4(1-iy_r)}\bigg)}.$$

Since $\Real \big[x_r^2\slash 4(1-iy_r)\big]=x_r^2\slash 4(1+y_r^2),$ the change of variable $x_r=(1+y_r^2)^{1/2}u_r$  leads to \begin{equation}\label{chap:Ck,sec:HorHan,eq:onprod}\int_{(0,\infty)}|x_r^s\Ha(e_{1-iy_r})(x_r)|^2\,d\nu_r(x_r)\lesssim (1+y_r^2)^{s},\qquad s\geq0.\end{equation}
Now, observing that $(1+|x|)^{2s}\approx 1+x_1^{2s}+\ldots +x_d^{2s}$ and using \eqref{chap:Ck,sec:HorHan,eq:onprod} we arrive at
$$\|w^s(\cdot)\Ha(e_{\textbf{1}-iy})(\cdot )\|_{L^2}\lesssim \sum_{r=1}^d (1+y_r^2)^{s/2}\approx (1+|y|)^{s}.$$
The latter bound together with \eqref{chap:Ck,sec:HorHan,eq:inte}, Minkowski's integral inequality, and the Schwarz  inequality give
\begin{align*}&\|\Ha (m) w^s\|_{2}\lesssim \int_{\mathbb{R}^d}|\mathcal{F}(g)(y)|(1+|y|)^{s}\,dy \\&\lesssim \bigg( \int_{\mathbb{R}^d}|\mathcal{F}(g)(y)|^2(1+|y|)^{2s+d+2\varepsilon}\,dy\bigg)^{1/2}\bigg(\int_{\mathbb{R}^d}(1+|y|)^{-d-2\varepsilon}\,dy\bigg)^{1/2}\lesssim \|g\|_{W^{s+d/2+\varepsilon}_2(\mathbb R^d)}\end{align*}
for any fixed $\varepsilon>0.$
Since $g(\la)=n(\la)e^{\la_1+...+\la_d}=n(\la) (e^{\lambda_1+...+\la_d}\eta_0(\la)),$ for some $\eta_0\in C_c^{\infty}(A_{1/8,8}),$ we see that $\|g\|_{W^{s+d/2+\varepsilon}_2(\mathbb R^d)}\leq C_{\eta_0} \|n\|_{W^{s+d/2+\varepsilon}_2(\mathbb R^d)},$ which implies \eqref{chap:Ck,sec:HorHan,eq:eq21}. \end{proof}

Using ideas of Mauceri and Meda \cite{vectvalMM} combined with  the fact that the Hankel transform is  an $L^2$-isometry we can improve Lemma \ref{lem:Lporr} in the following way.
\begin{lem}
\label{lem:linterr}
Assume that $\alpha_r\geq 1\slash 2$ for $r=1,...,d$. Then  for every  $s,\varepsilon>0,$ there is a constant $C(s,\varepsilon,d,\alpha)$ such that
if  $m(\la)=n(\la_1^2, ...,\la_d^2)$,  $\supp \, n\subseteq A_{1/2,2}$, then
\begin{equation*}
\|\Ha (m) w^s\|_{2} \leq C(s,\varepsilon,d,\alpha) \|n\|_{W_2^{s+\varepsilon}(\mathbb R^d)}.
\end{equation*}
\end{lem}
\begin{proof}
Let $h\in L^2(\mathbb R^d)$ be such that $n=h* G_{s+\varepsilon}$. Set $s'=(s+\varepsilon)(6+d)\slash 2\varepsilon$,
\newline $\theta=2\varepsilon \slash  (6+d),$ so that $(s'-3-d/2)\theta=s.$
Define $n_z$ by
 \begin{equation*}
\mathcal{F}(n_{z})(\xi)=\mathcal F h(\xi)(1+|\xi|^2)^{-s'z\slash 2}, \ \ 0\leq  \text{Re}\, z\leq 1.
\end{equation*}
 Clearly, $n_z=h* G_{s'z}$,  $\text{Re}\, z>0$, and $n=n_\theta$.
 Let $\eta_0$ be a $C_{c}^{\infty}$ function supported in $A_{1/4,4},$ equal to $1$ on $A_{1/2,2},$ and let $N_z(\la)=n_z(\la)\eta_0(\la)$. Then $\supp \, N_z\subseteq A_{1/4,4}$ and $\mathcal{F}(N_z)=\mathcal{F}(n_z)* \mathcal{F}(\eta_0).$
  Define
  $$m_z(\la)=n_z(\la_1^2,...,\la_d^2)\quad \textrm{ and } \quad M_z(\la)=N_z(\la_1^2,...,\la_d^2).$$
Since $\alpha_r\geq 1\slash 2$  for every $r=1,...,d$, we have that $M_z \in L^2$ and $\|M_z\|_{2} \leq C_{\alpha} \|N_z\|_{L^2(\mathbb R^d,dx)}.$

Let $g$ be an arbitrary $C_c^{\infty}(X)$ function with $\|g\|_{2}=1.$ Set
\begin{equation}\label{chap:Ck,sec:HorHan,eq:functionF}F(z)=\int_{X} \Ha(M_z)(x)(1+|x|)^{(s'-3-d\slash 2)z}g(x)\,d\nu(x).
\end{equation}
Then $F$ is holomorphic in the strip $\Sigma=\{z: 0<\textrm{Re}\, z<1\}$
and also continuous and bounded on its closure $\bar \Sigma$.
Using Plancherel's formula and the facts that $\supp \,N_z\subseteq A_{1/4,4}$ and $\mathcal{F}(\eta_0)\in\mathcal{S}(\mathbb{R}^d),$ for $\textrm{Re} \,z=0,$ we get
\begin{align*}
|F(z)|\leq \|\Ha(M_z)\|_{2}=\|M_z\|_{2}\leq C_{\alpha} \|N_z\|_{L^2(\mathbb{R}^d,dx)}\approx \|\mathcal{F}N_z\|_{L^2(\mathbb{R}^d,d\xi)}\leq C_{\eta_0,s',\theta} \|n\|_{W_2^{s+\varepsilon}(\mathbb R^d)}.
\end{align*}
If $\textrm{Re}\, z=1$, then  applying in addition Lemma \ref{lem:Lporr} we obtain
\begin{align*}
|F(z)|& \leq \|\Ha(M_z)w^{s'-3-d\slash 2}\|_{2}  \lesssim\|N_z\|_{W_2^{s'}(\mathbb R^d)}\\
&\lesssim C_{\eta_0}\|n_z\|_{W_2^{s'}(\mathbb R^d)}=C\| h\|_{L^2(\mathbb R^d)} = C \|n\|_{W_2^{s+\varepsilon}(\mathbb R^d)}.
\end{align*}
From the Phragm\'{e}n-Lindel\"{o}f principle we get  $|F(\theta)|\lesssim\|n\|_{W^{s+\varepsilon}_2(\mathbb R^d)}.$ Taking the supremum over all such $g$ we arrive at
$$\|\Ha(M_{\theta})w^{(s'-3-d/2)\theta}\|_{2}\lesssim \|n\|_{W_2^{s+\varepsilon}(\mathbb R^d)}.$$
Recall that $n=n_\theta=N_\theta$, so that also $m=m_\theta = M_{\theta},$ hence, in view of $(s'-3-d/2)\theta=s$, we get the desired conclusion. \end{proof}

Notice, that the assumption $\alpha_r\geq 1/2$ was crucial to get $\|M_z\|_{2} \leq C_{\alpha} \|N_z\|_{L^2(\mathbb R^d,dx)}$
in the proof of Lemma \ref{lem:linterr}. This was due to the fact that, in the case $d>2,$ $\lambda\in \supp \,N_z \subseteq A_{1/4,4}$ no longer implies that each $\lambda_r\neq0.$ In the full range of $\alpha$'s we can overcome this difficulty, but with different Sobolev condition.

\begin{lem}
\label{lem:Linterrw}
If we relax the conditions on $\alpha_r$ in Lemma \ref{lem:linterr} by assuming that $\alpha_r > -\frac{1}{2}$, then
\begin{equation*}
\|\Ha (m) w^s\|_{2} \leq C(s,\varepsilon,d,\alpha) \|n\|_{\mathcal L^\infty_{s+\varepsilon}(\mathbb R^d)}.
\end{equation*}
\end{lem}

\begin{proof} We argue similarly to the proof of Lemma \ref{lem:linterr}. Indeed, write $n=h* G_{s+\varepsilon}$, where $h\in L^\infty (\mathbb R^d)$.
Since $\supp \, n\subset A_{1\slash 2, 2}$, one can prove that $h\in L^2(\mathbb R^d)$ and $\| h\|_{L^2(\mathbb R^d)}\leq C_{s,\varepsilon,d} \| n\|_{\mathcal L^\infty_{s+\varepsilon}}$.

Set $s'=(2s+\varepsilon)(6+d)\slash 2\varepsilon$, $\theta=2\varepsilon\slash (6+d)$ and
define
$$
N_z(\lambda) =\eta_0(\lambda)\, h* G_{s'z+\varepsilon \slash 2}(\lambda), \ \ \lambda\in\mathbb R^d, \ \ 0\leq \text{Re}\, z\leq 1.$$
Then for every $z\in \bar \Sigma$ the function $N_z(\lambda) $ is continuous and supported in $A_{1\slash 4,4}$. Let $M_z(\lambda)=N_z(\lambda_1^2,...,\lambda_d^2)$.  Clearly, $M_\theta=m$. Moreover, by \eqref{chap:Ck,sec:HorHan,eq:potential},
\begin{equation*}\|M_z \|_{2}\lesssim \| M_z\|_{L^\infty}\lesssim\| N_z\|_{L^\infty(\mathbb R^d)}\lesssim  \| h\|_{L^\infty(\mathbb R^d)}=\|n\|_{\mathcal L^\infty_{s+\varepsilon}(\mathbb R^d)}.\end{equation*}

We now use the new functions $M_z$ to define a bounded  holomorphic  function $F(z)$ by the formula \eqref{chap:Ck,sec:HorHan,eq:functionF}.
Obviously $|F(z)|\lesssim \| n\|_{\mathcal L^\infty_{s+\varepsilon}}$ for $\text{Re}\, z=0$. To estimate $F(z)$ for $\text{Re}\, z=1$ we utilize  Lemma \ref{lem:Lporr} and obtain
\begin{equation*} |F(z)|\leq  \|\Ha(M_z)w^{s'-3-d\slash 2}\|_{2}\lesssim\|N_z\|_{W_2^{s'}(\mathbb R^d)} \lesssim C_{\eta_0} \| h\|_{L^2(\mathbb R^d)}
\lesssim\|n\|_{\mathcal L^\infty_{s+\varepsilon}(\mathbb R^d)}.
\end{equation*}
An application of the  Phragm\'{e}n-Lindel\"{o}f principle for $z=\theta$ finishes the proof. \end{proof}
\vskip 1em

 We will also need the following off-diagonal estimate (see \cite[Lemma 2.7]{dphanhar}).
\begin{lem}
\label{lem:Lfarr}
Assume that $\alpha\in \Rdp$ and let $\delta>0.$ Then there is $C_d>0$ such that  for every $y\in X$ and $R,t>0,$ we have
$$
\int_{|x-y|>R}|\tau^{y}(f_t)(x)|\,d\nu(x)\leq C_d(rt)^{-\delta}\|f\|_{L^1(X, w^{\delta}(x)d\nu(x))}.
$$
\end{lem}

\begin{proof} By homogeneity it suffices to prove the lemma for $t=1$. Let $B$ be the left-hand side of the inequality from the lemma. If $|x-y|>R$ then there is $k\in \{1,...,d\}$ such that  $|x_r-y_r|>R/\sqrt{d}$. Hence,
$$
B\leq \sum_{r=1}^d \int_{|x_r-y_r|>R/\sqrt d}|\tau^{y}(f)(x)|\,d\nu(x)
=\sum_{r=1}^dB_r.
$$

It is known that the generalized translations can be also expressed as
\begin{equation}
\label{chap:Ck,sec:HorHan,eq:translaform}
\tau^y f(x)= \int_{|x_1-y_1|}^{x_1+y_1}...\int_{|x_d-y_d|}^{x_d+y_d}\,f(z_1,...,z_d)
\,dW_{x_1,y_1}(z_1)...dW_{x_d,y_d}(z_d),
\end{equation}
with $W_{x_r,y_r}$ being a probability measure supported in $[|x_r-y_r|,x_r+y_r]$ (see \cite{Ha}).  Thus,
\begin{equation*}
B_r=\int_{|x_r-y_r|>R/\sqrt{d}}\Big|\,\int_{|x_1-y_1|}^{x_1+y_1}...
\int_{|x_d-y_d|}^{x_d+y_d}\,f(z_1,...,z_d)\,dW_{x_1,y_1}(z_1)...dW_{x_d,y_d}(z_d)\Big|\,d\nu(x).
\end{equation*}
Then, introducing the factor $z_r^{\delta}z_r^{-\delta}$ to  the inner integral in the above formula and denoting $g(x)=|f(x)|x_r^{\delta}$, we see that
\begin{align*}B_r&\leq C R^{-\delta} \int_{X}\int_{|x_1-y_1|}^{x_1+y_1}...
\int_{|x_d-y_d|}^{x_d+y_d}g(z)\,dW_{x_1,y_1}(z_1)...dW_{x_d,y_d}(z_d)\,d\nu(x)\\
&\leq C R^{-\delta} \|\tau^y g\|_{1}\leq C R^{-\delta} \|f\|_{L^1(X, w^{\delta}d\nu)},\end{align*}
where in the last inequality we have used the fact that $\tau^y$ is a contraction on $L^1.$ \end{proof}

Let $T_t(x,y)=\tau^{y}\Ha(e^{-t|\la|^2})(x)$ be the integral kernels of the heat semigroup corresponding to $L=\sum_{r=1}^d L_r$. Clearly,
$$T_t(x,y)=T_t^{(1)}(x_1,y_1)\cdots T_t^{(d)}(x_d,y_d), $$
where $T_t^{(r)}(x_r,y_r)$ is the one-dimensional heat kernel associated with the operator $L_r$.
\begin{lem}
\label{lem:Lneaa}
There is a constant $C >0$ such that
$$
\int_{X}|T_1(x,y)-T_1(x,y')|\,d\nu(x)\leq C_{d,\alpha}|y-y'|,\qquad y,y'\in X.
$$
\end{lem}

\begin{proof} The proof is a direct consequence of the   one-dimensional result, see \cite[Theorem 2.1]{GosSt1}, together with the equality  $$\int_0^\infty\,T_1^{(r)}(x_r,y_r)d\nu_r(x_r)=1,\qquad r=1,...,d.$$ \end{proof}

We are now ready to prove Theorem \ref{thm:HorHan}. The scheme of the proof takes ideas from \cite{Horm1}.

Assume that \eqref{chap:Ck,sec:HorHan,eq:conmulto} holds for some $\beta> Q /2.$ Fix  $\psi\in C_c^{\infty}(A_{1/2,2})$ satisfying \eqref{chap:Ck,sec:HorHan,eq:suma}. Let $$K(x,y)=\sum_{j\in\mathbb{Z}}K_j(x,y)=\sum_{j\in\mathbb{Z}}\tau^{y}\Ha(m_{j})(x),$$ where $m_j(\la)=\psi(2^{-j}(\la_1^2,...,\la_d^2))m(\la)=(\psi(2^{-j}\cdot)n(\cdot))(\la_1^2,...,\la_d^2).$ To prove that $\mathcal T_m$ is indeed a Calder\'{o}n-Zygmund operator associated with the kernel $K(x,y)$ we need to verify that it satisfies the H\"{o}rmander integral condition, i.e.\,
\begin{equation}
\label{chap:Ck,sec:HorHan,eq:horm}
\sup_{y,y'\in X}\int_{|x-y|>2|y-y'|}|K(x,y)-K(x,y')|\,d\nu(x)<\infty,
\end{equation}
and the association condition
\begin{equation}
\label{chap:Ck,sec:HorHan,eq:asoc}
\mathcal T_m f(x)=\int_{X} K(x,y)f(y)\,d\nu(y)
\end{equation}
for compactly supported $f\in L^{\infty}(X)$ such that $x\notin \supp \, f$.

We start by proving \eqref{chap:Ck,sec:HorHan,eq:horm}.
It suffices to show that
\begin{equation*} D_j(y,y')=\int_{|x-y|>2|y-y'|}|K_j(x,y)-K_j(x,y')|\,d\nu(x)\leq C_j:=C(j,d,\alpha),\end{equation*}
with $\sum_{j\in\mathbb{Z}}C_j<\infty.$

Let $R=2|y-y'|$ and assume first $j>- 2\log_2 R.$
 Define
$$\tm_j(\la)=m_j(2^{j/2} \la)=(\psi(\cdot)n(2^{j}\cdot))(\la_1^2,...,\la_d^2).$$
Note that $\supp \,(\psi(\cdot)n(2^{j}\cdot))\subseteq A_{1/2,2}.$ From \eqref{chap:Ck,sec:HorHan,eq:dil+tra} we see that $$\Ha(m_{j})(x)=2^{jQ/2}\Ha (\tm_j)(2^{j/2} x)=(\Ha(\tm_j))_{2^{j/2}}(x).$$ From Schwarz's inequality, Lemma \ref{lem:linterr}, and the assumption \eqref{chap:Ck,sec:HorHan,eq:conmulto} we get
\begin{align}\nonumber\int_{X} |\Ha(\tm_j)|w^{\delta}\,d\nu &\leq \bigg(\int_{X}|\Ha(\tm_j)|^2 w^{Q+4\delta}\,d\nu\bigg)^{1/2}\bigg(\int_{X}w^{-Q-2\delta}\,d\nu\bigg)^{1/2}\\
\label{chap:Ck,sec:HorHan,eq:weight}&\lesssim C_{\delta}\|\psi(\cdot) n(2^{j}\cdot)\|_{W^{\beta}_2(\mathbb R^d)}\lesssim C_{\delta},
\end{align}
for sufficiently small $\delta>0.$ Consequently, from Lemma \ref{lem:Lfarr} it follows that
\begin{align*}
D_j(y,y')&\lesssim \int_{|x-y|>R} |\tau^y (\Ha(\tm_j))_{2^{j/2}}(x)|\,d\nu(x)+\int_{|x-y'|>R/2} |\tau^{y'} (\Ha(\tm_j))_{2^{j/2}}(x)|\,d\nu(x)\\
&\lesssim (2^{j/2}R)^{-\delta}\int_{X} |\Ha(\tm_j)|w^{\delta}\,d\nu\lesssim C_{\delta}(2^{j/2}R)^{-\delta},
\end{align*}
so that
$\sum_{j> -2\log_2 R} D_j(y,y')\leq C(\beta,d,\alpha).$

\newcommand{\mmm}{m^{[0]}}
\newcommand{\mmmm}{m^{[\infty]}}

Assume now $j\leq- 2\log_2 R.$ Decompose $\tm_j(\la)=\tilde{\theta}_j(\la)  e^{-|\la|^2},$ so that we have $\tilde{\theta}_j(\la)=(\psi(\cdot)\exp(\cdot_1+...+\cdot_d)n(2^{j}\cdot))(\la_1^2,...,\la_d^2).$ Clearly, $\psi(\la)e^{\la_1+...+\la_d}$ is a $C_c^{\infty}$ function supported in $A_{1/2,2}.$ Denote $\tilde{\Theta}_j(x)=\Ha(\tilde{\theta}_j)(x).$ Since $\Ha(m_j)=(\Ha(\tm_j))_{2^{j/2}}$ and $\Ha(\tm_j)=\tilde{\Theta}_j \natural \Ha(e^{-|\la|^2})$ (which is a consequence of \eqref{chap:Ck,sec:HorHan,eq:convo}), by using \eqref{chap:Ck,sec:HorHan,eq:dil+tra}, we get
\begin{align*}
K_j(x,y)-K_j(x,y')&=(\tau^{2^{j/2} y} \Ha(\tm_j))_{2^{j/2}}(x)-(\tau^{2^{j/2} y'} \Ha(\tm_j))_{2^{j/2}}(x)\\
&=(\tilde{\Theta}_j \natural (T_1(\cdot,2^{j/2}y)-T_1(\cdot,2^{j/2}y')))_{2^{j/2}}(x).
\end{align*}
Proving \eqref{chap:Ck,sec:HorHan,eq:weight} with $\tm_j$ replaced by
$\tilde{\theta}_j$ and $\delta=0$ poses no difficulty. Hence, from Lemma \ref{lem:Lneaa} and \eqref{chap:Ck,sec:HorHan,eq:young} we obtain
$$D_j(y,y')\leq \|\tilde{\Theta}_j\|_{1}\|T_1(\cdot,2^{j/2}y)-T_1(\cdot,2^{j/2}y')\|_{1}\lesssim 2^{j/2}|y-y'|.$$
Consequently, $\sum_{j\leq- 2\log_2 R}D_j(y,y')\leq C(\beta,d,\alpha)$ and the proof of \eqref{chap:Ck,sec:HorHan,eq:horm} is finished.\\

 Now we turn to the proof of \eqref{chap:Ck,sec:HorHan,eq:asoc}. Let $f\in L^\infty$ be a compactly supported function and $x\not\in \supp \, f$. Then, there are $R_2>R_1>0$ such that
 $$\int_{X} K_j(x,y)f(y)\,d\nu(y)=\int_{R_2>|x-y|>R_1}K_j(x,y)f(y)\,d\nu(y).$$

 Since $\tau^y (\Ha(m_j))(x)=\tau^x (\Ha(m_j))(y),$ proceeding as in the first part of the proof of \eqref{chap:Ck,sec:HorHan,eq:horm} we can easily check that $\sum_{j>- 2\log_2 R}|K_j(x,y)|$ is integrable over $\{y\in X\,:\,|x-y|>R\}.$
 Hence, using the dominated convergence theorem (recall that $f\in L^{\infty}$),  \begin{equation}\label{chap:Ck,sec:HorHan,eq:kern}\sum_{j>-2\log_2R}\,\int_{X} K_j(x,y)f(y)\,d\nu(y)=\int_{X} \sum_{j>-2\log_2R}\,K_j(x,y)f(y)\,d\nu(y).
 \end{equation}
 From \eqref{chap:Ck,sec:HorHan,eq:convo} it follows that
 \begin{equation}\label{chap:Ck,sec:HorHan,eq:multi}\mathcal T_{m_j}f(x)=H(m_j)\natural f(x)=\int_{X} K_j(x,y)f(y)\,d\nu(y),\end{equation}
 with $\mathcal T_{m_j}$ defined as in \eqref{chap:Ck,sec:HorHan,eq:spHan}.
 Since the Hankel transform is an $L^2$-isometry, from the dominated convergence theorem we conclude  that $\sum_{j>-2\log_2R} \mathcal T_{m_j}f=\mathcal T_{\mmmm} f,$ where the sum converges in $L^2$ and $\mmmm=\sum_{j>-2\log_2R}m_j.$ Hence, combining \eqref{chap:Ck,sec:HorHan,eq:kern} and \eqref{chap:Ck,sec:HorHan,eq:multi}, we obtain
 $$\mathcal T_{\mmmm}f(x)=\int_{X} \sum_{j>-2\log_2R}\,K_j(x,y)f(y)\,d\nu(y),$$
for $\textrm{a.e.}\,\,x$ outside
$\supp \, f.$

The function $\mmm=m-\mmmm$ is bounded and compactly supported. Consequently, from \eqref{chap:Ck,sec:HorHan,eq:convo} we get $\mathcal T_{\mmm}f(x)=\Ha(\mmm) \natural f (x).$ Moreover, we see that $\sum_{j\leq -2 \log_2 R}|m_j(\la)|\leq C |m(\la)|\leq C$. Hence, from \eqref{chap:Ck,sec:HorHan,eq:translaform} we conclude $$\tau^{y}(\Ha\mmm)(x)=\sum_{j\leq -2 \log_2 R}\tau^y(\Ha m_j)(x),$$
 so that
$$\mathcal T_{\mmm}f(x)=\int_{X} \sum_{j\leq-2\log_2R}\,K_j(x,y)f(y)\,d\nu(y).$$
Then $\mathcal T_{m}f(x)=\mathcal T_{\mmm}f(x)+\mathcal T_{\mmmm}f(x)=\int_{X} \,K(x,y)f(y)\,d\nu(y),$ as desired. The proof of Theorem \ref{thm:HorHan} is thus finished.

Let us finally comment that the proof of Theorem \ref{thm:HorHanw} goes in the same way as that of Theorem \ref{thm:HorHan}. The only difference is that we use Lemma \ref{lem:Linterrw} instead of Lemma \ref{lem:linterr}.

\section[Marcinkiewicz type multipliers for the Dunkl transform]{Marcinkiewicz type multipliers for the Dunkl transform}

        \label{chap:Ck,sec:MarDun}
We prove a Marcinkiewicz type multiplier theorem for the Dunkl transform in the setting of the reflection group $G$ isomorphic to $\mathbb{Z}^d_2.$ The results presented in this section are taken from \cite[Section 3]{jaDun}. As in the previous section we heavily rely here on the concrete translation structure associated with the Dunkl transform introduced below.

Consider the system of Dunkl operators, $T=(T_1,\ldots,T_d),$
\begin{equation*}
T_rf(x)=T_r^{\alpha_r}f(x)=\partial_{x_r} f(x)+\alpha_r \frac{f(x)-f(\sigma_rx)}{x_r},\qquad f\in C^1(\mathbb{R}^d),
\end{equation*}
associated with the reflection group $G$ isomorphic to $\mathbb{Z}^d_2.$ Here $\sigma_{r}x$ is the reflection of $x$ in the hyperplane orthogonal to $e_r$ (the $r$-th coordinate vector in $\mathbb{R}^d$), while the parameter $\alpha=(\alpha_1,\ldots,\alpha_d)\in[0,\infty)^d$. It is known, and in fact easily verifiable in our case, that $T_r,$ $r=1,\ldots,d,$ commute. Moreover, they are anti-symmetric in $L^2(\mathbb{R}^d,\nu_{\alpha}).$ The measure $\nu_{\alpha}$ is the one from the previous section, however this time regarded on $\mathbb{R}^d,$ specifically, $\nu_{\alpha}=\nu_{\alpha_1}\otimes\cdots\otimes\nu_{\alpha_d},$ with $d\nu_{\alpha_r}(x_r)=|x_r|^{2\alpha_r}dx_r,$ $r=1,\ldots,d.$

Throughout this section we write briefly $L^p,$ $\|\cdot\|_p,$ and $\|\cdot\|_{p\to p}$ instead of $L^p(\mathbb{R}^d,\nu_{\alpha}),$ $\|\cdot\|_{L^p(\mathbb{R}^d,\nu_{\alpha})},$ and $\|\cdot\|_{L^p(\mathbb{R}^d,\nu_{\alpha})\to L^p(\mathbb{R}^d,\nu_{\alpha})},$ $1\leq p\leq\infty,$ respectively. By $\lesssim$ we indicate that the constant in the inequality is independent of significant quantities, however, not necessarily of $\alpha$ or $d.$

The Dunkl transform is defined by
\begin{equation*}D_{\alpha} f(x)= \frac1{c_{\alpha}}\int_{\mathbb{R}^d}E_{\alpha}(-ix,y)f(y)\,d\nu_{\alpha}(y)\,dy,\qquad f\in\mathcal{S}(\mathbb{R}^d),\end{equation*}
where $$c_{\alpha}=\int_{\mathbb{R}^d}e^{-|x|^2/2}\,d\nu_{\alpha}(x)=2^{-|\alpha|-d/2}\prod_{r=1}^{d}\frac{1}{\Gamma(\alpha_r+1/2)},$$ while $E_{\alpha}(z,w),$ $z=(z_1,\ldots,z_d),$ $w=(w_1,\ldots,w_d),$ $z,w\in \mathbb{C}^d,$ is the so called Dunkl kernel. Note that in the general setting of Dunkl operators $E_{\alpha}$ does not have an explicit expression. In our case the Dunkl kernel has a product form $E_{\alpha}(z,w)=\prod_{r=1}^d E_{\alpha_r}(z_r,w_r),$ from which it follows that the Dunkl transform is a composition of one-dimensional Dunkl transforms acting on separate variables, i.e. $D_{\alpha}=D_{\alpha_1}\cdots D_{\alpha_d}.$ Moreover, each one-dimensional Dunkl kernel $E_{\alpha_r}$ does have an explicit expression in terms of Bessel functions, see \cite[p.\ 1604]{NoSt1}. The operator $D_{\alpha}$ extends to an $L^2(\nu_{\alpha})$ isometry such that \begin{equation}\label{chap:Ck,sec:MarDun,eq:invdun} D_{\alpha}^{-1}f=D_{\alpha}(f^{\vee}),\qquad f\in \mathcal{S}(\mathbb{R}^d),\end{equation}
where $f^{\vee}(x)=f(-x).$ We also have $D_{\alpha}(f(\sigma\cdot))(\cdot)=(D_{\alpha}f)(\sigma\cdot),$ $\sigma\in G,$ i.e., the Dunkl transform commutes with the action of $G.$

In addition, $D_{\alpha}$ diagonalizes simultaneously the Dunkl operators $T_r,$ i.e. \begin{equation*}T_r D_{\alpha} f=-D_{\alpha}(i y_r f),\qquad \mathcal D_\alpha T_r f =i x_r D_{\alpha}f.\end{equation*} Therefore multivariate spectral multipliers for the Dunkl transform
\begin{equation}
\label{chap:Ck,sec:MarDun,eq:dunklmult}
\T_m^D f=\T_m f=D_\alpha^{-1} m D_\alpha f
\end{equation}
 coincide with the multivariate spectral multipliers for the system $(-iT_1,\ldots,-iT_d).$ Note that the particular choices of $m,$ $m=m_r(x)=x_r/|x|,$ $r=1,\ldots,d,$ lead to the Riesz-Dunkl transforms \begin{equation}\label{chap:Ck,sec:MarDun,eq:Riesz-Dunkl}R_r^{D}f=D_\alpha^{-1} m_r D_\alpha f.\end{equation}

In this section the smoothness requirement will be stated in terms of appropriate Sobolev space norms. For a vector $s=(s_1,\ldots,s_d)\in\mathbb{R}^d$ define the $L^2$ Sobolev (Euclidean) space with dominating mixed smoothness of order $s,$ by
\begin{equation*}W_{s}=\{f\in L^2(\mathbb{R}^d, dx)\colon \mathcal{F}f\in L^2(\mathbb{R}^d, w_{2s}(y)\,dy)\},\end{equation*}
with the norm
\begin{equation*}\|f\|_{W_s}^2=\int_{\mathbb{R}^d}|\mathcal{F}f (y) w_s(y)|^2\, dy,\end{equation*}
where $\mathcal{F}$ is the Fourier transform and $w_s(y)=\prod_{r=1}^d(1+y_r^2)^{s_r/2}.$ The main theorem we prove in this section is Theorem \ref{thm:dunklmarmult}. In the statement of this theorem, by $\psi$ we denote a non-zero $C_c^{\infty}([-2,-1/2]\cup[1/2,2])$ function such that $\sum_{l\in\mathbb{Z}}|\psi(2^{-l}\xi)|^2=1,$ $\xi\neq0.$ We also set $\Psi(y_1,\ldots,y_d)=\psi(y_1)\cdots\psi(y_d).$
\begin{thm}
\label{thm:dunklmarmult}
Let $m$ be a bounded measurable function on $\mathbb{R}^d.$ If for some $s>\alpha+\bf{\frac{1}{2}}$
\begin{equation}
\label{chap:Ck,sec:MarDun,eq:conem}
\|m\|_{W_{s,loc}}\equiv\sup_{j\in\mathbb{Z}^d} \|\Psi\, m(2^{j_1}\cdot,\ldots,2^{j_d}\cdot)\|_{W_s}<\infty,
\end{equation}
then the multiplier operator $\T_m^{D}$ defined by \eqref{chap:Ck,sec:MarDun,eq:dunklmult} is bounded on all the $L^p$ spaces, $1<p<\infty.$
\end{thm}
\begin{remark1} The theorem is true if we replace $\Psi$ by any non-zero function $\Phi\in C^{\infty}(\mathbb{R}^d)$ supported in $([-2,-1/2]\cup[1/2,2])^d$ and such that \begin{equation*}\sum_{j\in\mathbb{Z}^d}|\Phi(2^{-j_1}\xi_1,\ldots,2^{-j_d}\xi_d)|^2=C,\,\qquad \xi\neq 0.\end{equation*} We consider $\Psi$ having a tensor product structure because the proofs are simpler in this case.\end{remark1}
\begin{remark2} It is perhaps worth mentioning that if $\alpha={\bf 0},$ then Theorem \ref{thm:dunklmarmult} coincides with the Marcinkiewicz multiplier theorem for the Fourier transform, with the sharp smoothness threshold.
\end{remark2}

In \cite[Theorem 3.1]{Sifi1} the authors obtained a one-dimensional version of Theorem \ref{thm:dunklmarmult}, assuming $m$ satisfies certain H\"ormander's type condition of integer order greater than $\alpha+1/2.$ It is not hard to verify that, for $d=1,$ their condition implies condition \eqref{chap:Ck,sec:MarDun,eq:conem} of our Theorem \ref{thm:dunklmarmult}. Thus, even in dimension $d=1$ our result is a generalization of \cite[Theorem 3.1]{Sifi1}. In \cite[Theorem 3]{Sol1} the author obtained a one-dimensional version of the H\"ormander multiplier theorem for the Dunkl transform assuming Sobolev regularity of order $\alpha+1/2.$ The difference between this result and Theorem \ref{thm:dunklmarmult} in $d=1$ is that we allow the multiplier function $m$ to be more irregular near zero. Perhaps the simplest example of a function $m$ which fails to satisfy the assumption of \cite[Theorem 3]{Sol1} but (trivially) satisfies our assumption \eqref{chap:Ck,sec:MarDun,eq:conem} is $m(x)=\textrm{sgn}(x).$ We should also mention that some results for the (general) multipliers for the Dunkl transform were proved in \cite[Theorem 3.1]{Dai1}. These however are obtained for radial multiplier functions $m,$ whereas in our case $m$ does not need to be radial.

To prove Theorem \ref{thm:dunklmarmult} we need several lemmata. The lemma below is a variant of \cite[Lemma 1.1]{NoSt1}, relating the Dunkl transform to the (modified) Hankel transform.
\begin{lem}[cf. {\cite[Lemma 1.1]{NoSt1}}]
\label{lem:dunhan1}
Let $d=1$ and let $\Hla$ be the (modified) Hankel transform as defined by \eqref{chap:Ck,sec:HorHan,eq:HanTrandef} in the previous section. Then
\begin{equation}
\label{chap:Ck,sec:MarDun,eq:dunhan1}
D_{\alpha}f(x)=\Hla(f_e)(|x|)-ix\Hlap \left(\frac{f_o}{y}\right)(|x|), \qquad f\in \mathcal{S}(\mathbb{R}),
\end{equation}
where $f_e$ is the even part, while $f_o$ is the odd part of $f.$
\end{lem}

Using Lemma \ref{lem:dunhan1} we obtain a comparison of the norms on the $L^2$ Dunkl-Sobolev space with dominating mixed smoothness and $W_{s}.$
\begin{lem}
\label{lem:compadunsob}
Let $s\in \Rdp$ and let $m_0\in W_s$ be a bounded function supported in $([-2,-1/2]\cup[1/2,2])^d.$ Then
\begin{equation}
\label{chap:Ck,sec:MarDun,eq:ineqdunsob}
\|w_s D_{\alpha}^{-1} m_0\|_{2}\leq C_s \|m_0\|_{W_s}.
\end{equation}
\end{lem}
\begin{proof} Recalling \eqref{chap:Ck,sec:MarDun,eq:invdun} we see that it is enough to prove a version of \eqref{chap:Ck,sec:MarDun,eq:ineqdunsob} with $D_{\alpha}^{-1}$ replaced by $D_{\alpha}.$ Note that in our case the $d$-dimensional Dunkl transform is a composition of the one-dimensional transforms acting on separate variables. Therefore, since the weight $w_s$ has a product form it suffices to prove Lemma \ref{lem:compadunsob} in the one-dimensional case. From \eqref{chap:Ck,sec:MarDun,eq:dunhan1} we obtain
\begin{align*}
&\|w_s D_{\alpha} m_0\|_{L^2(\mathbb{R},\nu_{\alpha})}\leq \|w_s D_{\alpha}\big((m_0)_e\big)\|_{L^2(\mathbb{R},\nu_{\alpha})}+\|w_s D_{\alpha}\big((m_0)_o\big)\|_{L^2(\mathbb{R},\nu_{\alpha})}\\
&=\sqrt{2}\left(\int_{\mathbb{R}_+}|w_s(x) \Hla \big((m_0)_e\big)(x)|^2\,d\nu_{\alpha}(x)\right)^{1/2}\\
&+ \sqrt{2}\left(\int_{\mathbb{R}_+}\left|w_s(x) x \Hlap \left(\frac{(m_0)_o}{y}\right)(x)\right|^2\,d\nu_{\alpha}(x)\right)^{1/2}\\
&=\sqrt{2}\left( \|w_s \Hla \left((m_0)_e\right)\|_{L^2(\mathbb{R}_+,\nu_{\alpha})}+\left\|w_s \Hlap \left(\frac{(m_0)_o}{y}\right)\right\|_{L^2(\mathbb{R}_+,\nu_{\alpha+1})}\right).
\end{align*}
By using, for instance, \cite[Lemma 2.9]{dphanhar} the latter quantity is bounded by \begin{equation*}C'_s(\|(m_0)_e\|_{W_s}+\|((m_0)_o)/y\|_{W_s})\leq C'_s(\|m_0\|_{W_s}+\|m_0/y\|_{W_s}).\end{equation*} Since $m$ is supported outside $0,$ we further have $\|m_0/y\|_{W_s}\leq C_s \|m_0\|_{W_s},$ as desired. \end{proof}

The proof of Theorem \ref{thm:dunklmarmult} requires also using the translation/dilation structure connected with Dunkl operators. Let $\tau^{y},$ $y\in\mathbb{R}^d$ be the $d$-dimensional generalized Dunkl translation\footnote{This notation coincides with the one for the Hankel translation from the previous section. Since neither of the two translations is used outside their respective sections, this collision of symbols should not cause any confusion.} given by $\tau^y f (x)=\tau^{y_1}\cdots\tau^{y_d}f (x),$ see \cite{Rosbes} or \cite[Theorem 7.1]{convmax}. Each one-dimensional component of the generalized translation, $\tau^{s}f(t),$ is in fact the translation $\delta_{-s}*\delta_t(f)$ in the associated signed hypergroup on $\mathbb{R}.$ Moreover,
\begin{equation}
\label{chap:Ck,sec:MarDun,eq:transexp}
\begin{split}
\tau^{s} f(t)= &\frac{1}{2}\int_{-1}^1 f(\sqrt{t^2+s^2-2stu})\left(1+\frac{t-s}{\sqrt{t^2+s^2-2stu}}\right)\Phi_{\alpha_r}(u)\,du\\
 & +\frac{1}{2}\int_{-1}^1 f(-\sqrt{t^2+s^2-2stu})\left(1-\frac{t-s}{\sqrt{t^2+s^2-2stu}}\right)\Phi_{\alpha_r}(u)\,du,
\end{split}
\end{equation}
where $\Phi_{\alpha_r}(u)=b_{\alpha_r}(1+u)(1-u^2)^{\alpha_r-1}.$ Here $b_{\alpha_r}$ is a certain normalizing constant. Note that the formula \eqref{chap:Ck,sec:MarDun,eq:transexp} is in fact valid only for $\alpha_r>0.$ If some $\alpha_r=0$ we end up with the usual translation, for which the properties stated in Lemmata \ref{lem:probacom} and \ref{lem:dilest} are trivial. Thus, in the proofs of those lemmata we assume $\alpha_r>0,$ $r=1,\ldots,d.$

It is now convenient to introduce the following terminology, cf. \cite[p.\ 6]{NoStIm}. We say that a function on $\mathbb{R}^d$ is $\varepsilon$-symmetric, $\varepsilon\in\{0,1\}^d,$ if for each $r=1,\ldots,d,$ $f$ is either even or odd with respect to the $r$-th  coordinate according to whether $\varepsilon_r=0$ or $\varepsilon_r=1,$ respectively. In short, $f$ is $\varepsilon$-symmetric if and only if $f(\sigma_r x)=(-1)^{\varepsilon_r}f(x),$ $r=1,\ldots,d.$ Any function $f$ on $\mathbb{R}^d$ can be split into a sum of $\varepsilon$-symmetric functions $\varepsilon,$ $f=\sum_{\varepsilon\in\{0,1\}^d}f_{\varepsilon}$ in such a way that
\begin{equation}
\label{chap:Ck,sec:MarDun,eq:compasym}
\max_{\varepsilon\in\{0,1\}^d}\|f_{\varepsilon}\|_{p}\leq \|f\|_{p} \leq \sum_{\varepsilon\in\{0,1\}^d} \|f_{\varepsilon}\|_{p},\qquad f\in L^p. \end{equation}
Note that in the case $d=1$ the above procedure simply reduces to considering the even and odd parts of a function on $\mathbb{R}.$

We will also need the following modification of $\tau^y.$ We set $\tau_{\varepsilon_r}^s =\tau^s,$ if $\varepsilon_r=0$ and
\begin{equation*}\tau_{\varepsilon_r}^{s} f(t)=b_{\alpha_r}\int_{-1}^1 f(\sqrt{t^2+s^2-2stu}) (1-u^2)^{\alpha_r-1}\,du,\end{equation*}
if $\varepsilon_r=1.$ Note that $\tau_{1}^s$ is positivity preserving, contrary to the translation $\tau_0^s=\tau^s$ (which preserves positivity only on even functions). Now we define $\tau_{\varepsilon}^yf(x)=\tau_{\varepsilon_1}^{y_1}\cdots\tau_{\varepsilon_d}^{y_d}f (x).$ Observe that if a function $f_{\varepsilon}$ on $\mathbb{R}^d$ is $\varepsilon$-symmetric, then $|f_{\varepsilon}|^2$ is ${\bf 0}$-symmetric and, consequently, $\tau^{y}_{\varepsilon}(|f_{\varepsilon}|^2)\geq 0$.
\begin{lem}
\label{lem:probacom}
Let $f_{\varepsilon}\in\mathcal{S}(\mathbb{R}^d)$ be an $\varepsilon$-symmetric function, $\varepsilon\in\{0,1\}^d.$ Then
\begin{equation}
\label{chap:Ck,sec:MarDun,eq:ineprob1}
|\tau^y f_{\varepsilon}(x)|^2\lesssim\tau_{\varepsilon}^y (|f_{\varepsilon}|^2)(x).
\end{equation}
\end{lem}
 \begin{proof} It suffices to prove the one-dimensional version of \eqref{chap:Ck,sec:MarDun,eq:ineprob1}. Assume first that $f$ is even. Then from \eqref{chap:Ck,sec:MarDun,eq:transexp} it follows that
 $\tau^{s} f(t)=\int f(z)\Phi_{\alpha,s,t}(z)\,dz,$ where for each $s,t,$ $\Phi_{\alpha,s,t}$ is a probability density. Consequently, \begin{equation*}|\tau^{s} f(t)|^2=\left|\int f(z)\Phi_{\alpha,s,t}(z)\,dz\right|^2\leq \int |f(z)|^2\Phi_{\alpha,s,t}(z)\,dz=\tau^{s}(|f|^2)(t).\end{equation*}
 Assume now that $f$ is odd. Then, using \eqref{chap:Ck,sec:MarDun,eq:transexp} we have
 \begin{equation*}\tau^{s} f(t)=\int_{-1}^1 f\left(\sqrt{t^2+s^2-2stu}\right)\frac{t-s}{\sqrt{t^2+s^2-2stu}}\Phi_{\alpha}(u)\,du.\end{equation*} Observing that
\begin{equation*}\frac{|t-s|}{\sqrt{t^2+s^2-2stu}}\leq 2 (1+u)^{-1},\qquad t,s\in\mathbb{R},\quad u\in (-1,1),\end{equation*}
 see \cite[Theorem 7.7]{convmax}, we obtain
 $|\tau^{s} f(t)|\leq \tau_1^s |f|(t)= \int |f(z)|\tilde{\Phi}_{\alpha,s,t}.$ Here $\{\tilde{\Phi}_{\alpha,s,t}\}_{s,t}$ is a family of measures with uniformly bounded total variation. Consequently, $|\tau^{s} f(t)|^2\lesssim\tau_1^s |f|^2(t),$ as desired. \end{proof}
\indent It is known that $3^{-1}\tau^y$ is a contraction on all the $L^p$ spaces, $1\leq p\leq \infty.$ For any two suitable functions $f,g$ define the Dunkl convolution $f\star g$ by
\begin{equation*}f\star g\,(x)=\int_{\mathbb{R}^d}f(y)\tau^x g^{\vee}(y)\,d\nu_{\alpha}(y).\end{equation*}
Then we have $\|f\star g\|_p\leq 3^d\|f\|_1 \|g\|_p.$ Moreover, as in the Fourier or Hankel transform case, the Dunkl transform turns convolution into multiplication, i.e. $D_{\alpha}(f\star g)=D_{\alpha}f\,D_{\alpha}\,g.$ For $f\in L^1$ and $\la=(\la_1,\ldots,\la_d)\in \Rdp$ define the $L^1$-dilation of $f$ by \begin{equation*}(\delta_{\la} f)(x)={\la}^{-2\alpha-\bf{1}}f(\la_1^{-1}x_1,\ldots,\la_d^{-1}x_d).\end{equation*}
Then $D_{\alpha} (\delta_{\la} f)(x)=D_{\alpha} f(\la x).$

Combining Lemmata \ref{lem:compadunsob} and \ref{lem:probacom} gives the following key lemma.
\begin{lem}
\label{lem:dilest}
Assume that $\textrm{supp}\,m_0\subseteq([-1/2,2]\cup[1/2,2])^d$ and that $m_0\in W_s$ for some $s>\alpha+{\bf \frac{1}{2}}.$ Fix $\la=(\la_1,\ldots,\la_d)\in \Rdp$ and define \begin{equation*}T_{\lambda}f(x)=D_{\alpha}^{-1}(m_0(\la_1\xi_1,\ldots,\la_{d} \xi_d)D_{\alpha} f(\xi))(x), \qquad f\in\mathcal{S}(\mathbb{R}^d).\end{equation*}  Then
\begin{equation*}
|T_{\lambda}f(x)|^2\leq C_{s,\alpha}\| m_0 \|_{W_s}^2\int_{\mathbb{R}^d} \sum_{\varepsilon\in\{0,1\}^d}|f_{\varepsilon}(y)|^2 \ta^y \delta_{\la}(w_{-2s}) (x)\,d\nu_{\alpha}(y),
\end{equation*}
where $f=\sum_{\varepsilon\in\{0,1\}^d}f_{\varepsilon},$ with $f_{\varepsilon}$ being the $\varepsilon$-symmetric part of $f.$
\end{lem}
\begin{proof} We prove the lemma in the case $d=2.$ For $d>2$ the proof is analogous. The operator $T_{\lambda}f$ can be represented as a generalized convolution \begin{align*}&T_{\lambda}f(x)=D_{\alpha}^{-1}(m_0(\la_1\cdot,\la_2\cdot))\star f(x)=\int_{\mathbb{R}^2}D_{\alpha}^{-1}(m_0(\la_1\cdot,\la_2\cdot))(y)\,\tau^y f(x)d\nu_{\alpha}(y)\\&
=\int_{\mathbb{R}^2}\la^{-2\alpha-1}D_{\alpha}^{-1}(m_0)(\la_1^{-1}y_1,\la_2^{-1}y_2)\,\tau^y f(x)d\nu_{\alpha}(y).
\end{align*}
Introducing the factor $w_{s} (\la_1^{-1}y_1,\la_2^{-1}y_2)\times w_{-s} (\la_1^{-1}y_1,\la_2^{-1}y_2)$ under the integral and applying Schwarz's inequality together with the change of variable $(y_1,y_2)\rightarrow (\la_1^{-1}y_1,\la_2^{-1}y_2)$ we get \begin{equation*}|T_{\lambda}f(x)|^2\leq \|w_sD_{\alpha}^{-1} m_0 \|_{2}^2 \int_{\mathbb{R}^2} |\tau^y f(x)|^2 \delta_{\la}(w_{-2s}) (y)\,d\nu_{\alpha}(y).\end{equation*} Using Lemmata \ref{lem:compadunsob} and \ref{lem:probacom} we obtain.
\begin{align*}|T_{\lambda}f(x)|^2&\leq C_{s} \|m_0 \|_{W_s}^2 \int_{\mathbb{R}^2} |\tau^y f(x)|^2 \delta_{\la}(w_{-2s}) (y)\,d\nu_{\alpha}(y) \\
&\leq C_{s,\alpha} \|m_0 \|_{W_s}^2 \sum_{\varepsilon\in\{0,1\}^d}\int_{\mathbb{R}^2} \tau_{\varepsilon}^y (|f_{\varepsilon}|^2)(x) \delta_{\la}(w_{-2s}) (y)\,d\nu_{\alpha}(y).\end{align*}
At this point, the product structures of the (modified) Dunkl convolution and the measure $\nu_{\alpha}$ allow us to focus only on the one-dimensional situation.
First, for $f_e,$ the even part of $f,$ we have \begin{equation*}\int_{\mathbb{R}} \tau_{0}^y (|f_{e}|^2)(x) \delta_{\la}(w_{-2s})(y)\,d\nu_{\alpha}(y)=\int_{\mathbb{R}}  |f_{e}|^2(y) \tau^y\delta_{\la}(w_{-2s})(x)\,d\nu_{\alpha}(y),\end{equation*} as desired. Consider now $f_o,$ the odd part of $f.$ Let $g\in L^1(\mathbb{R},\nu_{\alpha})$ be a radial (in our case even) function. Then, the change of variable $y\to -y,$ gives
\begin{align*}
&b_{\alpha}\int_{\mathbb{R}}g(y)\int_{-1}^1 |f_o|^2(\sqrt{x^2+y^2-2xy u})\,u\,(1-u^2)^{\alpha-1}\,du\,dy=0.
\end{align*}
Consequently,
\begin{align*}
&\int_{\mathbb{R}}g(y)\tau_{1}^y (|f_o|^2)(x)\,d\nu_{\alpha}(y)=\int_{\mathbb{R}}g(y)\,b_{\alpha}\int_{-1}^1 |f_o|^2(\sqrt{x^2+y^2-2xy u})(1-u^2)^{\alpha-1}\,du\,dy\\
&= \int_{\mathbb{R}}g(y)\,b_{\alpha}\int_{-1}^1 |f_o|^2(\sqrt{x^2+y^2-2xy u})(1+u)(1-u^2)^{\alpha-1}\,du\,dy\\
&=\int_{\mathbb{R}}g(y)\tau^y (|f_o|^2)(x)\,d\nu_{\alpha}(y).
\end{align*}
Taking $g(y)=\delta_{\lambda}(w_{-2s})(y)$ we are done. \end{proof}
The next ingredient we need is a specific maximal function. For $r=1,\ldots,d,$ let
\begin{equation*}M_rf(x)=\sup_{t_r>0}\left|\frac{\int_{\mathbb{R}}f(x_1,\ldots, x_{r-1},y_r,x_{r+1}\ldots,x_d)\tau^{x_r}\chi_{[-t_r,t_r]}(y_r)\,d\nu_{\alpha_r}(y_r)}{\nu_{\alpha_r}([-t_r,t_r])}\right|,\end{equation*}
i.e., $M_r$ is the one-dimensional maximal function in the $\mathbb{Z}_2$ context (as studied in \cite{convmax}), applied to $f$ as a function of variable $y_r.$ Let $M_Pf(x)=M_1\circ\cdots\circ M_d (|f|)(x).$
 From the product structure of $M_P$ and \cite[Theorem 6.1]{convmax} (in the one-dimensional case) it follows that $M_P$ is bounded on $L^p,$ $1<p<\infty.$ Moreover, the following is true.
\begin{lem}
\label{lem:maxboundlem}
Let $s>\alpha+{\bf 1/2},$ so that $w_{-2s}\in L^1.$ Then, for a.e.\ $x\in\mathbb{R}^d,$
\begin{equation}
\label{chap:Ck,sec:MarDun,eq:maxbound}
|\delta_{\lambda}(w_{-2s})\star h(x)|\leq C_s M_P(|h|)(x),\qquad h\in L^p,\quad 1\leq p<\infty,
\end{equation}
with a constant $C_s$ independent of $\lambda\in\Rdp.$
\end{lem}
\begin{proof} Observe first that in the one-dimensional case \eqref{chap:Ck,sec:MarDun,eq:maxbound} is a consequence of \cite[Theorem 7.5]{convmax} (here $2\lambda_k+2=2\alpha+1$). Indeed, if $s>\alpha+1/2,$ then taking $\phi_0(|x|)=w_{-2s}(x)$ we see that all the assumptions of \cite[Theorem 7.5]{convmax} are satisfied (note that due to a slight misprint in the statement of \cite[Theorem 7.5]{convmax} we should actually assume that $\int_0^{\infty}R^{2\lambda_k+2} |\phi_0'(R)|\,dR<\infty,$ which is in fact precisely what we have).

In the multi-dimensional case we use the one-dimensional version together with the product structure of the measure $\nu_{\alpha},$ and the Dunkl translation and convolution. \end{proof}
The last component we need in order to prove Theorem \ref{thm:dunklmarmult} is a multivariate Littlewood-Paley theory for the Dunkl transform expressed in the following.


\begin{thm}
\label{thm:litpalmar}
Fix $1<p<\infty$ and let $\psi \in C_c^{\infty}(\mathbb{R})$ be a function supported in $[-4,-1/4]\cup[1/4,4].$ Define $S_{j},$ $j\in\mathbb{Z}^d,$ by \begin{equation*}D_{\alpha} (S_{j}f)(y)=\psi(2^{-j_1}y_1)\cdots\psi(2^{-j_d}y_d)D_{\alpha} f(y).\end{equation*}
Then
\begin{equation}
\label{chap:Ck,sec:MarDun,eq:ineje}
\|f\|_p\lesssim\bigg\|\bigg(\sum_{j\in\mathbb{Z}^d}|S_{j} f|^2\bigg)^{1/2}\bigg\|_{p}\lesssim \|f\|_{p},
\end{equation}
where the left hand side of the above inequality holds under the additional assumption $\sum_{j_r\in\mathbb{Z}}|\psi(2^{-j_r}\xi)|^2=C,$ $\xi\neq0.$
\end{thm}
\begin{proof}[Proof (sketch)] For a moment let $n\in\{1,\ldots,d\}$ be fixed. We claim that for $f\in L^p(\mathbb{R},\nu_{\alpha_r}),$ and arbitrary $\varepsilon_r^{j_r}\in \{-1,1\}^{\mathbb{Z}}$
\begin{equation}
\label{chap:Ck,sec:MarDun,eq:epsa}
\bigg\|\sum_{j_r\in\mathbb{Z}}\varepsilon_r^{j_r} S_{j_r} f\bigg\|_{L^p(\mathbb{R},\nu_{\alpha_r})}\leq C_{p,\alpha_r} \|f\|_{L^p(\mathbb{R},\nu_{\alpha_r})},
\end{equation}
where $D_{\alpha_r} (S_{j_r} f)(y_r)=\psi(2^{-j_r}y_r)D_{\alpha_r} f (y_r)$ and the constant $C_{p,\alpha_r}$ is independent of $\{\varepsilon_r^{j_r}\}_{j_r\in\mathbb{Z}}.$ If \eqref{chap:Ck,sec:MarDun,eq:epsa} holds, then since $S_{j}=S_{j_1}\cdots S_{j_d},$ the standard Rademacher function (Khintchine's inequality) trick (see Lemma \ref{lem:randkhin}) together with Fubini's theorem easily imply the right hand side of \eqref{chap:Ck,sec:MarDun,eq:ineje}. Then, the left hand side follows by using the polarization method. Coming back to \eqref{chap:Ck,sec:MarDun,eq:epsa} we observe that it is a consequence of \cite[Theorem 3.1]{Sifi1}. \end{proof}

Note that in Theorem \ref{thm:litpalmar} we are interested in 'product' functions $\Psi=\psi\otimes\cdots\otimes \psi$ and multi-parameter dilations $\Psi(2^{-j_1}y_1,\cdots,2^{-j_d}y_d),$ $j=(j_1,\ldots,j_d)\in\mathbb{Z}^d.$ Therefore the results of \cite{meja}, which assume that $\Psi$ is radial and deal with one-parameter dilations $\Psi(2^{-l}y),$ $l\in\mathbb{Z},$ are not applicable in our case.

Having collected Lemmata \ref{lem:dilest}, \ref{lem:maxboundlem}, and Theorem \ref{thm:litpalmar} we proceed to the rather standard proof of Theorem \ref{thm:dunklmarmult}.

\begin{proof}[Proof of Theorem \ref{thm:dunklmarmult}] Once again, for simplicity, we take $d=2.$ In view of \eqref{chap:Ck,sec:MarDun,eq:compasym}, with no loss of generality we may assume that $f$ is $\varepsilon$-symmetric, for some $\varepsilon\in\{0,1\}^2.$ By a density argument we can further restrict to $f\in\mathcal{S}(\mathbb{R}^d).$ Let $\psi$ be the function from Theorem \ref{thm:dunklmarmult} and set $D_{\alpha}(S_{j,k}f)=\psi(2^{-j}y_1)\psi(2^{-k}y_2)D_{\alpha} f(y_1,y_2).$ Let $\widetilde{\psi}$ be a non-zero even $C_c^{\infty}(\mathbb{R})$ function supported in $[-4,-1/4]\cup[1/4,4]$ and equal to $1$ on $[-2,-1/2]\cup[1/2,2].$ Defining \begin{equation*}D_{\alpha} (\tilde{S}_{j,k}f)(y_1,y_2)=\widetilde{\psi}(2^{-j}y_1)\widetilde{\psi}(2^{-k}y_2)D_{\alpha} f(y_1,y_2)\end{equation*} we easily see that $S_{j,k}\tilde{S}_{j,k}=S_{j,k}.$ Indeed, the operator $S_{j,k}$ is defined by means of the given function $\psi,$ which is supported in $[-2,-1/2]\cup[1/2,2].$ Now, from the left hand side of \eqref{chap:Ck,sec:MarDun,eq:ineje} applied to the family $\{S_{j,k}\}$
\begin{equation*}\|\T_m f\|_{p}\lesssim \bigg\|\bigg(\sum_{j,k\in\mathbb{Z}}|S_{j,k} \T_m f|^2\bigg)^{1/2}\bigg\|_{p}=\bigg\|\bigg(\sum_{j,k\in\mathbb{Z}}|S_{j,k} \T_m \tilde{S}_{j,k}f|^2\bigg)^{1/2}\bigg\|_{p}.\end{equation*}
Let $g_{j,k}=\tilde{S}_{j,k}f.$ Observe that $g_{j,k}\in\mathcal{S}(\mathbb{R}^d),$ $j,k\in\mathbb{Z}.$ Moreover, $(g_{j,k})_{\varepsilon}=g_{j,k},$ i.e. $g_{j,k}$ is $\varepsilon$-symmetric. The latter is clear once we recall that $f$ is $\varepsilon$-symmetric, $\tilde{\psi}\otimes\tilde{\psi}$ is ${\bf 0}$-symmetric and the Dunkl transform commutes with the action of $G.$ Assume for a moment that $p>2$ and let $1/(p/2)'+2/p=1.$ Then there exists $h\in L^{(p/2)'}(\nu_{\alpha})$ with norm $1$ such that
\begin{equation*}\bigg\|\bigg(\sum_{j,k\in\mathbb{Z}}|S_{j,k} \T_m g_{j,k}|^2\bigg)^{1/2}\bigg\|_{p}=\bigg(\sum_{j,k\in\mathbb{Z}}\,\int_{\mathbb{R}^2}|S_{j,k} \T_m\, g_{j,k}(x)|^2 h(x)\,d\nu_{\alpha}(x)\bigg)^{1/2}.\end{equation*} Since the multiplier associated with $S_{j,k} \T_m$ is $\Psi(2^{-j}y_1,2^{-k}y_2)m(y)$ and $g_{j,k}$ is an $\varepsilon$-symmetric function from $\mathcal{S}(\mathbb{R}^d),$ using Lemma \ref{lem:dilest} we have
\begin{align*}&\int_{\mathbb{R}^2}|S_{j,k} \T_m\, g_{j,k}(x)|^2 |h(x)|\,d\nu_{\alpha}(x)\\& \leq \|m\|^2_{W_{s,loc}}\int_{\mathbb{R}}\int_{\mathbb{R}}|g_{j,k}(y)|^2 \tau^{y}(\delta_{2^{-j},2^{-k}}(w_{-2s}))(x)\,d\nu_{\alpha}(y)\, |h(x)|\,d\nu_{\alpha}(x).\end{align*}
From the above inequality, using Fubini's theorem and Lemma \ref{lem:maxboundlem}, we get \begin{align*}&\bigg(\sum_{j,k\in\mathbb{Z}}\,\int_{\mathbb{R}^2}|S_{j,k} \T_m\, g_{j,k}(x)|^2 h(x)\,d\nu_{\alpha}(x)\bigg)^{1/2}\\&\lesssim \bigg(\int_{\mathbb{R}^2}\,\sum_{j,k\in\mathbb{Z}}|g_{j,k}(y)|^2\, |h|\star \delta_{2^{-j},2^{-k}}(w_{-2s}) (y)\,d\nu_{\alpha}(y)\bigg)^{1/2}\\&\lesssim \bigg(\int_{\mathbb{R}^2}\,\sum_{j,k\in\mathbb{Z}}|g_{j,k}(y)|^2\, M_P(h) (y)\,d\nu_{\alpha}(y)\bigg)^{1/2}.\end{align*}
Hence, applying H\"{o}lder's inequality together with the $L^p$ boundedness of $M_P$ we arrive at
\begin{equation*}\|\T_m f\|_{p}\lesssim \bigg\|\bigg(\sum_{j,k\in\mathbb{Z}}|\tilde{S}_{j,k}f|^2\bigg)^{1/2}\bigg\|_{p}.\end{equation*} Observing that, by Theorem \ref{thm:litpalmar}, the family $\{\tilde{S}_{j,k}\}$ satisfies the right hand side inequality in \eqref{chap:Ck,sec:MarDun,eq:ineje} we get the desired conclusion. Now, for $1<p<2,$ a duality argument completes the proof. \end{proof}

Finally, as a corollary of Theorem \ref{thm:dunklmarmult} and Lemma \ref{lem:dunhan1} we obtain the following Marcinkiewicz type multiplier theorem for the Hankel transform from Section \ref{chap:Ck,sec:HorHan}. Note that Theorem \ref{thm:MarHan} was proved earlier in \cite[Appendix]{cit:ja} without using the link with the Dunkl transform setting. Throughout the statement and proof of Theorem \ref{thm:MarHan} by $\psi$ we denote a non-zero $C_c^{\infty}((0,\infty))$ function supported in $[1/2,2]$ and such that $\sum_{l\in\mathbb{Z}}|\psi(2^{-l}\xi)|^2=1,$ $\xi>0.$ We also set $\Psi(y_1,\ldots,y_d)=\psi(y_1)\cdots\psi(y_d).$
\begin{thm}
\label{thm:MarHan}
Let $m$ be a bounded measurable function on $\Rdp.$ Assume that
\begin{equation}
\label{chap:Ck,sec:MarDun,eq:conemHan}
\sup_{j\in\mathbb{Z}^d} \|\Psi\, m(2^{j_1}\cdot,\ldots,2^{j_d}\cdot)\|_{W_s}<\infty,
\end{equation}
for some $s>\alpha+\bf{\frac{1}{2}}.$ Then the multiplier operator for the Hankel transform $\T_m^{\Hi},$ given by \eqref{chap:Ck,sec:HorHan,eq:spHan}, is bounded on all the $L^p(\Rdp,\nu_{\alpha})$ spaces, $1<p<\infty.$
\end{thm}
\begin{proof}
For a given $1<p<\infty$ take a function $f\in L^2(\Rdp,d\nu_{\alpha})\cap L^p(\Rdp,d\nu_{\alpha}).$ Consider the ${\bf 0}$-symmetric extensions of $f,$ $m,$ and $\Psi,$ i.e.\ the functions $\tilde{f},$ $\tilde{m}$ and $\tilde{\Psi}$ on $\mathbb{R}^d$ satisfying the identities $\tilde{f}(x)=f(|x_1|,\ldots,|x_d|),$ $\tilde{m}(x)=m(|x_1|,\ldots,|x_d|)$ and $\tilde{\Psi}(x)=\psi(|x_1|)\cdots \psi(|x_d|),$ for $x\in\mathbb{R}^d.$

Then, the function $\T_{\tilde{m}}^D(\tilde{f})$ is also ${\bf 0}$-symmetric, and from Lemma \ref{lem:dunhan1} it follows that
$\T_m^{\Hi}(f)(x)=\T_{\tilde{m}}^D(\tilde{f})(x),$ $x\in\Rdp.$ Now, it is not hard to verify that
\begin{equation*}
C_{d,s}^{-1}\|\Psi\, m(2^{j_1}\cdot,\ldots,2^{j_d}\cdot)\|_{W_s} \leq \|\tilde{\Psi}\,\tilde{m}(2^{j_1}\cdot,\ldots,2^{j_d}\cdot)\|_{W_s}\leq C_{d,s}\|\Psi\, m(2^{j_1}\cdot,\ldots,2^{j_d}\cdot)\|_{W_s}.
\end{equation*}
Thus, using the assumption \eqref{chap:Ck,sec:MarDun,eq:conemHan} and Theorem \ref{thm:dunklmarmult} with $\tilde{m}$ and $\tilde{\Psi}$ in place of $m$ and $\Psi,$ we obtain the boundedness of the operator $\T_{\tilde{m}}^D$ on $L^p.$ Consequently,
\begin{align*}
2^{d/p}\|\T_m^{\Hi}(f)\|_{L^p(\Rdp,\nu_{\alpha})}=\|\T_{\tilde{m}}^D(\tilde{f})\|_{p}\leq C_{p,\tilde{m},\tilde{\Psi}}\|\tilde{f}\|_{p}\leq C_{p,m,\Psi}2^{d/p}\|f\|_{L^p(\Rdp,\nu_{\alpha})},
\end{align*}
and the proof is finished.
\end{proof}


    %
    %

    \newchapter{Systems of operators having $H^{\infty}$ functional calculus}{Systems of operators having $H^{\infty}$ functional calculus}{Systems of operators having $H^{\infty}$ functional calculus}
        \label{Chap:Hinf}
In the present chapter we turn our attention to systems of operators which only have $H^{\infty}$ functional calculus.

In Section \ref{chap:Hinf,sec:genMarHinf}, using Theorem \ref{thm:gen} and modifying the proof of Theorem \ref{thm:genMarCk}, we prove a Marcinkiewicz type multiplier theorem, see Theorem \ref{thm:genMarHinf}. Then, in Section \ref{chap:Hinf,sec:ExOp}, we focus on particular examples of the Ornstein-Uhlenbeck and the Laguerre operators (in the setting of Laguerre polynomial expansions). Finally, in Section \ref{chap:Hinf,sec:HolExt}, we prove a fairly general multivariate holomorphic extension theorem, see Theorem \ref{thm:HolExt}. As an application of this theorem, in Corollary \ref{cor:HolExtOULag} we obtain some multivariate holomorphic extension properties for joint multipliers of the Ornstein-Uhlenbeck and the Laguerre operators.



        \section[A Marcinkiewicz type multiplier theorem]{A Marcinkiewicz type multiplier theorem}
        \label{chap:Hinf,sec:genMarHinf}
Consider a general system of self-adjoint non-negative strongly commuting operators $L=(L_1,\ldots,L_d)$ on some space $L^2(X,\nu).$ The assumptions of Section \ref{chap:Intro,sec:Setting} are still in force. We also keep the short notation of Chapter \ref{Chap:General} and Section \ref{chap:Ck,sec:genMar}, for the $L^p$ space, $L^p$ norm and $L^p$ operator norm. Recall that, for a Borel measurable function $m$ on $\Rdp,$ the multiplier operator $m(L)$ is given by \eqref{chap:Intro,sec:Setting,eq:mdef}.

As a consequence of Cowling's result, \cite[Corollary 1]{Hanonsemi}, for each $1<p<\infty$ and $r=1,\ldots,d,$ we have
\begin{equation}
\label{chap:Hinf,sec:genMarHinf,eq:Cowest}
\|L_r^{iv}\|_{p\to p}\leq C_{p}(1+|v|)^{4|1/p-1/2|}\exp(\pi |1/p-1/2||v|),\qquad v\in\mathbb{R},
\end{equation}
where the constant $C_p$ is independent of the system $L.$ Recently Carbonaro and Dragi\v{c}evi\'{c} \cite[Proposition 11]{Carb-Drag} proved that, for every $1<p<\infty,$
\begin{equation}
\label{chap:Hinf,sec:genMarHinf,eq:Carb-Dragest0}
\|L_r^{iv}\|_{p\to p}\leq C_{p}(1+|v|)^{1/2}\exp(\pst|v|),\qquad v\in\mathbb{R},
\end{equation}
where $\pst=\arcsin|2/p-1|.$ By interpolation the above inequality leads to the bound
\begin{equation}
\label{chap:Hinf,sec:genMarHinf,eq:Carb-Dragest}
\|L_r^{iv}\|_{p\to p}\leq C_{p,\varepsilon}(1+|v|)^{(1+\varepsilon)|1/p-1/2|}\exp(\pst|v|),\qquad v\in\mathbb{R},
\end{equation}
valid for each $r=1,\ldots,d,$ every $1<p<\infty,$ and arbitrary $\varepsilon>0.$ Since clearly, $\pst<\pi |1/p-1/2|$ for $1<p<\infty,$ we see that \eqref{chap:Hinf,sec:genMarHinf,eq:Carb-Dragest} improves \eqref{chap:Hinf,sec:genMarHinf,eq:Cowest}.

On the other hand, if $\mL$ is the self-adjoint extension of the $d$-dimensional Ornstein-Uhlenbeck operator, then from \cite[Theorem 3.5]{hmm} and \cite[Theorem 1.2]{sharp} (combined with \cite[Theorem 3]{PrusSohr1}) it follows that, for every $1<p<\infty$
\begin{equation}\label{chap:Hinf,sec:genMarHinf,eq:OUimabound}C_{p,d}^{-1}\,e^{\pst|v|}\leq \|(\mL+I)^{iv}\|_{p\to p}\leq C_{p,d}\,e^{\pst|v|},\qquad v\in \mathbb{R}.\end{equation}
Note that the left hand side inequality above implies that the operator $\mL+I$ cannot have a Marcinkiewicz functional calculus. The right hand side inequality in \eqref{chap:Hinf,sec:genMarHinf,eq:OUimabound} shows that the polynomial factor in \eqref{chap:Hinf,sec:genMarHinf,eq:Carb-Dragest} can be dropped for particular operators satisfying the assumptions of Section \ref{chap:Intro,sec:Setting}, even when they do not have a Marcinkiewicz functional calculus.

Throughout this section we assume that there exist $\theta=(\theta_1,\ldots,\theta_d)\in [0,\infty)^d$ and, for each $1<p<\infty,$  $\phi_p=(\phi_p^1,\ldots,\phi_p^d)\in (0,\pi/2)^d,$ such that
\begin{equation}
\label{chap:Hinf,sec:genMarHinf,eq:polynomial}
\|L_r^{iv}\|_{p\to p}\leq \mC(p,L_r)(1+|v|)^{\theta_r|1/p-1/2|}\exp(\phi_p^r|v|),\qquad v\in\mathbb{R},\quad r=1,\ldots,d.
\end{equation}
In view of \cite[Section 5]{CowDouMcYa}, for each fixed $r=1,\ldots,d,$ the condition \eqref{chap:Hinf,sec:genMarHinf,eq:polynomial} is equivalent to the fact that the operator $L_r$ has an $H^{\infty}$ functional calculus in a sector slightly larger than $S_{\phi_p^r}.$ Moreover, \eqref{chap:Hinf,sec:genMarHinf,eq:polynomial} clearly implies that for each $1<p<\infty$
\begin{equation}
\label{chap:Hinf,sec:genMarHinf,eq:polynomialfull}
\|L^{iu}\|_{p\to p}\leq \mC(p,L) \prod_{r=1}^d(1+|u_r|)^{\theta_r|1/p-1/2|}\exp(\phi_p^r|u_r|),\qquad u=(u_1,\ldots,u_d)\in\mathbb{R}^d.
\end{equation}
As in Section \ref{chap:Ck,sec:genMar}, we write the constants $\mC(p,L_r)$ in \eqref{chap:Hinf,sec:genMarHinf,eq:polynomial} and $\mC(p,L)$ in \eqref{chap:Hinf,sec:genMarHinf,eq:polynomialfull} in the calligraphic font because we need to keep track of them in Theorem \ref{thm:genMarHinf}.

In order to state the main result of this section we need some more notation. For a function $m\colon \overline{{\bf S}_{\phi_p}}\to \mathbb{C},$ $\phi_p=(\phi_p^1,\ldots,\phi_p^d),$ and $\varepsilon\in \{-1,1\}^d,$ set $$m(e^{i\varepsilon\phi_p}\lambda)=m(e^{i\varepsilon_1\phi_p^1}\la_1,\ldots,e^{i\varepsilon_d\phi_p^d}\la_d),\qquad \la=(\la_1,\ldots,\la_d)\in \mathbb{R}^{d}_+.$$ Note that by (multivariate) Fatou's theorem, every $m\in H^{\infty}({\bf S}_{\phi_p})$ admits a non-tangential limit a.e.\ on the boundary of ${\bf S}_{\phi_p}.$ Slightly abusing the notation we continue writing $m$ for this extension.

        The main result of this section, Theorem \ref{thm:genMarHinf}, is close to a previous result of Albrecht et al.\ \cite[Theorem 5.4]{AlFrMc}. The difference is that we assume less on the multiplier function $m$ (it does not need to be holomorphic in a larger polysector), but more on the operators $L_r,$ $r=1,\ldots,d$ (they are contractions on all $L^p$ spaces). Theorem \ref{thm:genMarHinf} is also a multivariate generalization of a result of Garc\'ia-Cuerva's et al., \cite[Theorem 2.2]{funccalOu}. Recall that the operators $L_r,$ $r=1,\ldots,d,$ satisfy all the assumptions of Section \ref{chap:Intro,sec:Setting}, $m(L)$ is defined by \eqref{chap:Intro,sec:Setting,eq:mdef}, while $\theta=(\theta_1,\ldots,\theta_d)$ appears in \eqref{chap:Hinf,sec:genMarHinf,eq:polynomialfull}.

\begin{thm}
\label{thm:genMarHinf}
Let $1<p<\infty$ be given. Assume $m\in H^{\infty}({\bf S}_{\phi_p})$ is such that the boundary value functions $m(e^{i\varepsilon\phi_p}\lambda),$ $\varepsilon \in \{-1,1\}^d,$ satisfy the Marcinkiewicz condition \eqref{chap:Intro,sec:Notation,eq:Marcon} of some order $\rho>|1/p-1/2|\theta+{\bf 1}.$ Then the multiplier operator $m(L)$ is bounded on $L^p$ and
$$\|m(L)\|_{p\to p}\leq C_{p,d}\, \mC(p,L)\,\max_{\varepsilon \in\{-1,1\}^d}\|m(e^{i\varepsilon\phi_p}\cdot)\|_{Mar,\rho}.$$
\end{thm}
\begin{proof}
The proof is a multivariate modification of the proof of \cite[Theorem 2.2]{funccalOu} (cf. also the proof of \cite[Theorem 4.2]{jaOU}). We shall demonstrate that $m$ satisfies the assumption \eqref{chap:General,eq:thm:gen} from Theorem \ref{thm:gen}. Clearly, it is enough to obtain, for some $N\in\mathbb{N}^d,$ $N>\rho,$ the uniform in $u=(u_1,\ldots,u_d)\in\mathbb{R}^d$ bound \begin{align}\nonumber&\sup_{t\in \Rdp}|\M(m_{N,t})(u)|\\
&\leq C_{N,d}\, \prod_{r=1}^d(1+|u_r|)^{-\rho_r|1/p-1/2|}\exp(-\phi_p^r|u_r|)\,\max_{\varepsilon \in\{-1,1\}^d}\|m(e^{i\varepsilon\phi_p}\cdot)\|_{Mar,\rho}\label{chap:Hinf,sec:genMarHinf,bound}.\end{align}

We focus on proving \eqref{chap:Hinf,sec:genMarHinf,bound}. Defining $\Rde=\{x\in\mathbb{R}^d\colon \varepsilon_r x_r\geq 0,\, r=1,\ldots,d\},$ with $\varepsilon\in \{-1,1\}^d,$ we see that it suffices to obtain \eqref{chap:Hinf,sec:genMarHinf,bound} separately on each $\Rde.$ Observe that, for each fixed $N\in \mathbb{N}^d,$ $t\in \Rdp$ and $u\in\mathbb{R}^d,$  $${\bf S}_{\phi_p}\ni z \mapsto m_{N,t}(z)z^{-iu-{\bf 1}}=t^{N}\la^{N-iu-{\bf 1}}\exp(-2^{-1}\langle z, t\rangle)m(z)$$ is a bounded holomorphic function on ${\bf S}_{\phi_p},$ which is rapidly (exponentially) decreasing when $\Real(z_r)\to \infty,$ $r=1,\ldots,d.$ Thus, for each $\varepsilon\in\{-1,1\}^d,$ we can use (multivariate) Cauchy's integral formula to change the path of integration in the integral defining $\M(m_{N,t})(u)$ to the poly-ray $\{(e^{i\varepsilon_1\phi_p^1}\la_1,\ldots,e^{i\varepsilon_d\phi_p^d}\la_d)\colon \la \in \Rdp\}.$ Then we  obtain
\begin{align}\nonumber
&e^{-\langle u,(\varepsilon \phi_p)\rangle}\M(m_{N,t})(u)\\ \label{chap:Hinf,sec:genMarHinf,Mform}
&=e^{i\langle N,(\varepsilon \phi_p)\rangle}\int_{\Rdp}t^N\la^N\exp(-2^{-1}\langle (e^{i\varepsilon_1\phi_p^1}t_1,\ldots, e^{i\varepsilon_d\phi_p^d}t_d),\la\rangle)m(e^{i\varepsilon\phi_p}\lambda)\la^{-iu}\,\frac{d\la}{\la},
\end{align}
where $\varepsilon \phi_p=(\varepsilon_1 \phi_p^1,\ldots,\varepsilon_d \phi_p^d).$
Using \eqref{chap:Hinf,sec:genMarHinf,Mform}, we can essentially repeat the proof of Theorem \ref{thm:genMarCk} (cf.\ \eqref{chap:Ck,sec:genMar,eq:estiHinf}) with $m$ replaced by $m(e^{i\varepsilon\phi_p}\cdot)$, thus proving \eqref{chap:Hinf,sec:genMarHinf,bound} for $-u\in\Rde.$
\end{proof}
\begin{remark} From (multivariate) Cauchy's integral formula we see that if $m\in H^{\infty}({\bf S}_{\varphi_p}),$ for some $\varphi_p=(\varphi_p^1,\ldots,\varphi_p^d)>\phi_p,$ then $\|m(e^{i\varepsilon\phi_p}\cdot)\|_{Mar,\rho}\leq C(\rho,\varepsilon,\varphi_p)\|m\|_{H^{\infty}({\bf S}_{\varphi_p})},$ for arbitrary $\rho\in \mathbb{N}^d.$ Consequently, we have
$$\|m(L)\|_{p\to p}\leq C_{p,d}\,\mC(p,L)\,\|m\|_{H^{\infty}({\bf S}_{\varphi_p})}.$$\end{remark}

By referring to \eqref{chap:Hinf,sec:genMarHinf,eq:Carb-Dragest} instead of \eqref{chap:Hinf,sec:genMarHinf,eq:polynomial} we also obtain the following seemingly weaker corollary. The applicability of Corollary \ref{cor:genMarHinfdimindep} lies in the fact that it is independent of the particular choice of the system of operators $L=(L_1,\ldots,L_d),$ as long as it satisfies all the assumptions of Section \ref{chap:Intro,sec:Setting}. Recall that $\pst=\arcsin|2/p-1|.$
\begin{cor}
\label{cor:genMarHinfdimindep}
Let $1<p<\infty$ be given and take $\varphi_p>(\pst,\ldots,\pst).$ If $m\in H^{\infty}({\bf S}_{\varphi_p}),$ then the operator $m(L)$ is bounded on $L^p$ and
$$\|m(L)\|_{p\to p}\leq C_{p,d}\,\|m\|_{H^{\infty}({\bf S}_{\varphi_p})}.$$
Moreover, the constant $C_{p,d}$ is independent of the particular system $L$ we consider.
\end{cor}
\begin{proof}
We just need to use the Remark following the proof of Theorem \ref{thm:genMarHinf} and \eqref{chap:Hinf,sec:genMarHinf,eq:Carb-Dragest} instead of \eqref{chap:Hinf,sec:genMarHinf,eq:polynomial}. The crucial fact is that the constant $C_{p,\varepsilon}$ in  \eqref{chap:Hinf,sec:genMarHinf,eq:Carb-Dragest} does not depend on the operator $L_r,$ $r=1,\ldots,d.$
\end{proof}
\begin{remark}
The corollary can be also deduced from a variant of \cite[Theorem 5.4]{AlFrMc} (with $d$ operators) combined with \cite[Theorem 1]{Carb-Drag}.
\end{remark}
        \section[Examples of systems with $H^{\infty}$ joint functional calculus]{Examples of systems with $H^{\infty}$ joint functional calculus}
         \label{chap:Hinf,sec:ExOp}
         We present two particular examples of systems of operators which are known to have only an $H^{\infty}$ functional calculus. The results from this section are contained in \cite{jaOU}.

         The first of these systems is that of the one-dimensional Ornstein-Uhlenbeck operators on $L^2(\mathbb{R}^d,\gamma),$ where $\gamma(x)=\pi^{-d/2}e^{-|x|^2}.$ The operators $\mL_r$ are given by
        $$\mL_r=-\frac{1}{2}\frac{\partial^2}{\partial{x_r}^2}+x_r\frac{\partial}{\partial x_r}, \qquad r=1,\ldots,d.$$ Then, for each $r=1,\ldots,d,$ the operator $\mL_r$ is symmetric and non-negative on $C_c^{\infty}(\mathbb{R}^d),$ with respect to the inner product on $L^2(\mathbb{R},\gamma_r);$ here $\gamma_r$ is the Gaussian probability measure $\gamma_r(x_r)=\pi^{-1/2}e^{-x_r^2}.$

        Moreover, the classical one-dimensional Hermite polynomials $H_{k_r}$ in the variable $x_r,$ see \cite{szego}, are eigenfunctions of $\mL_r$ corresponding to the eigenvalues $k_r\in\mathbb{N}_0,$ i.e.\ $\mL_r H_{k_r}=k_r H_{k_r}.$ Then, for each $r=1,\ldots,d,$ a standard procedure allows us to consider a self-adjoint extension of $\mL_r$ (still denoted by the same symbol) defined by $$\mL_rf(x_r)=\sum_{k_r=0}^{\infty}k_r\langle f,\nH_{k_r}\rangle_{L^2(\mathbb{R},\gamma_r)} \nH_{k_r}(x_r),$$ on the usual domain $$\Dom(\mL_r)=\{f\in L^2(\mathbb{R},\gamma_r)\colon \sum_{k_r=0}^{\infty}k_r^2 |\langle f,\nH_{k_r}\rangle_{L^2(\mathbb{R},\gamma_r)}|^2<\infty\}.$$
Here $\nH_{k_r}=\|H_{k_r}\|_{L^2(\mathbb{R},\gamma_r)}^{-1}H_{k_r}$ is the $L^2(\mathbb{R},\gamma_r)$ normalized one--dimensional Hermite polynomial of degree $k_r$ in the $x_r$ variable.

Using tensor products the operators $\mL_r$ are then lifted to $L^2(\mathbb{R}^d,\gamma)$ via \eqref{chap:Intro,sec:Notation,eq:tensnot}. Observe that $\mL_r$ thus defined admits the following spectral resolution
\begin{equation}\label{chap:Hinf,sec:ExOp,eq:system} \mL_rf=\sum_{k\in\mathbb{N}_0^d}k_r \langle f,\bnH_{k}\rangle_{L^2(\mathbb{R}^d,\gamma)} \bnH_{k},\qquad f\in L^2(\mathbb{R}^d,\gamma),\end{equation}
where $\bnH_k=\nH_{k_1}\otimes\cdots\otimes \nH_{k_d}.$  The operator given by \eqref{chap:Hinf,sec:ExOp,eq:system} is self-adjoint on the domain $$\Dom(\mL_r)=\{f\in L^2(\mathbb{R}^d,\gamma)\colon \sum_{k\in\mathbb{N}^d_0}k_r^2 |\langle f,\bnH_{k}\rangle_{L^2(\mathbb{R}^d,\gamma)}|^2<\infty\}.$$  Moreover, $(\mL_1,\ldots,\mL_d)$ is a system of strongly commuting self-adjoint operators on $L^2(\mathbb{R}^d,\gamma)$.

In this case, for $m\colon \mathbb{N}_0^d\to \mathbb{C},$ the joint spectral multiplier operator defined by \eqref{chap:Intro,sec:Setting,eq:mdef} is
\begin{equation}
\label{chap:Hinf,sec:ExOp,eq:def}
m(\mL_1,\ldots,\mL_d)f=\sum_{k\in\mathbb{N}^d_0}m(k_1,\ldots,k_d)\langle f, \bnH_k\rangle_{L^2(\mathbb{R}^d,\gamma)} \bnH_k,\qquad f\in L^2(\mathbb{R}^d,\gamma).
\end{equation}

It is known that the operators $\mL_r,$ $r=1,\ldots,d,$ do not have a Marcinkiewicz functional calculus. This can be seen as a consequence of a holomorphic extension theorem, see \cite[Theorem 3.5]{hmm}. Another way to justify this fact is to refer to the lower bound
\begin{equation}
\label{chap:Hinf,sec:ExOp,eq:lowBOU}
e^{\pst|u|}\leq C_p \|(\mL_{r,+})^{iu_r}\|_{L^p(\mathbb{R},\gamma_r)\to L^p(\mathbb{R},\gamma_r)},\qquad u_r \in \mathbb{R},\end{equation}
where \begin{equation}\label{chap:Hinf,sec:ExOp,eq:pst}\pst=\arcsin|2/p-1|,\end{equation} while $\mL_{r,+}$ is given by \eqref{chap:Hinf,sec:ExOp,eq:def} with $m(k)=m_r(k)=k_r\chi_{\{k_r\in\mathbb{N}\}},$ $r=1,\ldots,d.$ The bound \eqref{chap:Hinf,sec:ExOp,eq:lowBOU} is implicit in \cite{funccalOu} and restated in \cite[p.\ 448]{sharp}.

For the system $(\mL_1,\ldots,\mL_d)$ we have the following multivariate multiplier theorem, which is a generalization of \cite[Theorem 1]{funccalOu} and, to some extent, \cite[Theorem 1.2]{sharp}.
\begin{thm}
\label{thm:multOU} Assume that $m\in H^{\infty}({\bf S}_{\pst}),$ for some $p\in(1,\infty)\setminus\{2\}.$ Assume also that the boundary value functions $\Rdp\ni\la\mapsto m(e^{i\varepsilon\pst}\la),$ $\varepsilon=(\varepsilon_1,\ldots,\varepsilon_d)\in\{-1,1\}^d,$ satisfy the $d$-dimensional Marcinkiewicz condition \eqref{chap:Intro,sec:Notation,eq:Marcon} of some order $\rho> \bf{1},$ and that all the lower dimensional operators $m(\omega_1\mL_1,\ldots,\omega_d\mL_d)$ with $\omega=(\omega_1,\ldots,\omega_d)\in\{0,1\}^d,$ $\omega\neq {\bf 1},$
are bounded on $L^p(\mathbb{R}^d,\gamma).$ Then $m(\mL_1,\ldots,\mL_d)$ is bounded on $L^p(\mathbb{R}^d,\gamma).$
\end{thm}
\begin{remark} There are several ways to ensure the $L^p$ boundedness of the lower dimensional operators. One is to assume that $m$ is continuous on $[0,\infty)^d.$ Another is to assume that it satisfies appropriate lower dimensional Marcinkiewicz conditions of some orders greater than $\bf{1}.$ The sufficiency of the continuity condition is indicated in the proof of Theorem \ref{thm:multOU}. To see that satisfying lower dimensional Marcinkiewicz conditions is enough, one needs to use an easy inductive argument.\end{remark}
\begin{proof}[Proof of Theorem \ref{thm:multOU}]
Let $1<p<\infty$ be fixed. We start by showing that the bound
\begin{equation}\label{chap:Hinf,sec:ExOp,eq:imah}\|(\mL_{r}+\delta I)^{iu_r}\|_{L^p(\mathbb{R},\gamma_r)\to L^p(\mathbb{R},\gamma_r)}\leq C_p\, e^{\pst |u_r|},\qquad u_r\in\mathbb{R},\end{equation}
holds uniformly in $\delta>0.$

Applying the case $d=1$ of \cite[Theorem 4.3]{sharp} with $b=0$ we see that the lower bound in \eqref{chap:Hinf,sec:ExOp,eq:lowBOU} is in fact sharp, i.e.\ we have
\begin{equation}
\label{chap:Hinf,sec:ExOp,eq:upBOU}
\|(\mL_{r,+})^{iu_r}\|_{L^p(\mathbb{R},\gamma_r)\to L^p(\mathbb{R},\gamma_r)}\leq C_p\, e^{\pst|u|},\qquad u_r \in \mathbb{R}.\end{equation}
For each $r=1,\ldots,d,$ let $\mP_0^r f(x) = \langle f, \nH_0\rangle_{L^2(\mathbb{R},\gamma_r)} \nH_0(x_r)$ (that is $\mP_0^r$ is the projection onto the subspace spanned by $\nH_0$ applied to $f$ as a function of the variable $x_r$) and let $\mP_1^r=I-\mP_0^r$ (that is $\mathcal{P}_1^r$ is the projection onto the subspace spanned by $\{\nH_{k_r}\}_{k_r\in\mathbb{N}}$). Then, from the definition of $\mL_{r,+}$ it follows that $\mL_{r,+}=\mL_r\mP_1^r,$ and consequently,
$$(\mL_{r}+\delta I)^{iu_r}=(\mL_{r}+\delta I)^{iu_r}\,\mP_1^r+\delta^{iu_r}\,\mP_0^r=(\mL_{r,+}+\delta I)^{iu_r}\mP_1^r+ \delta^{iu_r}\mP_0^r.$$
Hence, recalling \eqref{chap:Hinf,sec:ExOp,eq:upBOU} and applying \cite[Theorem 3]{PrusSohr1}, together with the boundedness on $L^p(\mathbb{R},\gamma_r)$ of the projections $\mP_{\omega_r}^r,$ $\omega_r\in\{0,1\},$ we obtain the desired bound \eqref{chap:Hinf,sec:ExOp,eq:imah}.

Now, from Theorem \ref{thm:genMarHinf} together with \eqref{chap:Hinf,sec:ExOp,eq:imah}, we see that under our assumptions on $m$ the operator $m(\mL_1+\delta,\ldots,\mL_d+\delta)$ is bounded on $L^p(\mathbb{R}^d,\gamma).$ Moreover,
\begin{equation} \label{chap:Hinf,sec:ExOp,eq:inepsi}\|m(\mL_1+\delta,\ldots,\mL_d+\delta)\|_{L^p(\mathbb{R}^d,\gamma)\to L^p(\mathbb{R}^d,\gamma)}\leq C(p,d,m),\end{equation} independently of $\delta>0.$

Note that, if we additionally assume that $m$ is continuous on $[0,\infty)^d,$ then
$$\lim_{\delta\to 0^+}m(\mL_1+\delta,\ldots,\mL_d+\delta) f=m(\mL_1,\ldots,\mL_d)f,\qquad f\in L^2(\mathbb{R}^d,\gamma),$$
strongly in $L^2(\mathbb{R}^d,\gamma).$ Hence, \eqref{chap:Hinf,sec:ExOp,eq:inepsi} together with a density argument give the boundedness of $m(\mL_1,\ldots,\mL_d).$ In particular, using the fact that each $\mL_r,$ $r=1,\ldots,d,$ vanishes on functions not depending on the $r$-th variable, we can also obtain the boundedness of all the lower dimensional operators.

In the general case we proceed as follows. Since the projections $\mP_0^r,$ $r=1,\ldots,d,$ are bounded on $L^p(\mathbb{R}^d,\gamma),$  the same is true for the operator $\mP_1^1\cdots\mP_1^d.$ Consequently, from \eqref{chap:Hinf,sec:ExOp,eq:inepsi} it follows that
\begin{equation}
 \label{chap:Hinf,sec:ExOp,eq:inepsiP}
 \|m(\mL_1+\delta,\ldots,\mL_d+\delta)\mP_1^1\cdots\mP_1^d\|_{L^p(\mathbb{R}^d,\gamma)\to L^p(\mathbb{R}^d,\gamma)}\leq C(p,d,m).
\end{equation}

Observe now that
\begin{align*}
m(\mL_1+\delta,\ldots,\mL_d+\delta)\mP_1^1\cdots\mP_1^d f&=
\sum_{k\in\mathbb{N}^d}m(k+\delta {\bf 1})\langle f, \bnH_k\rangle_{L^2(\mathbb{R}^d,\gamma)} \bnH_k,\\
m(\mL_1,\ldots,\mL_d)\mP_1^1\cdots\mP_1^d f&= \sum_{k\in\mathbb{N}^d}m(k)\langle f, \bnH_k\rangle_{L^2(\mathbb{R}^d,\gamma)} \bnH_k.
\end{align*}
Since $m$ is continuous and bounded on $\Rdp,$ we have $$\lim _{\delta\rightarrow 0^+}m(\mL_1+\delta,\ldots,\mL_d+\delta)\mP_1^1\cdots\mP_1^d f=m(\mL)\mP_1^1\cdots\mP_1^d f,\qquad f\in L^2(\mathbb{R}^d,\gamma),$$ strongly in $L^2(\mathbb{R}^d,\gamma).$ Now, from the above equality and \eqref{chap:Hinf,sec:ExOp,eq:inepsiP} we conclude that
$$
\|m(\mL_1,\ldots,\mL_d)\mP_1^1\cdots\mP_1^d f\|_{L^p(\mathbb{R}^d,\gamma)\to L^p(\mathbb{R}^d,\gamma)}\leq C(p,d,m) \|f\|_{L^p(\mathbb{R}^d,\gamma)}.
$$
Using the decomposition $I=\mP_1^1\cdots\mP_1^d+\sum_{\omega\in \{0,1\}^d,\, \omega\neq {\bf 1}}\mP_{\omega_1}^1\cdots\mP_{\omega_d}^d,$ we see that
\begin{align*}
m(\mL_1,\ldots,\mL_d)&=m(\mL_1,\ldots,\mL_d)\mP_1^1\cdots\mP_1^d+\sum_{\omega\in \{0,1\}^d,\, \omega\neq {\bf 1}}m(\omega_1\mL_1,\ldots,\omega_d\mL_d)\mP_{\omega_1}^1\cdots\mP_{\omega_d}^d\\
&=m(\mL_1,\ldots,\mL_d)\mP_1^1\cdots\mP_1^d+\textrm{ lower dimensional operators}.
\end{align*}
Since the projections $\mP_{\omega_1}^1\cdots\mP_{\omega_d}^d$ are bounded on all $L^p(\mathbb{R}^d,\gamma),$ $1<p<\infty,$ it follows that $m(\mL_1,\ldots,\mL_d)$ is indeed bounded under the imposed assumptions.
\end{proof}

The other system of operators considered in this section consists of the one-dimensional Laguerre operators. These are defined by $$\mL_r^{\alpha_r}=-x_r\frac{\partial^2}{\partial{x_r}^2}+(\alpha_r+1-x_r)\frac{\partial}{\partial x_r}, \qquad r=1,\ldots,d.$$ The Laguerre operators are symmetric on $C_c^{\infty}(\Rdp),$ with respect to the inner product in $L^2(\mathbb{R}_+,\mu_r^{\alpha_r});$ where, for $\alpha_r>-1$ \begin{equation*}d\mu_r^{\alpha_r}(x_r)=\frac{x_r^{\alpha_r}e^{-x_r}}{\Gamma(\alpha_r+1)}\,dx_r,\qquad x_r>0,\quad r=1,\ldots,d.\end{equation*}

 Additionally, the classical one-dimensional Laguerre polynomials $L_{k_r}^{\alpha_r}$ in the variable $x_r,$ see \cite{szego}, are eigenfunctions of $\mL_r^{\alpha_r}$ corresponding to the eigenvalues $k_r\in\mathbb{N}_0,$ i.e.\ $\mL_r^{\alpha_r} L_{k_r}^{\alpha_r}=k_r L_{k_r}^{\alpha_r}.$ Then, as in the case of the system of Ornstein-Uhlenbeck operators, for each $r=1,\ldots,d,$ a standard argument allows us to consider a self-adjoint extension of $\mL_r^{\alpha_r}$ given by \begin{equation}\label{chap:Hinf,sec:ExOp,eq:systemLag}\mL_r^{\alpha_r}f(x_r)=\sum_{k_r=0}^{\infty} k_r \langle f,\nL_{k_r}\rangle_{L^2(\mathbb{R}_+,\,\mu^{\alpha_r}_r)} \nL_{k_r}^{\alpha_r}(x_r),\end{equation} on the domain $$\textrm{Dom}(\mL_r^{\alpha_r})=\{f\in L^2(\mathbb{R}_+,\mu_r^{\alpha_r})\colon \sum_{k_r=0}^{\infty}k_r^2 |\langle f,\nL_{k_r}^{\alpha_r}\rangle_{L^2(\mathbb{R}_+,\,\mu_r^{\alpha_r})}|^2<\infty\}.$$ Here by $\nL_{k_r}^{\alpha_r}=\|L_{k_r}^{\alpha_r}\|_{L^2(\mathbb{R}_+,\mu_r^{\alpha_r})}^{-1}L_{k_r}^{\alpha_r}$ we denote the $L^2(\mathbb{R}_+,\mu_r^{\alpha_r})$ normalized one--dimensional Laguerre polynomial of degree $k_r$ and order $\alpha_r$ in the $x_r$ variable.

A tensor product reasoning, similar to the one presented in the first part of this section for the system of Ornstein-Uhlenbeck operators $\mL_r,$ leads us to define the joint spectral multipliers $m(\mL^{\alpha})$ of the system $\mL^{\alpha}=(\mL_1^{\alpha_1},\ldots,\mL^{\alpha_d}_d),$ $\alpha=(\alpha_1,\ldots,\alpha_d),$ as
\begin{align}
\label{chap:Hinf,sec:ExOp,eq:defL}
m(\mL^{\alpha})f=\sum_{k\in\mathbb{N}^d_0}m(k_1,\ldots,k_d)\langle f, \bnL_k^{\alpha}\rangle_{L^2(\mathbb{R}_+^d,\,\mu_{\alpha})} \bnL_k^{\alpha},\quad f\in L^2(\mathbb{R}_+^d,\mu_{\alpha}).
\end{align}
Here $m$ is a function on $\mathbb{N}^d_0$ and $\mu_{\alpha}=\mu_1^{\alpha_1}\otimes\cdots\otimes \mu_d^{\alpha_d},$ $\bnL_{k}^{\alpha}=\nL_{k_1}^{\alpha_1}\otimes\cdots\otimes \nL_{k_d}^{\alpha_d}.$

Just as the Ornstein-Uhlenbeck operators, the Laguerre operators $\mL_r^{\alpha_r},$ $r=1,\ldots,d,$ also do not have a Marcinkiewicz functional calculus. This can be easily deduced from \cite[Theorem 2]{sas}.

The counterpart of Theorem \ref{thm:multOU} for the Laguerre system is Theorem \ref{thm:multLag}. This theorem is to some extent a multivariate generalization of \cite[Theorem 1]{sas}.
\begin{thm}
\label{thm:multLag} Fix a parameter $\alpha\in[0,\infty)^d.$ Assume that $m\in H^{\infty}({\bf S}_{\pst}),$ for some $p\in(1,\infty)\setminus\{2\}.$ Assume also that the boundary value functions $\Rdp\ni \la \mapsto m(e^{i\varepsilon\pst}\la),$ $\varepsilon=(\varepsilon_1,\ldots,\varepsilon_d)\in\{-1,1\}^d,$ satisfy the $d$-dimensional Marcinkiewicz condition \eqref{chap:Intro,sec:Notation,eq:Marcon} of some order $\rho> {\bf 3/2},$ and that all the lower dimensional operators $m(\omega_1\mL_1^{\alpha_1},\ldots,\omega_d\mL_d^{\alpha_d})$ with $\omega=(\omega_1,\ldots,\omega_d)\in\{0,1\}^d,$ $\omega\neq {\bf 1},$ are bounded on $L^p(\mathbb{R}^d_+,\,\mu_{\alpha}).$ Then $m(\mL^{\alpha})$ is bounded on $L^p(\mathbb{R}^d_+,\,\mu_{\alpha}).$
\end{thm}
\begin{proof}[Proof (sketch)] The proof uses Theorem \ref{thm:genMarHinf} and is almost the same as the proof of Theorem \ref{thm:multOU}. The only difference is that instead of \eqref{chap:Hinf,sec:ExOp,eq:imah} we use the general bounds
$$\|(\mL_{r}^{\alpha_r}+\delta I)^{iu_r}\|_{L^p(\mathbb{R}_+,\,\mu_r^{\alpha_r})\to L^p(\mathbb{R}_+,\,\mu_r^{\alpha_r})}\leq C_p (1+|u_r|)^{1/2}e^{\pst |u_r|},\qquad u_r\in\mathbb{R},$$
following from \eqref{chap:Hinf,sec:genMarHinf,eq:Carb-Dragest0}.
We omit the details.

\end{proof}
        \section[A multivariate holomorphic extension theorem]{A multivariate holomorphic extension theorem}
         \label{chap:Hinf,sec:HolExt}
         \numberwithin{equation}{section}
In this section we consider a system of operators $L=(L_1,\ldots,L_d),$ such that each $L_r$ is self-adjoint and non-negative on $L^2(X_r,\nu_r),$  $r=1,\ldots,d.$ We impose neither \eqref{chap:Intro,sec:Setting,eq:contra} nor \eqref{chap:Intro,sec:Setting,eq:noatomatzero}, however we assume that the spectra of $L_r,$ $r=1,\ldots,d,$ are discrete. The results of this section are also contained in \cite{jaOU}.

Let $\{e_k\}_{k\in\mathbb{N}^d_0},$ $e_k=e^1_{k_1}\otimes\ldots\otimes e^d_{k_d}$ be an orthonormal basis of some space $L^2(X,\nu),$ which is linearly dense in all the $L^p(X,\nu),$ $1<p<\infty,$ spaces. Here $X=X_1\times\cdots\times X_d,$ $\nu=\nu_1\otimes\cdots\otimes\nu_d,$ while $k=(k_1,\ldots,k_d)$ is a multi-index. We assume that for each $r=1,\ldots,d,$ $\{e^r_{k_r}\}_{k_r=0,1,\ldots}$ is an eigenfunction decomposition of $L_r$ in $L^2(X_r,\nu_r)$ with eigenvalues $0\leq\la^r_{0}<\la^r_{1}<\ldots,$ i.e.\ $$ L_rf=\sum_{k_r=0}^{\infty}\la^r_{k_r}\langle f,e^r_{k_r}\rangle_{L^2(X_r,\nu_r)} e^r_{k_r}$$ on the domain $$\textrm{Dom}(L_r)=\{f\in L^2(X_r,\nu_r)\,:\, \sum_{k_r=0}^{\infty}(\la^r_{k_r})^2|\langle f,e^r_{k_r}\rangle_{L^2(X_r,\nu_r)}|^2<\infty\}.$$ In this setting, for a Borel measurable function $m_r\colon[0,\infty)\rightarrow \mathbb{C}$ the spectral multipliers of the operators $L_r$ are
\begin{equation*}
m_r(L_r)f=\sum_{k_r=0}^{\infty}m_r(\la^r_{k_r})\langle f, e_{k_r}^r\rangle_{L^2(X_r,\nu_r)} e_{k_r}^r,\qquad f\in L^2(X_r,\nu_r), \qquad r=1,\ldots,d.
\end{equation*}

Similarly to the previous section, a tensor product reasoning allows us to regard $L_r,$ $r=1,\ldots,d,$ as strongly commuting operators acting on the full space $L^2(X,\nu).$ Then, for a Borel measurable function $m\colon[0,\infty)^d\rightarrow \mathbb{C},$ the joint spectral multipliers of the system $L$ defined with accordance to \eqref{chap:Intro,sec:Setting,eq:mdef} are given by
\begin{equation*}
m(L)=m(L_1,\ldots,L_d)=\sum_{k\in\mathbb{N}^d_0}m(\la^1_{k_1},\ldots,\la^d_{k_d})\langle f, e_k\rangle_{L^2(X,\nu)} e_k,\qquad f\in L^2.
\end{equation*}

In the main theorem of this section we show that, if, for each $L_r,$ $r=1,\ldots,d,$ their $L^p,$ $p>1,$ uniform spectral multipliers have a certain holomorphic extension property, then the same is true for the $L^p,$ $p>1,$ uniform joint spectral multipliers of the system $L$.

\begin{thm}
\label{thm:HolExt}
Fix $p>1.$ Assume that for each $r=1,\ldots,d,$ there is a sector $S_{\phi_p^r},$ $0<\phi_p^r<\pi/2,$ with the following property: if $m_r\colon [0,\infty)\rightarrow \mathbb{C}$ is a function which is bounded on $[0,\infty),$ continuous on $\mathbb{R}_+,$ and such that $\sup_{t_r>0}\|m_r(t_r L_r)\|_{L^p(X_r,\nu_r)\to L^p(X_r,\nu_r)}<\infty,$ then $m_r$ extends to a bounded holomorphic function in $S_{\phi_p^r},$ and
\begin{equation}
\label{chap:Hinf,sec:HolExt,eq:thm:HolExt:sep}
\|m_r\|_{H^{\infty}(S_{\phi_p^r})}\leq \sup_{t_r>0}\|m_r(t_r L_r)\|_{L^p(X_r,\nu_r)\to L^p(X_r,\nu_r)}.
\end{equation}
Then the following is true: every function $m\colon[0,\infty)^d\rightarrow \mathbb{C}$ which is bounded on $[0,\infty)^d,$ continuous on $\Rdp$ and such that \begin{equation}\label{chap:Hinf,sec:HolExt,eq:thm:HolExt:asumm}A_p^L:= \sup_{t\in\Rdp}\|m(t_1L_1,\ldots,t_d L_d)\|_{L^p(X,\nu)\to L^p(X,\nu)}<\infty,\end{equation} extends to a bounded holomorphic function of several variables in ${\bf S}_{\phi_p},$ $\phi_p=(\phi_p^1,\ldots,\phi_p^d);$ moreover
\begin{equation}
\label{chap:Hinf,sec:HolExt,eq:thm:HolExt:jointuni}
\|m\|_{H^{\infty}({\bf S}_{\phi_p})}\leq\sup_{t\in\Rdp}\|m(t_1L_1,\ldots,t_d L_d)\|_{L^p(X,\nu)\to L^p(X,\nu)}.
\end{equation}
\end{thm}
\begin{remark1}
The theorem is also true under the slightly weaker assumption that each $S_{\phi_p^r}$ is a not necessarily symmetric sector around $\mathbb{R}^+.$ Note that in the original paper, see \cite[Theorem 3.1]{jaOU}, we did not impose any constraints on the set $S_{\phi_p^r}.$ However, as the assumption on $m_r$ is dilation invariant, the restriction to non-symmetric sectors is in fact no loss of generality.
\end{remark1}
\begin{remark2}
If we additionally assume that each of the operators $L_r$ preserves the class of real-valued functions, then we lose nothing by considering only symmetric sectors. This follows from a version of the Schwarz reflection principle. Moreover, under this assumption, using the self-adjointness of $L_r$ it can be shown that we may take ${\bf S}_{\phi_p}={\bf S}_{\phi_q},$ whenever $1/p+1/q=1.$
\end{remark2}
\begin{proof}[Proof of Theorem \ref{thm:HolExt}]Recall that by Hartog's theorem (see \cite{krantz}) a function $f:U\rightarrow \mathbb{C},$ where $U$ is an open subset of $\mathbb{C}^d,$ is holomorphic in several variables if and only if it is holomorphic in each variable $z_r,$ $r=1,\ldots,d,$ while the other variables are held constant. The key ingredient of the proof here is the bound \eqref{chap:Hinf,sec:HolExt,eq:thm:HolExt:sep}. The reasoning we present has been indicated to us by prof.\ Fulvio Ricci. We use induction on $d.$ When $d=1,$ there is nothing to do. Assume that we proved the theorem for some $d-1$ and let $m\colon[0,\infty)^d\rightarrow \mathbb{C}$ be a bounded continuous function on $[0,\infty)^d,$ such that \eqref{chap:Hinf,sec:HolExt,eq:thm:HolExt:asumm} holds.

Fix $t^{(1)}=(t_2,\ldots,t_d)\in\mathbb{R}^{d-1}_+$ and $$f,g\in \textrm{span}\{e_{k_2}^2\otimes\cdots\otimes e_{k_{d}}^d\colon k_r\in\mathbb{N},\, r=2,\ldots,d\}.$$ Note that by our assumptions the latter set is dense in every $L^q(X^{(1)},\nu^{(1)}),$ $1< q<\infty,$ where $X^{(1)}=X_2\times\cdots\times X_d,$ $\nu^{(1)}=\nu_2\otimes\cdots\otimes\nu_d.$ Indeed, it can be easily verified that $\{e_{k_2}^2\otimes\cdots\otimes e_{k_{d}}^d\}$ satisfies the condition from \cite[Corollary 2.5]{basis}. Set
\begin{equation}
\label{chap:Hinf,sec:HolExt,eq:mwi}
\tilde{m}_{t^{(1)},f,g}(t_1)=\langle m(t_1,t_2L_2,\ldots,t_d L_d)f, g\rangle_{L^2(X^{(1)},\nu^{(1)})},\qquad t_1\in [0,\infty).
\end{equation}
Then, it is not hard to see that for $f_1,g_1\in L^p(\nu_1)\cap L^{p'}(\nu_1)$ we have
\begin{align*}
\langle\tilde{m}_{t^{(1)},f,g}(t_1L_1)f_1,g_1\rangle_{L^2(X_1,\nu_1)}= \langle m(t_1L_1,t_2L_2,\ldots, t_{d} L_{d})(f_1\otimes f),g_1\otimes g\rangle_{L^2(X,\nu)}.
\end{align*}

Consequently, from the assumption that $m$ satisfies \eqref{chap:Hinf,sec:HolExt,eq:thm:HolExt:asumm}, we obtain $$\|\tilde{m}_{t^{(1)},f,g}(t_1L_1)\|_{L^p(X_1,\nu_1)\to L^p(X_1,\nu_1)}\leq A_p^L  \|f\|_{L^p(X^{(1)},\nu^{(1)}))}\|g\|_{L^{p'}(X^{(1)},\nu^{(1)})}.$$ Clearly, $\tilde{m}_{t^{(1)},f,g}(\cdot)$ is a bounded continuous function on $[0,\infty),$ hence from \eqref{chap:Hinf,sec:HolExt,eq:thm:HolExt:sep} it follows that $\tilde{m}_{t^{(1)},f,g}(t_1)$ extends to $H^{\infty}(S_{\phi_p^1})$ and (denoting this extension by the same symbol)
\begin{equation}
\label{chap:Hinf,sec:HolExt,eq:ext}
\|\tilde{m}_{t^{(1)},f,g}(\cdot)\|_{H^{\infty}(S_{\phi_p^1})}\leq A_p^L \|f\|_{L^p(X^{(1)},\nu^{(1)}))}\|g\|_{L^{p'}(X^{(1)},\nu^{(1)})}.
\end{equation}

Since in particular $\tilde{m}_{(t_2\la_2(\la_1^2)^{-1},\ldots,t_d\la_d(\la_1^d)^{-1}),e^{(1)}_1,e^{(1)}_1}(t_1)=m(t_1,t_2\la_2,\ldots,t_d\la_d),$ where $e^{(1)}_1=e_1^2\otimes\cdots\otimes e_1^{d},$ we have the bounded holomorphic extension $m(z_1,t_2\la_2,\ldots,t_d\la_d).$ Moreover, for $f,g\in \textrm{span}\{e^2_{k_2}\otimes\cdots\otimes e^d_{k_{d}}\colon k_r\in\mathbb{N}_0,\,r=2,\ldots,d\}$ we see that $$z_1\mapsto \tilde{m}_{t^{(1)},f,g}(z_1)\qquad \textrm{and}\qquad z_1 \mapsto \langle m(z_1,t_2L_2,\ldots,t_dL_d)f, g\rangle_{L^2(X^{(1)},\nu^{(1)})}$$ are two holomorphic functions which agree on the positive real half-line. By uniqueness of the analytic continuation $$\tilde{m}_{t^{(1)},f,g}(z_1)=\langle m(z_1,t_2L_2,\ldots,t_dL_d)f, g\rangle_{L^2(X^{(1)},\nu^{(1)})},$$ i.e.\ \eqref{chap:Hinf,sec:HolExt,eq:mwi} still holds for the extension of $\tilde{m}_{t^{(1)},f,g}.$ Hence, from \eqref{chap:Hinf,sec:HolExt,eq:ext} we infer that \begin{equation*}\sup_{t_2,\ldots,t_d>0}\|m(z_1,t_2L_2,\ldots, t_d L_d)\|_{L^p(X^{(1)},\nu^{(1)})\to L^p(X^{(1)},\nu^{(1)})}\leq A_p^L,\end{equation*} uniformly in $z_1\in S_{\phi_p^1}.$

Denote ${\bf S}^{(1)}=S_{\phi_p^2}\times\cdots\times S_{\phi_p^d}.$ Then, from the inductive hypothesis applied separately for each $z_1\in S_{\phi_p^1},$ we obtain the holomorphic function ${\bf S}^{(1)}\ni (z_2,\ldots,z_d)\mapsto m(z_1,z_2,\ldots,z_d),$ which satisfies \begin{equation}\label{chap:Hinf,sec:HolExt,eq:thm:HolExt:uniboundd}\|m(z_1,\cdot)\|_{H^{\infty}({\bf S}^{(1)})}\leq A_p^L,\qquad z_1\in S_{\phi_p^1}.\end{equation} In summary, for each $z_1\in S_{\phi_p^1},$ the function $m(z_1,z_2,\ldots,z_d)$ is holomorphic in ${\bf S}^{(1)},$ and satisfies the desired bound \eqref{chap:Hinf,sec:HolExt,eq:thm:HolExt:jointuni}. We also know that for each $t^{(1)}=(t_2,\ldots,t_d)\in\mathbb{R}^{d-1}_+,$ the function $m(z_1,t^{(1)})$ is holomorphic on $S_{\phi_p^1}$.

It remains to prove that for each $z^{(1)}=(z_2,\ldots,z_d)\in {\bf S}^{(1)},$ the function $m(z_1,z^{(1)})$ is holomorphic in $S_{\phi_p^1}.$ The rest of the proof of Theorem \ref{thm:HolExt} is devoted to justifying this statement.

We begin by showing that for each $t^{(1)}\in \Rdpm$ and $k=(k_2,\ldots,k_d)\in \mathbb{N}^{d-1}_0,$ the derivative $\partial^{k}_{(1)}m(z_1,t^{(1)})=\partial^{k_2}_2\cdots \partial^{k_d}_dm(z_1,t^{(1)})$ is holomorphic on $S_{\phi_p^1}.$ It is enough to focus on the case $k=(0,\ldots,1,\ldots,0),$ with $1$ on the $(i-1)$-th component, $i=2,\ldots,d,$ as the general argument is an iteration of the one used in this particular case.

Let $t^{(1)}\in \Rdpm$ be fixed. Then, from Taylor's theorem we have
\begin{align*}& m_{n,t^{(1)}}(z_1):=n\bigg(m(z_1,t_2,\ldots,t_i+\frac{1}{n},\ldots,t_d)-m(z_1,t^{(1)})\bigg)\\
&=\partial_{i}m(z_1,t^{(1)})+\frac{1}{n}\partial^{2}_{i}m(z_1,t_2,\ldots,\xi_i(n),\ldots,t_d):=\partial_{i}m(z_1,t^{(1)})+R_{n,t^{(1)}}(z_1),\end{align*}
where $\xi_i(n)\in [t_i,t_i+1/n].$ Fix $z_1\in S_{\phi_p^1}$ for a moment and recall that $m(z_1,\cdot)$ is a holomorphic function of $d-1$ variables on ${\bf S}^{(1)}.$ Thus, using the uniform bound \eqref{chap:Hinf,sec:HolExt,eq:thm:HolExt:uniboundd} together with multivariate Cauchy's integral formula (applied to $\partial^{2}_{i}m(z_1,\cdot)$), we see that the sequence $R_{n,t^{(1)}}(z_1)$ converges to zero, uniformly in $z_1\in S_{\phi_p^1}.$ Consequently, $\{m_{n,t^{(1)}}(z_1)\}_{n}$ is a Cauchy sequence in $H^{\infty}(S_{\phi_p^1}),$ and its limit $\partial_{i}m(z_1,t^{(1)})$ is a holomorphic function on $S_{\phi_p^1}.$

Since $m(z_1,\cdot)$ is holomorphic on ${\bf S}^{(1)}$, for each fixed $z_1\in S_{\phi_p^1}$ and $t^{(1)}\in \Rdpm,$ it can be expanded into a power series (in several variables) in a poly-disc arround $t^{(1)}.$ Specifically we have
\begin{equation}\label{chap:Hinf,sec:HolExt,eq:thm:HolExt:unipower}m(z_1,z^{(1)})=\sum_{k\in \mathbb{N}^{d-1}_0}\frac{\partial^{k}_{(1)}m(z_1,t^{(1)})}{k_2!\cdots k_d!}(z^{(1)}-t^{(1)})^{k},\end{equation}
as long as $$z^{(1)}\in D(t_2,t_2\sin \phi_p^2)\times\cdots \times D(t_d,t_d\sin \phi_p^d):=D(t^{(1)}),$$
where, for $w_r\in \mathbb{C}$ and $R_r>0$ we denote $D(w_r,R_r)=\{z_r\in \mathbb{C}\colon |z_r-w_r|< R_r\}.$
Moreover, from \eqref{chap:Hinf,sec:HolExt,eq:thm:HolExt:uniboundd} and (multivariate) Cauchy's integral formula, it follows that for each fixed $z^{(1)}\in D(t^{(1)}),$ the series given by \eqref{chap:Hinf,sec:HolExt,eq:thm:HolExt:unipower} converges uniformly in $z_1\in S_{\phi_p^1}.$ Thus, we conclude that for each $t^{(1)}\in \Rdpm$ and $z^{(1)}\in D(t^{(1)}),$ the function $m(\cdot,z^{(1)})$ is holomorphic in $S_{\varphi_p^1}.$ Noting that $\bigcup_{t^{(1)}\in \Rdpm}D(t^{(1)})={\bf S}^{(1)}$ we are done.
\end{proof}

As an immediate corollary from Theorem \ref{thm:HolExt}, we obtain the following holomorphic extension theorem for multipliers connected with the systems of Ornstein-Uhlenbeck and Laguerre operators considered in Section \ref{chap:Hinf,sec:ExOp}. In the statement of Corollary \ref{cor:HolExtOULag} the symbol $(X,\nu)$ denotes $(\mathbb{R}^d,\gamma)$ in the case of the system of Ornstein-Uhlenbeck operators, or $(\Rdp,\mu_{\alpha}),$ in the case of the system of Laguerre operators. Recall that $\pst$ is given by \eqref{chap:Hinf,sec:ExOp,eq:pst}.
\begin{cor}
\label{cor:HolExtOULag}
Let $L=(L_1,\ldots,L_d)$ be either the system of Ornstein-Uhlenbeck operators or the system of Laguerre operators from Section \ref{chap:Hinf,sec:ExOp}. Fix $p\in (1,\infty)\setminus\{2\}$ and assume that $m\colon[0,\infty)^d\to \mathbb{C}$ is bounded on $[0,\infty)^d$ continuous on $\mathbb{R}_{+}^{d}$ and
$$\sup_{t\in\Rdp}\|m(t_1L_1,\ldots,t_dL_d)\|_{L^p(X,\nu)\to L^p(X,\nu)}<\infty.$$
Then $m$ extends to a bounded holomorphic function in ${\bf S}_{\pst}$ and
$$\|m\|_{H^{\infty}({\bf S}_{\pst})}\leq \sup_{t\in\Rdp}\|m(t_1L_1,\ldots,t_dL_d)\|_{L^p(X,\nu)\to L^p(X,\nu)}<\infty.$$
\end{cor}
\begin{proof}
We use Theorem \ref{thm:HolExt} together with \cite[Theorem 3.5 (i)]{hmm} (in the case of the system $(\mL_1,\ldots,\mL_d)$) or \cite[Theorem 2 (i)]{sas} (in case of the system $\mL^{\alpha}$). Note that in the both cases $\phi_p^r=\pst,$ $r=1,\ldots,d.$
\end{proof}

It is worth to remark that, in the case of both systems there are results for $p=1,$ see \cite[Theorem 3.8 (ii)]{jaOU} and \cite[Theorem 5.2 (ii)]{jaOU}. Moreover, at least in the case of the system $\mL,$ it is possible to obtain various other holomorphic extension theorems, see \cite[Corolarry 3.24]{jaOU}. However, these results require using some specific properties of the Ornstein-Uhlenbeck or Laguerre operators and thus we decided not to include them in the thesis.


    %
    %

   \newchapter{An application of Corollary \ref{cor:genMarHinfdimindep} to Riesz transforms}{An application of Corollary \ref{cor:genMarHinfdimindep} to Riesz transforms}{An application of Corollary \ref{cor:genMarHinfdimindep} to Riesz transforms}

     \label{chap:Riesz}
     \numberwithin{equation}{chapter}
        In the present chapter we demonstrate how the $L^p$ boundedness of certain one-dimensional Riesz transforms implies the $L^p$ boundedness of appropriate $d$-dimensional Riesz transforms. Moreover, the $L^p$ bound we obtain is independent of the dimension. The key trick here is to use the $H^{\infty}$ joint functional calculus for commuting operators from Corollary \ref{cor:genMarHinfdimindep}.

        The crucial result in this chapter is Theorem \ref{thm:genRiesz}. In the statement of this theorem and throughout its proof we keep the setting and notation of Sections \ref{chap:Intro,sec:Setting} and \ref{chap:Hinf,sec:genMarHinf}. Recall that various operators $m(L)$ built on $L$ are defined by the multivariate spectral theorem via \eqref{chap:Intro,sec:Setting,eq:mdef}. In particular, each $m(L)$ is defined on the domain given by \eqref{chap:Intro,sec:Setting,eq:mdom}; for example, the domain of the operator $L_1+\ldots +L_d$ is the subspace
         $$\{f\in L^2(X,\nu)\colon \int_{[0,\infty)^d}(\la_1+\ldots+\la_d)^2 dE_{f,f}(\la)\}.$$
         All the formalities that are not properly explained in the proof of Theorem \ref{thm:genRiesz} can be easily derived from Proposition \ref{prop:funpro}. We decided not to write them explicitly in the proof in order not to obscure its idea.

        \begin{thmc}
         \label{thm:genRiesz} Let $L=(L_1,\ldots,L_d)$ be a general system of non-negative self-adjoint strongly commuting operators on $L^2(X,\nu)$ satisfying all the assumptions of Section \ref{chap:Intro,sec:Setting}. Fix  $\sigma\in(0,\infty).$ Then the operators $L_r^{\sigma}(\sum_{r=1}^d L_r)^{-\sigma},$ $r=1,\ldots,d,$ are bounded on all $L^p(X,\nu)$ spaces, $1<p<\infty,$  and
         \begin{equation}
         \label{chap:Riesz,eq:thm:genRiesz}
         \|L_r^{\sigma}(L_1+\cdots+L_d)^{-\sigma}\|_{p\to p}\leq C_{p,\sigma},
         \end{equation}
         where the constant $C_{p,\sigma}$ is independent of both $r=1,\ldots,d$ and the ('dimension') $d.$
         \end{thmc}
        \begin{proof}
        Clearly, by the multivariate spectral theorem, the operator $L_r^{\sigma}(\sum_{r=1}^d L_r)^{-\sigma}$ is bounded on $L^2(X,\nu).$

        Denote $L_{(r)}=\sum_{s\neq r} L_s,$ so that $\sum_{r=1}^d L_r=L_r+L_{(r)}.$ Then $(L_r,L_{(r)})$ is a pair of strongly commuting self-adjoint non-negative operators on $L^2(X,\nu),$ which satisfy all the assumptions of Section \ref{chap:Intro,sec:Setting}. Indeed, the $L^p$ contractivity property \eqref{chap:Intro,sec:Setting,eq:contra} of $\exp(-tL_{(r)})$ follows from the identity $e^{-tL_{(r)}}=\prod_{s\neq r}e^{-t L_{s}},$ $t>0;$ while \eqref{chap:Intro,sec:Setting,eq:noatomatzero} is immediate once we note that $0\leq E_{L_{(r)}}(\{0\})\leq E_{L_s}(\{0\})=0,$ $s=1,\ldots,d,$ $s\neq r.$

        Consequently, from Corollary \ref{cor:genMarHinfdimindep}, it follows that if $$m\in H^{\infty}({\bf S}_{(\varphi_1,\varphi_2)})\qquad \textrm{for some}\qquad \varphi_r>\pst=\arcsin|2/p-1|,\quad r=1,2,$$ then
        \begin{equation}
        \label{chap:Riesz,eq:Hinfbound}
        \|m(L_r,L_{(r)})\|_{p\to p}\leq C_{p}\|m\|_{H^{\infty}({\bf S}_{(\varphi_1,\varphi_2)})}.
        \end{equation}
        Moreover, the constant $C_p=C_{p,2}$ in \eqref{chap:Riesz,eq:Hinfbound} is independent of $(L_r,L_{(r)}).$ Taking $m_{\sigma}(z_1,z_2)=z_1^{\sigma}(z_1+z_2)^{-\sigma}$ (here we consider the principal branch of the complex power function) it is not hard to see that $m_{\sigma}$ is holomorphic in ${\bf S}_{\pi/2}$ and uniformly bounded in every ${\bf S}_{\phi},$ $\phi=(\phi_1,\phi_2)\in(0,\pi/2)^2.$ Thus, putting $m_{\sigma}$ and $\varphi=(\pst+\varepsilon_p,\pst+\varepsilon_p)$ in \eqref{chap:Riesz,eq:Hinfbound} (here we take $\varepsilon_p$ small enough so that $\pst+\varepsilon_p<\pi /2$), we obtain \eqref{chap:Riesz,eq:thm:genRiesz}, finishing the proof.
        \end{proof}

        Now we consider systems of operators $L=(L_1,\ldots,L_d)$ such that each $L_r$ is non-negative and self-adjoint on some $L^2(X_r,\nu_r),$ where $(X_r,\nu_r),$ $r=1,\ldots,d,$ is a $\sigma$-finite measure space. We assume that each $L_r,$ $r=1,\ldots,d$ satisfies the contractivity condition \eqref{chap:Intro,sec:Setting,eq:contra} (with respect to $L^p(X_r,\nu_r)$) and the atomlessness condition \eqref{chap:Intro,sec:Setting,eq:noatomatzero}. Denote $X=X_1\times\cdots\times X_d,$ $\nu=\nu_1\otimes\cdots\otimes \nu_d$ and, for $1\leq p\leq \infty,$ $L^p=L^p(X,\nu),$ $\|\cdot\|_p=\|\cdot\|_{L^p}$ and $\|\cdot\|_{p\to p}=\|\cdot\|_{L^p\to L^p}.$
        Recalling the discussion in the last paragraphs of Section \ref{chap:Intro,sec:Notation}, we know that the operators $L_r,$ $r=1,\ldots,d,$ can be also regarded as non-negative self-adjoint strongly commuting operators on $L^2,$ that satisfy \eqref{chap:Intro,sec:Setting,eq:contra} (with respect to $L^p$) and \eqref{chap:Intro,sec:Setting,eq:noatomatzero}.

      Let us now see how Theorem \ref{thm:genRiesz} formally implies dimension free bounds of certain Riesz transforms. A multi-dimensional Riesz transform of order $j_r\in\mathbb{N}$ associated to the system $L=(L_1,\ldots,L_d)$ is an operator of the form $\delta_r^{j_r}(\sum_{s=1}^d L_s)^{-j_r/2}.$ Here $\delta_r$ is a certain operator acting on a dense subspace of $L^2(X_r,\nu_r)$ (hence, by a tensorization argument also on a dense subspace of $L^2$). Often $\delta_r$ and $L_r$ are related by $L_r=\delta_r^*\delta_r+a_r,$ with $a_r\geq 0;$ in particular this is the case in Theorem \ref{thm:RieszOrt} and Corollary \ref{cor:RieszgenLu_Piqu}. The following corollary of Theorem \ref{thm:genRiesz} is rather informal, though, as we shall soon see, it can be easily formalized in many concrete cases.

     \begin{corc}
     \label{cor:Scheme}
     Let $r=1,\ldots,d,$ be fixed. Assume that, for some $j_r \in \mathbb{N},$ the one-dimensional Riesz transform $\delta_r^{j_r} L_r^{-j_r/2}$ of order $j_r$ is bounded on $L^p(X_r,\nu_r).$ Then the multi-dimensional Riesz transform of order $j_r$ is bounded on $L^p$ and
      \begin{equation}
      \label{chap:Riesz,eq:corschembound}
      \|\delta_r^{j_r}(L_1+\cdots+L_d)^{-j_r/2}\|_{p\to p}\leq C_{p,j_r}\|\delta_r^{j_r}L_r^{-j_r/2}\|_{L^p(X_r,\nu_r)\to L^p(X_r,\nu_r)}.\end{equation}
       \end{corc}
     \begin{proof}
     We decompose
      \begin{equation*}
      \delta_r^{j_r}(L_1+\cdots+L_d)^{-j_r/2}=(\delta_r^{j_r}L_r^{-j_r/2})(L_r^{j_r/2}(L_1+\cdots+L_d)^{-j_r/2}).
      \end{equation*}
      Since the system $L$ satisfies the assumptions of Section \ref{chap:Intro,sec:Setting}, using Theorem \ref{thm:genRiesz} with $\sigma=j_r/2$ and the fact that $\|\delta_r^{j_r} L_r^{-j_r/2}\|_{L^p(X_r,\nu_r)\to L^p(X_r,\nu_r)}=\|\delta_r^{j_r} L_r^{-j_r/2}\|_{p \to p}$ (cf.\ \eqref{chap:Intro,sec:Notation,eq:tensnormeq}), we obtain the desired bound \eqref{chap:Riesz,eq:corschembound}.
     \end{proof}
\begin{remark1}
In some cases we need a variant of the corollary that allows the operators $L_r$ to violate the atomlessness condition \eqref{chap:Intro,sec:Setting,eq:noatomatzero}. Then, we need to add appropriate projections in the definitions of Riesz transforms. More details are provided in the specific cases when we use such a variant.
\end{remark1}
\begin{remark2}
The argument used in the proof of the corollary bears a resemblance with the method of rotations by Calder\'on and Zygmund, see \cite{Cal-Zyg1} or \cite[Corollary 4.8]{Duo}. Indeed, when applied to the classical (multi-dimensional) Riesz transforms on $\mathbb{R}^d,$ this method allows us to deduce their $L^p$ boundedness from the $L^p$ boundedness of the (one-dimensional) directional Hilbert transforms. However, the method of rotations does not give a dimension free bound for the classical Riesz transforms.
\end{remark2}

Till the end of this chapter we focus on rigorous applications in particular cases of Corollary \ref{cor:Scheme} or its variations.

First observe that Corollary \ref{cor:Scheme} implies a dimension free estimate for the norms on $L^p(\mathbb{R}^d,dx),$ $1<p<\infty,$ of the classical Riesz transforms $R_r,$ $r=1,\ldots,d,$ cf.\ \cite{SteinRiesz}. In this case $L_r$ coincides with the self-adjoint extensions on $L^2(\mathbb{R}^d,dx)$ of the operators $-\partial_r^2,$ $r=1,\ldots,d,$ initially defined on $C_c^{\infty}(\mathbb{R}^d).$ Then $L=(L_1,\ldots,L_d)$ is a system of strongly commuting operators on $L^2(\mathbb{R}^d,dx),$ satisfying all the assumptions of Section \ref{chap:Intro,sec:Setting}. The one-dimensional Riesz transforms $\delta_r L_r^{-1/2}=\partial_r(-\partial_r^2)^{-1/2}$ of order $1$ and the operators $L_r^{1/2}(\sum_{r=1}^d L_r)^{-1/2},$ $r=1,\ldots,d,$  are defined as the Fourier multipliers $$\mF(\delta_r L_r^{-1/2} f)(\xi)={\rm sgn}\, \xi_r\,\mF f(\xi)=\frac{\xi_r}{|\xi_r|}\mF f(\xi),\qquad f\in L^2(\mathbb{R}^d,dx)$$ and $$\mF(L_r^{1/2}(L_1+\cdots+ L_d)^{-1/2}f)(\xi)=\frac{|\xi_r|}{|\xi|}\mF f(\xi),\qquad f\in L^2(\mathbb{R}^d,dx),$$ respectively. Since the multi-dimensional Riesz transform $R_r$ is given by the Fourier multiplier $$\mF(R_r f)(\xi)=\frac{\xi_r}{|\xi|}\mF f,\qquad f\in L^2,$$
we see that, for $r=1,\ldots,d,$
$$R_r f=(\delta_r L_r^{-1/2}) (L_r^{1/2}(L_1+\cdots+ L_d)^{-1/2})f,\qquad f\in L^2(\mathbb{R}^d,dx). $$
In order to obtain the desired dimension free bounds, it suffices to note that $\delta_r L_r^{-1/2}$ is a constant times the Hilbert transform in the $x_r$ variable and to apply \eqref{chap:Riesz,eq:corschembound} with $j_r=1$.

An argument completely analogous to the one from the preceding paragraph can be also used for the Riesz-Dunkl transforms defined by \eqref{chap:Ck,sec:MarDun,eq:Riesz-Dunkl} in Section \ref{chap:Ck,sec:MarDun}. Observe that Theorem \ref{thm:RieszDunkl} below improves \cite[Theorem 3.1]{Sifi2} in the case $G\cong\mathbb{Z}_2^d$. Indeed, the bound from \cite[Theorem 3.1]{Sifi2} is clearly dependent on the dimension, as it uses the underlying structure of the space of homogenous type.

\begin{thmc}
\label{thm:RieszDunkl}
Fix $r\in\mathbb{N}$ and let $R_r^D$ be the $r$-th $d$-dimensional, $d\geq r,$ Riesz-Dunkl transform defined by \eqref{chap:Ck,sec:MarDun,eq:Riesz-Dunkl}. Then,
$$\|R_r^D\|_{L^p(\mathbb{R}^d,\nu_{\alpha})\to L^p(\mathbb{R}^d,\nu_{\alpha})}\leq C(p,\alpha_r),\qquad 1<p<\infty,$$
where the constant $C_{p,\alpha_r}$ is independent of the dimension $d.$
\end{thmc}
\begin{proof}
We apply Corollary \ref{cor:Scheme}. The proof is a repetition of the argument given for the classical Riesz transforms. We just need to replace the Fourier transform by the Dunkl transform in the context of $G\cong\mathbb{Z}_2^d.$ The boundedness of the one-dimensional Riesz-Dunkl (Hilbert-Dunkl) transform follows from Theorem \ref{thm:dunklmarmult}. To see that the one-dimensional Dunkl heat semigroups $\{e^{-t_r T_r^2}\}_{t_r>0},$ $r=1,\ldots,d,$ satisfy \eqref{chap:Intro,sec:Setting,eq:contra} and \eqref{chap:Intro,sec:Setting,eq:noatomatzero} we refer to \cite[Section 4]{Rosherm}
\end{proof}

\section[Riesz transforms for classical orthogonal expansions]{Riesz transforms for classical orthogonal expansions}
\numberwithin{equation}{section}
The next corollary of Theorem \ref{thm:genRiesz} focuses on Riesz transforms connected with various classical orthogonal expansions. A perfect setting for the present section is the one given by Nowak and Stempak in \cite{NSRiesz}.

In order not to introduce a separate notation for each case of the orthogonal expansions, for precise definitions we kindly refer the reader to the list of examples in \cite[Section 7]{NSRiesz}. In Theorem \ref{thm:RieszOrt} we consider first order multi-dimensional Riesz transforms $R_r$ in one of the following settings from \cite{NSRiesz}:
\begin{enumerate}[a)]
\item Hermite polynomial expansions $\{H_{k}\}_{k\in \mathbb{N}_0^d}$, see \cite[Section 7.1]{NSRiesz};
\item Laguerre polynomial expansions $\{L_{k}^{\alpha}\}_{k\in \mathbb{N}_0^d},$ $\alpha\in(-1,\infty)^d$, see \cite[Section 7.2]{NSRiesz};
\item Jacobi polynomial expansions $\{P_{k}^{\alpha,\beta}\}_{k\in \mathbb{N}_0^d},$ $\alpha,\beta\in(-1,\infty)^d$, see \cite[Section 7.3]{NSRiesz};
\item Hermite function expansions $\{h_{k}\}_{k\in \mathbb{N}_0^d}$, see \cite[Section 7.4]{NSRiesz};
\item Laguerre expansions of Hermite type $\{\varphi_{k}^{\alpha}\}_{k\in \mathbb{N}_0^d},$ $\alpha\in(-1,\infty)^d$, see \cite[Section 7.5]{NSRiesz};
\item Laguerre expansions of convolution type $\{\ell_{k}^{\alpha}\}_{k\in \mathbb{N}_0^d},$ $\alpha\in(-1,\infty)^d$, see \cite[Section 7.6]{NSRiesz};
\item Jacobi function expansions $\{\phi_{k}^{\alpha,\beta}\}_{k\in \mathbb{N}_0^d},$ $\alpha,\beta \in(-1,\infty)^d$, see \cite[Section 7.7]{NSRiesz};
\item Fourier-Bessel expansions $\{\psi_{k}^{v}\}_{k\in \mathbb{N}^d},$ $v\in(-1,\infty)^d,$ with the Lebesgue measure, see \cite[Section 7.8]{NSRiesz} (there $\nu$ is used in place of $v$).
\end{enumerate}
As showed in \cite{NSRiesz}, each of the operators $R_r,$ $r=1,\ldots,d,$ is bounded on an appropriate $L^2(X,\nu)$ (in the terminology of \cite{NSRiesz} $\mu$ is used instead of $\nu$).

Moreover, in all the settings a)-h) the operators $R_r,$ $r=1,\ldots,d,$ are in fact bounded on $L^p,$ for $1<p<\infty.$ Additionally, in the settings a) - e), it is known that their norms as operators on $L^p$ are independent of the dimension $d,$ and on their respective parameters (appropriately restricted in some cases). In the setting a) this is a consequence of \cite{Mey1}, see also \cite{Gut1} and \cite{Pis1}. For the dimension free bounds in the settings b) and c), see \cite{No1}, and \cite{NSj2}, respectively. In the setting d) the dimension free boundedness is proved in \cite{HRST}, while in the setting e) in \cite{StWr}.

Theorem \ref{thm:RieszOrt} below is a fairly general result giving dimension free boundedness for the norms on $L^p$ of first order Riesz transforms connected with the orthogonal expansions listed in items a)-h)\footnote{Recently Forzani et al.\ obtained dimension free $L^p$ estimates for Riesz transforms connected with polynomial expansions, see the presentation \cite{SasForSc}. Their results overlap with our Theorem \ref{thm:RieszOrt} in the settings a)-c).}. In fact the theorem is a tool for reducing the problem of proving dimensionless and parameterless bounds to proving parameterless bounds in one dimension. Note that, in the cases f) - h), it seems that the obtained results are new. Additionally, in the cases b) and e), we improve the range of the admitted parameters. In the statement of Theorem \ref{thm:RieszOrt} by $\diamond_{r}$ we denote the $r$-th parameter in one of the cases b)-c), e)-h). By convention, in the cases a) and d), the symbol $C(p,\diamond_r)$ denotes a constant depending only on $p.$
\begin{thm}
\label{thm:RieszOrt} Fix $r\in\mathbb{N},$ and let $R_r$ be the $r$-th first order $d$-dimensional, $d \geq r,$ Riesz transform in one of the settings a)-h), as defined in \cite[Sections 7.1-7.8]{NSRiesz}. Then, $R_r$ is bounded on $L^p$ and
\begin{equation}
\label{chap:Riesz,eq:RieszOrtbound}
\|R_r\|_{p\to p}\leq C(p,\diamond_r),\qquad 1<p<\infty,
\end{equation}
where the constant $C(p,\diamond_r)$ is independent of both the dimension $d$ and all the other parameters $\diamond_{r'},$ $r'\neq r.$ The bound \eqref{chap:Riesz,eq:RieszOrtbound} in the following settings below holds for:
\begin{enumerate}[a)]
\setcounter{enumi}{1}
\item parameter $\alpha\in(-1,\infty)^d,$ the constant $C(p,\alpha_r)$ is independent of $\alpha_r;$
\item parameters $\alpha,\beta\in[-1/2,\infty)^d,$ the constant $C(p,(\alpha_r,\beta_r))$ is independent of $\alpha_r,\beta_r;$
\setcounter{enumi}{4}
\item parameter $\alpha\in (\{-1/2\}\cup [1/2,\infty))^d,$ the constant $C(p,\alpha_r)$ is independent of $\alpha_r;$
\item parameter $\alpha\in (-1,\infty)^d;$
\item parameters $\alpha,\beta\in(\{-1/2\}\cup [1/2,\infty))^d;$
\item parameter $\nu\in (\{-1/2\}\cup [1/2,\infty))^d.$
\end{enumerate}
\end{thm}
\begin{proof}
We focus on proving one particular model case of the theorem, namely the case b). The proofs in the other settings are similar to the one presented below. Whenever we are in one of the settings a)-h), the symbol $L=(L_1,\ldots,L_d)$ denotes the system of operators considered in this setting. Note that each $L_r,$ $r=1,\ldots,d,$ is self-adjoint and non-negative on $L^2(X_r,\nu_r).$ It may happen however that some of the operators $L_r$ does not satisfy one or both of the conditions \eqref{chap:Intro,sec:Setting,eq:contra} or \eqref{chap:Intro,sec:Setting,eq:noatomatzero}. Throughout the proof we do not distinguish between $L_r$ and $L_r\otimes I_{(r)}=I\otimes\cdots \otimes L_r\otimes \cdots \otimes I.$

Since the operators $L_r,$ $r=1,\ldots,d,$ act on separate variables, they (or rather their tensor product versions) commute strongly. Hence, by the multivariate spectral theorem, we can define the projections $\Pi_{0}=\chi_{\{\la_1+\cdots+\la_d>0\}}(L)$ and $\Pi_{0,r}=\chi_{\{\la_r>0\}}(L),$ $r=1,\ldots,d,$ via \eqref{chap:Intro,sec:Setting,eq:mdef}. Note that in some of the cases a)-h) or for some values of the parameters the operators $\Pi_0$ or $\Pi_{0,r}$ are trivial, i.e.\ equal to the identity operator.

We prove formally the case b) of Theorem \ref{thm:RieszOrt}. The proof is a minor variation on Corollary \ref{cor:Scheme}. In this case we show that the constant $C(p,\alpha_r)$ in \eqref{chap:Riesz,eq:RieszOrtbound} is also independent of $\alpha_r\in(-1,\infty).$ The notation used here is the one introduced in the second part of Section \ref{chap:Hinf,sec:ExOp}. In particular the operators $L_r=\mL_{r}^{\alpha_r}$ are given by \eqref{chap:Hinf,sec:ExOp,eq:systemLag} and $L^p=L^p(\Rdp,\mu_{\alpha}).$ In this case the operator $\delta_r=\sqrt{x_r}\partial_r$ from \cite[p.\ 690]{NSRiesz} acts on the Laguerre polynomial $\bnL_k^{\alpha},$ $k\in\mathbb{N}^d_0$ by $$\delta_r(\bnL_k^{\alpha})=-\sqrt{k_r}\sqrt{x_r}\bnL_{k-e_r}^{\alpha+e_r};$$
 here $e_r$ is the $r$-th coordinate vector and, by convention, $\bnL_{k-e_r}^{\alpha+e_r}=0,$ if $k_r=0.$ The ($d$-dimensional) Riesz transform $R_r$ and the one-dimensional Riesz transform are then well defined on finite linear combinations of Laguerre polynomials by $R_r=\delta_r (\sum_{r=1}^d L_r)^{-1/2}\Pi_{0}$ and $\delta_r L_r^{-1/2}\Pi_{0,r},$ respectively. Here the operators $(\sum_{r=1}^d L_r)^{-1/2}\Pi_{0}$ and $L_r^{-1/2}\Pi_{0,r}$ are given by \eqref{chap:Hinf,sec:ExOp,eq:defL}, with the appropriate functions $m.$

Observe that the Riesz transform $R_r$ can be written as
\begin{equation}
\label{chap:Riesz,eq:RieszLaPoldecom}
R_r f=(\delta_rL_r^{-1/2}\Pi_{0,r})(L_r^{1/2}(L_1+\cdots+ L_d)^{-1/2}\Pi_{0,r})f,
\end{equation}
whenever $f$ is a finite linear combination of Laguerre polynomials. To show \eqref{chap:Riesz,eq:RieszLaPoldecom}, first note that for such $f,$
\begin{equation*}
R_rf=(\delta_rL_r^{-1/2}\Pi_{0,r})(L_r^{1/2}(L_1+\cdots+ L_d)^{-1/2}\Pi_{0})f+\delta_r(I-\Pi_{0,r})((L_1+\cdots+ L_d)^{-1/2}\Pi_0).
\end{equation*}
Then, since $\delta_r$ vanishes on $$\Ran(I-\Pi_{0,r})=\{(I-\Pi_{0,r})g\colon g\in L^2(X,\nu)\}=\{g\in L^2(X,\nu)\colon g\textrm{ is independent of $x_r$}\}$$ and $\Pi_{0,r}\Pi_0=\Pi_{0,r},$ we obtain \eqref{chap:Riesz,eq:RieszLaPoldecom}.

In the case b) the boundedness on $L^p(X_r,\nu_r)$ of the one-dimensional Riesz transform $\delta_rL_r^{-1/2}\Pi_{0,r}$ follows from \cite[Theorem 3 (b)]{Muck1}. Moreover, the bound is also independent of the parameter $\alpha_r\in(-1,\infty).$ Additionally, the operator $\Pi_{0,r}$ is also bounded on all $L^p(X_r,\nu_r),$ $1<p<\infty,$ with a bound independent of $\alpha_r.$ Thus, since finite combinations of Laguerre polynomials are dense in $L^p,$ $1<p<\infty,$ coming back to \eqref{chap:Riesz,eq:RieszLaPoldecom} we see that it suffices to prove the dimension and parameter free $L^p$ boundedness of the operator $(L_r^{1/2}(\sum_{r=1}^d L_r)^{-1/2}\Pi_{0,r}).$ Note that the operators $L_r,$ $r=1,\ldots,d,$ satisfy the assumption \eqref{chap:Intro,sec:Setting,eq:contra} from Section \ref{chap:Intro,sec:Setting}, see \cite[Section 2]{NScontr}.  However, we cannot directly apply Theorem \ref{thm:genRiesz}, since none of these operators satisfies the assumption \eqref{chap:Intro,sec:Setting,eq:noatomatzero}.

To overcome this obstacle we consider the auxiliary systems $$L_{\varepsilon}=(L_1+\varepsilon,\ldots,L_d+\varepsilon),\qquad \varepsilon>0.$$ Using Theorem \ref{thm:genRiesz} for the systems $L_{\varepsilon}$ we obtain
\begin{equation}
\label{chap:Riesz,eq:RieszLaPolaux}
\|(L_r+\varepsilon)^{1/2}(L_1+\cdots+L_d+d\varepsilon)^{-1/2}\|_{p\to p}\leq C_{p}, \qquad 1<p<\infty,
\end{equation}
with a constant $C_p$ independent both of the dimension $d$ and $\varepsilon>0.$ Now, since $L_r$ has discrete spectrum, the set $\sigma(L_r)\setminus\{0\}$ is separated from $0.$ Consequently, from the multivariate spectral theorem,
\begin{equation}
\label{chap:Riesz,eq:RieszLalimit}
L_r^{1/2}(L_1+\cdots +L_d)^{-1/2}\Pi_{0,r}=\lim_{\varepsilon\to 0^+}(L_r+\varepsilon)^{1/2}(L_1+\cdots+L_d+d\varepsilon)^{-1/2}\Pi_{0,r},\end{equation}
in the strong $L^2$ sense. Using the above equality together with \eqref{chap:Riesz,eq:RieszLaPolaux} and the fact that the operator $\Pi_{0,r}$ is bounded on $L^p(X_r,\nu_r),$ $1<p<\infty,$ we obtain the desired dimension free bound for $L_r^{1/2}(\sum_{r=1}^d L_r)^{-1/2}\Pi_{0,r}.$

We do not give proofs for the cases a) and c)-h). Note that in some of these settings the proofs are simpler then in the setting b). That is because we can apply directly Corollary \ref{cor:Scheme}, without using the auxiliary systems $L_{\varepsilon}.$

We finish the proof by pointing out exemplary references for the boundedness of appropriate one-dimensional Riesz transforms and for the $L^p$ contractivity of appropriate semigroups.

For the first of these topics the references are e.g.:\ \cite{Muck2} for the case a), \cite{NSj2} for the case c) (the bound being independent of the parameters), \cite{ThanRieszHerm} for the case d), \cite{NSz1} for the case f), \cite[Theorem 1.14, Corollary 17.11]{Muck3} for the case g), and \cite{CiSt1} for the case h). In the case e), a combination of the proof of \cite[Theorem 3.3]{NSRieszLagHerm} from \cite[Sections 5,7]{NSRieszLagHerm} (for small $\alpha_r$), and \cite[Theorem 3.1 and Theorem 5.1]{StWr} (for large $\alpha_r$), gives a bound independent of $\alpha_r\in \{-1/2\}\cup[1/2,\infty).$

For the $L^p$ contractivity property of appropriate semigroups the reader can consult e.g.: \cite[Proposition 1.1. i)]{funccalOu} for the case a); \cite{NScontr} for the cases b), e), and f); \cite[Section 2]{NSj1} for the cases c) and g); \cite{HRST} for the case d). In the case h) the $L^p$ contractivity follows from the Feynman-Kac formula.

In those references in which multi-dimensional expansions are considered, we only need to use the one-dimensional case $d=1.$

\end{proof}

\section[Riesz transforms on products of discrete groups]{Riesz transforms on products of discrete groups with polynomial volume growth}
\label{chap:Riesz,sec:disc}
We shall now see how Corollary \ref{cor:Scheme} applies to the context introduced in the title of the present section.

Let $G$ be a discrete group. Assume that there is a finite set $U$ containing the identity and generating $G.$ We impose that $G$ has polynomial volume growth, i.e., there is $\alpha\in \mathbb{N}_0,$ such that
$
|U^n|\leq C n^{\alpha},
$ where $|F|$ denotes the counting measure (Haar measure) of $F\subset G.$
Throughout this section by $l^p,$ $\|\cdot\|_p,$ and $\|\cdot\|_{p\to p},$ we mean $l^p(G),$ $\|\cdot\|_{l^p(G)},$ and $\|\cdot\|_{l^p(G)\to l^p(G)},$ respectively.

Fix a finitely supported symmetric probability measure $\mu$ on $G,$ such that $\supp \mu$ generates $G.$ Then the operator $$P f(x)=P_{\mu} f(x)=f*\mu(x)=\sum_{y\in G}\mu(x^{-1}y) f(y)=\sum_{y\in G}\mu(y) f(xy),$$
is a contraction on all $l^p,$ $1\leq p\leq \infty.$ Since $\mu$ is symmetric, $P$ is self adjoint on $l^2$, thus $L=(I-P)$ is self-adjoint and non-negative on $l^2.$ Moreover, $L$ satisfies \eqref{chap:Intro,sec:Setting,eq:contra}. Indeed, we have
$e^{-tL}=e^{-t}\sum_{n=0}^{\infty}\frac{P^n }{n!},$ so that
$$\|e^{-tL}\|_{p\to p}\leq e^{-t}\sum_{n=0}^{\infty}\frac{\|P\|^n_{p\to p}}{n!}\leq 1, \qquad 1\leq p\leq \infty.$$

Additionally, since $\supp \mu$ generates $G,$ it can be shown that, if $Pf=f$ for some $f\in l^2,$ then $f$ is a constant. For the sake of completeness we give a short prove of this observation\footnote{We thank Gian Maria Dall'Ara for showing us this argument.}. Since the measure $\mu$ is real-valued it suffices to focus on real-valued $f\in l^2.$ Then, either $f$ or $-f$ attains a maximum on $G.$ Assume the former holds (the proof in the latter case is analogous) and let $x\in G$ be such that $\sup_{y\in G}f(y)\leq f(x).$ Since $f(x)=Pf(x)=\sum_{y\in \supp \mu}\mu(y)f(xy),$ we thus have $f(x)=f(xy),$ for all $y\in \supp \mu.$ Now, an inductive argument shows that $f(x)=f(xy^n),$ for all $y\in \supp \mu$ and every $n\in \mathbb{N}.$ Thus, using the assumption that $\supp \mu$ generates $G,$ we obtain that $f$ is constant on $G.$

From the previous paragraph it follows that, in the case $|G|=\infty,$ if $Pf=f$ and $f\in l^2,$ then $f=0.$ Hence, if $G$ is infinite, then $L$ satisfies \eqref{chap:Intro,sec:Setting,eq:noatomatzero}, i.e.\ $E_{L}(\{0\})=0.$

For fixed $g_0\in G$ denote $\partial_{g_0} f(x)=f(xg_0)-f(x).$ Then $\partial_{g_0}$ is a bounded operator on all $l^p,$ $1\leq p\leq \infty.$ Let $R$ be the discrete Riesz transform considered in \cite{AlexdisRiesz} (see also \cite{HebSC}). From \cite[Theorem 2.4]{AlexdisRiesz} (or \cite[Section 8]{HebSC}) it follows that $R$ is bounded on all $l^p,$ $1<p<\infty$ and of weak type $(1,1).$ Note that, if $G$ is infinite, then $E_{L}(\{0\})=0$ and the formula $R=\partial_{g_0}L^{-1/2}=\partial_{g_0}(I-P)^{-1/2}$ holds on a dense subset of $l^2.$ If $|G|=K,$ for some $K\in\mathbb{N}$, then, for all $f\in l^2$ we have $R=\partial_{g_0}L^{-1/2}\pi_0,$ where $\pi_0$ is the projection onto the orthogonal complement of the constants given by\ $\pi_0f=f-|G|^{-1}\sum_{y\in G}f(y).$

Now we consider multi-dimensional Riesz transforms on direct products $G^d.$ Let $P_r=P\otimes I_{(r)}$ be given by \eqref{chap:Intro,sec:Notation,eq:tensnot}. Set $L_r=L\otimes I_{(r)}=(I-P_r)$ and, in the case of finite $G,$ denote \begin{equation}\label{chap:Riesz,sec:disc,eq:Pi0def}\Pi_0f=f-\frac{1}{|G|^d}\sum_{y=(y_1,\ldots,y_d)\in G^d}f(y).\end{equation} The $d$-dimensional Riesz transform $R_r$ is then defined by \begin{equation}\label{chap:Riesz,sec:disc,eq:Rieszmultdef}
R_r=\left\{ \begin{array}{lr}
 (\partial_{g_0}\otimes I_{(r)})(L_1+\cdots +L_d)^{-1/2}, &\mbox{if $|G|=\infty$}, \\
 (\partial_{g_0}\otimes I_{(r)})(L_1+\cdots +L_d)^{-1/2}\Pi_0, &\mbox{if $|G|<\infty.$}
       \end{array} \right.
\end{equation}
Recall that the measure $\mu$ in the definition of $P_r=P\otimes I_{(r)}=P_{\mu}\otimes I_{(r)},$ $r=1,\ldots,d,$ is a symmetric probability measure, such that $\supp \mu$ is finite and generates $G$.

By using Corollary \ref{cor:Scheme} (or its variation) we prove the following.
\begin{thm}
\label{thm:Rieszdiscgen}
Let $G$ be a discrete group of polynomial volume growth. Then the $d$-dimensional Riesz transforms $R_r,$ given by \eqref{chap:Riesz,sec:disc,eq:Rieszmultdef}, are bounded on all $l^p(G^d),$ $1<p<\infty.$ Moreover, for each $r=1,\ldots,d,$
\begin{equation}
\|R_r \|_{l^p(G^d)\to l^p(G^d)}\leq C_p ,\qquad 1<p<\infty,
\label{chap:Riesz,sec:disc,eq:Rieszdiscgen}
\end{equation}
where the constant $C_p$ is independent of $d.$ Consequently,
\begin{equation}\label{chap:Riesz,sec:disc,eq:Rieszdiscgenvectp<2}
\bigg\|\left(\sum_{r=1}^d|R_r f|^2\right)^{1/2}\bigg\|_{l^p(G^d)}\leq C_p\, d\,\|f\|_{l^p(G^d)},\qquad 1<p<2,\end{equation}
\begin{equation}\bigg\|\left(\sum_{r=1}^d|R_r f|^2\right)^{1/2}\bigg\|_{l^p(G^d)}\leq C_p\, \sqrt{d}\, \|f\|_{l^p(G^d)},\qquad 2<p<\infty. \label{chap:Riesz,sec:disc,eq:Rieszdiscgenvectp>2}
\end{equation}
\end{thm}
\begin{remark}
In \cite{BaRus1} and \cite{Rus1} Badr and Russ studied discrete Riesz transforms on graphs. Applying Corollary \ref{cor:Scheme}, Theorem \ref{thm:Rieszdiscgen} can be generalized to products of graphs that were studied in \cite{BaRus1} and \cite{Rus1}.
\end{remark}
\begin{proof}[Proof of Theorem \ref{thm:Rieszdiscgen}]
The proof of \eqref{chap:Riesz,sec:disc,eq:Rieszdiscgen} is just another application of Corollary \ref{cor:Scheme} (or its variants).

We start with showing \eqref{chap:Riesz,sec:disc,eq:Rieszdiscgen} in the case $|G|=\infty.$ After decomposing
\begin{equation*}R_r=\big((\partial_{g_0}\otimes I_{(r)})L_r^{-1/2}\big)\big(L_r^{1/2}(L_1+\cdots+ L_d)^{-1/2}\big)\end{equation*}
it suffices to use the boundedness of the one-dimensional Riesz transform, see \cite[Theorem 2.4]{AlexdisRiesz}, together with Corollary \ref{cor:Scheme}. Note that $L$ satisfies \eqref{chap:Intro,sec:Setting,eq:contra} and \eqref{chap:Intro,sec:Setting,eq:noatomatzero}, hence the same is true for the operators $L_r,$ $r=1,\ldots,d.$

To obtain \eqref{chap:Riesz,sec:disc,eq:Rieszdiscgen} in the case $|G|<\infty,$ we proceed similarly as in the proof of case b) of Theorem \ref{thm:RieszOrt}. Denoting $$\Pi_{0,r}f(x)=(\pi_0\otimes I_{(r)})f(x)=f-\frac{1}{|G|}\sum_{y_r\in G}f(x_1,\ldots,x_{r-1},y_r,x_{r+1},\ldots,x_d),$$ we see that $\Pi_{0,r}\Pi_0=\Pi_0\Pi_{0,r}=\Pi_{0,r}.$ Since $(\partial_{g_0}\otimes I_{(r)})(I-\Pi_{0,r})=0,$ we rewrite $R_r$ as
\begin{align*}R_r&=\big((\partial_{g_0}\otimes I_{(r)})L_r^{-1/2}\Pi_{0,r}\big)\big(L_r^{1/2}(L_1+\cdots + L_d)^{-1/2}\Pi_{0,r}\big)\\&=(R\otimes I_{(r)})(L_r^{1/2}(L_1+\cdots + L_d)^{-1/2}\Pi_{0,r}).\end{align*}
Using the boundedness of  $R\otimes I_{(r)}$ on $l^p(G^d)$ we are left with showing that the operator $L_r^{1/2}(L_1+\cdots + L_d)^{-1/2}\Pi_{0,r}$ is bounded on $l^p(G^d),$ uniformly in $d.$ This can be done exactly as in in the proof of case b) of Theorem \ref{thm:RieszOrt}. Namely, we apply Theorem \ref{thm:genRiesz} to the auxiliary systems $L_{\varepsilon}=(L_1+\varepsilon,\ldots,L_d+\varepsilon),$ to get a uniform in $\varepsilon>0$ and $d$ bound for the $l^p(G^d)$ norms of the operators $(L_r+\varepsilon)^{1/2}(L_1+\cdots + L_d+d\varepsilon)^{-1/2}.$ Using the identity
$$L_r^{1/2}(L_1+\cdots +L_d)^{-1/2}\Pi_{0,r}=\lim_{\varepsilon\to 0^+}(L_r+\varepsilon)^{1/2}(L_1+\cdots+L_d+d\varepsilon)^{-1/2}\Pi_{0,r},$$
cf.\ \eqref{chap:Riesz,eq:RieszLalimit},
together with the dimension free $l^p(G^d)$ boundedness of $\Pi_{0,r},$ we thus obtain the desired dimension free boundedness of $L_r^{1/2}(L_1+\cdots + L_d)^{-1/2}\Pi_{0,r}.$

Now we focus on proving \eqref{chap:Riesz,sec:disc,eq:Rieszdiscgenvectp<2} and \eqref{chap:Riesz,sec:disc,eq:Rieszdiscgenvectp>2} . The former inequality is a simple consequence of the fact that the $l^2$ norm of $(R_1,\ldots,R_d)$ is smaller than its $l^1$ norm. To prove \eqref{chap:Riesz,sec:disc,eq:Rieszdiscgenvectp>2} we use Minkowski's integral inequality (first inequality below, note that here we need $p/2>1$) together with \eqref{chap:Riesz,sec:disc,eq:Rieszdiscgen} (second inequality below), obtaining
\begin{align*}
&\bigg\|\left(\sum_{r=1}^d|R_r f|^2\right)^{1/2}\bigg\|_{l^p(G^d)}=\bigg\|\sum_{r=1}^d|R_r f|^2\bigg\|^{1/2}_{l^{p/2}(G^d)}\\
&\leq \left(\sum_{r=1}^d \|R_r f\|^2_{l^p(G^d)}\right)^{1/2}\leq C_p\, \sqrt{d}\, \|f\|_{l^p(G^d)}.
\end{align*}
\end{proof}

We finish this section by showing how Theorem \ref{thm:Rieszdiscgen} can be used to prove a version of \cite[Theorem 2.8]{Lu_Piqu1} by Lust-Piquard, that applies to all cyclic groups. Till the end of this section we assume that $G$ is a cyclic group, that is, there exists $g_0\in G$ such that $\{g_0,g_0^{-1}\}$ generates $G.$
In this case $G$ is abelian and isomorphic to either $(\mathbb{Z},+)$ or $\mathbb{Z}_K=(\{0,\ldots,K-1\},+_{K}),$ where $+_{K}$ denotes addition modulo $K.$ Note that the results of \cite{Lu_Piqu1} include only the cases $G\cong \mathbb{Z}_3,$ $G\cong \mathbb{Z}_4,$ $G\cong \mathbb{Z},$ whereas the case $G\cong \mathbb{Z}_2$ is studied in \cite{Lu_Piqu2}.

Defining $\mu=\mu_{g_0}=(\delta_{g_0}+\delta_{g_0^{-1}})/2,$ we see that $\mu$ is a symmetric probability measure with $\supp \mu=\{g_0,g_0^{-1}\}$ generating $G.$ Moreover, \begin{equation*} \partial_{g_0}\partial_{g_0}^{*}=\partial_{g_0}^*\partial_{g_0}=2(I-P),\end{equation*} where, as we recall, $Pf=P_{\mu}f=\mu*f.$ Consequently, we have
\begin{equation}\label{chap:Riesz,sec:disc,eq:connect Gen-Lus}
\frac{1}{\sqrt{2}}R_r=\left\{ \begin{array}{lr}
 (\partial_{g_0}\otimes I_{(r)})\big(\sum_{r=1}^d (\partial_{g_0}\otimes I_{(r)})(\partial_{g_0}\otimes I_{(r)})^{*}\big)^{-1/2}, &\mbox{if $|G|=\infty$}, \\
(\partial_{g_0}\otimes I_{(r)})\big(\sum_{r=1}^d (\partial_{g_0}\otimes I_{(r)})(\partial_{g_0}\otimes I_{(r)})^{*}\big)^{-1/2}\Pi_0, &\mbox{if $|G|<\infty,$}
       \end{array} \right.
\end{equation}
with $\Pi_0$ given by \eqref{chap:Riesz,sec:disc,eq:Pi0def}. Thus, for the specific choice of $\mu=\mu_{g_0},$ our Riesz transform $R_r$ coincides with $\sqrt{2}$ times the Riesz transform considered in \cite[p.\ 307]{Lu_Piqu1}. Since, the group $G$ clearly has polynomial volume growth, we obtain the following Corollary of Theorem \ref{thm:Rieszdiscgen}, which is a generalization of \cite[Theorem 2.8]{Lu_Piqu1} to all cyclic groups.
\begin{cor}
\label{cor:RieszgenLu_Piqu}
Let $G$ be a cyclic group. Then the discrete Riesz transforms $2^{-1/2}R_r$ considered in \cite{Lu_Piqu1} and given by \eqref{chap:Riesz,sec:disc,eq:connect Gen-Lus}, satisfy \eqref{chap:Riesz,sec:disc,eq:Rieszdiscgen}, \eqref{chap:Riesz,sec:disc,eq:Rieszdiscgenvectp<2}, and \eqref{chap:Riesz,sec:disc,eq:Rieszdiscgenvectp>2}.
\end{cor}
\begin{remark1}
Our method, contrary to the one used in \cite{Lu_Piqu1}, does not prove dimension free estimates for the vector of Riesz transforms. Note that, under the assumptions that $G$ is a locally compact abelian group, such that $g_0$ spans an infinite subgroup, by \cite[Theorem 2.8]{Lu_Piqu1} the inequality \eqref{chap:Riesz,sec:disc,eq:Rieszdiscgenvectp>2} holds with a constant independent of $d.$ However, the counterexample given in \cite[Proposition 2.9]{Lu_Piqu1}, shows that, even for $G=\mathbb{Z},$ the constant in \eqref{chap:Riesz,sec:disc,eq:Rieszdiscgenvectp<2} is dependent on $d.$
\end{remark1}
\begin{remark2}
The constant $C_p$ in \eqref{chap:Riesz,sec:disc,eq:Rieszdiscgen}, \eqref{chap:Riesz,sec:disc,eq:Rieszdiscgenvectp<2}, and \eqref{chap:Riesz,sec:disc,eq:Rieszdiscgenvectp>2}, is also independent on the considered cyclic group $G.$ To see this we apply transference methods due to Coifman and Weiss. Namely, from \cite[Corollary 3.16]{CWtr}, it is not hard to deduce that the norms of the one-dimensional Riesz transform on $l^p(\mathbb{Z}_K)$ and $l^p(\mathbb{Z})$ are related by $\|R\|_{l^p(\mathbb{Z}_K) \to l^p(\mathbb{Z}_K)}\leq 2\|R\|_{l^p(\mathbb{Z})\to l^p(\mathbb{Z})}.$ Since every cyclic group $G$ is isomorphic to either $\mathbb{Z}$ or $\mathbb{Z}_K$ we thus have $\|R\|_{l^p(G) \to l^p(G)}\leq 2\|R\|_{l^p(\mathbb{Z})\to l^p(\mathbb{Z})}$ Hence, recalling that $C_p$ is of the form $C'_{p}\|R\|_{l^p(G) \to l^p(G)},$ where $C'_p$ depends only on $p$ and not on $G,$ we obtain the desired $G$ independence of $C_p.$
\end{remark2}

    %
    %

    \newchapter{Combinations of Marcinkiewicz and $H^{\infty}$ functional calculi}{Combinations of Marcinkiewicz and $H^{\infty}$ functional calculi}{Combinations of Marcinkiewicz and $H^{\infty}$ functional calculi}
    \label{Chap:CkHinf}
\numberwithin{equation}{section}
Here we investigate 'mixed' systems of operators $(L,A),$ with $L$ consisting of operators having $H^{\infty}$ functional calculi and $A$ consisting of operators having Marcinkiewicz functional calculi. In Section \ref{chap:CkHinf,sec:genMarHinfCk} we provide a general Marcinkiewicz type multiplier theorem in this setting. Then, in Section \ref{chap:CkHinf,sec:OA} we focus on a specific situation of the system $(\mL,A),$ with $\mL$ being the $d$-dimensional Ornstein-Uhlenbeck operator and $A$ being an operator whose heat semigroup has a kernel satisfying certain Gaussian bounds.



        \section[A Marcinkiewicz type multiplier theorem]{A Marcinkiewicz type multiplier theorem}
        \label{chap:CkHinf,sec:genMarHinfCk}

         Throughout this section we consider a system of $d=n+l$ operators $$(L,A):=((L_1,\ldots,L_n),(A_1\ldots,A_l)),$$ which are positive, self-adjoint and strongly commuting on some space $L^2(X,\nu).$ We also impose that the system $(L,A)$ satisfies the assumptions of Section \ref{chap:Intro,sec:Setting}, i.e.\ all the operators $L_r,$ $r=1,\ldots,n,$ and $A_r,$ $r=1,\ldots,l,$ satisfy the conditions \eqref{chap:Intro,sec:Setting,eq:contra} and \eqref{chap:Intro,sec:Setting,eq:noatomatzero}. Then, if $m$ is a Borel measurable function on $\Rdp,$  the multivariate spectral theorem allows us to define $m(L,A)$ via \eqref{chap:Intro,sec:Setting,eq:mdef}. As earlier, for the sake of brevity, we write $L^p$ instead of $L^p(X,\nu)$ and $\|\cdot\|_p$ instead of $\|\cdot\|_{L^p(X,\nu)}.$ The symbol $\|T\|_{p\to p}$ denotes the operator norm of $T$ acting on $L^p.$

        We impose that, for each fixed $1<p<\infty,$ the imaginary powers $L_r^{iu_r},$ $u_r\in\mathbb{R},$ $r=1,\ldots,n,$ satisfy \eqref{chap:Hinf,sec:genMarHinf,eq:polynomialfull}. Strictly speaking, we assume that
        there exist $\theta=(\theta_1,\ldots,\theta_n)\in [0,\infty)^n$ and $\phi_p=(\phi_p^1,\ldots,\phi_p^n)\in (0,\pi/2)^n,$ such that
\begin{equation}
\label{chap:CkHinf,sec:genMarHinfCk,eq:exponential}
\|L^{i(u_1,\ldots,u_n)}\|_{p\to p}\leq \mC(p,L)\,\prod_{r=1}^{n}(1+|u_r|)^{\theta_r|1/p-1/2|}\exp(\phi_p^r|u_r|),\qquad (u_1,\ldots,u_n)\in \mathbb{R}^n_+.
\end{equation}
Using Theorem \ref{thm:genMarHinf} it is not hard to see that \eqref{chap:CkHinf,sec:genMarHinfCk,eq:exponential} is equivalent to the fact that the system $L$ has an $H^{\infty}$ joint functional calculus.

For the operators $A_r,$ $r=1,\ldots,l,$ we assume \eqref{chap:Ck,sec:genMar,eq:polynomialfull}, i.e.\ that there is a vector of positive real numbers $\sigma=(\sigma_1,\ldots,\sigma_l),$ such that for every $1<p<\infty$ and $r=1,\ldots,l$
\begin{equation}
\label{chap:CkHinf,sec:genMarHinfCk,eq:polynomial}
\|A_r^{i(u_{n+1},\ldots,u_d)}\|_{p\to p}\leq \mC(p,A)\, \prod_{r=n+1}^{d}(1+|u_r|)^{\sigma_r|1/p-1/2|},\qquad (u_{n+1},\ldots,u_d)\in\Rlp.
\end{equation}
Recall that, by Corollary \ref{cor:genMarCkequiv}, the condition \eqref{chap:CkHinf,sec:genMarHinfCk,eq:polynomial} is equivalent to the fact that the system $A$ has a Marcinkiewicz joint functional calculus.

Note that, since the constants in \eqref{chap:CkHinf,sec:genMarHinfCk,eq:exponential} and \eqref{chap:CkHinf,sec:genMarHinfCk,eq:polynomial} appear in Theorem \ref{thm:genMarCkHinf}, similarly to Sections \ref{chap:Ck,sec:genMar} and \ref{chap:Hinf,sec:genMarHinf}, they are written in the calligraphic font.

Modifying slightly the notation from Section \ref{chap:Hinf,sec:genMarHinf}, for a function $m\colon \overline{{\bf S}_{\phi_p}}\times \Rlp\to \mathbb{C},$ $\phi_p=(\phi_p^1,\ldots,\phi_p^n),$ and $\varepsilon\in \{-1,1\}^n$ we set $$m(e^{i\varepsilon\phi_p}\underline{\lambda},\overline{\la})=m(e^{i\varepsilon_1\phi_p^1}\la_1,\ldots,e^{i\varepsilon_d\phi_p^n}\la_n,\la_{n+1},\ldots,\la_d),$$ with $\la=((\la_1,\ldots,\la_n),(\la_{n+1}\ldots,\la_{d}))=(\underline{\la},\overline{\la})\in \mathbb{R}^{n}_+\times \mathbb{R}^{l}_+.$

The following theorem is a generalization of Theorems \ref{thm:genMarCk} and \ref{thm:genMarHinf}.
\begin{thm}
\label{thm:genMarCkHinf}
Fix $p>1$ and let $m\colon {\bf S}_{\phi_p}\times \Rlp \to \mathbb{C}$ be a bounded function with the following property: for each fixed $\overline{\la}\in \mathbb{R}_+^l,$ $m(\cdot,\overline{\la})\in H^{\infty}({\bf S}_{\phi_p})$ and the boundary value functions\footnote{They exist by (multivariate) Fatou's theorem} $m(e^{i\varepsilon\phi_p}\underline{\lambda},\overline{\la}),$ $\varepsilon\in \{-1,1\}^{n},$ satisfy the $d$-dimensional Marcinkiewicz condition \eqref{chap:Intro,sec:Notation,eq:Marcon} of some order $\rho>|1/p-1/2|(\theta,\sigma)+{\bf 1}.$ Then the multiplier operator $m(L,A)$ is bounded on $L^p$ and
$$\|m(L,A)\|_{p\to p}\leq C_{p,d}\, \mC(p,L)\,\mC(p,A)\, \sup_{\varepsilon \in\{-1,1\}^n}\|m(e^{i\varepsilon\phi_p}\underline{\lambda},\overline{\la})\|_{Mar,\rho}.$$
\end{thm}
\begin{proof}[Outline of the proof]
The proof is a combination of the techniques used in the proofs of Theorems \ref{thm:genMarCk} and \ref{thm:genMarHinf}. Once again we show that $m$ satisfies \eqref{chap:General,eq:thm:gen} from Theorem \ref{thm:gen}. As previously, it is enough to obtain, for some $N\in\mathbb{N}^d,$ $N>\rho,$ the bound
\begin{align}\nonumber
&\sup_{t\in\Rdp}|\M(m_{N,t})(u)|\leq C_{N,\rho}\sup_{\varepsilon \in\{-1,1\}^n}\|m(e^{i\varepsilon\phi_p}\underline{\lambda},\overline{\la})\|_{Mar,\rho}\\
&\times\prod_{r=1}^n(1+|u_r|)^{-\rho_r|1/p-1/2|}\exp(-\phi_p^r|u_r|)\prod_{r=n+1}^d (1+|u_r|)^{-\rho_r|1/p-1/2|},\label{chap:CkHinf,sec:genMarHinfCk,eq:claim}
\end{align}
uniformly in $u\in \mathbb{R}^d;$ recall that $$\M(m_{N,t})(u)=\int_{\mathbb{R}^n_+}\int_{\mathbb{R}^l_+}t^N \la^N \exp (-\langle t, \la\rangle)m(\la)\la^{-iu}\,\frac{d\la}{\la}.$$ The proof of  \eqref{chap:CkHinf,sec:genMarHinfCk,eq:claim} is similar to the proof of Theorem \ref{thm:genMarHinf} when we consider the first $n$ integrals defining $\M(m_{N,t})$ and similar to the proof of Theorem \ref{thm:genMarCk} when we consider the last $l$ integrals defining $\M(m_{N,t}).$ We omit the details.
\end{proof}

        \section[The system (OU, $A$)]{The system (OU, $A$) with $A$ having a Marcinkiewicz functional calculus}
        \label{chap:CkHinf,sec:OA}
Here we consider a pair of operators $(\mL,A),$ where $\mL$ is the $d$-dimensional Ornstein-Uhlenbeck (OU) operator, while $A$ is an operator having certain Gaussian bounds on its heat kernel (which implies that $A$ has a Marcinkiewicz functional calculus). We also assume that $A$ acts on a space of homogenous type $(Y,\zeta,\mu).$  The main theorem of this section is Theorem \ref{thm:OA}. It states that Laplace transform type multipliers of $(\mL,A)$ are bounded from the $H^1(Y,\mu)$-valued $L^1(\mathbb{R}^d,\gamma)$ to $L^{1,\infty}(\gamma \otimes \mu).$ Here $H^1(Y,\mu)$ is the atomic Hardy space in the sense of Coifman and Weiss \cite{CW}, while $\gamma$ is the Gaussian measure on $\mathbb{R}^d$ given by $d\gamma(x)=\pi^{-d/2}e^{-|x|^2}dx.$ We finish this section by showing that the considered weak type $(1,1)$ property interpolates well with the boundedness on $L^2,$ see Theorem \ref{thm:interH1}.

In what follows we denote by $\mL$ the $d$-dimensional Ornstein-Uhlenbeck operator
$$
-\frac{1}{2}\Delta+\langle x,\nabla\rangle.
$$
Note that $\mL$ is then a sum $\sum_{r=1}^d \mL_r$ of the one-dimensional Ornstein-Uhlebneck operators $\mL_r$ considered in Section \ref{chap:Hinf,sec:ExOp}. It is easily verifiable that $\mL$ is symmetric on $C_c^{\infty}(\mathbb{R}^d)$ with respect to the inner product on $L^2(\mathbb{R}^d,\gamma).$ The operator $\mL$ is also essentially self-adjoint on $C_c^{\infty}(\mathbb{R}^d),$ and we continue writing $\mL$ for its unique self-adjoint extension.

It is well known that $\mL$ can be expressed in terms of Hermite polynomials by
$$\mL f=\sum_{k\in\mathbb{N}^d_0} |k| \langle f,\bnH_k\rangle_{L^2(\mathbb{R}^d,\gamma)}\bnH_k=\sum_{j=0}^{\infty}jP_jf,$$
on the natural domain
$$\Dom(\mL)=\{f\in L^2(\mathbb{R}^d,\gamma)\colon \sum_{k\in\mathbb{N}^d_0}|k|^2 \langle f,\bnH_k\rangle_{L^2(\mathbb{R}^d,\gamma)}<\infty\}.$$
Here $|k|=k_1+\ldots+k_d$ is the length of a multi-index $k\in \mathbb{N}^d_0,$ $\bnH_k$ denotes the $L^2(\mathbb{R}^d,\gamma)$ normalized $d$-dimensional Hermite polynomial of order $k,$ while
$$P_jf=\sum_{|k|=j}\langle f,\bnH_k\rangle_{L^2(\mathbb{R}^d,\gamma)} \bnH_k,\qquad j\in\mathbb{N}_0,$$ is the projection onto the eigenspace of $\mL$ with eigenvalue $j.$

For a bounded function $m\colon\mathbb{N}\to \mathbb{C},$ the spectral multipliers of $\mL$ are defined by $m(\mL)=m(\mL_1+\cdots +\mL_d)=\tilde{m}(\mL_1,\ldots,\mL_d),$ where $\tilde{m}(k)=m(|k|).$ Thus, using \eqref{chap:Hinf,sec:ExOp,eq:def}, for $f\in L^2(\mathbb{R}^d,\gamma)$ we have
\begin{equation*}
m(\mL)f=\sum_{k\in\mathbb{N}^d_0} m(|k|) \langle f,\bnH_k\rangle_{L^2(\mathbb{R}^d,\gamma)}\bnH_k=\sum_{j=0}^{\infty}m(j)P_jf.
\end{equation*}

Let $m$ be a function, which is bounded on $[0,\infty)$ and continuous on $\mathbb{R}_+.$ We say that $m$ is an $L^p(\mathbb{R}^d,\gamma)$-uniform multiplier of $\mL,$ whenever
$$\sup_{t>0}\|m(t \mL)\|_{L^p(\mathbb{R}^d,\gamma)\to L^p(\mathbb{R}^d,\gamma)}<\infty.$$
Observe that by the spectral theorem the above bound clearly holds for $p=2.$ Using \cite[Theorem 3.5 (i)]{hmm} it follows that, if $m$ is an $L^p(\mathbb{R}^d,\gamma)$-uniform multiplier of $\mL$ for some $1<p<\infty,$ $p\neq 2,$ then $m$ necessarily extends to a holomorphic function in the sector $S_{\phi_p^{*}}$ (recall that $\pst=\arcsin|2/p-1|$). Assume now that $m(t\mL)$ is of weak type $(1,1)$ with respect to $\gamma,$ with a weak type constant which is uniform in $t>0.$ Then, since the sector $S_{\phi_p^{*}}$ approaches the right half-plane $S_{\pi/2}$ when $p\to 1^+,$ using the Marcinkiewicz interpolation theorem we see that the function $m$ is holomorphic (but not necessarily bounded) in $S_{\pi/2}$. An example of such an $m$ is a function of Laplace transform type in the sense of Stein \cite[pp. 58, 121]{topics}, i.e.\ $m(z)=z\int_0^{\infty}e^{-zt}\kappa(t)\,dt,$ with $\kappa\in L^{\infty}(\mathbb{R}_+,dt).$\footnote{Taking $\kappa(t)=e^{-it},$ so that $m(z)=z/(z+i),$ we see that these multipliers may be unbounded on $S_{\pi/2}.$}

Let now $A$ be a non-negative, self-adjoint operator defined on a space $L^2(Y,\mu),$ where $Y$ is equipped with a metric $\zeta$ such that $(Y,\zeta,\mu)$ is a space of homogenous type, i.e.\ $\mu$ is a doubling measure. For simplicity we assume that $\mu(Y)=\infty,$ and that for all $x_2\in Y,$ the function $(0,\infty)\ni R\mapsto \mu(B_{\zeta}(x_2,R))$ is continuous and $\lim_{R\to 0}\mu(B_{\zeta}(x_2,R))=0.$ We further impose on $A$ the assumptions  \eqref{chap:Intro,sec:Setting,eq:contra} and \eqref{chap:Intro,sec:Setting,eq:noatomatzero} of Section \ref{chap:Intro,sec:Setting}. Throughout this section we also assume that the heat semigroup $e^{-tA}$ has a kernel $e^{-tA}(x_2,y_2),$ $x_2,y_2\in Y,$ which is continuous on $\mathbb{R}^+\times Y\times Y,$ and satisfies the following Gaussian bounds.
\begin{equation}
\label{chap:CkHinf,sec:OA,eq:gausbound} 0\le e^{-tA}(x_2,y_2)\leq \frac{C}{\mu (B(x_2,\sqrt{t}))}\exp(-c\zeta(x_2,y_2)^2\slash t),
\end{equation}
We also impose that for some $\delta>0,$ if $2\zeta(y_2,y'_2)\leq \zeta(x_2,y_2),$ then
\begin{equation}
\label{chap:CkHinf,sec:OA,eq:heatlipsch}|e^{-tA}(x_2,y_2)-e^{-tA}(x_2,y'_2)|\leq \left(\frac{\zeta(y_2,y'_2)}{\sqrt{t}}\right)^{\delta}\frac{C}{\mu(B(x,\sqrt{t}))}\exp(-c \zeta(x_2,y_2)^2\slash t),
\end{equation}
while in general,
\begin{equation}
\label{chap:CkHinf,sec:OA,eq:heatlipschngauss}|e^{-tA}(x_2,y_2)-e^{-tA}(x_2,y'_2)|\leq  \left(\frac{\zeta(y_2,y'_2)}{\sqrt{t}}\right)^{\delta}\frac{C}{\mu(B(x,\sqrt{t}))}.
\end{equation}

From \cite[Theorem 2.1]{Sik} (or rather its version for a single operator), it follows that, under \eqref{chap:CkHinf,sec:OA,eq:gausbound}, the operator $A$ has a finite order Marcinkiewicz functional calculus on $L^p(Y,\mu),$ $1<p<\infty$. Examples of operators $A$ satisfying \eqref{chap:CkHinf,sec:OA,eq:gausbound}, \eqref{chap:CkHinf,sec:OA,eq:heatlipsch}, and \eqref{chap:CkHinf,sec:OA,eq:heatlipschngauss} include, among others, the Laplacian $-\Delta$ and the harmonic oscillator $-\Delta+|x|^2$ on $L^2(\mathbb{R}^d,dx),$ or the Bessel operator $L_{B}$ considered in Section \ref{chap:Ck,sec:HorHan} (see \cite[Lemma 4.2]{dpw}).

Denote by $H^1=H^1(Y,\zeta,\mu)$ the atomic Hardy space in the sense of Coifman-Weiss \cite{CW}. More precisely, we say that a measurable function $b$ is an $H^1$-atom, if there exists a ball $B=B_{\zeta}\subseteq Y$, such that $\supp\, b \subset B,$ $\|b\|_{L^{\infty}(Y,\mu)}\leq 1/ \mu(B),$ and $\int_{Y}b(x_2)d\mu(x_2) =0.$ The space $H^1$ is defined as the set of all $g\in L^1(Y,\mu),$ which can be written as $g= \sum_{j=1}^{\infty} c_j b_j,$ where $b_j$ are atoms and $\sum_{j=1}^{\infty} |c_j|<\infty,$ $c_j\in\mathbb{C}.$ We equip $H^1$ with the norm
$
\|f\|_{H^1}=\inf \sum_{j=1}^{\infty} |c_j|,$
where the infimum runs over all absolutely summable $\{c_j\}_{j\in\mathbb{N}},$ for which $g= \sum_{j=1}^{\infty} c_j b_j,$ with $b_j$ being $H^1$-atoms. Note that from the very definition of $H^1$ we have $\|g\|_{L^1(Y,\mu)}\leq \|g\|_{H^1}.$

It can be shown that under \eqref{chap:CkHinf,sec:OA,eq:gausbound}, \eqref{chap:CkHinf,sec:OA,eq:heatlipsch}, and \eqref{chap:CkHinf,sec:OA,eq:heatlipschngauss}, the space $$H^1_{max}=\{g\in L^1(Y,\mu)\colon \sup_{t>0}|e^{-tA}g|\in L^1(Y,\mu)\}$$ coincides with the atomic $H^1,$ i.e., there is a constant $C_{\mu}$ such that
\begin{equation}
\label{chap:CkHinf,sec:OA,eq:maxchar}
C_{\mu}^{-1}\|g\|_{H^1} \leq \big\|\sup_{t>0}|e^{-tA}g|\big\|_{L^1(Y,\mu)}\leq C_{\mu} \|g\|_{H^1},\qquad g\in H^1(Y).
\end{equation}
The proof of \eqref{chap:CkHinf,sec:OA,eq:maxchar} is similar to the proof of \cite[Proposition 4.1 and Lemma 4.3]{dpw}. The main trick is to replace the metric $\zeta$ with the measure distance (see \cite{CW})
$$\tilde{\zeta}(x_2,y_2)=\inf\{\mu(B)\colon B\textrm{ is a ball in Y},\, x_2,y_2\in B\},$$
 change the time $t$ via $$\mu(B(y,\sqrt t))=s,\qquad y\in Y,\quad t,\,s>0,$$ and apply Uchiyama's Theorem, see \cite[Corollary 1']{Uchi}. We omit the details. Note that taking $r=e^{-t},$ we can restate \eqref{chap:CkHinf,sec:OA,eq:maxchar} as
\begin{equation}
\label{chap:CkHinf,sec:OA,eq:maxcharr}
C_{\mu}^{-1}\|g\|_{H^1} \leq \big\|\sup_{0<r<1}|r^Ag|\big\|_{L^1(Y,\mu)}\leq C_{\mu} \|g\|_{H^1},\qquad g\in H^1(Y).
\end{equation}

For fixed $0<\varepsilon<1/2,$ define $M_{A,\varepsilon}(g)(x)=\int_Y\sup_{\varepsilon<r<1-\varepsilon}|r^A(x_2,y_2)||g(y_2)|\,d\mu(y_2).$ Then, a short reasoning using the Gaussian bound \eqref{chap:CkHinf,sec:OA,eq:gausbound} and the doubling property of $\mu$ gives
\begin{equation}
\label{chap:CkHinf,sec:OA,eq:maxcharrL1}
\left\|M_{A,\varepsilon}(g)\right\|_{L^1(Y,\mu)}\leq C_{\mu,\varepsilon} \|g\|_{L^1(Y,\mu)},\qquad g\in L^1(Y,\mu).
\end{equation}

Denote by $L^1_{\gamma}(H^1)$ the Banach space of those Borel measurable functions $f$ on $\mathbb{R}^d\times Y$ such that the norm
\begin{equation}\label{chap:CkHinf,sec:OA,eq:L1H1norm}\|f\|_{L^1_{\gamma}(H^1)}=\int_{\mathbb{R}^d}\|f(x_1,\cdot)\|_{H^1}\,d\gamma(x),\end{equation}
is finite. In other words $L^1_{\gamma}(H^1)$ is the $L^1(\gamma)$ space of $H^1$-valued functions. Moreover, it is the closure of $$L^1_{\gamma}(\mathbb{R}^d)\odot H^1:=\bigg\{f\in L^1_{\gamma}(H^1)\colon f=\sum_{j}f_j^1\otimes f_j^2,\quad f_j^1 \in L^1_{\gamma}(\mathbb{R}^d),\, f_j^2\in H^1\bigg\}$$ in the norm given by \eqref{chap:CkHinf,sec:OA,eq:L1H1norm}.

In place of $\mL$ and $A$ we now consider the tensor product operators $\mL \otimes I$ and $I \otimes A.$ For the sake of brevity we write $L^p,$ $\|\cdot\|_{p}$ and $\|\cdot\|_{p \to p},$ instead of $L^p(\mathbb{R}^d \otimes Y, \gamma \otimes \mu),$ $\|\cdot\|_{L^p},$ and $\|\cdot\|_{L^p \to L^p},$ respectively. We shall also use the space $L^{1,\infty}:=L^{1,\infty}(\mathbb{R}^d\times Y,\gamma \otimes \mu),$ equipped with the quasinorm
 \begin{equation}\label{chap:CkHinf,sec:OA,eq:weaknormgauss}\|f\|_{L^{1,\infty}}=\sup_{s>0}s (\gamma \otimes \mu)(\mathbb{R}^d\times Y\colon |f(x)|>s).\end{equation}
 Let $S$ be an operator which is of weak type $(1,1)$ with respect to $\gamma \otimes \mu.$ Then, $\|S\|_{L^1\to L^{1,\infty}}=\sup_{\|f\|_{1}=1}\|Sf\|_{L^{1,\infty}}$ is the best constant in its weak type $(1,1)$ inequality.

Let $m$ be a bounded function defined on $[0,\infty)\times \sigma(A),$ and let $m(\mL,A)$ be a joint spectral multiplier of $(\mL,A),$ as in \eqref{chap:Intro,sec:Setting,eq:mdef}. Assume that for each $t>0,$ the operator $m(t\mL,A)$ is of weak type $(1,1)$ with respect to $\gamma\otimes \mu,$ with a weak type $(1,1)$ constant uniformly bounded with respect to $t.$ Then, from what was said before, we may conclude\footnote{At least in the case when $A$ has a discrete spectrum.} that for each fixed $a\in \sigma(A)$ the function $m(\cdot,a)$ has a holomorphic extension to the right half-plane. We limit ourselves to $m$ being of the following Laplace transform type:
\begin{equation} \label{chap:CkHinf,sec:OA,eq:mult1} m(\la,a)=m_{\kappa}(\la,a):=\la\int_0^{\infty}e^{-\la t}e^{-a t}\kappa(t)\,dt,\qquad (\la,a)\in [0,\infty)\times \mathbb{R}_+,\end{equation}
with $\kappa\in L^{\infty}(\mathbb{R}_+,dt).$ In what follows we denote $\|\kappa\|_{\infty}=\|\kappa\|_{L^{\infty}(\mathbb{R}_+,dt)}.$

Observe that under the assumptions we made on $A,$ the function $m_{\kappa}$ gives a well defined bounded operator $m_{\kappa}(\mL,A)$ on $L^2.$ Indeed, since $\chi_{\{a=0\}}(\mL,A)=0,$ we have
$$m_{\kappa}(\mL,A)=m_{\kappa}(\mL,A)\chi_{\{a>0\}}(\mL,A).$$ Moreover, $m_{\kappa}(0,a)=0$ for $a>0,$ and, consequently, the function $m_{\kappa}(\la,a)\chi_{\{a>0\}}$ is bounded on $[0,\infty)\times \mathbb{R}_+.$ Now, using the multivariate spectral theorem we see that $m_{\kappa}(\mL,A)$ is bounded on $L^2.$

The operator $m_{\kappa}(\mL,A)$ is also bounded on all $L^p$ spaces, $1<p<\infty.$ Indeed, after decomposing $$m_{\kappa}(L,A)=\frac{L}{L+A}(L+A)\int_0^{\infty}e^{- t(L+A)}\kappa(t)\,dt,$$ we may apply Theorem \ref{thm:genRiesz} (with $\sigma=1$) and the general Stein's multiplier theorem \cite[Corollary 3, p.121]{topics}. Moreover, we have $\|m\|_{p\to p}\leq C_p,$ with universal constants $C_p,$ $1<p<\infty.$ The boundedness of $m_{\kappa}(\mL,A)$ on $L^p$ follows also directly from Theorem \ref{thm:genMarCkHinf}.

However, the following question is left open: is $m_{\kappa}(\mL,A)$ also of weak type $(1,1)?$ The main theorem of this section is a positive result in this direction.
\begin{thm}
\label{thm:OA}
Let $\mL$ be the Ornstein-Uhlenbeck operator on $L^2(\mathbb{R}^d,\gamma)$ and let $A$ be a non-negative self-adjoint operator on $L^2(Y,\zeta,\mu),$ satisfying all the assumptions of Section \ref{chap:Intro,sec:Setting} and such that its heat kernel satisfies \eqref{chap:CkHinf,sec:OA,eq:gausbound}, \eqref{chap:CkHinf,sec:OA,eq:heatlipsch} and  \eqref{chap:CkHinf,sec:OA,eq:heatlipschngauss}, as described in this section. Let $\kappa$ be a bounded function on $\mathbb{R}_+$ and let $m_{\kappa}$ be given by \eqref{chap:CkHinf,sec:OA,eq:mult1}. Then the multiplier operator $m_{\kappa}(\mL,A)$ is bounded from $L^1_{\gamma}(H^1)$ to $L^{1,\infty}(\gamma\otimes \mu),$ i.e.\
\begin{equation}
\label{chap:CkHinf,sec:OA,eq:weakH1type}
(\gamma\otimes \mu)(\{x\in \mathbb{R}^d\times Y\colon |m_{\kappa}(\mL,A)f(x)|>s\})\leq \frac{C_{d,\mu}\ki}{s}\|f\|_{L^1_{\gamma}(H^1)},\qquad s>0.
\end{equation}
\end{thm}
\begin{remark}
Observe that $L^2 \cap L^1_{\gamma}(H^1)$ is dense in $L^1_{\gamma}(H^1).$ Thus, it is enough to prove \eqref{chap:CkHinf,sec:OA,eq:weakH1type} for $f\in L^2 \cap L^1_{\gamma}(H^1).$
\end{remark}
Altogether, the proof of Theorem \ref{thm:OA} is rather long and technical, thus for the sake of the clarity of the presentation we do not provide all details. We use a decomposition of the kernel of the operator $T:=m_{\kappa}(\mL,A)$ into the global and local parts with respect to the Gaussian measure in the first variable. The local part will turn out to be of weak type $(1,1)$ (with respect to $\gamma \otimes \mu$) in the ordinary sense. For both the local and global parts we use ideas and some estimates from Garc\'ia-Cuerva's et al.\ \cite{high} and \cite{laptype}.

Set $\kappa^{\varepsilon}=\kappa\chi_{[\varepsilon,1/\varepsilon]},$ $0<\varepsilon<1.$ Then, using the multivariate spectral theorem together with the fact that $A$ satisfies \eqref{chap:Intro,sec:Setting,eq:noatomatzero}, we see that $\lim_{\varepsilon\to 0^+}m_{\kappa^{\varepsilon}}((\mL,A))=m_{\kappa}((\mL,A)),$ strongly in $L^2.$ Consequently, we also have convergence in the measure $\gamma\otimes \mu$. Since, clearly $\|\kappa^{\varepsilon}\|_{L^{\infty}(\mathbb{R}^+)}\leq \ki,$ it suffices to prove \eqref{chap:CkHinf,sec:OA,eq:weakH1type} for $\kappa$ such that $\supp \kappa\subseteq[\varepsilon,1/\varepsilon].$ Thus, throughout the proof of Theorem \ref{thm:OA} we assume (often without further mention) that $\kappa$ is supported away from $0$ and $\infty.$ Additionally, the symbol $\lesssim$ denotes that the estimate is independent of $\kappa.$

In the proof of Theorem \ref{thm:OA} the variables with subscript $1,$ e.g.\ $x_1,y_1,$ are elements of $\mathbb{R}^d,$ while the variables with subscript $2,$ e.g.\ $x_2,y_2,$ are taken from $Y.$

We start with introducing some notation and terminology. Define $$L^{\infty}_c=\{f\in L^{\infty}\colon \supp f \textrm{ is compact}\}=\{f\in L^{\infty}(\mathbb{R}^d\times Y, \Lambda \otimes \mu)\colon \supp f \textrm{ is compact}\},$$
where $\Lambda$ is Lebesgue measure on $\mathbb{R}^d.$ Denoting $L^p(\mathbb{R}^d\times Y, \Lambda\otimes \mu):=L^p(\Lambda\otimes \mu),$ we see that for each $1\leq p<\infty,$ $L^{\infty}_c$ is a dense subspace of both $L^p$ and $L^p(\Lambda\otimes \mu).$ In particular, any operator which is bounded on $L^2$ or $L^2(\Lambda\otimes \mu)$ is well defined on $L^{\infty}_c.$ We also need the weak space $L^{1,\infty}(\Lambda\otimes \mu):=L^{1,\infty}(\mathbb{R}^d\times Y,\Lambda\otimes \mu)$ equipped with the quasinorm given by \eqref{chap:CkHinf,sec:OA,eq:weaknormgauss} with $\gamma$ replaced by $\Lambda.$ An operator $S$ is of weak type $(1,1)$ precisely when
$$\|S\|_{L^1(\Lambda\otimes \mu) \to L^{1,\infty}(\Lambda\otimes \mu)}=\sup_{\|f\|_{L^1(\Lambda\otimes \mu)=1}}\|Sf\|_{L^{1,\infty}(\Lambda\otimes \mu)}<\infty.$$

Let $\eta$ be the product metric on $\mathbb{R}^d\times Y,$
\begin{equation}\label{chap:CkHinf,sec:OA,eq:eta}\eta(x,y)=\max(|x_1-y_1|,\zeta(x_2,y_2)),\qquad x,y\in \mathbb{R}^d\times Y.
\end{equation}
Then it is not hard to see that the triple $(\mathbb{R}^d\times Y,\eta,\Lambda\otimes \mu)$ is a space of homogenous type.
 \begin{defi}
 \label{defi:kernel}
 We say that a function $S(x,y)$ defined on the product $(\mathbb{R}^{d}\times Y)\times (\mathbb{R}^d\times Y)$ is a kernel of a linear operator $S$ defined on $L^{\infty}_c$ if, for every $f\in L^{\infty}_c$ and a.e.\ $x\in \mathbb{R}^d\times Y,$
$$Sf(x)=\int_{\mathbb{R}^d}\int_Y S(x,y)f(y)\,d\mu(y_2)\,dy_1.$$
\end{defi}
\begin{remark1}
We do not restrict to $x\not\in \supp f;$ the operators we consider later on are well defined in terms of their kernels for all $x.$ This is true because of the assumption that $\kappa$ is supported away from $0$ and $\infty.$
\end{remark1}
\begin{remark2}
The reader should keep in mind that the inner integral defining $Sf(x)$ is taken with respect to the Lebesgue measure $dy_1$ rather than the Gaussian measure $d\gamma(y_1).$ The reason for this convention is the form of Mehler's formula we use, see \eqref{chap:CkHinf,sec:OA,eq:MehlformOUint}.
\end{remark2}

Let $\mM_r(x_1,y_1),$ $x_1,y_1 \in \mathbb{R}^d,$ $0<r<1,$  denote Mehler's kernel in $\mathbb{R}^d,$ i.e.\ the kernel of the operator $r^{\mL}=e^{-t\mL},$ with $r=e^{-t}.$ It is well known that, for $0<r<1,$
\begin{equation}
\label{chap:CkHinf,sec:OA,eq:MehlformOU}
\mM_r(x_1,y_1)=\pi^{-d/2}(1-r^2)^{-d/2}\exp\bigg(-\frac{|rx_1-y_1|^2}{1-r^2}\bigg),\qquad x_1,y_1\in\mathbb{R}^d.
\end{equation}
and that, for all $g\in L^p(\mathbb{R}^d,\gamma)$ with $1\leq p\leq \infty,$
\begin{equation}\label{chap:CkHinf,sec:OA,eq:MehlformOUint} r^{\mL}g(x_1)=\int_{\mathbb{R}^d}\mM_r(x_1,y_1)g(y_1)\,dy_1,\qquad x_1\in \mathbb{R}^d.\end{equation}
In particular, using \eqref{chap:CkHinf,sec:OA,eq:MehlformOUint} it can be deduced that $\{e^{-t\mL}\}_{t>0}$ satisfies the contractivity condition \eqref{chap:Intro,sec:Setting,eq:contra}. Additionally, a short computation using \eqref{chap:CkHinf,sec:OA,eq:MehlformOU} gives
\begin{equation}
\label{chap:CkHinf,sec:OA,eq:comM}
\begin{split}
\partial_r\,\mM_r(x_1,y_1)= &\pi^{-d/2}\left(dr-2r\frac{|rx_1-y_1|^2}{1-r^2}-\langle rx_1-y_1,x_1\rangle\right)(1-r^2)^{-d/2-1}\\
&\times \exp\bigg(-\frac{|rx_1-y_1|^2}{1-r^2}\bigg).
\end{split}
\end{equation}
From the above we see that, if $\varepsilon<r<1-\varepsilon,$ for some $0<\varepsilon<1/2,$ then
\begin{equation}
\label{chap:CkHinf,sec:OA,eq:comMestfar01}
|\partial_r\,\mM_r(x_1,y_1)|\lesssim C_{\varepsilon}(1+|x_1|).
\end{equation}

Note that, since $\kappa$ is supported away from $0$ and infinity, the function $\kappa_{\log}(r)=\kappa(-\log r),$ $0<r<1,$ is supported away from $0$ and $1,$ say in an interval $[\varepsilon,1-\varepsilon],$ $0<\varepsilon<1/2.$ In what follows, slightly abusing the notation, we keep the symbol $\kappa$ for the function $\kappa_{\log}.$

The change of variable $r=e^{-t}$ leads to the formal equality
$$T=\int_{0}^{1}\kappa(r)\mL r^{\mL}r^A \,\frac{dr}{r}=\int_{0}^{1}\kappa(r)\partial_r r^{\mL}r^A \,dr.$$ Suggested by the above we define the kernel
\begin{equation*}
K(x,y)=\int_0^{1}\partial_r \mM_r(x_1,y_1)\,r^A(x_2,y_2)\kappa(r)\,dr,\qquad x_1,y_1\in\mathbb{R}^d,\quad x_2,y_2\in Y,
\end{equation*}
with $r^A(x_2,y_2)=e^{(\log r) A}(x_2,y_2).$ Then we have.

\begin{lem}
\label{lem:kerKopT}
The function $K$ is a kernel of $T$ in the sense of Definition \ref{defi:kernel}.
\end{lem}
\begin{proof}[Proof (sketch)]
It is enough to show that for $f,h \in L_c^{\infty}$ we have
\begin{equation}
\label{chap:CkHinf,sec:OA,eq:itisenough}
\langle Tf, h\rangle=\int_{\mathbb{R}^d\times Y}\int_{\mathbb{R}^d\times Y}K(x,y)f(y)h(x)\,d(\Lambda\otimes \mu)(y)\,d(\gamma\otimes \mu)(x).\end{equation}

From the multivariate spectral theorem together with Fubini's theorem we see that
\begin{equation}
\label{chap:CkHinf,sec:OA,eq:multFubin}
\langle m(\mL,A)f, g\rangle_{L^2}=\int_0^1\kappa(r) \langle \mL r^{\mL-1} r^A f, h\rangle_{L^2}\,dr, \qquad f,\, h\in L^2.\end{equation}
Now, by the multivariate spectral theorem $\mL r^{\mL-1} (r^A f)=(\partial_r r^{\mL}) (r^A f),$ where on right hand side we have the Fr\'echet derivative in $L^2.$ Thus, $\langle \mL r^{\mL-1} r^A f, h\rangle_{L^2}$ is the limit (as $\delta\to 0$) of \begin{align}\label{chap:CkHinf,sec:OA,eq:limandint}
&\delta^{-1}\langle((r+\delta)^{\mL}-r^{\mL}) r^A f, h\rangle_{L^2}\\
&\nonumber=
\int_{\mathbb{R}^d\times Y}\int_{\mathbb{R}^d\times Y}\frac{\mM_{r+\delta}(x_1,y_1)-\mM_{r}(x_1,y_1)}{\delta}r^{A}(x_2,y_2)f(y)\, h(x)\,d(\Lambda\otimes \mu)(y) \,d(\gamma\otimes\mu)(x).\end{align}
We kindly refer the reader to Proposition \ref{prop:funpro} for details on the above applications of the multivariate spectral theorem.

Since $f,g\in L_c^{\infty},$ using \eqref{chap:CkHinf,sec:OA,eq:maxcharrL1}, \eqref{chap:CkHinf,sec:OA,eq:comMestfar01}, and the dominated convergence theorem we justify taking the limit inside the integral in \eqref{chap:CkHinf,sec:OA,eq:limandint} and obtain
$$\langle \mL r^{\mL-1} r^A f, h\rangle_{L^2}=\int_{\mathbb{R}^d\times Y}\int_{\mathbb{R}^d\times Y}\partial_r\mM_{r}(x_1,y_1)r^{A}(x_2,y_2)f(y)\, h(x)\,d(\Lambda\otimes \mu)(y) \,d(\gamma\otimes\mu)(x).$$
Plugging the above formula into \eqref{chap:CkHinf,sec:OA,eq:multFubin}, and using Fubini's theorem (which is allowed by \eqref{chap:CkHinf,sec:OA,eq:maxcharrL1}, \eqref{chap:CkHinf,sec:OA,eq:comMestfar01} and the fact that $\supp \kappa \subseteq [\varepsilon, 1-\varepsilon]$), we arrive at \eqref{chap:CkHinf,sec:OA,eq:itisenough}, as desired.
\end{proof}
Let $N_s,$ $s>0,$ be given by
\begin{equation*}
N_s=\big\{(x_1,y_1)\in\mathbb{R}^d\times \mathbb{R}^d\colon |x_1-y_1|\leq \frac{s}{1+|x_1|+|y_1|}\big\}.
\end{equation*}
We call $N_s$ the local region with respect to the Gaussian measure $\gamma$ on $\mathbb{R}^d.$ This set (or its close variant) is very useful when studying maximal operators or multipliers for $\mL.$ After being applied by Sj\"ogren in \cite{Sj1}, it was used in \cite{funccalOu}, \cite{high}, \cite{laptype}, and \cite{sharp}, among others.

The local and global parts of the operator $T$ are defined, for $f\in L_c^{\infty},$ by
\begin{equation}
\label{chap:CkHinf,sec:OA,eq:glob1}
T^{glob}f(x)=\int_{\mathbb{R}^d}\int_Y (1-\chi_{N_2}(x_1,y_1))K(x,y)f(y)\,d\mu(y_2)\,dy_1,
\end{equation}
and
\begin{equation*}
T^{loc}f(x)=Tf(x)-T^{glob}f(x),
\end{equation*}
respectively. The estimates from Proposition \ref{prop:proglob} demonstrate that the integral \eqref{chap:CkHinf,sec:OA,eq:glob1} defining $T^{glob}$ is absolutely convergent for a.e.\ $x,$ whenever $f\in L^1.$

Note that the cut-off considered in \eqref{chap:CkHinf,sec:OA,eq:glob1} is the rough one from \cite[p.\ 385]{high} (though only with respect to $x_1,y_1$) rather than the smooth one from \cite[p.\ 288]{laptype}. In our case, using a smooth cut-off with respect to $\mathbb{R}^d$ does not simplify the proofs. That is because, even a smooth cut-off with respect to $x_1,y_1$ may not preserve a Calder\'on-Zygmund kernel in the full variables $(x,y).$ Consequently, a scheme based on \cite{laptype} may not work in our setting, thus to prove Theorem \ref{thm:OA} we use methods similar to those from \cite{high}.

We begin with proving the desired weak type $(1,1)$ property for $T^{glob}.$ Since
$$T^{glob}f(x)=\int_{0}^{1} \int_{\mathbb{R}^d}\partial_r \mM_r(x_1,y_1)\chi_{N_2^c}(x_1,y_1)\,r^A(f(y_1,\cdot))(x_2)\,dy_1\,\kappa(r)\,dr$$
and $\supp \kappa \subseteq [\varepsilon,1-\varepsilon]$ we have
\begin{equation}
\label{chap:CkHinf,sec:OA,eq:glob2}
\begin{split}
&|T^{glob}f(x)|\leq \|\kappa\|_{\infty}\int_{\varepsilon}^{1-\varepsilon} \int_{\mathbb{R}^d}|\partial_r \mM_r(x_1,y_1)|\chi_{N_2^c}(x_1,y_1)|r^A(f(y_1,\cdot))(x_2)|\,dy_1\,dr\\
&\leq \|\kappa\|_{\infty} \int_0^{1} \int_{\mathbb{R}^d}|\partial_r \mM_r(x_1,y_1)|\chi_{N_2^c}(x_1,y_1)\sup_{\varepsilon<r<1-\varepsilon}|r^A(f(y_1,\cdot))(x_2)|\,dy_1\,dr\\
&:=\|\kappa\|_{\infty}T_*^{glob}f(x).
\end{split}
\end{equation}
Moreover, the following proposition holds.
\begin{pro}
\label{prop:proglob}
The operator $T_*^{glob}$ is well defined on $L^1$ and bounded from $L^1_{\gamma}(H^1)$ to $L^{1,\infty}(\gamma\otimes \mu),$ with a bound independent of $0<\varepsilon<1/2.$ Thus, $T^{glob}$ is also well defined on $L^1$ and we have
\begin{equation*}
(\gamma\otimes \mu)(\{x\in \mathbb{R}^d\times Y\colon |T^{glob}f(x)|>s\})\leq \frac{C_{d,\mu}\ki}{s}\|f\|_{L^1_{\gamma}(H^1)},\qquad s>0.
\end{equation*}
\end{pro}
\begin{proof}
Clearly, from \eqref{chap:CkHinf,sec:OA,eq:glob2} it follows that  it suffices to focus on $T_*^{glob}.$

Using the finite sign change argument, i.e.\ the inequality (2.3) from the proof of \cite[Lemma 2.1]{laptype}, we see that
$$T_*^{glob}f(x)\leq C \int_{\mathbb{R}^d}\sup_{0<r<1}\mM_r(x_1,y_1)\chi_{N_2^c}(x_1,y_1)f^*_2(y_1,x_2)\,dy_1, $$
where $f^*_2(x_1,x_2)=\sup_{\varepsilon<r<1-\varepsilon}|r^A(f(x_1,\cdot))(x_2)|.$ Moreover, from \cite[Theorem 3.8]{laptype} and \cite[Lemma 2]{Sj1} it follows that the operator $$L^1(\mathbb{R}^d,\gamma)\ni g\mapsto \int_{\mathbb{R}^d}\sup_{0<r<1}\mM_r(x_1,y_1)\chi_{N_2^c}(x_1,y_1)|g|(y_1)\,dy_1:=T^{*}_1 g(x_1),$$ is of weak type $(1,1)$ with respect to $\gamma.$
Hence, using Fubini's theorem we have
\begin{align}
&\nonumber(\gamma\otimes \mu)(\{x\in \mathbb{R}^d\times Y\colon |T_*^{glob}f(x)|>s\})= \int_{Y}\gamma(\{x_1\in\mathbb{R}^d\colon |T_*^{glob}f(x)|>s\})\,d\mu(x_2)\\ \nonumber
&\leq \int_{Y}\gamma(\{x_1\in\mathbb{R}^d\colon |T^{*}_1(f^*_2(\cdot,x_2))(x_1)|>s\})\,d\mu(x_2)\leq  \int_{Y}\frac{C_d}{s}\int_{\mathbb{R}^d}f^*_2(x_1,x_2)\,d\gamma(x_1)\,d\mu(x_2)\\
&=\frac{C_d}{s} \int_{\mathbb{R}^d}\int_{Y}\sup_{\varepsilon<r<1-\varepsilon}|r^A(f(x_1,\cdot))(x_2)|\,d\mu(x_2)\,d\gamma(x_1). \label{chap:CkHinf,sec:OA,eq:FubestTglob}
\end{align}
Now, from \eqref{chap:CkHinf,sec:OA,eq:maxcharrL1} we see that, for each fixed $0<\varepsilon<1/2,$ the operator $T_*^{glob}$ is of weak type $(1,1)$ with respect to $\gamma\otimes \mu;$ in particular, it is well defined for $f\in L^1.$ Finally, using \eqref{chap:CkHinf,sec:OA,eq:FubestTglob} and \eqref{chap:CkHinf,sec:OA,eq:maxcharr}, we obtain the (independent of $\varepsilon$) boundedness of $T_*^{glob}$ from $L^1_{\gamma}(H^1)$ to $L^{1,\infty}(\gamma\otimes \mu).$
\end{proof}

Now we turn to the local part $T^{loc}.$ As we already mentioned, $T^{loc}$ turns out to be of (classical) weak type $(1,1)$ with respect to $\gamma\otimes \mu.$
\begin{pro}
\label{thm:tlocalpart}
The operator $T^{loc}$ is of weak type $(1,1)$ with respect to $\gamma\otimes \mu,$ and $\|T^{loc}\|_{L^1\to L^{1,\infty}}\lesssim \ki.$ Thus, $T^{loc}$ is also bounded from $L^1_{\gamma}(H^1)$ to $L^{1,\infty}(\gamma\otimes \mu),$ and
\begin{equation*}
(\gamma\otimes \mu)(\{x\in \mathbb{R}^d\times Y\colon |T^{loc}f(x)|>s\})\leq \frac{C_{d,\mu}\ki}{s}\|f\|_{L^1_{\gamma}(H^1)},\qquad s>0.
\end{equation*}
\end{pro}
From now on we focus on the proof of Proposition \ref{thm:tlocalpart}. The key ingredient is a comparison (in the local region) of the kernel $K$ with a certain convolution kernel $\tilde{K}$ in the variables $(x_1,y_1),$ i.e.\ depending on $(x_1-y_1,x_2,y_2).$ We also heavily exploit the fact that in the local region $N_2$ the measure $\gamma\otimes \mu$ is comparable with $\Leb\otimes \mu.$

For further reference we restate \cite[Lemma 3.1]{laptype}.
\begin{lem}
\label{lem:Lem31laptype}
There exists a family of balls on $\mathbb{R}^d$
$$B_j=B\left(x_1^j,\frac{1}{20(1+|x_1^j|)}\right),$$
such that:
\begin{enumerate}[i)]
\item the family $\{B_j\colon j\in \mathbb{N}\}$ covers $\mathbb{R}^d$;
\item the balls $\{\frac14 B_j \colon j\in \mathbb{N}\}$ are pairwise disjoint;
\item for any $\beta > 0$, the family $\{\beta B_j \colon j\in \mathbb{N}\}$ has bounded overlap, i.e.;
$\sup \sum_j \chi_{\beta B_j}(x_1)\leq C$;
\item $B_j \times 4B_j \subseteq N_1$ for all $j\in \mathbb{N}$;
\item if $x_1 \in B_j,$ then $B(x_1,\frac{1}{20(1+|x_1|)}) \subseteq 4B_j$;
\item for any measurable $V\subseteq 4B_j,$ we have $\gamma(V)\approx e^{-|x_1^j|^2}\Lambda(V).$
\end{enumerate}
\end{lem}

The next lemma we need is a two variable version of \cite[Lemma 3.3]{laptype} (see also the following remark). The proof is based on Lemma \ref{lem:Lem31laptype} and proceeds as in \cite{laptype}. The only ingredient we need to add is an appropriate use of Fubini's theorem. In Lemma \ref{lem:Lem33laptype} by $\nu$ we denote one of the measures $\gamma$ or $\Lambda.$
\begin{lem}
\label{lem:Lem33laptype}
Let $S$ be a linear operator defined on $L_c^{\infty}$ and
set
$$S_1f(x)=\sum_j \chi_{B_j}(x_1)S(\chi_{4B_j}(y_1)f)(x),$$
where $B_j$ is the family of balls from Lemma \ref{lem:Lem31laptype}. We have the following:
\begin{enumerate}[i)]
\item If $S$ is of weak type $(1,1)$ with respect to the measure $\nu\otimes \mu,$ then $S_1$ is of weak type $(1,1)$ with respect to both $\gamma \otimes \mu$ and $\Lambda \otimes \mu;$ moreover, $$\|S_1\|_{L^1\to L^{1,\infty}}+\|S_1\|_{L^1(\Lambda\otimes \mu)\to L^{1,\infty}(\Lambda\otimes \mu)}\lesssim \|S\|_{L^1(\nu\otimes \mu)\to L^{1,\infty}(\nu\otimes \mu)}.$$
\item If $S$ is bounded on $L^p(\mathbb{R}^d\times Y,\nu\otimes \mu),$ for some $1<p<\infty,$ then $S_1$ is bounded on both $L^p$ and $L^p(\Lambda \otimes\mu);$ moreover,
    $$\|S_1\|_{p\to p}+\|S_1\|_{L^p(\Lambda\otimes \mu)\to L^p(\Lambda\otimes \mu)}\lesssim \|S\|_{L^p(\nu\otimes \mu)\to L^p(\nu\otimes \mu)}.$$
\end{enumerate}
\end{lem}

We proceed with the proof of Proposition \ref{thm:tlocalpart}. Decompose $T=D+\tilde{T},$ where,
\begin{align*}
Df&=\int_{0}^{1}\kappa(r)\,\partial_r [r^{\mathcal{L}}-e^{\frac{1}{4}(1-r^2)\Delta}]\,r^{A}f\,dr,\\
\tilde{T}f&=\int_{0}^{1}\kappa(r)\,\partial_r e^{\frac{1}{4}(1-r^2)\Delta}\, r^{A}f\,dr,
\end{align*}
with $\Delta$ being the self-adjoint extension of the Laplacian on $L^2(\mathbb{R}^d,\Lambda).$ Observe that, by the multivariate spectral theorem applied to the system $(-\Delta,A),$ the operator $\tilde{T}$ is bounded on $L^2(\Lambda\otimes\mu).$ Consequently, $\tilde{T}$ and thus also $D=T-\tilde{T},$ are both well defined on $L_c^{\infty}.$

We start with considering the operator $\tilde{T}.$ First we demonstrate that $$\tilde{T}=\int_{0}^{1}\kappa(r)\,\partial_r e^{\frac{1}{4}(1-r^2)\Delta}\, r^{A}\,dr$$ is a Calder\'on-Zygmund operator on the space of homogenous type $(\mathbb{R}^d\times Y, \eta, \Leb\otimes \mu);$ recall that $\eta$ is defined by \eqref{chap:CkHinf,sec:OA,eq:eta}. In what follows $\tilde{K}$ is given by
\begin{equation*}
\tilde{K}(x,y)=\int_{0}^{1}\kappa(r)\,\partial_r\mW_r(x_1-y_1)\, r^{A}(x_2,y_2)\,dr,
\end{equation*}
with
\begin{equation}
\label{chap:CkHinf,sec:OA,eq:mWrdef}
\mW_r(x_1,y_1)=\pi^{-d/2}(1-r^2)^{-d/2}\exp\bigg(-\frac{|x_1-y_1|^2}{1-r^2}\bigg).\end{equation}
In the proof of Lemma \ref{lem:calzyg} we often use the following simple bound
$$\int_0^{\infty}t^{-\alpha}\exp(-\beta t^{-1})\,dt\lesssim \beta^{-\alpha+1},\qquad \alpha>1,\quad\beta >0,$$
cf.\ \cite[Lemma 1.1]{StTor}, without further mention.

\begin{lem}
\label{lem:calzyg}
The operator $\tilde{T}$ is a Calder\'on-Zygmund operator associated with the kernel $\tilde{K}.$ More precisely, $\tilde{T}$ is bounded on $L^2(\mathbb{R}^d\times Y, \eta, \Leb\otimes \mu),$ with
\begin{equation}
\label{chap:CkHinf,sec:OA,eq:L2tilT}
\|\tilde{T}\|_{L^2( \Leb\otimes \mu)\to L^2( \Leb\otimes \mu)}\lesssim \ki,
\end{equation}
and its kernel satisfies standard Calder\'on-Zygmund estimates, i.e.\ the growth estimate
\begin{equation}
\label{chap:CkHinf,sec:OA,eq:growth}
|\tilde{K}(x,y)|\lesssim \frac{\ki}{(\Leb\otimes \mu)(B(x,\eta(x,y)))},\qquad x\neq y,
\end{equation}
and, for some $\delta>0,$ the smoothness estimate
\begin{align}
\label{chap:CkHinf,sec:OA,eq:smooth}
|\tilde{K}(x,y)-\tilde{K}(x,y')|\lesssim \left(\frac{\eta(y,y')}{\eta(x,y)}\right)^{\delta}\frac{\ki}{(\Leb\otimes \mu)(B(x,\eta(x,y)))},\qquad 2\eta(y,y')\leq \eta(x,y).
\end{align}
Consequently $\tilde{T}$ is of weak type $(1,1)$ with respect to $\Leb\otimes \mu,$ and
$$(\Leb\otimes \mu)(x\in \mathbb{R}^d\times Y\colon |\tilde{T}f(x)|>s)\leq \frac{C_{d,\mu}\ki}{s}\|f\|_{L^1(Y,\mu)},\qquad s >0.$$
\end{lem}
\begin{proof}
As we have already remarked, by spectral theory $\tilde{T}$ is bounded on $L^2(\Leb\otimes \mu),$ and we easily see that \eqref{chap:CkHinf,sec:OA,eq:L2tilT} holds. Additionally, an argument similar to the one used in the proof of Lemma \ref{lem:kerKopT} shows that $\tilde{T}$ is associated with the kernel $\tilde{K}$ even in the sense of Definition \ref{defi:kernel}.

We now pass to the proofs of the growth and smoothness estimates and start with demonstrating \eqref{chap:CkHinf,sec:OA,eq:growth}. An easy calculation shows that
\begin{equation}
\partial_r\mW_r(x_1-y_1)=\pi^{-d/2}r(1-r^2)^{-d/2-1}\exp\bigg(-\frac{|x_1-y_1|^2}{1-r^2}\bigg)\left[d-2\frac{|x_1-y_1|^2} {1-r^2}\right]. \label{chap:CkHinf,sec:OA,eq:comker}
\end{equation}
Hence, we have for $x_1,y_1 \in\mathbb{R}^d$
\begin{equation}
\label{chap:CkHinf,sec:OA,eq:comestt}
|\partial_r\mW_r(x_1-y_1)|\lesssim r(1-r)^{-d/2-1}\exp\bigg(-\frac{|x_1-y_1|^2}{4(1-r)}\bigg),\qquad 0<r<1.
\end{equation}
For further use we remark that the above bound implies
\begin{equation}
\label{chap:CkHinf,sec:OA,eq:comesttint}
\int_0^{1}|\partial_r\mW_r(x_1-y_1)|\,dr \lesssim |x_1-y_1|^{-d},\qquad x_1,y_1\in\mathbb{R}^d,\quad x_1\neq y_1.
\end{equation}

From \eqref{chap:CkHinf,sec:OA,eq:comestt} we see that
\begin{equation}
\label{chap:CkHinf,sec:OA,eq:comesttt}
|[\partial_r\mW_r(x_1-y_1)]_{r=e^{-t}}|
\leq
\left\{
\begin{array}{rl}
 &C t^{-d/2-1}\exp\bigg(-\frac{|x_1-y_1|^2}{ct}\bigg),\qquad t\leq 1, \\
  & C  e^{-t}\exp\bigg(-c|x_1-y_1|^2\bigg),\qquad t>1.
 \end{array} \right.
\end{equation}
Thus, coming back to the variable $t=-\log r$ and then using \eqref{chap:CkHinf,sec:OA,eq:gausbound}, we arrive at
\begin{align*}
&|\tilde{K}(x,y)|\lesssim \ki \int_0^{\infty}t^{-d/2-1}\exp\bigg(-\frac{|x_1-y_1|^2}{ct}\bigg)\frac{1}{\mu(B(x_2,\sqrt{t}))}\exp\bigg(-\frac{\zeta^2(x_2,y_2)}{ct}\bigg)\,dt.
\end{align*}
A standard argument using the doubling property of $\mu$ (cf. \eqref{chap:CkHinf,sec:OA,eq:double}) shows that we can further estimate
\begin{align*}
&|\tilde{K}(x,y)|\lesssim \frac{\ki}{\mu(B(x_2,\eta(x,y)))} \int_0^{\infty}t^{-d/2-1}\exp\bigg(-\frac{\eta^2(x,y)}{2ct}\bigg)\,dt.
\end{align*}
The last integral is bounded by a constant times $\eta^{d}(x,y),$ which equals $C_d \Leb(B_{|\cdot|}(x_1,\eta(x,y))).$ Thus, \eqref{chap:CkHinf,sec:OA,eq:growth} follows once we note that
$$ \frac{1}{\Leb(B_{|\cdot|}(x_1,\eta(x,y))\mu(B_{\zeta}(x_2,\eta(x,y)))}=\frac{1}{(\Leb\otimes \mu) (B(x,\eta(x,y)))}.$$

We now focus on the smoothness estimate \eqref{chap:CkHinf,sec:OA,eq:smooth}, which is enough to obtain the desired weak type $(1,1)$ property of $\tilde{T}.$ We decompose the difference in \eqref{chap:CkHinf,sec:OA,eq:smooth} as
\begin{equation*}
\tilde{K}(x,y)-\tilde{K}(x,y')=[\tilde{K}(x,y)-\tilde{K}(x,y'_1,y_2)]+[\tilde{K}(x,y'_1,y_2)-\tilde{K}(x,y')]\equiv I_1+I_2.
\end{equation*}
Till the end of the proof of \eqref{chap:CkHinf,sec:OA,eq:smooth} we assume $\eta(x,y)\geq 2 \eta(y,y'),$ so that $\eta(x,y)\approx \eta(x,y').$

We start with estimating $I_2$ and consider two cases. First, let $|x_1-y_1|\leq \zeta(x_2,y_2).$ Then, $\eta(x,y)=\zeta(x_2,y_2)\geq 2\eta(y,y')\geq 2\zeta(y_2,y'_2)$ and consequently, $\zeta(x_2,y'_2)\approx \zeta(x_2,y_2).$ Now, coming back to the variable $t=-\log r$ and using \eqref{chap:CkHinf,sec:OA,eq:comesttt} we have
\begin{align*}
|I_2|\lesssim \ki \int_0^{\infty} t^{-d/2-1}\exp\bigg(-\frac{|x_1-y'_1|^2}{ct}\bigg)\big|e^{-tA}(x_2,y_2)-e^{-tA}(x_2,y'_2)\big|\,dt.
\end{align*}
Hence, from \eqref{chap:CkHinf,sec:OA,eq:heatlipsch} it follows that
\begin{align}
\label{chap:CkHinf,sec:OA,eq:quan}
|I_2|\lesssim \ki \zeta(y_2,y'_2)^{\delta}\int_0^{\infty}t^{-d/2-1-\delta/2}\frac{1}{\mu(B(x_2,\sqrt {t})}\exp\bigg(-\frac{\eta^2(x,y)}{ct}\bigg)\,dt
\end{align}
Using the doubling property of $\mu$ it is not hard to see that
\begin{equation}
\label{chap:CkHinf,sec:OA,eq:double}
\frac{1}{\mu(B(x_2,\sqrt t)}\exp\bigg(-\frac{\eta^2(x,y)}{ct}\bigg)\lesssim \frac{1}{\mu(B(x_2,\eta(x,y)))}\exp\bigg(-\frac{\eta^2(x,y)}{2ct}\bigg),
\end{equation}
and consequently,
\begin{align*}
|I_2|&\lesssim \ki \zeta(y_2,y'_2)^{\delta}\frac{1}{\mu(B(x_2,\eta(x,y)))}\int_0^{\infty}t^{-d/2-1-\delta/2}\exp\bigg(-\frac{\eta^2(x,y)}{3ct}\bigg)\,dt
\\&\lesssim \ki\zeta(y_2,y'_2)^{\delta}\eta(x,y)^{-\delta}\frac{1}{\mu(B(x_2,\eta(x,y)))}(\eta^2(x,y))^{-d/2},
\end{align*}
thus proving that
\begin{equation}
\label{chap:CkHinf,sec:OA,eq:I2}
|I_2|\lesssim \bigg(\frac{\zeta(y_2,y'_2)}{\eta(x,y)}\bigg)^{\delta}\frac{\ki}{(\Leb\otimes\mu)(B(x,\eta(x,y)))}.
\end{equation}

Assume now that $\zeta(x_2,y_2)\leq |x_1-y_1|.$ In this case $\eta(x,y)=|x_1-y_1|>2\eta(y,y')\geq 2|y_1-y'_1|,$ so that $|x_1-y_1|\approx |x_1-y'_1|.$ Hence, proceeding similarly as in the previous case (this time we use \eqref{chap:CkHinf,sec:OA,eq:heatlipschngauss} instead of \eqref{chap:CkHinf,sec:OA,eq:heatlipsch}), we obtain
\begin{align*}
|I_2|\lesssim \ki \zeta(y_2,y'_2)^{\delta} \int_0^{\infty}t^{-d/2-1-\delta/2}\frac{1}{\mu(B(x_2,\sqrt{t}))}\exp\bigg(-\frac{\eta^2(x,y)}{ct}\bigg)\,dt.
\end{align*}
The latter quantity has already appeared in \eqref{chap:CkHinf,sec:OA,eq:quan} and has been estimated by the right hand side of \eqref{chap:CkHinf,sec:OA,eq:I2}.

Now we pass to $I_1.$ A short computation based on \eqref{chap:CkHinf,sec:OA,eq:comker} gives
\begin{equation*}
\pi^{d/2}\partial_{z_j}\partial_r \mW_r (z)=-2r(1-r^2)^{-d/2-2}z_j\left(d+2-2\frac{|z|^2}{1-r^2}\right)\exp\bigg(\frac{-|z|^2}{1-r^2}\bigg),\qquad z\in \mathbb{R}^d.
\end{equation*}
From the above inequality it is easy to see that
$$|\partial_{z_j}\partial_r\mW_{r}(z)|\lesssim r(1-r)^{-d/2-3/2}\exp\bigg(\frac{-|z|^2}{2(1-r^2)}\bigg),$$
and consequently, after the change of variable $e^{-t}=r,$
$$|\partial_{z_j}\partial_r\mW_{r}(z)\big|_{r=e^{-t}}|\lesssim t^{-d/2-3/2}\exp\bigg(\frac{-|z|^2}{ct}\bigg),\qquad 0<t<\infty.$$
Hence, from the mean value theorem it follows that for $|x_1-y_1|\geq 2 |y_1-y'_1|,$
\begin{align}
\label{chap:CkHinf,sec:OA,eq:lipmW}
&|\left[\partial_r\mW_{r}(x_1-y_1)-\partial_r\mW_{r}(x_1-y'_1)\right]_{r=e^{-t}}|\lesssim \frac{|y_1-y'_1|}{\sqrt{t}} \, t^{-d/2-1}\exp\bigg(\frac{-|x_1-y_1|^2}{ct}\bigg)\,
\end{align}
while for arbitrary $x_1,y_1,$
\begin{align}
\label{chap:CkHinf,sec:OA,eq:lipngaussmW}
&|\left[\partial_r\mW_{r}(x_1-y_1)-\partial_r\mW_{r}(x_1-y'_1)\right]_{r=e^{-t}}|\lesssim \frac{|y_1-y'_1|}{\sqrt{t}}\, t^{-d/2-1}.
\end{align}
Moreover, at the cost of a constant in the exponent, the expression $|y_1-y'_1|/\sqrt{t}$ from the right hand sides of \eqref{chap:CkHinf,sec:OA,eq:lipmW} and \eqref{chap:CkHinf,sec:OA,eq:lipngaussmW} can be replaced by $(|y_1-y'_1|t^{-1/2})^{\delta},$ for arbitrary $0<\delta \leq 1.$
If $|y_1-y'_1|\leq \sqrt{t},$ this is a consequence of \eqref{chap:CkHinf,sec:OA,eq:lipmW} and \eqref{chap:CkHinf,sec:OA,eq:lipngaussmW}, while if $|y_1-y'_1|\geq \sqrt{t}$ it can be deduced from \eqref{chap:CkHinf,sec:OA,eq:lipmW} and \eqref{chap:CkHinf,sec:OA,eq:comesttt}. Similarly as it was done for $I_2,$ to estimate $I_1$ we consider two cases.

Assume first $|x_1-y_1|\geq\zeta(x_2,y_2),$ so that $\eta(x,y)=|x_1-y_1|> 2\eta(y,y')\geq 2|y_1-y'_1|$ and $|x_1-y_1|\approx |x_1-y'_1|.$ Therefore, using \eqref{chap:CkHinf,sec:OA,eq:gausbound} and the version of \eqref{chap:CkHinf,sec:OA,eq:lipmW} with $(|y_1-y_1'|t^{-1/2})^{\delta}$ in place of $|y_1-y_1'|/\sqrt{t},$ we obtain
\begin{align*}
|I_1|\lesssim \ki |y_1-y'_1|^{\delta}\int_0^{\infty} t^{-d/2-1-\delta/2}\frac{1}{\mu(B(x_2,\sqrt{t}))}\exp\bigg(-\frac{\eta^2(x,y)}{ct}\bigg)\,dt.
\end{align*}
Almost the same quantity appeared already in \eqref{chap:CkHinf,sec:OA,eq:quan}, thus employing once again previous techniques, we end up with
\begin{equation}
\label{chap:CkHinf,sec:OA,eq:I1}
|I_1|\lesssim \bigg(\frac{|y_1-y'_1|}{\eta(x,y)}\bigg)^{\delta}\frac{\ki}{(\Leb\otimes \mu)(B(x,\eta(x,y)))}.
\end{equation}

Assume now that $|x_1-y_1|<\zeta(x_2,y_2),$ so that $\eta(x,y)=\zeta(x_2,y_2)>2\eta(y,y')\geq 2\zeta(y_2,y'_2)$ and $\zeta(x_2,y_2)\approx \zeta(x_2,y'_2).$ This time, from \eqref{chap:CkHinf,sec:OA,eq:gausbound} and the $\delta$ version of \eqref{chap:CkHinf,sec:OA,eq:lipngaussmW} we have
\begin{align*}
|I_1|\lesssim \ki |y_1-y'_1|^{\delta}\int_0^{\infty} t^{-d/2-1-\delta/2}\frac{1}{\mu(B(x_2,\sqrt{t}))}\exp\bigg(-\frac{\eta^2(x,y)}{ct}\bigg)\,dt,
\end{align*}
which has been already estimated by the right hand side of \eqref{chap:CkHinf,sec:OA,eq:I1}.

Finally, \eqref{chap:CkHinf,sec:OA,eq:smooth} follows after collecting the bounds \eqref{chap:CkHinf,sec:OA,eq:I2} and \eqref{chap:CkHinf,sec:OA,eq:I1}, thus finishing the proof of Lemma \ref{lem:calzyg}.
\end{proof}

Now we focus on the operator $D=T-\tilde{T}.$ Since $T$ and $\tilde{T}$ are associated with the kernels $K$ and $\tilde{K},$ respectively, $D$ is associated with
$$D(x,y)=\int_{0}^1\int_{0}^{1}\kappa(r)\,\partial_r [r^{\mathcal{L}}(x_1,y_1)-e^{\frac{1}{4}(1-r^2)\Delta}(x_1-y_1)]\,r^{A}(x_2,y_2)\,dr.$$
Using \eqref{chap:CkHinf,sec:OA,eq:maxcharrL1}, \eqref{chap:CkHinf,sec:OA,eq:comestt}, and the fact that $\supp \kappa\subseteq [\varepsilon,1-\varepsilon],$  it is not hard to see that
$$\tilde{T}^{glob}f(x)=\int_{\mathbb{R}^d}\int_Y \chi_{N_2^c}\tilde{K}(x,y)f(y)\,d\mu(y_2)\,dy_1,$$
is a well defined and bounded operator on $L^1(\Lambda\otimes \mu).$ Thus,
$$D^{glob}f(x):=T^{glob}f(x)-\tilde{T}^{glob}f(x)=\int_{\mathbb{R}^d}\int_Y \chi_{N_2^c}D(x,y)f(y)\,d\mu(y_2)\,dy_1,$$
is a well defined operator on $L_c^{\infty}.$ Consequently, $D^{loc}f(x):=Df(x)-D^{glob}f(x)$ is also a.e.\ well defined for $f\in L_c^{\infty}.$ Moreover, we have $D^{loc}=T^{loc}-\tilde{T}^{loc},$ where $\tilde{T}^{loc}:=\tilde{T}-\tilde{T}^{glob}.$

We shall need two auxiliary lemmata. Recall that $\M_r$ and $\mW_r$ are given by \eqref{chap:CkHinf,sec:OA,eq:MehlformOU} and \eqref{chap:CkHinf,sec:OA,eq:mWrdef}, respectively.
\begin{lem}
\label{lem:diffest}
If $(x_1,y_1)\in N_2,$ then we have
\begin{align}\nonumber
D_I(x_1,y_1)&:=\int_0^1 |\partial_r\mathcal{M}_r(x_1,y_1)-\partial_r\mW_r(x_1-y_1)|\,dr\\
&\leq
\left\{
\begin{array}{rl}
 &C \frac{1+|x_1|}{|x_1-y_1|^{d-1}},  \\
  & C (1+|x_1|)\log \frac{C}{|x_1||x_1-y_1|},
 \end{array} \right.
 {\rm  if  }
\begin{array}{rl}
 & d>1, \\
  & d=1.
 \end{array}
 \label{chap:CkHinf,sec:OA,eq:bounddiffest}
\end{align}
\end{lem}
\begin{proof}
We proceed similarly to the proof of \cite[Lemma 3.9]{high}.
Since for $(x_1,y_1)$ from the local region $N_2$ we have
\begin{equation}
\label{chap:CkHinf,sec:OA,eq:localreg}
|x_1-y_1|^2+(r-1)^2|x_1|^2-C(1-r)\leq |rx_1-y_1|^2\leq (r-1)^2|x_1|^2+|x_1-y_1|^2+C(1-r),
\end{equation}
therefore
$$\frac{|rx_1-y_1|^2}{1-r^2}\exp\bigg(-\frac{|rx_1-y_1|^2}{1-r^2}\bigg)\lesssim \exp\bigg(-\frac{|x_1-y_1|^2}{2(1-r^2)}\bigg),$$
and
\begin{align}
\label{chap:CkHinf,sec:OA,eq:estmix}&|\langle rx_1-y_1,x_1\rangle|\exp\bigg(-\frac{|rx_1-y_1|^2}{2(1-r^2)}\bigg)\\ \nonumber
&\lesssim |rx_1-y_1|\exp\bigg(-\frac{|x_1-y_1|^2}{4(1-r)}\bigg)|x_1|\exp\bigg(-c(1-r)|x_1|^2\bigg)\lesssim 1.
\end{align}
Thus, using \eqref{chap:CkHinf,sec:OA,eq:comM} we obtain for $(x_1,y_1)\in N_2,$
\begin{equation}\label{chap:CkHinf,sec:OA,eq:comMestt} |\partial_r\,\mM_r(x_1,y_1)|\lesssim (1-r)^{-d/2-1}\exp\bigg(-\frac{|x_1-y_1|^2}{4(1-r)}\bigg),\qquad 0<r<1.\end{equation}
Note that the above inequality implies
\begin{equation}\label{chap:CkHinf,sec:OA,eq:comesttintM} \int_0^1|\partial_r\,\mM_r(x_1,y_1)|\lesssim |x_1-y_1|^{-d},\qquad (x_1,y_1)\in N_2.\end{equation}

Using \eqref{chap:CkHinf,sec:OA,eq:comMestt} and \eqref{chap:CkHinf,sec:OA,eq:comestt} we easily see that
$$\int_0^{1/2}|\partial_r\mathcal{M}_r(x_1,y_1)-\partial_r\mW_r(x_1-y_1)|\,dr\lesssim 1,\qquad (x_1,y_1)\in N_2$$
which is even better then the estimate we want to prove.

Now we consider the integral over $(1/2,1).$ Denoting $r(x_1)=\max (1/2, 1-|x_1|^2)$ and using once again \eqref{chap:CkHinf,sec:OA,eq:comMestt} and \eqref{chap:CkHinf,sec:OA,eq:comestt} we obtain
$$\int_{1/2}^{r(x_1)}|\partial_r\mathcal{M}_r(x_1,y_1)-\partial_r\mW_r(x_1-y_1)|\,dr\lesssim \int_{1/2}^{r(x_1)}(1-r)^{-d/2-1}e^{-\frac{|x_1-y_1|^2}{4(1-r)}}\,dr.$$
The above quantity is exactly the one estimated by the right hand side of \eqref{chap:CkHinf,sec:OA,eq:bounddiffest} in the second paragraph of the proof of \cite[Lemma 3.9]{high}. It remains to estimate the integral taken over $(r(x_1),1).$ Using the formulae \eqref{chap:CkHinf,sec:OA,eq:comM} and \eqref{chap:CkHinf,sec:OA,eq:comker} together with \eqref{chap:CkHinf,sec:OA,eq:estmix} we write
\begin{equation*}
\int_{r(x_1)}^1|\partial_r\mM_r(x_1,y_1)-\partial_r\mW_r(x_1-y_1)|\,dr\lesssim J_1+J_2,
\end{equation*}
with
\begin{align*}
J_2=&\int_{r(x_1)}^1 (1-r)^{-d/2-1}\left|\big[d-\frac{2|rx_1-y_1|^2}{1-r^2}\big]\exp\bigg(-\frac{|rx_1-y_1|^2}{(1-r^2)}\bigg)\right.
\\&-\left.\big[d-\frac{2|x_1-y_1|^2}{1-r^2}\big]\exp\bigg(-\frac{|x_1-y_1|^2}{1-r^2}\bigg)\right|\,dr,\\
J_2=&|x_1|\int_{r(x_1)}^1(1-r)^{-d/2-1/2}\exp\bigg(-\frac{|x_1-y_1|^2}{8(1-r)}\bigg)\,dr.
\end{align*}
The quantity $J_2$ has been already estimated in the proof of \cite[Lemma 3.9, p.12]{high}, thus we focus on $J_1.$ For fixed $r, x_1,y_1$ denote
$$\ph(s)=\phi_{r,x_1,y_1}(s)=\left(d-\frac{2|sx_1-y_1|^2}{1-r^2}\right)\exp\bigg(-\frac{|sx_1-y_1|^2}{1-r^2}\bigg),$$
so that
$$J_1=\int_{r(x_1)}^1 (1-r)^{-d/2-1}|\phi_{r,x_1,y_1}(1)-\phi_{r,x_1,y_1}(r)|\,dr.$$
Since
$$\ph'(s)=2\left(-\frac{(d+2)\langle sx_1-y_1,x_1\rangle}{1-r^2}+2\frac{\langle sx_1-y_1,x_1\rangle|sx_1-y_1|^2}{(1-r^2)^2}\right)\exp\bigg(-\frac{|sx_1-y_1|^2}{1-r^2}\bigg),$$
by using \eqref{chap:CkHinf,sec:OA,eq:localreg} and \eqref{chap:CkHinf,sec:OA,eq:estmix} with $r$ replaced by $s,$ we obtain
\begin{equation*}
|\ph'(s)|\lesssim |x_1|(1-r)^{-1/2}\exp\bigg(-\frac{|x_1-y_1|^2}{8(1-r)}\bigg).
\end{equation*}
Thus, by the mean value theorem
\begin{align*}
|J_1|\lesssim |x_1|\int_{r(x_1)}^1(1-r)^{-d/2-1/2}\exp\bigg(-\frac{|x_1-y_1|^2}{8(1-r)}\bigg)\,dr=J_2.
\end{align*}
Recalling that $J_2$ was estimated before, we conclude the proof.
\end{proof}

The following shows that the local parts of $T$ and $\tilde{T}$ inherit their boundedness properties.
\begin{lem}
\label{lem:locinher}
Let $S$ denote one of the operators $T$ or $\tilde{T},$ and let $\nu$ be any of the measures $\gamma$ or $\Lambda.$ Then $S^{loc}$ is bounded on $L^2(\nu\otimes \mu):=L^2(\mathbb{R}^d\times Y,\nu\otimes \mu),$ and
\begin{equation}
\label{chap:CkHinf,sec:OA,eq:locinherL2}
\|S^{loc}\|_{L^2(\nu\otimes \mu)\to L^2(\nu\otimes \mu)}\lesssim \ki.\end{equation}
Moreover, $\tilde{T}^{loc}$ is of weak type $(1,1)$ with respect to $\nu\otimes \mu,$ with $\nu=\gamma$ or $\nu=\Lambda,$ and
\begin{equation}
\label{chap:CkHinf,sec:OA,eq:locinherweakTtil}
(\nu\otimes\mu)(x\in \mathbb{R}^d\times Y\colon |\tilde{T}^{loc}f(x)|>s)\leq\frac{C_{d,\mu}\ki}{s}\|f\|_{L^1(\nu\otimes \mu)}.
\end{equation}
\end{lem}
\begin{proof}
In what follows $S(x,y)$ denotes either the kernel $K(x,y)$ of $T,$ or the kernel $\tilde{K}(x,y)$ of $\tilde{T}.$ Recall that in both the cases the integral defining $S^{glob}f(x)$ is absolutely convergent.

The proof is analogous to the proof of \cite[Proposition 3.4]{laptype}. Let $B_j$ be the family of balls in $\mathbb{R}^d$ from Lemma \ref{lem:Lem31laptype}. Take $f\in L_c^{\infty}$ and, for $x_1\in B_j,$ decompose
\begin{align*}
&S^{loc}f(x)=Sf(x)-S^{glob}f(x)\\
&=S(f\chi_{4B_j}(y_1))(x)+S(f\chi_{(4B_j)^c}(y_1))-\int_{\mathbb{R}^d}\int_Y \chi_{N_2^c}S(x,y)(x_1,y_1)f(y)\, d\mu(y_2)\,dy_1\\
&=S(f\chi_{4B_j}(y_1))(x)+\int_{\mathbb{R}^d}\int_Y \chi_{(4B_j)^c}(y_1)-\chi_{N_2^c}(x_1,y_1))S(x,y)f(y)\, d\mu(y_2)\,dy_1.
\end{align*}

Multiplying by $\chi_{B_j}(x_1)$ and summing over $j,$ we arrive at the inequality
\begin{align*}
S^{loc}f(x)&\leq \int_{\mathbb{R}^d}\sum_j \chi_{B_j}(x_1)|\chi_{N_2}(x_1,y_1)-\chi_{4B_j}(y_1)|\int_{Y}|S(x,y)||f(y)|\,d\mu(y_2)\,dy_1 \\
&+\sum_j \chi_{B_j}(x_1)|S(f(y)\chi_{4B_j}(y_1))(x)|
:= S_2(f)+S_1(f),
\end{align*}
Recall that $T$ is bounded on $L^2,$ while $\tilde{T}$ is bounded on $L^2(\Lambda\otimes \mu).$ Hence, both for $S=T$ and $S=\tilde{T}$, using Lemma \ref{lem:Lem33laptype} we see that $S_1$ is bounded on $L^2(\nu\otimes \mu),$ and $\|S_1\|_{L^2(\nu\otimes \mu)\to L^2(\nu\otimes \mu)}\lesssim \ki.$ Moreover, from Lemma \ref{lem:calzyg} we know that $\tilde{T}$ is of weak type $(1,1)$ with respect to $\Lambda \otimes \mu,$ and $\|\tilde{T}\|_{L^1(\Lambda\otimes \mu)\to L^{1,\infty}(\Lambda\otimes \mu)}\lesssim \ki.$ Hence, using once again Lemma \ref{lem:Lem33laptype}, we see that in the case $S=\tilde{T},$ we have $\|S_1\|_{L^1(\nu\otimes \mu)\to L^{1,\infty}(\nu\otimes \mu)}\lesssim \ki.$

It remains to consider $S_2,$ for which we show boundedness on both $L^1(\nu\otimes \mu)$ and $L^{\infty}(\nu\otimes \mu).$ Here we need the following estimate, valid for $(x_1,y_1)\in N_2,$ and $f\in L_c^{\infty},$
\begin{equation}
\label{chap:CkHinf,sec:OA,eq:SL1norm}
\int_{Y}\int_Y|S(x,y)f(y_1,y_2)|\,d\mu(y_2)\,d\mu(x_2)\lesssim \ki |x_1-y_1|^{-d}\|f(y_1,\cdot)\|_{L^1(Y,\mu)}.
\end{equation}
Recall that $r^A$ is a contraction on $L^1(Y,\mu).$ Thus, for $S=T$ the above follows from \eqref{chap:CkHinf,sec:OA,eq:comesttintM}, while, for $S=\tilde{T}$ it is a consequence of \eqref{chap:CkHinf,sec:OA,eq:comesttint}.

We start with the boundedness on $L^1(\nu\otimes \mu)$ and denote $g(y_1)=\|f(y_1,\cdot)\|_{L^1(Y,\mu)}.$ From Lemma \ref{lem:Lem31laptype} iv) it follows that $\sum_j \chi_{B_j}(x_1)|\chi_{N_2}(x_1,y_1)-\chi_{4B_j}(y_1)|$ is supported in $N_2.$
Hence, using Fubini's theorem and \eqref{chap:CkHinf,sec:OA,eq:SL1norm}, we obtain
\begin{align*}&\|S_2(f)(x_1,\cdot)\|_{L^1(Y,\mu)}\\
&\lesssim
\ki  \int_{\mathbb{R}^d}\sum_j \chi_{B_j}(x_1)|\chi_{N_2}(x_1,y_1)-\chi_{4B_j}(y_1)||x_1-y_1|^{-d}\,|g(y_1)|\,d\nu(y_1).
\end{align*}
From that point we proceed exactly as in the $T^2$ part of the proof of \cite[Proposition 3.4]{laptype}, arriving at
$
\|S_2(f)\|_{L^1(\nu\otimes \mu)}\lesssim\int_{\mathbb{R}^d}g(y_1)\,d\nu(y_1)= \|f\|_{L^1(\nu\otimes \mu)}.
$
The proof of the $L^{\infty}(\nu\otimes \mu)$ boundedness of $S_2$ is similar.
\end{proof}

Applying Lemmata \ref{lem:diffest} and \ref{lem:locinher} we now prove the following.
\begin{lem}
\label{prop:dloc}
The operator $D^{loc}$ is bounded on all the spaces $L^p(\Leb\otimes \mu).$  Moreover,
\begin{equation}
\label{chap:CkHinf,sec:OA,eq:DlocLpbound}
\|D^{loc}\|_{L^p(\Leb\otimes \mu) \to L^p(\Leb\otimes \mu)}\lesssim \ki, \qquad 1\leq p\leq \infty.\end{equation}
\end{lem}
\begin{proof} Observe that $D^{loc}$ may be expressed as
$$D^{loc}f(x)=\int_{\mathbb{R}^d}\chi_{N_2}(x_1,y_1)\int_0^1\, \kappa(r) \left[\partial_r\mathcal{M}_r(x_1,y_1)-\partial_r\mW_r(x_1-y_1)\right] (r^A f)(y_1,x_2)\,dr\,dy_1,$$
at least for $f\in L_c^{\infty}.$ Moreover, the estimates below imply that the integral defining $D^{loc}$ is actually absolutely convergent, whenever $f\in L^p(\Lambda\otimes\mu),$ for some $1\leq p\leq \infty.$

Using Fubini's theorem, and the $L^1(Y,\mu)$ contractivity of $r^A,$
\begin{align*}
&\|\kappa\|_{\infty} ^{-1}\|D^{loc}f\|_{L^{1}(\Leb\otimes \mu)}\\
&\leq \int_{\mathbb{R}^d}\int_{\mathbb{R}^d}\chi_{N_2}\int_0^1 |\partial_r\mathcal{M}_r(x_1,y_1)-\partial_r\mW_r(x_1-y_1)|\int_{Y}|(r^A|f|)(y_1,x_2)|\, d\mu(x_2)\,dr\,dy_1dx_1\\
&\leq  \int_{Y}\int_{\mathbb{R}^d}\int_{\mathbb{R}^d}\chi_{N_2}D_I(x_1,y_1)|f(y_1,x_2)| \,dy_1dx_1\,d\mu(x_2).
\end{align*}
Now, using Lemma \ref{lem:diffest} it can be shown that the singularity of $\chi_{N_2}D_I(x_1,y_1)$ is integrable in $x_1.$ Moreover, $\int_{\mathbb{R}^d}\chi_{N_2}D_I(x_1,y_1)\,dx_1\leq C,$ where $C$ is independent of $y_1.$ Thus, applying Fubini's theorem we obtain $\|D^{loc}\|_{L^{1}(\Leb\otimes \mu)\to L^{1}(\Leb\otimes \mu)}\leq C\|\kappa\|_{\infty}$. Since in the local region $ |x_1|\leq 2+ |y_1|\leq 4+|x_2|$ and $\chi_{N_2}(x_1,y_1)=\chi_{N_2}(y_1,x_1),$ the singularity of $\chi_{N_2}D_I(x_1,y_1)$ is also integrable in $y_1.$ Hence, using Fubini's theorem and the $L^{\infty}(Y,\mu)$ contractivity of $r^A,$ we have $\|D^{loc}\|_{L^{\infty}(\Leb\otimes \mu)\to L^{\infty}(\Leb\otimes \mu)}\leq C\|\kappa\|_{\infty}.$ Interpolating between the $L^1(\Leb\otimes \mu)$ and $L^{\infty}(\Leb\otimes \mu)$ bounds for $D^{loc}$ we finish the proof of \eqref{chap:CkHinf,sec:OA,eq:DlocLpbound}.
\end{proof}

Recalling that $T^{loc}=\tilde{T}^{loc}+D^{loc},$ and using Lemmata \ref{lem:locinher} and \ref{prop:dloc}, we see that the local part $T^{loc}$ is of weak type $(1,1)$ with respect to $\Lambda \otimes \mu.$ Moreover, the weak type $(1,1)$ constant is less than or equal to $C_{d,\mu}\ki.$ Thus, repeating the proof of \eqref{chap:CkHinf,sec:OA,eq:locinherweakTtil} in Lemma \ref{lem:locinher} with $\tilde{T}$ replaced by $T,$ we finish the proof of Proposition \ref{thm:tlocalpart}. After combining Propositions \ref{prop:proglob} and \ref{thm:tlocalpart}, the proof of Theorem \ref{thm:OA} is completed.

As we observed before, besides being bounded from $L^1_{\gamma}(H^1)$ to $L^{1,\infty}(\gamma \otimes \mu)$ and on $L^2,$ the operator $m_{\kappa}(\mL,A)$ is also bounded on all the $L^p$ spaces, $1<p<2.$ Till the end of this section we focus on showing that this interpolation property remains true for general operators.
\begin{thm}
\label{thm:interH1}
Let $S$ be an operator which is bounded from $L^1_{\gamma}(H^1)$ to $L^{1,\infty}(\gamma\otimes \mu),$ and from $L^2$ to $L^2.$ Then $S$ is bounded on all $L^p$ spaces, $1<p<2.$
\end{thm}

The main ingredient of the proof is a Calder\'on-Zygmund decomposition of a function $f(x_1,x_2),$ with respect to the variable $x_2,$ when $x_1$ is fixed, see Lemma \ref{lem:calzygdec}. For the decomposition we present it does not matter that we consider $\mathbb{R}^d$ with the measure $\gamma.$ The important assumption is that $(Y,\zeta,\mu)$ is a space of homogenous type. Therefore till the end of the proof of Lemma \ref{lem:calzygdec} we consider a more general space $L^1:=L^1(X\times Y, \nu \otimes \mu).$ Here $\nu$ is an arbitrary $\sigma$-finite Borel measure on $X.$ Recall that, by convention, elements of $X$ are denoted by $x_1,y_1,$ while elements of $Y$ are denoted by $x_2,y_2.$

It is known that in every space of homogenous type in the sense of Coifman-Weiss there exists a family of {\it disjoint} 'dyadic' cubes, see \cite[Theorem 2.2]{dyad}. Here we use \cite[Theorem 2.2]{dyad} to $(Y,\zeta,\mu).$ Let $\mathcal{Q}_l$ be the set of all dyadic cubes of generation $l$ in the space $(Y,\zeta,\mu).$ Note that $l\to\infty$ corresponds to 'small' cubes, while $l\to-\infty$ to 'big' cubes. We define the $l$-th generation dyadic average and the dyadic maximal function with respect to the second variable, by
\begin{equation*}
E_{l}f(x)=\sum_{Q\in\mathcal{Q}_l}\frac{1}{\mu(Q)}\int_{Q}f(x_1,y_2)\,d\mu(y_2)\,\chi_{Q}(x_2),\end{equation*}
and
\begin{equation}\label{chap:CkHinf,sec:OA,eq:dyadmaxdef}\mathcal{D}f(x)=\sup_{l} E_{l}|f|(x),\end{equation}
respectively.

We prove the following Calder\'on-Zygmund type lemma.
\begin{lem}
\label{lem:calzygdec}
Fix $s>0$ and let $f\in L^1$ be a continuous non-negative function on $X\times Y.$  Then there exist Borel measurable functions $g$ ('good') and $\{b_j\}$ ('bad') such that  $f=g+b:=g+\sum_{j} {b_j},$  and:
\begin{enumerate}[(i)]
\item $\|g\|_{L^1}+\sum_j \|b_j\|_{L^1}\leq 4\|f\|_{L^1};$
\item $|g(x)|\leq C_{\mu} s,$ for $x=(x_1,x_2)\in X\times Y;$
\item each function $b_j$ is associated with unique dyadic cube $Q_j.$ Moreover, the functions $b_j$ are supported in disjoint measurable sets $S_j=F_j\times Q_j$ such that for each fixed $x_1\in X,$ we have $\sum_{j}\mu (S_j(x_1))\leq s^{-1}\int_Y f(x)\,d\mu(x_2),$ where $S_j(x_1)=\{x_2\colon x\in S_j\}.$ Additionally, for each fixed $j\in\mathbb{Z}$ and $x_1\in X,$ $\int_{Y}b_j(x)\,d\mu(x_2)=0,$ and either, there exists a 'cube' such that $Q_{j(x_1)}=S_j(x_1)$ and $\supp(b_j(x_1,\cdot))\subset Q_{j(x_1)},$ or $S_j(x_1)=\emptyset$ and $b_j(x_1,\cdot)\equiv 0;$
\item If, for fixed $x_1\in X$ the set $S_j(x_1)$ is non empty (hence in view of (iii) $S_j(x_1)=Q_{j(x_1)}$), then
 $$C_{\mu}^{-1}s\leq \frac{1}{\mu(Q_{j(x_1)})}\int_{Q_{j(x_1)}}f(x)\,d\mu(x_2)\leq C_{\mu} s;$$
\item $$\{x\in X\times Y\colon \mathcal{D}(f)(x)>s\}=\bigcup_{j}F_j\times Q_j=\bigcup S_j.$$
\end{enumerate}
\end{lem}
\begin{proof}
In part the lemma is intuitively quite clear. The fact we do need to prove is that the decomposition can be done in a 'measurable' way.

Since $f$ is continuous $E_{l}f$ is measurable on $X\times Y.$ Therefore
$$\Omega_l=\left\{x\in X\times Y\colon E_{l}f(x)> s,\quad E_{l'}f(x)\leq s \textrm{ for } l'<l\right\}$$
are measurable subsets of $X\times Y.$ Moreover, the sets $\Omega_l$ are pairwise disjoint and satisfy
\begin{equation}\label{chap:CkHinf,sec:OA,eq:sumOmel} \Omega:=\bigcup_l \Omega_l=\{x\in X\times Y\colon \mathcal{D}(f)(x)>s\}\end{equation}. Setting $\Omega=\bigcup_l\Omega_l$ we see that if $x\in \Omega^c,$ then $f(x)\leq s.$

Observe now that for each fixed $x_1\in X,$ if $z_{Q_{\alpha}^l}$ denotes the center of the cube $Q_{\alpha}^l,$ then $E_{l}f(x)=E_{l}f(x_1,z_{Q_{\alpha}^l}),$ for all $x_2\in Q_{\alpha}^l.$ Therefore, a short reasoning shows that $\Omega_l=\bigcup_{\alpha} F_{\alpha,l} \times Q_{\alpha}^l\equiv \bigcup_{\alpha} S_{\alpha,l},$ where
\begin{equation*}
F_{\alpha,l}=\left\{x_1\in X\colon E_{l}f(x_1,z_{Q_{\alpha}^l})> s,\quad E_{l'}f(x_1,z_{Q})\leq s \textrm{ for } Q\supset Q_{\alpha}^l,\,Q \in \mathcal{Q}_{l'},\,l'<l\right\}.
\end{equation*}
From the continuity of $f$ it follows that the sets $S_{\alpha,l}$ are $\nu \otimes \mu$ measurable. Moreover, $\Omega=\bigcup_{\alpha,l}S_{\alpha,l},$ where the sum runs over $(\alpha,l)$ corresponding to all cubes and the sets $S_{\alpha,l}$ are pairwise disjoint. Hence, recalling \eqref{chap:CkHinf,sec:OA,eq:sumOmel}, we obtain (v).

Note that some of the sets $S_{\alpha,l}$ may be empty, as well as the sets $S_{\alpha,l}(x_1)=\{x_2\colon x\in S_{\alpha,l}\}.$ However, if for some $x_1\in X$ the set $S_{\alpha,l}(x_1)$ is not empty, then $S_{\alpha,l}(x_1)$ coincides with a cube $Q^{\alpha}_l(x_1).$ In fact the just presented construction may be phrased as follows: $x_1\in F_{\alpha,l}$ is and only if the cube $Q^{\alpha}_{l}$ has been chosen as one of the cubes for the Calder\'on-Zygmund decomposition of the function $f(x_1,\cdot).$

Since the set of pairs $(\alpha,l)$ is countable from now on we associate with each $j$ a pair $(\alpha,l)$ and a cube $Q_{\alpha,l}.$ Then $S_j= F_j \times Q_j$ are the sets from (iii). Next we set
\begin{align*}
&g(x)=f(x)\chi_{\Omega^c}+\sum_j\frac{1}{\mu(Q_{j})}\int_{Q_{j}}f(x_1,y_2)\,d\mu(y_2)\,\chi_{S_j}(x),\\
&b(x)=\sum_{j}b_{j}(x)=\sum_{j}\left(f(x)-\frac{1}{\mu(Q_{j})}\int_{Q_{j}}f(x_1,y_2)\,d\mu(y_2)\right)\chi_{S_j}(x),
\end{align*}
so that $f=g+\sum_{j}b_{j}.$ Also, since each set $S_j$ is uniquely associated with the dyadic cube $Q_j,$ the same holds for the functions $b_j.$ Let $x_1\in X$ be fixed. Then either $S_j(x_1)$ is or is not empty. In the second case $S_j(x_1)=Q_{j}(x_1),$ for some cube $Q_{j}(x_1).$ Moreover, the cubes $Q_j(x_1)$ are pairwise disjoint. Hence,
\begin{align*}
&\sum_j\int_{S_j(x_1)}\frac{1}{\mu(Q_{j})}\int_{Q_{j}}f(x_1,y_2)\,d\mu(y_2)\,\chi_{S_j}(x)\,d\mu(x_2)\\
&=\sum_{j\colon S_j(x_1)\neq \emptyset}\int_{Q_j(x_1)}\frac{1}{\mu(Q_{j}(x_1))}\int_{Q_{j}(x_1)}f(x_1,y_2)\,d\mu(y_2)\,\,d\mu(x_2)\leq \int_Y f(x)\,d\mu(x_2)
\end{align*}
and consequently, since $\chi_{S_j}(x)=\chi_{S_j(x_1)}(x_2),$ using Fubini's theorem we obtain
\begin{align*}
&\iint\limits_{X\times Y}\sum_j\frac{1}{\mu(Q_{j})}\int_{Q_{j}}f(x_1,y_2)\,d\mu(y_2)\,\chi_{S_j}(x)\, d\nu(x_1)\,d\mu(x_2)\\
&=\int_X \bigg(\sum_j\int_{S_j(x_1)}\frac{1}{\mu(Q_{j})}\int_{Q_{j}}f(x_1,y_2)\,d\mu(y_2)\,\chi_{S_j}(x)\,d\mu(x_2)\bigg)\,d\nu(x_1)\\
&\leq \int_{X\times Y}f(x_1,y_2)\,d\mu(y_2)\,d\nu(x_1)=\|f\|_{L^1}. \label{chap:CkHinf,sec:OA,eq:estint}
\end{align*}
From the above we obtain $\|g\|_{L^1}\leq 2\|f\|_{L^1}$ and $\sum_j \|b_j\|_{L^1}\leq  2\|f\|_{L^1},$ thus proving (i).

Now we pass to (ii). Since $|f(x)|\leq s,$ for $x\in \Omega^c$ and the sets $S_j$ are disjoint it suffices to show that,
\begin{equation}
\label{chap:CkHinf,sec:OA,eq:toshow}
\frac{1}{\mu(Q_{j})}\int_{Q_{j}}f(x_1,y_2)\,d\mu(y_2)\chi_{S_j}(x)\leq Cs,\qquad \textrm{ for } x_2\in S_{j}(x_1).
\end{equation} If $x_2\in S_j(x_1),$ then $S_{j}(x_1)=Q_{j}(x_1),$ for some $Q_{j}(x_1)\in\mathcal{Q}_l.$ Moreover, there exists $\tilde{Q}_{j}(x_1)\supset Q_{j}(x_1),$ with $\tilde{Q}_{j}(x_1)\in \mathcal{Q}_{l-1}.$ Then, since $x_2\in S_j(x_1),$  $$\frac{1}{\mu(\tilde{Q}_{j}(x_1))}\int_{\tilde{Q}_{j}(x_1)}f(x_1,y_2)\,d\mu(y_2)=E_{l'}f(x)\leq s.$$ Therefore, a standard argument, based on the doubling property of $\mu,$ gives
$$\frac{1}{\mu(\tilde{Q}_{j}(x_1))}\int_{\tilde{Q}_{j}(x_1)}f(x)\,d\mu(y_2)\leq \frac{C_{\mu}}{\mu(\tilde{Q}_{j}(x_1))}\int_{\tilde{Q}_{j}(x_1)}f(x_1,y_2)\,d\mu(y_2)\leq C_{\mu} s.$$ Hence, \eqref{chap:CkHinf,sec:OA,eq:toshow} and thus also (ii) is proved.

Observe that from the very definition of the sets $S_j$ we have
\begin{equation}
\label{chap:CkHinf,sec:OA,eq:comless}
\frac{1}{\mu(Q_{j})}\int_{Q_{j}}f(x_1,y_2)\,d\mu(y_2)\chi_{S_j}(x)> s,\qquad \textrm{ for } x_2\in S_{j}(x_1).
\end{equation}
Combining the above with \eqref{chap:CkHinf,sec:OA,eq:toshow} we obtain item (iv).

It remains to prove the property (iii). The inequality $\sum_{j}\mu(S_j(x_1))\leq s^{-1}\int_{Y}f(x)\,d\mu(x_2)$ follows from \eqref{chap:CkHinf,sec:OA,eq:comless}. If $S_j(x_1)=\emptyset$ then obviously, $b_j(x_1,\cdot)=0.$ If $S_j(x_1)$ is not empty, then $S_j(x_1)=Q_{j}(x_1),$ for some $j(x_1),$ so that $\supp b_j(x_1,\cdot)\subset Q_{j}(x_1).$ In either case $\int_{Y}b_j(x)\,d\mu(x_2)=\int_{S_j(x_1)}b_j(x)\,d\mu(x_2)=0.$
\end{proof}
Using Lemma \ref{lem:calzygdec} we now prove Theorem \ref{thm:interH1}. The proof follows the scheme from \cite[Theorem D, pp.\ 596, 635--637]{CW} by Coifman and Weiss.
\begin{proof}[Proof of Theorem \ref{thm:interH1}]
Fix $1<q<p$ and set $\mD^q(f)=(\mD(|f|^q))^{1/q},$ with $\mD$ given by \eqref{chap:CkHinf,sec:OA,eq:dyadmaxdef}. Then, since $\mD$ is bounded on $L^p$ and $1<q<p,$ the same is true for $\mD^q.$

Fix a continuous function $0\leq f\in L^p$ and let
\begin{equation}\label{chap:CkHinf,sec:OA,eq:levelset}\Theta^{s}=\{x\colon \mD^{q}(f)>s\}=\{x\colon \mD(f^q)>s^q\}.\end{equation} From item (v) of Lemma \ref{lem:calzygdec} it follows that $$\Theta^{s}=\bigcup_j F_j \times Q_j=\bigcup S_j,$$ where the sets $S_j$ satisfy properties (i)-(iv) from Lemma \ref{lem:calzygdec} with $s^{q}$ in place of $s$ and $f^q$ in place of $f.$ In particular
\begin{equation}
\label{chap:CkHinf,sec:OA,eq:meanprop}
\frac{1}{\mu(Q_j)}\int_{Q_j}f^q\,d\mu(x_2)\approx s^{q},\qquad x_1\in F_j.
\end{equation}

Decompose $f=g_{s}+b_{s}=g_{s}+\sum_j b_{j,s}$ with
\begin{align*}
g_{s}&=g=f(1-\chi_{\Theta^{s}})+\sum_j\frac{1}{\mu(Q_j)}\int_{Q_j}f(x)\,d\mu(x_2)\chi_{S_j}\\
b_{j,s}&=b_j=\left(f(x)-\frac{1}{\mu(Q_j)}\int_{Q_j}f(x_1,y_2)\,d\mu(y_2)\right)\chi_{S_j}.
\end{align*}
If we fix $x_1\in F_j,$ then because $|b_j|\leq |f|+\frac{1}{\mu(Q_j)}\int_{Q_j}f(x_1,y_2)\,d\mu(y_2)\chi_{S_j},$ using \eqref{chap:CkHinf,sec:OA,eq:meanprop} and H\"older's inequality, we obtain
\begin{equation}
\label{chap:CkHinf,sec:OA,eq:dyad-ball}
\begin{split}&\left(\int_{Q_j}|b_j|^{q}\,d\mu(x_2)\right)^{1/q}\\&\leq \left(\int_{Q_j}|f|^{q}\,d\mu(x_2)\right)^{1/q}+\left(\int_{Q_j}\left|\frac{1}{\mu(Q_j)}\int_{Q_j}f(x_1,y_2)\,d\mu(y_2)\right|^{q}\chi_{Q_j}\,d\mu(x_2)\right)^{1/q}\\
&\lesssim s\mu(Q_j)^{1/q}+\left(\int_{Q_j}\frac{1}{\mu(Q_j)}\int_{Q_j}f^q(x_1,y_2)\,d\mu(y_2)\chi_{Q_j}\,d\mu(x_2)\right)^{1/q}\lesssim s \mu(Q_j)^{1/q}.
\end{split}
\end{equation}
Let $\underline{B}(Q_j)$ be the ball included in $Q_j$ from \cite[Theorem 2.2 (2.8)]{dyad}, i.e.\ satisfying
$$\underline{B}(Q_j)\subset Q_j ,\qquad \mu(Q_j)\leq C_{\mu}\,\mu (\underline{B}(Q_j)).$$
Then, from \eqref{chap:CkHinf,sec:OA,eq:dyad-ball} it follows that
$$\left(\frac{1}{\mu(\underline{B}(Q_j))}\int_{\underline{B}(Q_j)}|b_j|^{q}\,d\mu(x_2)\right)^{1/q}\leq C_{\mu} s.$$
Consequently, for each fixed $x_1\in F_j,$ the function $$c_j(x_1,\cdot)=\frac{b_j(x_1,\cdot)}{C_{\mu}s \mu(\underline{B}(Q_j))}$$ is supported in $\underline{B}(Q_j)$ and satisfies \begin{equation}\label{chap:CkHinf,sec:OA,eq:atom}\|c_j(x_1,\cdot)\|_{L^q(Y,\frac{1}{\mu(\underline{B}(Q_j))}d\mu)}\leq \frac{1}{\mu(\underline{B}(Q_j))}.\end{equation}
The above inequality is also trivially satisfied if $x_1\not\in F_j,$ since then $c_j(x_1,\cdot)\equiv 0.$

From \eqref{chap:CkHinf,sec:OA,eq:atom} it follows that for each fixed $x_1\in \mathbb{R}^d$ the functions $c_j(x_1,\cdot)$ are $H^{1,q}(Y,\mu)$-atoms in the sense of Coifman-Weiss \cite[p.\ 591]{CW}, and thus $\|c_j\|_{H^{1,q}(Y,\mu)}=1$. Moreover, from the decomposition $b=\sum_j C_{\mu}s \mu(\underline{B}(Q_j))c_j$ we obtain
$$\|b(x_1,\cdot)\|_{H^{1,q}(Y,\mu)}\leq C_{\mu} s \sum_{j\colon x_1\in F_j}\mu(Q_j)=c_{\mu}s \sum_{j}\mu(S_j(x_1)).$$
Since the spaces $H^{1,q}(Y,\mu)$ and $H^{1}(Y,\mu)=H^{1,\infty}(Y,\mu)$ coincide, cf. \cite[Theorem A]{CW}, using Fubini's theorem and the disjointness of $S_j$ we obtain
\begin{equation}
\label{chap:CkHinf,sec:OA,eq:H1part}
\|b\|_{L^1_{\gamma}(H^1)}=\int_{\mathbb{R}^d}\|b(x_1,\cdot)\|_{H^{1}(Y,\mu)}\,d\gamma(x_1)\lesssim s \sum_{j}(\gamma\otimes \mu)(S_j)=C s (\gamma\otimes \mu)(\Theta^{s}).
\end{equation}

By the layer-cake formula we have
\begin{equation*}
p^{-1}\|Sf\|_{L^p}=\int_0^{\infty}s^{p-1}(\gamma\otimes \mu)(x\colon |Sf(x)|>s)\,ds,
\end{equation*}
and, consequently,
\begin{align*} &\|Sf\|_{L^p}\lesssim\int_0^{\infty}s^{p-1}(\gamma\otimes \mu)(x\colon |Sb_{s}(x)|>s/2)\,ds+
\int_0^{\infty}s^{p-1}(\gamma\otimes \mu)(x\colon |Sg_{s}(x)|>s/2)\,ds\\
&:= E_1+E_2.
\end{align*}
To estimate $E_1$ we use the weak type property of $S$ and \eqref{chap:CkHinf,sec:OA,eq:H1part}, obtaining
\begin{equation}
\label{chap:CkHinf,sec:OA,eq:thtaest}
E_1\lesssim\int_0^{\infty}s^{p-2}\|b_{s}\|_{L^1_{\gamma}(H^1)}\,ds\lesssim\int_0^{\infty}s^{p-1} (\gamma\otimes \mu)(\Theta^{s})\,ds=\|\mD^q(f)\|_p^p\lesssim\|f\|_{p}^p.
\end{equation}

Passing to $E_2,$ the layer-cake formula together with the $L^2$ boundedness of $S$ and Chebyshev's inequality produce
\begin{align*}
&p^{-1}E_2\lesssim C \int_0^{\infty}s^{p-3}\|g_{s}\|_2^2\,ds\\
&=\int_0^{\infty}s^{p-3}\int_{\Theta^{s}}|g_{s}|^2\,d\gamma\,d\mu\,ds+\int_0^{\infty}s^{p-3}\int_{(\Theta^{s})^c}|g_{s}|^2\,d\gamma\,d\mu\,ds:= E_{2,1}+E_{2,2}.
\end{align*}
From \eqref{chap:CkHinf,sec:OA,eq:levelset}, \eqref{chap:CkHinf,sec:OA,eq:meanprop} and the definition of $g_{s}$ we see that
$|g_{s}|\leq C s,$ and consequently,
\begin{align*}
E_{2,1}\leq \lesssim\int_0^{\infty}s^{p-1}(\gamma\otimes \mu)(\Theta^{s})\,ds.
\end{align*}
The above quantity has already been estimated, see \eqref{chap:CkHinf,sec:OA,eq:thtaest}. Now we focus on $E_{2,2}.$ Since $g_{s}=f$ outside of $\Theta^{s}$ and $f\leq\mD^q(f),$ using Fubini's theorem we have
\begin{align*}
E_{2,2}\lesssim\int_{\mathbb{R}^d\times Y}|f(x)|^2\int_{f}^{\infty}s^{p-3}\,ds \,d(\gamma\otimes \mu)\lesssim\int_{\mathbb{R}^d\times Y}|f(x)|^p\,d(\gamma\otimes \mu),
\end{align*}
thus obtaining the desired estimate for $E_2$ and hence, finishing the proof of Theorem \ref{thm:interH1}.
\end{proof}


    %
    %

   \begin{appendices}

    \label{chap:App}

    \chapter{}

    \section{Joint spectral measure and integration}
    \numberwithin{equation}{section}
    \label{chap:App,sec:joint}
    We provide the definitions and briefly list the properties of the joint spectral measure and integration. In particular, in Section \ref{chap:App,sec:joint,subsec:mst}, we state the multivariate spectral theorem. We finish this appendix by proving the strong $L^p,$ $1<p<\infty,$ continuity of $u\mapsto L^{iu}.$ The exposition in Sections \eqref{chap:App,sec:joint,subsec:resid} - \ref{chap:App,sec:joint,subsec:mst} is based on the neat and comprehensive presentation in Martini's thesis \cite[Appendix A.3]{Martini_Phd}, to which we refer the reader for more details and further references.
    \subsection{Resolutions of the identity}
    \label{chap:App,sec:joint,subsec:resid}
    Let $(\Hi,\langle \cdot , \cdot\rangle_{\Hi})$ be a Hilbert space and let $\Omega$ be a locally compact second-countable Hausdorff topological space (in fact we are interested in $\Omega=\mathbb{R}^d,$ $d\geq 1$). A \textit{resolution of the identity} of $\Hi$ on $\Omega$ is a mapping $E$ which sends Borel subsets $\omega$ of $\Omega$ to orthogonal projections $E(\omega)$ on $\Hi$ (called \textit{spectral projections}), and which satisfies the following properties
    \begin{itemize}
     \setlength{\itemsep}{-20pt}
    \item $E(\Omega)=I;$\\
    \item $E(\omega_1\cup \omega_2)=E(\omega_1)+E(\omega_2)$, whenever $\omega_1\cap \omega_2=\emptyset;$\\
    \item $E(\omega)=\lim_n E(\omega_n)$ (strong limit), whenever $\omega=\bigcup\omega_n$ (increasing union);\\
    \item $E(\omega_1\cap\omega_2)=E(\omega_1)E(\omega_2)=E(\omega_2)E(\omega_1).$
    \end{itemize}

    If $E$ is a resolution of the identity of $\Hi$ on $\Omega,$ then for every $x,y\in \Hi$ the equality $$E_{x,y}(\omega)=\langle E(\omega)x,y\rangle_{\Hi}$$ defines a complex-valued Borel measure $E_{x,y}$ on $\Omega,$ which satisfies $$|E_{x,y}(\Omega)|\leq \|x\|\|y\|.$$ If $x=y,$ then $E_{x,x}$ is a positive Borel measure such that
    $$E_{x,x}(\Omega)=\|x\|^2.$$

    The \textit{support} $\supp E$ of $E$ on $\Omega$ is defined by
    $$\supp E= \Omega \setminus \bigcup\{\omega\subseteq \Omega\colon \omega \textrm{ is open and } E(\omega)=0\}.$$
    Since we assume that $\Omega$ is second-countable we always have $E(\supp E)=I.$

    \subsection{Spectral integration}
    Let $B(\Omega)$ be the vector space of complex-valued bounded Borel functions on a locally compact second-countable Hausdorff topological space $\Omega$ equipped with the norm $\|f\|=\sup_{\la \in \Omega}|f(\la)|.$ The space $B(\Omega)$ is a $C^*$ algebra with pointwise multiplication. If $E$ is a resolution of the identity of $\Hi$ on $\Omega$, then there is a unique continuous linear map $B(\Omega)\to \mathcal{B}(\Hi)$ which, for every Borel set $\omega\subseteq \Omega,$ maps $\chi_{\omega}$ to $E(\omega).$ The image of a function $f\in B(\Omega)$ via this map is denoted by $$\int_{\Omega}f dE\quad\textrm{or}\quad E[f]$$ and is called the \textit{spectral integral} of the function of $f$ with respect to $E.$ Then, for every $f\in B(\Omega)$ we have
    $$\left\langle E[f]\, x,y\right\rangle_{\Hi}=\int_{\Omega}f dE_{x,y},\qquad x,y\in \Hi.$$

    The definition of $E[f]$ can be extended to general Borel functions $f$ on $\Omega,$ leading to (possibly unbounded) operators $E[f]=\int_{\Omega}f dE.$ These are given for $x\in \Dom(E[f]),$
    $$\Dom\left(E[f]\right)=\left\{x\in \Hi\colon \int_{\Omega}|f|^2\,dE_{x,x}<\infty\right\},$$
    by the condition
    $$\left\langle E[f] x,y\right\rangle_{\Hi}=\int_{\Omega}f dE_{x,y},\qquad y\in \Hi.$$
    In fact it can be shown that for every Borel function $f:\Omega\to \mathbb{C},$ the set $\Dom(E[f])$ is a dense subspace of $\Hi$ and
    $$\|E[f]x\|^2=\int_{\Omega}|f|^2dE_{x,x},\qquad x\in \Dom(E[f]).$$
    Moreover, if $f$ is real-valued, then $E[f]$ is self-adjoint on $\Dom(E[f]).$
    \subsection{Multivariate spectral theorem}
    \label{chap:App,sec:joint,subsec:mst}
    Consider a system $T=(T_1,\ldots,T_d)$ of self-adjoint operators on a Hilbert space $\Hi$ (the most interesting instance for us being $\Hi=L^2(X,\nu)$). By the spectral theorem each $T_r,$ $r=1,\ldots,d,$ has an associated unique resolution of the identity $E_{T_r}$ of $\Hi$ on $\mathbb{R}.$ For each $r=1,\ldots,d,$ $E_{T_r}$ is also called \textit{the spectral measure} or \textit{the spectral resolution} of $T_r.$

    We say that the operators $T_r,$ $r=1,\ldots,d,$ commute strongly if their spectral projections commute pairwise. In this case the \textit{multivariate spectral theorem} states that there exists a unique resolution of the identity $E$ on $\mathbb{R}^d$ such that
    $$E(\omega_1\times\cdots\times \omega_d)=E_{T_1}(\omega_1)\cdots E_{T_d}(\omega_d)$$
    and
    $$T_r=\int_{\mathbb{R}^d}\la_r\, dE(\la)=\int_{\mathbb{R}}\la_r\, dE_{T_r}(\la_r),\qquad \la=(\la_1,\ldots,\la_d),$$
    see \cite[Appendix A.4.4]{Martini_Phd} and \cite[Theorem 4.10 and Theorems 5.21, 5.23]{schmu:dgen}. The resolution of the identity $E(\la)$ is called the \textit{joint spectral measure} for the system $T.$ Additionally, we have
    \begin{equation}\label{chap:App,sec:joint,sub:mst,eq:support}\supp E\subseteq \supp E_1\times\cdots\times \supp E_d,\end{equation} see \cite[Proposition 5.24 (ii)]{schmu:dgen}, however it may happen that the inclusion is sharp. The set $\supp E$ is also denoted by $\sigma(L)$ and called the \textit{joint spectrum} of the system $T.$ Since $\supp E_r=\sigma(T_r),$ where $\sigma(T_r)$ denotes the spectrum of $T_r,$ $r=1,\ldots,d,$ the inclusion \eqref{chap:App,sec:joint,sub:mst,eq:support} can be also restated as
    $$\sigma(T)\subset \sigma(T_1)\times\cdots\times \sigma(T_d).$$

    For a Borel measurable complex-valued function $m$ on $\mathbb{R}^d$ we define the spectral multiplier for the system $T$ by
        $$
        m(T)=m(T_1,\ldots,T_d)=\int_{\mathbb{R}^d} m(\la)dE(\omega)=\int_{\sigma(L)}m(\la)dE(\la),
        $$
    on the domain
    $$\Dom(m(T))=\bigg\{x\in \Hi\colon \int_{\mathbb{R}^d}|m(\la)|^2dE_{x,x}(\la)<\infty\bigg\}.$$

    In particular, if $T=L=(L_1,\ldots,L_d),$ where $L_r,$ $r=1,\ldots,d,$ are self-adjoint, non-negative operators operators on $L^2(X,\nu)$ from Section \ref{chap:Intro,sec:Setting}, then $m(L)$ is given by
     \begin{equation}
     \label{chap:App,sec:joint,sub:mst,eq:m(L)def}
        m(L)=m(L_1,\ldots,L_d)=\int_{[0,\infty)^d} m(\la)dE(\la)=\int_{\sigma(L)}m(\la)dE(\la),
     \end{equation}
    on the domain
    \begin{equation} \label{chap:App,sec:joint,sub:mst,eq:m(L)dom}\Dom(m(L))=\bigg\{f\in L^2(X,\nu)\colon \int_{[0,\infty]^d}|m(\la)|^2dE_{f,f}(\la)<\infty\bigg\}.\end{equation}
    If we assume additionally \eqref{chap:Intro,sec:Setting,eq:noatomatzero}, then the first integral in \eqref{chap:App,sec:joint,sub:mst,eq:m(L)def} may be taken over $\Rdp$ instead of $[0,\infty)^d.$

    We finish this appendix by summarizing properties of the functional calculus given by \eqref{chap:App,sec:joint,sub:mst,eq:m(L)def}. Here $\overline{m}$ denotes the complex conjugate of a function $m.$
    \begin{pro}[cf.\ {\cite[Theorem 4.16]{schmu:dgen}}]
    \label{prop:funpro}
    For Borel measurable functions $m,m_1,m_2$ on $[0,\infty)^d$ and $z_1,z_2\in\mathbb{C}$ we have:
    \begin{enumerate}[i)]
    \item The adjoint of the (possibly unbounded) operator $m(L)$ is the operator $\overline{m}(L),$ consequently, if $m$ is real, then $m(L)$ is self-adjoint;
    \item If $m$ is non-negative, then $m(L)$ is non-negative with the spectral measure given by $E_{m(L)}(\omega)=\chi_{m^{-1}(\omega)}(L),$ where $\omega$ is a Borel subset of  $[0,\infty);$
    \item $m(L)$ is a closed normal operator on $L^2(X,\nu);$
    \item If $f\in\Dom(m(L)),$ then $\|m(L)f\|_{L^2(X,\nu)}=\int_{[0,\infty)^d}|m(\la)|^2 dE_{f,f}(\la);$
    \item If $m$ is bounded, then $\langle m(L)f ,g \rangle_{L^2(X,\nu)} =\int_{[0,\infty)^d} m(\la) dE_{f,g}(\la),$ for $f,g \in L^2(X,\nu);$
    \item The operator $(z_1m_1+z_2m_2)(L)$ is the closure of $z_1m_1(L)+z_2m_2(L);$
    \item If $m_1,m_2$ are non-negative, then $m_1(L)+m_2(L)=(m_1+m_2)(L);$
    \item The operator $(m_1m_2)(L)$ is the closure of $m_1(L)m_2(L),$ consequently, if $m_1,m_2$ are bounded, then the operators $m_1(L)$ and $m_2(L)$ commute.
    \end{enumerate}
    \end{pro}
    \subsection{Strong $L^p$ continuity of $u\mapsto L^{iu}$}
     \label{chap:App,sec:joint,subsec:strongmeasLiu}
    In this subsection we prove the strong $L^p,$ $1<p<\infty,$ continuity of $u\mapsto L^{iu},$ postulated on p. \pageref{chap:General,pag:strong}. Recall that the operator $L^{iu},$ $u\in\mathbb{R}^d,$ is given by \eqref{chap:App,sec:joint,sub:mst,eq:m(L)def} with $m(\la)=m_u(\la)=\la^{iu},$ $\la \in \Rdp.$

    We show that under \eqref{chap:Intro,sec:Setting,eq:contra} and \eqref{chap:Intro,sec:Setting,eq:noatomatzero}, for every $1<p<\infty$ and $u_0\in\mathbb{R}^d,$
    \begin{equation}
    \label{chap:App,sec:joint,subsec:strongmeasLiu,eq:toshow}
   \lim_{u\to u_0} \|L^{iu}f-L^{iu_0}f\|_{p}=0,\qquad f\in L^p.
    \end{equation}
    Since the operators $L^{iu_r},$ $r=1,\ldots,d,$ commute and $\|L_1^{iu_1}\cdots L_d^{iu_d}\|_{p\to p}$ is locally bounded, see \eqref{chap:Hinf,sec:genMarHinf,eq:Cowest} (cf.\ \cite[Corollary 1 and Theorem 3]{Hanonsemi}), it is enough to prove \eqref{chap:App,sec:joint,subsec:strongmeasLiu,eq:toshow} for $d=1$ and a single operator $L.$

    Let $d=1$ and fix $1<p<\infty.$ Since $\{L^{iu}\}_{u\in\mathbb{R}}$ is a group of bounded operators on $L^p,$ by \cite[Theorem 1.6]{EnNa1} the notions of weak and strong convergence coincide for $\{L^{iu}\}_{u\in\mathbb{R}}.$ Thus, it suffices to prove that $\lim_{u\to u_0}L^{iu}f=L^{iu_0}f,$ weakly in $L^p.$ The weak convergence is a consequence of the strong $L^2$ continuity of $L^{iu}$ (which follows from the spectral theorem), the local boundedness of $\|L^{iu}\|_{p\to p},$ and the density of $L^2\cap L^{p'},$ $1=1/p+1/p',$ in $L^{p'}.$

    \section{Tensor products of operators}
     \label{chap:App,sec:tens}
     Throughout this section the $\sigma$-finite measure space $(X,\nu)$ has a product form $(X_1\times\cdots \times X_d,\nu_1\otimes \cdots\otimes \nu_d).$ We study properties of tensor products of operators on $L^2(X,\nu)$ or, in some cases, $L^p(X,\nu),$ $1\leq p<\infty.$ For functions $f_r\in L^2(X_r,\nu_r),$ $r=1,\ldots,d,$ we always identify $f_1\otimes \cdots \otimes f_d$ with the pointwise product $f(x_1)\cdots f(x_d),$ $x=(x_1,\ldots,x_d)\in X.$ We say that $f$ is a tensor product function, whenever $f=f_1\otimes\cdots\otimes f_d.$ By simple functions we mean finite sums of characteristic functions of measurable sets.

     Assume that for each $r=1,\ldots,d,$ we have an operator $T_r$ which is linear and defined on a dense domain $\Dom(T_r)$ in the Hilbert space $L^2(X_r,\nu_r).$ Let $T_1\odot\cdots\odot T_d$ be the operator given by
     $(T_1\odot\cdots\odot T_d)(f)=\sum_{j}T_1(f_{j}^1)\cdots T_d(f_{j}^d)$ on
        \begin{align*}&\Dom(T_1\odot\cdots\odot T_d)=\bigg\{f\in L^2(X,\nu)\colon f=\sum_{j}f_{j}^1\otimes\cdots\otimes f_{j}^d\textrm{ (finite sum)},\, f_{j}^r\in \Dom(T_r)\bigg\}.\end{align*}
      Then, a short Hilbert space argument shows that $T_1\odot\cdots\odot T_d$ is a well defined linear operator on $\Dom(T_1\odot\cdots\odot T_d),$ cf.\ \cite[Proposition 7.20(i)]{schmu:dgen}. In the case when all but one of the operators $T_r$ are identities we also set
    \begin{equation}
     \label{chap:App,sec:tens,eq:Itensinit}
        T\odot I_{(r)}=I_{L^2(X_1,\nu_1)}\odot\cdots \odot I_{L^2(X_{r-1},\nu_{r-1})}\odot T \odot I_{L^2(X_{r+1},\nu_{r+1})}\cdots \odot I_{L^2(X_d,\nu_d)}.
        \end{equation}

    The following proposition contains \eqref{chap:Intro,sec:Notation,eq:tensnormeq} as a special case.
    \begin{pro}
    \label{prop:Fubtensprop}
    Fix $1\leq p<\infty,$ and assume that each $T_r$ is bounded on $L^2(X_r,\nu_r)$ and $L^p(X_r,\nu_r).$ Then $T_1\odot\cdots\odot T_d$ extends (uniquely) to a bounded operator $T_1\otimes\cdots\otimes T_d$ on both $L^2(X,\nu)$ and $L^p(X,\nu).$ Moreover, we have
     \begin{equation}
     \label{chap:App,sec:tens,eq:Fubtens}
        \prod_{r=1}^d\|T_r\|_{L^p(X_r,\nu_r)\to L^p(X_r,\nu_r)}=\|T_1\otimes\cdots\otimes T_d\|_{L^p(X,\nu)\to L^p(X,\nu)}.
        \end{equation}
    \end{pro}
    \begin{proof}
    In what follows $q$ denotes either $p$ or $2.$

    Observe that $T_1\odot\cdots\odot T_d$ acts on a function $f\in \Dom(T_1\odot\cdots\odot T_d)$ by applying each of the operators $T_r$ only to the $r$-th variable of $f.$ Thus, using Fubini's theorem several times we obtain for such an $f,$
    $$\|(T_1\odot\cdots\odot T_d)(f)\|_{L^q(X,\nu)\to L^q(X,\nu)}\leq \prod_{r=1}^d\|T_r\|_{L^q(X_r,\nu_r)\to L^q(X_r,\nu_r)}\|f\|_{L^q(X,\nu)}.$$
   Additionally, the above bound can be extended to all of $L^q(X,\nu).$ Indeed, since the operators $T_r$ are defined on $L^q(X_r,\nu_r),$ the set $\Dom(T_1\odot\cdots\odot T_d)$ is dense in $L^q(X,\nu),$ as it contains finite sums of tensor product functions of measurable sets.

    Thus, we have the desired extension $T_1\otimes\cdots\otimes T_d,$ and we proved that the right hand side of \eqref{chap:App,sec:tens,eq:Fubtens} is less than or equal to its left hand side. The opposite inequality follows once we apply $T_1\otimes\cdots\otimes T_d$ to $f=f_1\otimes\cdots\otimes f_d,$ such that $\|f_r\|_{L^q(X_r,\nu_r)}=1,$ for $r=1,\ldots,d,$ and $\|T_r(f_r)\|_{L^q(X_r,\nu_r)}$ is almost $\|T_r\|_{L^q(X_r,\nu_r)}.$

    \end{proof}

    Assume now that for each $r=1,\ldots,d,$ we have a non-negative self-adjoint operator $L_r$ defined on a dense domain $\Dom(L_r)\subseteq L^2(X_r,\nu_r).$ The next proposition allows us to lift the operators $L_r$ to the space $L^2(X,\nu)$.
     \begin{pro}
    \label{prop:tensclosself}
    The operators $L_r\odot I_{(r)},$ $r=1,\ldots,d,$ given by \eqref{chap:App,sec:tens,eq:Itensinit} are essentially self-adjoint on $\Dom(L_r\odot I_{(r)}).$ Moreover, their closures $L_r\otimes I_{(r)}$ are non-negative.
    \end{pro}
    \begin{proof}[Proof (sketch)]
    We proceed similarly as in the proof of \cite[Theorem 7.23]{schmu:dgen}. Throughout the proof we fix $r=1,\ldots,d.$

    The fact that the operator $L_r\odot I_{(r)}$ is symmetric follows from an easy computation. The proof of its essential self-adjointness is based on an application of Nelson's theorem, which says that a symmetric operator having dense set of analytic vectors is essentially self-adjoint, see \cite[Theorem 7.16]{schmu:dgen}. We omit the details here.

    Since $L_r\odot I_{(r)}$ is symmetric and essentially self-adjoint, its closure, which we denote by $L_r\otimes I_{(r)},$ is a self-adjoint operator defined on a dense domain in $L^2(X,\nu).$ It remains to prove that it is non-negative.

    Clearly, by a density argument it is enough to show that $L_r\odot I_{(r)}$ is non-negative on its domain. Take $f=\sum_j f_j^1\otimes\cdots \otimes f_j^d,$  $f\in \Dom(L_r\odot I_{(r)}),$ and note that $L_r$ is self adjoint on the finite-dimensional vector space $G$ spanned by $\{f_{j}^r\}.$  Hence, there exists an orthonormal basis $\{g_1^r,\ldots,g_s^r\}$ of $G$ such that $\langle L_r g_i^r, g_{i'}^r\rangle_{L^2(X_r,\nu_r)}=0,$ for $i\neq i',$  and
    \begin{align*}f=\sum_{i,j}\langle f_j^r, g_i^r\rangle_{L^2(X_r,\nu_r)}g_i^r\otimes f_j^{(r)}=\sum_{i}g_i^r\otimes\left(\sum_{j}\langle f_j^r, g_i^r\rangle_{L^2(X_r,\nu_r)}f_j^{(r)}\right):= \sum_{i}g_i^r\otimes h_i,\end{align*} with $f_j^{(r)}$ denoting the tensor product function of $f_j^1$ up to $f_j^d,$ excluding $f_j^r.$

    Thus, setting $X^{(r)}=X_1\times \cdots \times X_{r-1}\times X_{r+1}\times \cdots \times X_d$ and $\nu^{(r)}=\nu_1\otimes \cdots \otimes \nu_{r-1}\otimes \nu_{r+1}\otimes \cdots \otimes \nu_d,$ and using the fact that $L_r\geq 0,$ we obtain
    \begin{align*}&\langle (L_r\odot I_{(r)}) f,f\rangle_{L^2(X,\nu)}=\sum_{i,i'}\langle L_r g_i^r, g_{i'}^r\rangle_{L^2(X_r,\nu_r)}\langle h_i, h_{i'}\rangle_{L^2(X^{(r)},\nu^{(r)})}\\
    &=\sum_{i}\langle L_r g_i^r, g_{i}^r\rangle_{L^2(X_r,\nu_r)}\langle h_i, h_{i}\rangle_{L^2(X^{(r)},\nu^{(r)})}\geq 0.\end{align*}
    \end{proof}

 We now proceed somewhat differently by first considering the spectral measure $E_{L_r}$ of $L_r.$ Observe that the map $E_{L_r}\otimes I_{(r)},$ sending $\omega_r\subseteq [0,\infty)$ to $E_{L_r}(\omega_r)\otimes I_{(r)},$ is a resolution of the identity of $L^2(X,\nu)$ on $[0,\infty).$ This is easily seen when $f$ is a sum of tensor product functions, while for general $f$ we use a density argument. The following holds.
\begin{pro}
\label{prop:tensspectmeas}
For each $r=1,\ldots,d,$ the spectral measure $E_{L_r\otimes I_{(r)}}$ of the operator $L_r\otimes I_{(r)}$ coincides with $E_{L_r}\otimes I_{(r)}.$ Moreover, the operators $L_r\otimes I_{(r)}$ commute strongly, and, for a bounded function $m$ on $[0,\infty)$ we have \begin{equation} \label{chap:App,sec:tens,eq:equam} m(L_r\otimes I_{(r)})f=(m(L_r)\otimes I_{(r)})f,\qquad f\in L^2(X,\nu).\end{equation}
\end{pro}
\begin{proof}[Proof (sketch)]
We start with proving the first part of the proposition. Let $r=1,\ldots,d,$ be fixed. By uniqueness of the spectral representation of a self-adjoint operator it is enough to show that
\begin{equation}
\label{chap:App,sec:tens,eq:Lrmeasrepr}
L_r\otimes I_{(r)}=\int \la_r\, d(E_{L_r}\otimes I_{(r)})(\la_r):=(E_{L_r}\otimes I_{(r)})[\la_r].
\end{equation}
Approximating $\la_r$ by simple functions, it can be verified that \eqref{chap:App,sec:tens,eq:Lrmeasrepr} holds for tensor product functions, hence also on $\Dom(L_r\odot I_{(r)}).$ Note that both  $L_r\otimes I_{(r)}$ and $(E_{L_r}\otimes I_{(r)})[\la_r]$ are closed (because they are self-adjoint) and $\Dom(L_r\odot I_{(r)})$ is a core of $L_r\otimes I_{(r)}.$ Consequently, the operator $(E_{L_r}\otimes I_{(r)})[\la_r]$ is an extension $L_r\otimes I_{(r)}.$ Since both of these operators are self-adjoint also $L_r\otimes I_{(r)}$ is an extension of $(E_{L_r}\otimes I_{(r)})[\la_r].$ The proof of \eqref{chap:App,sec:tens,eq:Lrmeasrepr} is thus finished.

Observe that if $T_r$ are bounded operators on $L^2(X_r,\nu_r),$ $r=1,\ldots,d,$ then $T_r\otimes I_{(r)}$ commute. Hence, using the first part of the proposition we obtain the desired strong commutativity of the operators $L_r\otimes I_{(r)}.$

Now we focus on \eqref{chap:App,sec:tens,eq:equam}. Assume first that $f=\sum_{j}f_j^1\otimes \cdots\otimes f_j^d.$ Then, if $m$ is a characteristic function, \eqref{chap:App,sec:tens,eq:equam} follows directly form the first part of the proposition. Since every bounded function can be approximated by simple functions, passing to the limit we obtain \eqref{chap:App,sec:tens,eq:equam} for all bounded $m.$ Finally, for general $f$ we use a density argument.
\end{proof}

Using Proposition \ref{prop:tensspectmeas} we can easily prove the statements from the second to the last paragraph of Section \ref{chap:Intro,sec:Notation}.

Indeed, directly from the identity $E_{L_r\otimes I_{(r)}}=E_{L_r}\otimes I_{(r)}$ it follows that if $L_r$ satisfies the atomlessness condition \eqref{chap:Intro,sec:Setting,eq:noatomatzero}, then the same is true for $L_r\otimes I_{(r)}.$ Moreover, using Proposition \ref{prop:Fubtensprop} and \eqref{chap:App,sec:tens,eq:equam}, we see that if $L_r$ satisfies the contractivity condition \eqref{chap:Intro,sec:Setting,eq:contra} (with respect to $L^p(X_r,\nu_r)$), then $L_r\otimes I_{(r)}$ satisfies the same condition (with respect to $L^p(X,\nu)$).
    \section{The Mellin transform}
     \label{chap:App,sec:Mel}
    \subsection{Definition and properties}
    \label{chap:App,sec:Mel,sub:dap}
    Recall that for a function $m\in L^1(\Rdp,\frac{d\la}{\la})$ the $d$-dimensional Mellin transform of $m$ is defined by
    \begin{equation}
    \label{chap:App,sec:Mel,eq:def}
    \M(m)(u)=\int_{\Rdp}\la^{-iu}\,m(\la)\,\frac{d\la}{\la},\qquad u\in\mathbb{R}^d.
    \end{equation}

    The Mellin transform is precisely the Fourier transform on the product multiplicative group $(\Rdp,\cdot),$ which is in fact isomorphic with $(\mathbb{R}^d,+)$ via the map $$\Rdp \ni\la=(\la_1,\ldots,\la_d)\mapsto h(\la)=(\log \la_1,\ldots, \log \la_d)\in \mathbb{R}^d.$$ Therefore \eqref{chap:App,sec:Mel,eq:def} can be reexpressed as
    \begin{equation}
    \label{chap:App,sec:Mel,eq:connection}
    \M(m)(u)=\F(m\circ h^{-1})(u);
    \end{equation}
    here $\F$ denotes the classical Fourier Transform, which is defined for $f\in L^1(\mathbb{R}^d, dx)$ by $$\F(f)(\xi)=\int_{\mathbb{R}^d}f(x)e^{-i\langle x, \xi\rangle}\,dx,\qquad \xi \in \mathbb{R}^d.$$

    Formula \eqref{chap:App,sec:Mel,eq:connection} allows to transfer all properties of the Fourier transform to the new context; e.g.\ we have Plancherel's formula,
    \begin{equation}
    \label{chap:App,sec:Mel,eq:Planch}
    \int_{\Rdp}|m(\la)|^2\,\frac{d\la}{\la}=\frac{1}{(2\pi)^{d}}\int_{\mathbb{R}^d}|\M(m)(u)|^2\,du,
    \end{equation}
    and the inversion formula for the Mellin transform,
    \begin{equation}
    \label{chap:App,sec:Mel,eq:inv}
    m(\la)=\frac{1}{(2\pi)^{d}}\int_{\mathbb{R}^d}\M(m)(u)\la^{iu}\,du,\qquad \la=(\la_1,\ldots,\la_d)\in \Rdp.
    \end{equation}
    These formulae are valid pointwise if $m$ is nice enough, e.g.\ in \eqref{chap:App,sec:Mel,eq:Planch} it suffices to take $m\in L^1(\Rdp,\frac{d\la}{\la})\cap L^2(\Rdp,\frac{d\la}{\la}),$ while in \eqref{chap:App,sec:Mel,eq:inv} any $m$ such that $m\in L^1(\Rdp,\frac{d\la}{\la})$ and $\M(m)\in L^1(\mathbb{R}^d,du).$
 \subsection{Measurability of $u\mapsto \sup_{\TT\in \Rdp}|\M(m_{N,\TT})(u)|$}
    \label{chap:App,sec:Mel,sub:rem}
We shall now justify the first statement in Remark 1 after Theorem \ref{thm:gen}. Namely, we show that, if $m$ is a bounded Borel measurable function on $\Rdp,$ then, for each fixed $N\in \mathbb{N}^d,$ the function
\begin{equation}
\label{chap:App,sec:Mel,sub:rem,eq:just1}
\mathbb{R}^d \ni u\mapsto \sup_{\TT\in\Rdp}|\M(m_{N,\TT})(u)|
\end{equation} is Borel measurable, where, we recall $m_{N,\TT}(\la)=t^N\la^N \exp(-\langle t ,\la\rangle)m(\la),$ $\la \in \Rdp.$ In fact we prove that, for each fixed $N\in\mathbb{N}^d$ and $u\in\mathbb{R}^d,$ the function $\Rdp\ni \TT\mapsto h_{N,u}(\TT):=\M(m_{N,\TT})(u)$ is continuous on $\Rdp.$ Thus, the supremum in \eqref{chap:App,sec:Mel,sub:rem,eq:just1} can be taken over a countable set, and since, clearly, $\mathbb{R}^d\ni u \mapsto \M(m_{N,\TT})(u)$ is continuous, we obtain that $\mathbb{R}^d\ni u \mapsto \sup_{\TT\in\Rdp}|\M(m_{N,\TT})(u)|$ is measurable.

We focus on showing the continuity of $h_{N,u}(\TT).$ Since $t\to t_0\in\Rdp$ we may assume that $\underline{t}<t<\overline{t},$ for some $\underline{t},\overline{t}\in \Rdp,$ so that
$$|t^N\la^N\exp(-\langle t,\la\rangle)|\leq |(\overline{t})^{N}\exp(-\langle \underline{t}, \la\rangle)|,\qquad \la\in\Rdp.$$
Since $m$ is bounded, using the dominated convergence theorem we obtain
\begin{align*}
\lim_{t\to t_0}h_{N,u}(\TT)&=\lim_{t\to t_0}\int_{\Rdp}\la^{-iu}t^N\la^N\exp(-\langle t,\la\rangle)m(\la)\,\frac{d\la}{\la}\\
&=\int_{\Rdp}\la^{-iu}\lim_{t\to t_0}\big[t^N\la^N\exp(-\langle t,\la\rangle)\big]m(\la)\,\frac{d\la}{\la}\\
&=h_{N,u}(\TT_0),
\end{align*}
as desired.
\end{appendices}
    %
    %


\begin{thebibliography}{10}
        \thispagestyle{myheadings}
        \addcontentsline{toc}{chapter}{\bibname}


\bibitem{Al} D.\ Albrecht, \textit{Functional calculi of commuting unbounded operators}, PhD Thesis, Monash University,
Australia (1994).
\bibitem{AlFrMc} D.\ Albrecht, E.\ Franks, and A.\ McIntosh, \textit{Holomorphic functional calculi and sums of commuting operators}, Bull.\ Aust.\ Math.\ Soc.\ 58 (1998), 291--305.
\bibitem{AlexdisRiesz} G. Alexopoulos, \textit{Random walks on discrete groups of polynomial volume growth}, Ann. Probab. (2) 30 (2002), 723-801.
\bibitem{s1} G.\ Alexopoulos, \textit{Spectral multipliers on Lie groups of polynomial growth}, Proc.\ Amer.\ Math.\ Soc.\, (3) 120 (1994), 973--979.
\bibitem{s2} G.\ Alexopoulos, \textit{Spectral multipliers for Markov chains}, J.\ Math.\ Soc.\ Japan (3) 56 (2004), 833--852.
\bibitem{Sifi1} B.\ Amri, A.\ Gasmi, and M.\ Sifi, \textit{Linear and Bilinear Multiplier Operators for the Dunkl Transform}, Mediterr.\ J.\ Math.\ (4) 7 (2010), 503--521.
\bibitem{Sifi2} B.\ Amri, M.\ Sifi, \textit{Riesz transforms for Dunkl transform}, Ann.\ Math.\ Blaise Pascal (1) 19 (2012), 247--262.

\bibitem{BaRus1} N. Badr, E. Russ, \textit{Interpolation of Sobolev Spaces, Littlewood-Paley Inequalities and Riesz Transforms on Graphs}, Publ.\ Mat.\ (2) 53 (2009), 273-328.
\bibitem{B} J.\ J.\ Betancor, A.\ J.\ Castro, J.\ Curbelo, \textit{Spectral multipliers for multidimensional Bessel operators}, J.\ Fourier Anal.\ Appl.\ (5) 17 (2011), 932--975.
\bibitem{BX} W.\ R.\ Bloom, Z.\ Xu, {\em Fourier multipliers for Lp on Ch\'ebli-Trimeche hypergroups}, Proc.\ London Math.\ Soc.\ (3) 80 (2000), 643--664.

\bibitem{Cal-Zyg1} A.\ P.\ Calder\'on, A.\ Zygmund, \textit{On singular integrals}, Amer.\ J.\ Math.\ (2) 78 (1956), 289--309.
\bibitem{Carb-Drag} A.\ Carbonaro, O.\ Dragi\v{c}evi\'{c}, \textit{Functional calculus for generators of symmetric contraction semigroups}, preprint (2013), arXiv:1308.1338v1.
\bibitem{s3} M.\ Christ, C.\ D.\ Sogge, \textit{The weak type $L^1$ convergence of eigenfunction expansions for pseudodifferential operators}, Invent.\ Math.\ (2) 94 (1998), 421--453.
\bibitem{CiSt1}
    \'O.\ Ciaurri, K.\ Stempak, \textit{Conjugacy for Fourier-Bessel expansions}, Studia Math.\ 176 (2006), 215--247.
\bibitem{CRW} R.\ R.\ Coifman, R.\ Rochberg, and G.\ Weiss, \textit{Applications of transference: The $L^p$ version of von Neumann's
inequality and Littlewood-Paley-Stein theory}, Linear Spaces and Approximation, Birkh\"auser, Basel, 1978, 53--67.
\bibitem{CW} R.\ R.\ Coifman, G.\ Weiss, \textit{Extensions of Hardy spaces and their use in analysis}, Bull.\ Amer.\
Math.\ Soc.\ 83 (1977), 569--615.
\bibitem{CWtr} R.\ R.\ Coifman, G.\ Weiss, \textit{Transference Methods in Analysis}, CBMS regional conference series in mathematics, No.\ 31, A.M.S., Providence, R.I., 1977.
\bibitem{CouKerPetr1} T.\ Couhlon, G.\ Kerkyacharian, and P.\ Petrushev, \textit{Heat kernel generated frames in the setting
of Dirichlet spaces}, J.\ Fourier Anal.\ Appl.\ 18 (2012), 995--1066.
\bibitem{Hanonsemi} M.\ G.\ Cowling, \textit{Harmonic analysis on semigroups}, Ann.\ of Math. 117 (1983), 267-283.
\bibitem{CowDouMcYa} M.\ G.\ Cowling, I.\ Doust, A.\ McIntosh, and A.\ Yagi, \textit{Banach space operators with a bounded $H^{\infty}$ functional
calculus}, Journal of Aust.\ Math.\ Society (Series A) 60 (1996), 51--89.
\bibitem{Hanonultracon} M.\ G.\ Cowling, S.\ Meda, \textit{Harmonic Analysis and Ultracontractivity}, Trans.\ Amer.\ Math.\ Soc.\ (2) 340 (1993), 733--752.
\bibitem{s4} M.\ G.\ Cowling, A.\ Sikora, \textit{A spectral multiplier theorem on $SU(2)$}, Math.\ Z.\ (1) 238 (2001), 1--36.

\bibitem{Dai1} F.\ Dai, H.\ Wang, \textit{A transference theorem for the Dunkl
Transform and its Applications},  J.\ Funct.\ Anal.\ (12) 258 (2010), 4052--4074.
\bibitem{Duo} J.\ Duoandikoetxea, \textit{Fourier Analysis}, Amer.\ Math.\ Soc.\, Providence, RI, 2001.
\bibitem{s5} X.\ T.\ Duong, A.\ McIntosh, \textit{Singular integral operators with non-smooth kernels on irregular domains}, Rev.\ Mat.\ Iberoamericana (2) 15 (1999), 233--265.
\bibitem{s6} X.\ T.\ Duong, E.\ M.\ Ouhabaz, and A.\ Sikora, \textit{Plancherel type estimates and sharp spectral multipliers}, J.\ Funct.\ Anal.\ (2) 196 (2002), 443--485.
\bibitem{dphanhar} J.\ Dziuba\'{n}ski, M.\ Preisner, {\it Multiplier theorem for Hankel transform
on Hardy spaces}, Monat.\ Math.\ 159 (2010) 1--12.
\bibitem{dpw} J.\ Dziuba\'nski, M.\ Preisner, and B.\ Wróbel, \textit{Multivariate Hörmander-type multiplier theorem for the Hankel transform}, J.\ Fourier Anal.\ Appl.\
(2) 19 (2013), 417--437.

\bibitem{EnNa1} K-J.\ Engel, R.\ Nagel, \textit{A Short Course on Operator Semigroups}, Universitext, Springer, 2006.

\bibitem{SasForSc} L.\ Forzani, E.\ Sasso, and R.\ Scotto \textit{Dimensional free $L^p$ estimates for Riesz Transforms
associated to Polynomial expansions}, presentation (2013), http://calvino.polito.it/\textasciitilde anfunz/Alba/slides/Sasso.pdf.
\bibitem{Fr1} A.\ J.\ Fraser, \textit{Marcinkiewicz multipliers on the Heisenberg group}, PhD dissertation, Princeton University, 1997.
\bibitem{Fr2} A.\ J.\ Fraser, \textit{An (n+1)-fold Marcinkiewicz multiplier theorem on the Heisenberg group}, Bull.\ Austral.\ Math.\ Soc.\ 63 (2001), 35--58.
\bibitem{Fr3} A.\ J.\ Fraser, \textit{Convolution kernels of (n+1)-fold Marcinkiewicz multipliers on the Heisenberg
group}, Bull.\ Austral.\ Math.\ Soc.\ (3) 64 (2001), 353--376.


\bibitem{funccalOu} J.\ Garc\'ia-Cuerva, G.\ Mauceri, S.\ Meda, P.\ Sj\"{o}gren, and J.\ L.\ Torrea, \textit{Functional Calculus for the Ornstein Uhlenbeck Operator}, J.\ Funct.\ Anal.\ 183 (2001), 413--450.
\bibitem{high}   J.\ Garc\'ia-Cuerva, G.\ Mauceri, P.\ Sj\"{o}gren, and J.\ L.\ Torrea, \textit{Higher-Order Riesz Operators for the Ornstein-Uhlenbeck Semigroup}, Potential Anal.\ 10 (1999), 379--407.
\bibitem{laptype} J.\ Garc\'ia-Cuerva, G.\ Mauceri, P.\ Sj\"{o}gren, and J.\ L.\ Torrea, \textit{Spectral multipliers for the Ornstein-Uhlenbeck semigroup}, J.\ Anal.\ Math.\ 78 (1999), 281--305.
\bibitem{GS}  G.\ Garrig\'os, A.\ Seeger, \textit{Characterizations of Hankel multipliers}, Math.\ Ann.\ 342 (2008) 31--68.
\bibitem{GT} G.\ Gasper, W.\  Trebels, \textit{Necessary conditions for Hankel multipliers},
Indiana Univ.\ Math.\ J.\ 31 (1982) 403--414.
\bibitem{GosSt1} J.\ Gosselin, K.\ Stempak, \textit{A Weak-Type Estimate for Fourier-Bessel Multipliers}, Proc.\
Amer.\ Math.\ Soc.\ (3) 106 (1989), 655--662.
\bibitem{Gut1} C.\ E.\ Guti\'errez, \emph{On the Riesz transforms for Gaussian measures}, J.\ Funct.\ Anal.\ 120 (1994), 107--134.

\bibitem{Ha} D.\ T.\ Haimo, {\em Integral equations associated with Hankel convolutions}, Trans.\ Amer.\ Math.\ Soc.\ 116 (1965), 91--109.
\bibitem{HRST}
E.\ Harboure, L.\ de Rosa, C.\ Segovia and J.\ L.\ Torrea,
\emph{$L^p$-dimension free boundedness for Riesz transforms associated to Hermite functions},
Math.\ Ann.\ 328 (2004), 653--682.
\bibitem{s7} W.\ Hebisch, \textit{A multiplier theorem for Schr\"odinger operators}, Colloq.\ Math.\ (2) 60/61 (1990), 659--664.
\bibitem{hmm} W.\ Hebisch, G.\ Mauceri, and S.\ Meda, \textit{Holomorphy of spectral multipliers of the
Ornstein-Uhlenbeck operator}, J.\ Funct.\ Anal.\ 210 (2004), 101--124.
\bibitem{HebSC} W. Hebisch, L. Saloff-Coste, \textit{Gaussian estimates for Markov chains and random walks on groups},
Ann. Probab. (2) 21 (1993), 673--709.
\bibitem{basis} C.\ Heil, \textit{Basis theory primer}, Springer Basel Ag, 2010.
\bibitem{Horm1} L.\ H\"{o}rmander, \textit{Estimates for translation invariant operators in $L^p$ spaces}, Acta Math.\ (1) 104 (1960), 93--140.
\bibitem{Horm2} L.\ H\"{o}rmander, \textit{The Analysis of Linear Partial Differential Operators. I-IV.} Springer-Verlag, Berlin, 1983.
\bibitem{dyad} T.\ Hyt\"onen, A.\ Kairema, \textit{Systems of dyadic cubes in a doubling metric space}, Colloq.\ Math.\ 126 (2012), 1--33.

\bibitem{krantz} S.\ G.\ Krantz, \textit{Function Theory of Several Complex Variables}, AMS Chelsea Publishing, Providence, Rhode Island, 1992.

\bibitem{LanLanMer} F.\ Lancien, G.\ Lancien, and C.\ Le Merdy,  \textit{A joint functional calculus for sectorial operators
with commuting resolvents}, Proc.\ London Math.\ Soc.\ (2) 77 (1998), 387--414.
\bibitem{2} N.\ N.\ Lebedev, \textit{Special Functions and Their Applications}, Dover, New York, 1972.
\bibitem{Lu_Piqu1} F. Lust-Piquard, \textit{Dimension free estimates for discrete Riesz transforms on products of abelian groups}, Adv. Math. (2) 185 (2004), 289--327.
\bibitem{Lu_Piqu2} F. Lust-Piquard, \textit{Riesz transforms associated with the number operator on the Walsh system and the
fermions}, J. Funct. Anal. 155 (1998), 263-285.

\bibitem{Marorg} J.\ Marcinkiewicz, \textit{Sur les multiplicateurs des séries de Fourier}, Studia Math.\ (1) 8 (1939), 78--91.
\bibitem{Martini_Phd} A.\ Martini, \textit{Algebras of differential operators on Lie groups and spectral multipliers}, PhD Thesis, Scuola Normale Superiore, Pisa, Italy (2009), arXiv:1007.1119.
\bibitem{Martini_JFA} A.\ Martini, \textit{Spectral theory for commutative algebras of differential operators on Lie groups}, J.\ Funct.\ Anal.\ (9) 260 (2011), 2767--2814.
\bibitem{Martini_Annales}  A.\ Martini, \textit{Analysis of joint spectral multipliers on Lie groups of polynomial growth}, Ann.\ Inst.\ Fourier (4) 62 (2012), 1215--1263.
\bibitem{vectvalMM} G.\ Mauceri, S.\ Meda,  \textit{Vector-valued multipliers on stratified groups}, Rev.\ Mat.\ Iberoamericana (3-4) 6 (1990), 141--154.
\bibitem{sharp} G.\ Mauceri, S.\ Meda, and P.\ Sj\"{o}gren, \textit{Sharp estimates for the Ornstein-Uhlenbeck operator}, Ann.\ Sc.\ Norm.\ Super.\ Pisa Cl.\ Sci.\ 3 (2004), 447--480.
\bibitem{cit:Me} S.\ Meda, \textit{A general multiplier theorem}, Proc.\ Amer.\ Math.\ Soc.\ (3) 110 (1990), 639--647.
\bibitem{Me_gfun} S.\ Meda, \textit{On the Littlewood-Paley-Stein g-Function}, Trans.\ Amer.\ Math.\ Soc.\ (6) 347 (1995), 2201--2212.
\bibitem{meja} H.\ Mejjaoli, \textit{Littlewood-Paley decomposition associated with the Dunkl
operators and paraproduct operators}, J.\ Inequal.\ Pure and Appl.\ Math, (4) 9 (2008), Article 95, 25 pp.
\bibitem{Mey1} P.\ A.\ Meyer, \textit{Transformationse de Riesz pour les lois gaussiennes}, S\'eminaire de Proba. XVIII, Springer Lecture Notes 1059 (1984), 179--293.
\bibitem{Mikhlin} S.\ G.\ Mikhlin, \textit{Multidimensional singular integrals and integral equations}, Translated from
the Russian by W.\ J.\ A.\ Whyte. Pergamon Press, Oxford-New York-Paris 1965.
\bibitem{Muck2} B.\ Muckenhoupt, \textit{Hermite conjugate expansions}, Trans.\ Amer.\ Math.\ Soc.\ 139 (1969), 243--260.
\bibitem{Muck1} B.\ Muckenhoupt \textit{Conjugate functions for Laguerre expansions}, Trans.\ Amer.\ Math.\ Soc.\ 147 (1970), 403--418.
\bibitem{Muck3} B.\ Muckenhoupt, \textit{Transplantation theorems and multiplier theorems for Jacobi series}, Mem.\ Amer.\ Math.\ Soc.\ 356 (1986).
\bibitem{Mu:RiSt1} D.\ M\"uller, F.\ Ricci, and E.\ M.\ Stein, \textit{Marcinkiewicz multipliers and multi-parameter structure
on Heisenberg (-type) groups I.}, Invent.\ Math.\, (2) 119 (1995), 199--233.
\bibitem{Mu:RiSt2} D.\ M\"uller, F.\ Ricci, and E.\ M.\ Stein, \textit{Marcinkiewicz multipliers and multi-parameter structure
on Heisenberg (-type) groups. II.}, Math.\ Z.\, (2) 221 (1996), 267--291.

\bibitem{No1} A.\ Nowak, \textit{On Riesz transforms for Laguerre expansions},
J.\ Funct.\ Anal.\ 215 (2004), 217-240.
\bibitem{NSj0} A.\ Nowak, P.\ Sj\"{o}gren, \textit{Calder\'{o}n-Zygmund Operators Related to Jacobi
Expansions}, J.\ Fourier Anal.\ Appl.\ 18 (2012), 717--749.
\bibitem{NSj1} A.\ Nowak, P.\ Sj\"ogren, \textit{Sharp estimates of the Jacobi heat kernel}, to appear in Studia Math., arXiv:1111.3145.
\bibitem{NSj2} A.\ Nowak, P.\ Sj\"ogren,
\textit{Riesz transforms for Jacobi expansions}, J.\ Anal.\ Math.\ 104 (2008), 341--369.
\bibitem{NSRiesz}
A.\ Nowak, K.\ Stempak,
\emph{$L^2$-theory of Riesz transforms for orthogonal expansions},
J.\ Fourier Anal.\ Appl.\ 12 (2006), 675--711.
\bibitem{NoSt1} A.\ Nowak, K.\ Stempak, \textit{Relating transplantation and multipliers for Dunkl and Hankel transforms}, Math.\ Nach.\ 281 (2008), 1604--1611.
\bibitem{NoStIm} A.\ Nowak, K.\ Stempak, \textit{Imaginary powers of the Dunkl Harmonic Oscillator}, SIGMA (016) 5 (2009).
\bibitem{NScontr}
A.\ Nowak and K.\ Stempak,
\emph{On $L^p$-contractivity of Laguerre semigroups},
Illinois J.\ Math.\ (4) 56 (2012).
\bibitem{NSRieszLagHerm} A.\ Nowak, K.\ Stempak, \textit{Riesz transforms and conjugacy for Laguerre function expansions of Hermite type}, J.\ Funct.\ Anal.\ 244 (2007) 399--433.
\bibitem{NSz1} A.\ Nowak, T.\ Z.\ Szarek, \textit{Calder\'on-Zygmund operators related to Laguerre function expansions of convolution type}, J.\ Math.\ Anal.\ Appl.\ (2) 388 (2012), 801--816.

\bibitem{Pis1}
G.\ Pisier,
\emph{Riesz transforms: A simpler proof of P.A. Meyer's inequality},
Lecture Notes in Math., Vol.\ 1321, 485--501, 1988.
\bibitem{PrusSohr1} J.\ Pr\"uss, H.\ Sohr, \textit{On operators with bounded imaginary powers in Banach spaces}, Math.\ Z.\ 203 (1990), 429--452.

\bibitem{Rosbes} M.\ R\"{o}sler, \textit{Bessel-type signed hypergroups on R}, in: H.\ Heyer, A.\ Mukherjea (eds.), Probability
measures on groups and related structures XI, Proc.\ Oberwolfach 1994, World Scientific, Singapore, 1995, 292--304.
\bibitem{Rosherm} M.\ R\"{o}sler, \textit{Generalized Hermite polynomials and the heat equation for Dunkl operators}, Comm.\ Math.\ Phys.\ 192 (1998), 519--542, q-alg/9703006.
\bibitem{Rus1} E. Russ, \textit{Riesz transforms on graphs for $1<p<2$}, Math.\ Scand.\, 87 (2000), 133-160.

\bibitem{sas} E.\ Sasso, \textit{Functional calculus for the Laguerre operator}, Math.\ Zeit.\ 249 (2005), 683--711.
\bibitem{schmu:dgen} K.\ Schm\"udgen, \textit{Unbounded Self-adjoint Operators on Hilbert Space}, Grad.\ Texts in Math.\ 265 (2012).
\bibitem{Sik} A.\ Sikora, \textit{Multivariable spectral multipliers and analysis of quasielliptic operators on fractals}, Indiana Univ.\ Math.\ J.\ 58 (2009), 317--334.
\bibitem{SikWr} A.\ Sikora, J.\ Wright, \textit{Imaginary powers of Laplace operators}, Proc.\ Amer.\ Math.\ Soc.\ (6) 129 (2000), 1745--1754.
\bibitem{Sj1} P.\ Sj\"ogren, \textit{On the maximal function for the Mehler kernel}, in Harmonic Analysis,
Cortona 1982, Springer Lecture Notes in Mathematics 992 (1983), 73--82.
\bibitem{Sol1} F.\ Soltani,  \textit{$L^p$-Fourier multipliers for the Dunkl operator on the real line}, J.\ Funct.\ Anal.\ 209 (2004), 16--35.
\bibitem{topics} E.\ M.\ Stein, \textit{Topics in Harmonic Analysis Related to the Littlewood-Paley Theory}, Annals Math.\ Studies, Princeton Univ.\ Press, 1970.
\bibitem{singular} E.\ M.\ Stein, \textit{Singular Integrals and Differentiability Properties of Functions}, Princeton Univ.\ Press, Princeton, 1971.
\bibitem{SteinRiesz} E.\ M.\ Stein, \textit{Some results in harmonic analysis in $\mathbb R^n$, $n\to \infty$},
Bull.\ Amer.\ Math.\ Soc.\ 9 (1983), 71--73.
\bibitem{StTor} K.\ Stempak, J.\ L.\ Torrea, \textit{Poisson integrals and Riesz transforms for Hermite function expansions with weights}, J.\ Funct.\ Anal.\ 202 (2003), 443--472.
\bibitem{StWr} K.\ Stempak, B.\ Wr\'obel, \textit{Dimension free $L^p$ estimates for Riesz transforms associated with Laguerre
function expansions of Hermite type}, Taiwanese J.\ Math.\ (1) 17 (2013), 63--81.
\bibitem{szego} G.\ Szeg\"o, \textit{Orthogonal Polynomials. Rev. Ed.}, Am.\ Math.\ Soc.\ Providence, 1959.

\bibitem{Thanherm} S.\ Thangavelu, \textit{Multipliers for Hermite expansions}, Rev.\ Mat.\ Iberoam.\ 3 (1987), 1--24.
\bibitem{ThanRieszHerm} S.\ Thangavelu, \textit{On conjugate Poisson integrals and Riesz transforms for the Hermite expansions}, Colloq.\ Math.\ 64 (1993), 103--113.
\bibitem{convmax} S.\ Thangavelu, Y.\ Xu, \textit{Convolution operator and maximal function for the Dunkl transform}, J.\ Anal.\ Math.\ (1) 97 (2005), 25--55.
\bibitem{T}  E.\ C.\ Titchmarsh, {\em Introduction to the Theory of Fourier Integrals}, Clareoton Press, Oxford, 1937.

\bibitem{Uchi} A.\ Uchiyama, \textit{A Maximal Function Characterization of $H^p$ on the Space of Homogeneous Type}, Trans.\ Amer.\ Math.\ Soc.\ 262 (1980) no. 2, 579--592.


\bibitem{W} G.\ N.\ Watson, {\em A Treatise on the Theory of Bessel Functions}, Cambridge University Press, Cambridge, 1966.
\bibitem{cit:ja} B.\ Wr\'obel, \textit{Multivariate spectral multipliers for tensor product orthogonal expansions},  Monatsh.\ Math.\ 168 (2012), 124--149. \textit{Erratum}, Monatsh. Math. 169 (2013), 113--115.
\bibitem{jaOU} B.\ Wr\'obel, \textit{Multivariate spectral multipliers for systems of Ornstein-Uhlenbeck operators}, Studia Math.\ (1) 216 (2013), 47--67.
\bibitem{jaDun} B.\ Wr\'obel, \textit{Multivariate spectral multipliers for the Dunkl transform and the Dunkl harmonic oscillator}, to appear in Forum Math., avaliable online, DOI: 10.1515/forum-2013-0039.


\end{thebibliography}
\end{document}